\documentclass[12pt]{scrartcl}
\usepackage{a4}   
\usepackage{amsmath}    
\usepackage{paralist}
\usepackage{amssymb} 
\usepackage{amsfonts}  
\usepackage{mathrsfs}  
\usepackage{dsfont}
\usepackage{latexsym} 
\usepackage{xcolor}
\usepackage{bbm,exscale}
\definecolor{Myblue}{rgb}{0,0,0.6}  
\usepackage[colorlinks,citecolor=Myblue,linkcolor=Myblue,urlcolor=Myblue,pdfpagemode=None]{hyperref}
\usepackage{amsthm}
\usepackage{accents}
\usepackage[square,numbers,sort&compress]{natbib} 
\usepackage[all,cmtip]{xy}
\usepackage{ifthen} 
\usepackage{bbding}
\usepackage{stmaryrd}  
\usepackage{wasysym}
\usepackage{verbatim}
\usepackage{bbding} 
\usepackage{soul}  
\usepackage[yyyymmdd,hhmmss]{datetime}
\usepackage{tikz}
\usepackage{tikz-cd}
\usepackage{tikz-3dplot}
\usepackage{pgfplots}
	\pgfplotsset{width=7cm,compat=1.8}
	\usetikzlibrary{decorations.markings}
	\usetikzlibrary{patterns}

\tikzset{
    string/.style={draw=#1, postaction={decorate}, decoration={markings,mark=at position .51 with {\arrow[color=#1]{>}}}},
    costring/.style={draw=#1, postaction={decorate}, decoration={markings,mark=at position .51 with {\arrow[draw=#1]{<}}}},
    ostring/.style={draw=#1, postaction={decorate}, decoration={markings,mark=at position .47 with {\arrow[draw=#1]{>}}}},
    ustring/.style={draw=#1, postaction={decorate}, decoration={markings,mark=at position .56 with {\arrow[draw=#1]{>}}}},
    oostring/.style={draw=#1, postaction={decorate}, decoration={markings,mark=at position .43 with {\arrow[draw=#1]{>}}}},
    uustring/.style={draw=#1, postaction={decorate}, decoration={markings,mark=at position .59 with {\arrow[draw=#1]{>}}}},
    directed/.style={string=blue!50!black}, 
    odirected/.style={ostring=blue!50!black}, 
    udirected/.style={ustring=blue!50!black}, 
    oodirected/.style={oostring=blue!50!black}, 
    uudirected/.style={uustring=blue!50!black},     
    redirected/.style={costring= blue!50!black},
    redirectedgreen/.style={costring= green!50!black},
    directedgreen/.style={string= green!50!black},
    redirectedlightgreen/.style={costring= green!65!black},
    directedlightgreen/.style={string= green!65!black},
}

\tikzset{-dot-/.style={decoration={
  markings,
  mark=at position 0.5 with {\fill circle (1.875pt);}},postaction={decorate}}}

\tikzset{
	Fdot/.style={circle, draw, fill, inner sep=0pt}, 
	Odot/.style={circle, draw, inner sep=0.1pt, minimum size=0.1cm}
	}

\def\nicedashedcolourscheme{\shadedraw[top color=blue!22, bottom color=blue!22, draw=gray, dashed]}
\def\nicedashedpalecolourscheme{\shadedraw[top color=blue!12, bottom color=blue!12, draw=gray, dashed]}
\def\nicehalfpalecolourscheme{\shadedraw[top color=blue!22, bottom color=blue!22, draw=white]}
\def\nicenotpalecolourscheme{\shadedraw[top color=blue!32, bottom color=blue!32, draw=white]}
\def\nicecolourscheme{\shadedraw[top color=blue!22, bottom color=blue!22, draw=blue!22]}
\def\nicepalecolourscheme{\shadedraw[top color=blue!12, bottom color=blue!12, draw=white]}
\def\nicenocolourscheme{\shadedraw[top color=gray!2, bottom color=gray!25, draw=white]}
\def\nicereallynocolourscheme{\shadedraw[top color=white!2, bottom color=white!25, draw=white]}
\def\boringcolourscheme{\draw[fill=blue!20, dashed]}

\newcommand\tikzzbox[1]
{#1}

  \vfuzz \hfuzz
\makeatletter
\newcommand{\raisemath}[1]{\mathpalette{\raisem@th{#1}}}
\newcommand{\raisem@th}[3]{\raisebox{#1}{$#2#3$}}
\makeatother

\renewcommand{\H}{\mathcal H}
\newcommand{\A}{\mathcal{A}}
\newcommand{\orb}{\mathcal{A}}
\newcommand{\sgn}{\mathrm{sgn}}
\renewcommand{\leq}{\leqslant}
\renewcommand{\geq}{\geqslant}

\newcommand{\E}{\text{e}} 
\newcommand{\I}{\text{i}}
\newcommand{\B}{\mathcal{B}}
\newcommand{\Borb}{\B_{\mathrm{orb}}}
\newcommand{\Beq}{\B_{\mathrm{eq}}}
\newcommand{\C}{\mathds{C}}
\newcommand{\D}{\mathds{D}}
\newcommand{\Ee}{\mathds{E}}
\newcommand{\K}{\mathds{K}}
\newcommand{\M}{\mathds{M}}
\newcommand{\N}{\mathds{N}}
\newcommand{\Q}{\mathds{Q}}
\newcommand{\R}{\mathds{R}}
\newcommand{\Z}{\mathds{Z}}
\newcommand{\IP}{\mathds{P}}
\def\1{\ifmmode\mathrm{1\!l}\else\mbox{\(\mathrm{1\!l}\)}\fi}
\newcommand{\one}{\mathbbm{1}}
\newcommand{\be}{\begin{equation}}
\newcommand{\ee}{\end{equation}}
\newcommand{\bes}{\begin{equation*}}
\newcommand{\ees}{\end{equation*}}
\newcommand{\cc}[1] {\overline{#1}}
\newcommand{\inv}[0]{{-1}}

\newcommand{\inver}[0]{\times}
\newcommand{\chirel}[0]{\chi_{\mathrm{sym}}}
\newcommand{\Dec}[0]{\mathrm{Dec}}
\newcommand{\interior}[1]{%
  {\kern0pt#1}^{\mathrm{o}}%
}
\newcommand{\Strat}[0]{\mathrm{Strat}}
\newcommand{\Stratdef}[0]{\mathrm{Strat}^{\mathrm{def}}}

\newcommand{\sVir}{\mathsf{sVir}}
\newcommand{\MF}{\operatorname{MF}_{\operatorname{bi}}}
\newcommand{\MFW}{\operatorname{MF}_{\operatorname{bi}}(W)}
\newcommand{\MFR}{\operatorname{MF}^\text{R}_{\operatorname{bi}}}
\newcommand{\DG}{\operatorname{DG}_{\operatorname{bi}}}
\newcommand{\DGW}{\operatorname{DG}_{\text{bi}}(W)}
\newcommand{\DGR}{\operatorname{DG}^\text{R}_{\text{bi}}}
\newcommand{\id}{\text{id}}
\newcommand{\KMF}{K_{0}(\operatorname{MF}_{\text{bi}}}
\newcommand{\Ext}{\operatorname{Ext}}
\newcommand{\Hom}{\operatorname{Hom}}
\newcommand{\End}{\operatorname{End}}
\newcommand{\modu}{\operatorname{mod}}
\def\LG{\mathcal{LG}}
\def\LGgr{\mathcal{LG}^{\mathrm{gr}}}
\def\LGgrs{\mathcal{LG}'^{\mathrm{gr}}}
\def\LGgrso{\mathcal{LG}'^{\mathrm{gr}}_{\mathrm{orb}}}
\def\LGs{\mathcal{LG}'}
\def\LGsorb{\mathcal{LG}'_{\mathrm{orb}}}
\def\LGorb{\mathcal{LG}_{\mathrm{orb}}}
\def\LGeq{\mathcal{LG}_{\mathrm{eq}}}
\newcommand{\hmf}{\operatorname{hmf}}
\newcommand{\HMF}{\operatorname{HMF}}
\newcommand{\ev}{\operatorname{ev}}
\newcommand{\eval}{\operatorname{eval}}
\newcommand{\tev}{\widetilde{\operatorname{ev}}}
\newcommand{\coev}{\operatorname{coev}}
\newcommand{\tcoev}{\widetilde{\operatorname{coev}}}
\def\lra{\longrightarrow}
\def\lmt{\longmapsto}
\DeclareMathOperator{\tr}{tr}
\DeclareMathOperator{\str}{str}
\DeclareMathOperator{\Jac}{Jac}
\def\Re{R^{\operatorname{e}}}
\DeclareMathOperator{\Res}{Res}
\newcommand*{\longhookrightarrow}{\ensuremath{\lhook\joinrel\relbar\joinrel\rightarrow}}
\newcommand*{\twoheadlongrightarrow}{\ensuremath{\relbar\joinrel\twoheadrightarrow}}
\newcommand{\Ga}[1]{\Gamma_{\hspace{-2pt}#1}}
\newcommand{\HIA}{\Hom(I,A)}
\newcommand{\ZA}{Z_A(\Hom(I,A))}
\newcommand{\gZA}{\!Z_A^\gamma(\Hom(I,A))}
\newcommand{\gA}{{}_{\gamma_A}A}
\newcommand{\Aginv}{A_{\gamma_A^{-1}}}
\newcommand{\picc}{\pi^{(\text{c,c})}_A}
\newcommand{\pirr}{\pi^{\text{RR}}_A}
\newcommand{\tpirr}{{\widetilde\pi}^{\text{RR}}} 
\newcommand{\im}{\operatorname{im}}
\DeclareMathOperator*{\eq}{=}
\DeclareMathOperator*{\congscript}{\cong}
\newcommand{\specflow}{\mathcal U_{-\frac{1}{2},-\frac{1}{2}}}
\newcommand{\Hil}{\mathcal{H}}
\newcommand{\Hpcc}{\mathcal{H}'_{\text{(c,c)}}}
\newcommand{\Hprr}{\mathcal{H}'_{\text{RR}}}
\newcommand{\Hcc}{\mathcal{H}_{\text{(c,c)}}^A}
\newcommand{\Hrr}{\mathcal{H}_{\text{RR}}^A}
\newcommand{\HccnoA}{\mathcal{H}_{\text{(c,c)}}}
\newcommand{\HrrnoA}{\mathcal{H}_{\text{RR}}}
\newcommand{\Hrrbo}{\Hom_A(X,{}_{\gamma_A}X)}
\newcommand{\buboorb}{\beta_X^{\text{orb}}}
\newcommand{\bobuorb}{\beta^X_{\text{orb}}} 
\def\Cong{C_g}
\def\Centg{N_g}
\def\alphaKK{\alpha^{\{K\}}}
\def\alphagKg{\alpha_g^{K(g)}}
\newcommand{\AGC}{A_G^c}

\newcommand{\Bord}{\operatorname{Bord}}
\newcommand{\Bordor}{\operatorname{Bord}_{n}^{\mathrm{or}}}
\newcommand{\Borddef}{\operatorname{Bord}^{\mathrm{def}}}
\newcommand{\Borddefn}[1] {\operatorname{Bord}^{\mathrm{def}}_{#1}}
\newcommand{\Bordoc}[1] {\operatorname{Bord}^{\mathrm{oc}}_{#1}}
\newcommand{\Bords}{\operatorname{Bord}_{3}^{\APLstar}}
\newcommand{\Sphere}{\operatorname{Sphere}^{\mathrm{def}}}
\newcommand{\Nbh}{\operatorname{\mathcal{N}}}
\newcommand{\Cube}{\operatorname{Cube}}
\newcommand{\Cubed}{\operatorname{Cube}_{3}^{\mathrm{def}}(\mathds{D})}
\newcommand{\Fradj}{\operatorname{\mathcal C_{\mathds D}^{adj}}}
\newcommand{\Bordstrat}{\operatorname{Bord}^{\mathrm{strat}}}
\newcommand{\Bordstratn}[1]  {\operatorname{Bord}^{\mathrm{strat}}_{#1}}
\newcommand{\Borddlong}{\operatorname{Bord}_{3}^{\mathrm{def}}(D_3,D_2,D_1)_{s,t,f}}
\newcommand{\Bordd[1]}{\operatorname{Bord}_{#1}^{\mathrm{def}}(\mathds{D})}
\newcommand{\Strator}{\operatorname{Strat}_{n}^{\mathrm{or}}}
\newcommand{\G}{\mathcal{G}}
\newcommand{\tz}{\mathcal T_\zz}
\newcommand{\dz}{\mathcal D_\zz}
\newcommand{\tzp}{\mathcal T_{\mathcal Z'}}
\newcommand{\Data}{\mathds{D}}
\newcommand{\Obj}{\mathrm{Obj}}
\newcommand{\zz}{\mathcal{Z}}
\newcommand{\Fdp}{\operatorname{\mathcal F}_{\textrm{d}}^{\textrm{p}}}
\newcommand{\zzd}{\mathcal{Z}^{\mathrm{def}}}
\newcommand{\zztriv}{\mathcal{Z}^{\mathrm{triv}}} 
\newcommand{\zzAtriv}{\mathcal{Z}_{\mathrm{triv}}^{\Cat{A}}} 
\newcommand{\euc}{\odot}
\newcommand{\euctwo}{\odot_{\geq 2}}
\newcommand{\ieuc}{\iota^\euc}
\newcommand{\peuc}{\pi^\euc}

\newcommand{\zzTVA}{\mathcal{Z}^{TV}_{\Cat{A}}} 
\newcommand{\zzRTC}{\mathcal{Z}^{RT}_{\Cat{C}}}
\newcommand{\tzztriv}{\mathcal{T}_{{\mathcal{Z}^{triv}}}}

\newcommand{\tzGamma}{\mathcal{T}_{{\mathcal{Z}^{\Gamma}}}}
\newcommand{\unit}{\operatorname{\mathbf{1}}}  
\newcommand{\bigslant}[2]{{\raisebox{.0em}{$#1$}\left/\raisebox{-.2em}{$#2$}\right.}}
\newcommand{\Cubedp}{\operatorname{Cube}_{3}^{\mathrm{def}}(\mathds{D}')}
\newcommand{\Borddp}{\operatorname{Bord}_{3}^{\mathrm{def}}(\mathds{D}')}
\newcommand{\Vect}{\operatorname{Vect}}
\newcommand{\Vectk}{\operatorname{Vect}_\Bbbk}

\newcommand{\eps}{\varepsilon}
\newcommand{\al}{\alpha}
\newcommand{\alb}{\overline{\alpha}}
\newcommand{\T}{\mathcal{T}}
\newcommand{\Ss}{\mathcal{S}}
\newcommand{\X}{\mathcal{X}}
\newcommand{\Y}{\mathcal{Y}}
\newcommand{\sta}{\boxempty}
\newcommand{\fus}{\otimes}
\newcommand{\sd}{^{\star}}
\newcommand{\dagg}{^{\dagger}}
\newcommand{\hash}{^{\#}}
\newcommand{\Set}{\mathrm{Set}}
\newcommand{\Ball}{\mathrm{Ball}}
\newcommand{\Sph}{\mathsf{Sph}}
\newcommand{\fork}{\pitchfork }
\newcommand{\Cat}[1]         {\operatorname{\mathcal{#1}}}
\newcommand{\Catpre}[1]         {\operatorname{\mathcal{#1}^{\mathrm{pre}}}}

\newcommand{\dX}{{}^\dagger\hspace{-1.8pt}X}
\newcommand{\dA}{{}^\dagger\hspace{-1.8pt}A}
\newcommand{\dsX}{{}^\dagger\hspace{-1.8pt}\mathcal{X}}
\newcommand{\deqX}{{}^\star\hspace{-1.8pt}X} 
\newcommand{\dseqX}{{}^\star\hspace{-1.8pt}\mathcal{X}} 
\newcommand{\dY}{{}^\dagger\hspace{-0.3pt}Y}
\newcommand{\dphi}{{}^\dagger\hspace{-0.9pt}\phi}
\newcommand{\dPhi}{{}^\dagger\hspace{-0.9pt}\Phi}

\newcommand{\opp}             {{\mathrm{op}}} 

\newcommand{\Alg}        {\operatorname{\mathsf{Alg}}}
\newcommand{\Frob}        {\operatorname{\mathsf{Frob}}}
\newcommand{\Lincat}        {\operatorname{\mathsf{Cat}^{ses}}}
\newcommand{\Deftqft}{\operatorname{TQFT}^{\mathrm{def}}}

\newcommand{\KVvect}   {\operatorname{KV-2Vect}}
\newcommand{\CYvect}   {\operatorname{CY-2Vect}}

\newcommand{\evx}[1]   {\operatorname{\mathsf{ev}}_{#1}}
\newcommand{\coevx}[1]   {\operatorname{\mathsf{coev}}_{#1}}
\newcommand{\evp}[1]   {\operatorname{\mathsf{ev}}_{#1}^{\prime}}
\newcommand{\coevp}[1]   {\operatorname{\mathsf{coev}}_{#1}^{\prime}}

\newcommand{\evc}[1]   {\operatorname{\mathsf{c-ev}}_{#1}}
\newcommand{\coevc}[1]   {\operatorname{\mathsf{c-coev}}_{#1}}
\newcommand{\evpc}[1]   {\operatorname{\mathsf{c-ev}}_{#1}^{\prime}}
\newcommand{\coevpc}[1]   {\operatorname{\mathsf{c-coev}}_{#1}^{\prime}}

\def\la               {{\rm l.a.}}
\def\ra               {{\rm r.a.}}
\def\rra              {{\rm r.r.a.}}
\def\lla               {{\rm l.l.a.}}
\def\Fun              {{\mathsf{Fun}}}
\def\Funbilin              {{\mathsf{Fun}^{\mathsf{bilin}}}}

\usepackage{color}

\newcommand{\rrr}[1]{{\color{red}{#1}}}
\newcommand{\rrR}[1]{{\color{red3}{#1}}}
\newcommand{\bbb}[1]{{\color{blue}{#1}}}
\definecolor{DarkViolet} {rgb}{0.580392,0.000000,0.827450}
\newcommand{\vio}[1]{{\color{DarkViolet}{#1}}}
\newcommand{\green}[1]{{\color{green}{#1}}}

\newcommand{\ques} [1] {\marginpar\textbf{qu}\textbf{{#1} ?}} 
\newcommand{\chn}[1]{\marginpar\textbf{changed}\textbf{{#1}}}
\newcommand{\note} [1] {\marginpar\textbf{note}\textbf{{#1}}}

\newcommand{\todo} [1] {\marginpar\textbf{todo} {\bbb{#1}  } }
\newcommand{\plan} [1] {\marginpar\textbf{Plan} { \bbb{#1}  } }
\newcommand{\GS} [1] {\marginpar{\small \vio{modified\\ by GS:} {}}{ \vio{#1} }} 
\newcommand{\GSC}[1] {\marginpar{\small \vio{comment \\ by GS:} {}}{ ~\\ \it \vio{#1} \\ }} 
\newcommand{\GSQ}[1] {\marginpar{\small \vio{question\\ by GS:} {}}{ ~\\ \it \vio{#1} \\ }} 
\newcommand{\IR} [1] {\marginpar{\small \rrr{modified\\ by IR:} {}}{ \rrr{#1} }} 
\newcommand{\IRC}[1] {\marginpar{\small \rrr{comment \\ by IR:} {}}{ ~\\ \it \rrr{#1} \\ }} 
\newcommand{\IRQ}[1] {\marginpar{\small \rrr{question\\ by IR:} {}}{ ~\\ \it \rrr{#1} \\ }} 
\newcommand{\NCMod} [1] {\marginpar{\small \vio{modified\\ by NC:} {}}{ \vio{#1} }} 
\newcommand{\NCC}[1] {\marginpar{\small \vio{comment \\ by NC:} {}}{ ~\\ \it \vio{#1} \\ }} 
\newcommand{\NCQ}[1] {\marginpar{\small \vio{question\\ by NC:} {}}{ ~\\ \it \vio{#1} \\ }} 

\newcommand\nxt{\noindent\raisebox{.08em}{\rule{.44em}{.44em}}\hspace{.4em}}
\newcommand\arxiv[2]      {\href{http://arXiv.org/abs/#1}{#2}}
\newcommand\doi[2]        {\href{http://dx.doi.org/#1}{#2}}
\newcommand\httpurl[2]    {\href{http://#1}{#2}}

\renewcommand{\labelenumi}{(\roman{enumi})}

\allowdisplaybreaks

\deffootnote[1em]{1em}{1em}{\textsuperscript{\thefootnotemark}}

\theoremstyle{definition}
\newtheorem{definition}{Definition}
\newtheorem{proposition}[definition]{Proposition}

\newtheorem{theoremdefinition}[definition]{Theorem and Definition}
\newtheorem{lemma}[definition]{Lemma}

\newtheorem{remark}[definition]{Remark}

\newtheorem{example}[definition]{Example}
\newtheorem{construction}[definition]{Construction}

\numberwithin{equation}{section}
\numberwithin{definition}{section}
\numberwithin{figure}{section}

\newcommand\void[1]{}

\begin{document}

\title{Orbifolds of \\ $\boldsymbol{n}$-dimensional defect TQFTs}

\author{%
\!\!\!\!\!\!\!Nils Carqueville$^*$ \quad
Ingo Runkel$^\#$ \quad
Gregor Schaumann$^*$%
\\[0.5cm]
   \hspace{-1.8cm}  \normalsize{\texttt{\href{mailto:nils.carqueville@univie.ac.at}{nils.carqueville@univie.ac.at}}} \\  %
   \hspace{-1.8cm}  \normalsize{\texttt{\href{mailto:ingo.runkel@uni-hamburg.de}{ingo.runkel@uni-hamburg.de}}} \\
   \hspace{-1.8cm}  \normalsize{\texttt{\href{mailto:gregor.schaumann@univie.ac.at}{gregor.schaumann@univie.ac.at}}}\\[0.1cm]
   \hspace{-1.2cm} {\normalsize\slshape $^*$Fakult\"at f\"ur Mathematik, Universit\"at Wien, Austria}\\[-0.1cm]
   \hspace{-1.2cm} {\normalsize\slshape $^\#$Fachbereich Mathematik, Universit\"{a}t Hamburg, Germany}\\[-0.1cm]
}

\date{}
\maketitle

\begin{abstract}
We introduce the notion of $n$-dimensional topological quantum field theory (TQFT) with defects as a symmetric monoidal functor on decorated stratified bordisms of dimension~$n$. 
The familiar closed or open-closed TQFTs are special cases of defect TQFTs, and for $n=2$ and $n=3$ our general definition recovers what had previously been studied in the literature. 

Our main construction is that of ``generalised orbifolds'' for any $n$-dimensional defect TQFT: 
Given a defect TQFT~$\zz$, one obtains a new TQFT $\zz_\A$ by decorating the Poincar\'{e} duals of triangulated bordisms with certain algebraic data~$\A$ and then evaluating with~$\zz$. 
The orbifold datum~$\A$ is constrained by demanding invariance under $n$-dimensional Pachner moves. 
This procedure generalises both state sum models and gauging of finite symmetry groups, for any~$n$. 
After developing the general theory, we focus on the case $n=3$.
\end{abstract}

\newpage

\tableofcontents

\section{Introduction}

Topological quantum field theory (TQFT) interrelates topology, higher categories and mathematical physics, with prominent ties also to algebra and geometry. 
In recent years the subject has enjoyed further attention via the study of $\infty$-categories and topological phases of matter. 

The simplest ``closed'' case of a TQFT is a symmetric monoidal functor $\Bord_n \to \Vectk$ for some positive integer~$n$, where $\Bord_n$ has closed smooth oriented $(n-1)$-dimensional manifolds as objects and diffeomorphism classes of bordisms as morphisms. 
To work towards classification results, to gain structural insight into their inner workings and interrelations, and to address questions in neighbouring fields such as knot theory or theoretical physics, it is desirable to enrich the basic notion of closed TQFTs. 
Two natural directions to do so are to ``extend'' and to add ``defects''. 

Extended TQFTs in $n$ dimensions are higher functors on higher bordism categories that involve $n$- and $(n-1)$-dimensional manifolds as in $\Bord_n$, but also manifolds with corners of lower dimension. 
Accordingly, for an extended TQFT also a higher symmetric monoidal target category has to be specified. 
On the other hand, the idea behind defect TQFTs $\zz\colon \Bordd[n] \to \Vectk$ is to concentrate all the enriched structure into the source category $\Bordd[n]$ while staying within the realm of ordinary symmetric monoidal 1-categories. 
As will be explained in detail in Section~\ref{sec:defectTQFTs}, objects and morphisms in $\Bordd[n]$ are $(n-1)$-dimensional oriented manifolds and their $n$-dimensional bordism classes, 
respectively, which come with a decomposition (or stratification) into $j$-dimensional submanifolds (or $j$-strata) for all $j\in \{ 0,1,\dots,n \}$. 
Furthermore, all strata are decorated by ``defect data'' from a given collection~$\D$. 
These decorated strata can be thought of as ``extended physical observables'', of which boundary conditions are familiar special cases. 
One may argue that this approach is closer to the original motivation to axiomatise structures from physics, laying emphasis on the combinatorics of defect conditions. 
An $n$-category may then subsequently be extracted from the functor~$\zz$ as an algebraic invariant (this has been worked out in detail for $n\in \{ 1,2,3 \}$, cf.~\cite{dkr1107.0495,CMS}).

\medskip

Defects can be used to describe symmetries of a closed $n$-dimensional TQFT $\zz_{\text{cl}}$: 
Given a representation~$\rho$ of a finite group~$G$ on $\zz_{\text{cl}}$, it may be possible to find a family $\{ \rho(g) \}_{g\in G}$ of $(n-1)$-dimensional defects whose ``fusion'' reproduces the product in $G$. 
The ``superposition'' $\A_G = \bigoplus_{g\in G} \rho(g)$ together with coherently chosen lower-dimensional defects (such a choice may be obstructed and is typically not unique) can be used to construct the $G$-orbifold theory $\zz_{\text{cl}}^G$.  
This is done by ``averaging'' over the symmetry -- a procedure implemented by evaluating~$\zz$ on a network of $\A_G$-defects Poincar\'{e} dual to a triangulation of the bordisms. 
One should think of $\A_G$ and its lower-dimensional defects as an algebraic structure encoding the symmetry $G$ and the necessary extra information needed to orbifold, or, in other words, gauge, that symmetry.

In 2~dimensions this is well-understood, with the algebraic structure of~$\A_G$ turning out to be that of a $\Delta$-separable symmetric Frobenius algebra \cite{frs0403157,cr1210.6363}.
More recently $G$-crossed fusion categories have been studied in connection with $G$-actions on 3-dimensional TQFTs, cf.~\cite{eno0909.3140, CGPW, BBCW1410.4540}. 
A study of orbifolds of $n$-dimensional TQFTs by finite groups~$G$ using the language of $G$-equivariant TQFTs rather than that of defects has been carried out in \cite{sw1705.05171}.

Algebraic structures~$\A$ like those above, describing defects in all dimensions (but which need not necessarily arise from group actions), are also at the centre of state sum constructions: 
special symmetric Frobenius algebras over a field~$\Bbbk$ in 2~dimensions \cite{bp9205031, FHK}, and spherical fusion categories over~$\Bbbk$ in 3~dimensions \cite{TurVir, BarWes}. 
Here again one decorates the Poincar\'{e} dual of a triangulation (or a similar type of decomposition) with the data of~$\A$ and then uses the algebraic structure to evaluate. 
Thus it is natural to generalise the notion of orbifold to encompass any system of defects which can be used to decorate suitable decompositions of bordisms. For the resulting orbifold to be well-defined one must impose the condition that the evaluation with the TQFT functor is invariant under the specific choice of decomposition. 
Then in particular one finds that ``state sum models are orbifolds of the trivial theory''. 

The idea of generalised orbifolds in 2~dimensions was first put forward in \cite{ffrs0909.5013} in the context of rational conformal field theory. 
Later in \cite{cr1210.6363} it was adapted to 2-dimensional defect TQFTs and developed into a Morita-type theory of $\Delta$-separable symmetric Frobenius algebras and their bimodules internal to any pivotal 2-category. 
Out of this emerged a notion of ``orbifold equivalence'' which has since found applications in algebra, geometry and mathematical physics, see e.\,g.\ \cite{CRCR, DanielIlkaNils, BCP2, cqv2015}. 

A good way to think of orbifolds in this generalised sense is the slogan 
\begin{quote}
``Carry out a state sum construction with defects internal to a given $n$-dimensional quantum field theory.'' 
\end{quote}
The present paper provides a way to make this idea precise and productive for topological QFTs of any dimension~$n$. 
It originally grew out of a desire to extend the theory of generalised orbifolds from 2~to 3~dimensions, with applications to topological invariants, tensor categories, quantum computation, and topological phases of matter in mind. 

\medskip
In this paper, our first contribution is to give a detailed definition of oriented defect TQFTs as symmetric monoidal functors\footnote{%
In fact the definition of defect TQFT works for any symmetric monoidal target category, and given the existence of certain limits, also all other constructions in this paper go through, cf.~Remark~\ref{rem:symmontarget}.} 
\be
\zz\colon \Bordd[n] \lra \Vectk
\ee
in any dimension~$n$.
The definition of the defect bordism category $\Bordd[n]$ (Definitions~\ref{def:Bordndef} and~\ref{def:d-dim-defect-data}) is inductive in~$n$ and controls the ways in which we allow defects, i.\,e.~decorated strata in stratified bordisms, to meet in terms of iterated cones and cylinders over basic configurations. 
In particular, we will describe how the ``defect data''~$\D$ include label sets~$D_j$ to decorate $j$-strata for all $j \in \{ 0,1,\dots,n \}$. 
We then prove that $n$-dimension defect TQFTs themselves form a symmetric monoidal category $\Deftqft_n$, see Proposition~\ref{prop:TQFTissymmon}. 

In Sections~\ref{subsec:point-defects-from} and~\ref{subsec:completewrtpointinsertions} we develop the theory further and in particular consider decorations for point defects. 
We show that without loss of generality in a defect TQFT one can identify labels for 0-strata with states associated to the surrounding sphere that are invariant under certain automorphisms of the sphere, and we prove that such states form an algebra.
In Section~\ref{subsec:completewrtpointinsertions} we exponentiate invertible point defects with the Euler characteristic of the surrounding stratum to construct 
a refinement of a given defect TQFT, which we call the ``Euler completion''. 
The details of these subsections are however not required to understand Sections~\ref{sec:orbifolds} and~\ref{sec:highercatfor}.

The other central notion which we introduce (in Definition~\ref{def:orbidatan}) is that of an ``orbifold datum''~$\A$ for a given defect TQFT $\zz\colon \Bordd[n] \to \Vectk$. 
An orbifold datum consists of a collection of labels $\A_j \in D_j$ to decorate the $j$-strata of Poincar\'{e} duals of triangulated bordisms for $j \in \{1,\dots,n\}$, as well as two labels $\orb_0^+, \orb_0^- \in D_0$ for 0-strata. 
The defining constraints on the orbifold datum~$\A$ are precisely that evaluation with~$\zz$ of $\A$-decorated bordisms is invariant under the choice of triangulation, i.\,e.~under oriented versions of Pachner moves (which we recall in Section~\ref{subsec:oritriaPachmo}). 
Our main result (Theorem~\ref{thmdef:orbifoldtheory}) is then the construction of the ``$\A$-orbifold theory'' $\zz_\A$: 

\medskip

\noindent
\textbf{Theorem. }
For every defect TQFT $\zz\colon \Bordd[n] \to \Vectk$ and every orbifold datum~$\A$ for~$\zz$, there is an associated closed TQFT
\be
\zz_\A \colon \Bord_n \lra \Vectk \, . 
\ee

\medskip

After a brief discussion in Section~\ref{subsec:2dimorbis} of how the above-mentioned $\Delta$-separable symmetric Frobenius algebras for $n=2$ fit into the general picture, we finally concentrate on the 3-dimensional case in Section~\ref{subsec:3dimorbis}.
For $n=3$, the number of defining constraints for an orbifold datum (i.\,e.~the number of independent oriented Pachner moves) is already 30. 
We will introduce (Definition~\ref{def:orbidata}) the notion of a ``special orbifold datum'' which involves only ten 
conditions that can be checked more easily in practice. 
Consistent with the general (expected) relation between state sum models and orbifolds, a special orbifold datum may be characterised as ``spherical fusion categories internal to Gray categories with duals'' (which need not have units), as we discuss in Section~\ref{sec:highercatfor}. 

We note that the main constructions in Sections~\ref{sec:defectTQFTs} and~\ref{sec:orbifolds} are exclusively in terms of ordinary symmetric monoidal categories and their functors. 
It is only in Section~\ref{sec:highercatfor} that we discuss higher categorical formulations -- which may prove worthwhile independently of their TQFT origin. 

\medskip

The only examples in the present paper are the invertible ``Euler defect TQFTs'' (Example~\ref{example:Eulerx}).
In the follow-up work \cite{CRS3} we will study examples of orbifolds of 3-dimensional defect TQFTs, namely Turaev-Viro models as orbifolds of the trivial TQFT, and two different types of $\Z_2$-orbifold of the Reshetikhin-Turaev theory for 
$\widehat{\mathfrak{sl}}(2)_k$.
	Reshetikhin-Turaev TQFT with defects is developed in the companion paper~\cite{CRS2}.

\subsubsection*{Acknowledgements} 

The work of N.\,C.~is partially supported by a grant from the Simons Foundation. 
N.\,C.~and G.\,S.~are partially supported by the stand-alone project P\,27513-N27 of the Austrian Science Fund. 
The authors acknowledge support by the Research Training Group 1670 of the German Research Foundation.

\section{Defect TQFTs}
\label{sec:defectTQFTs}

An $n$-dimensional defect TQFT is a symmetric monoidal functor from a category $\Bordd[n]$ of decorated stratified $n$-dimensional bordisms to $\Vectk$. 
To explain the details, we start in Section~\ref{subsec:strabord} by describing our conventions for 
stratifications. 
In Section~\ref{subsec:defectbords} we define $\Bordd[n]$ in two steps: first without decorations (Definition~\ref{def:Bordndef}) and then fully with decorations by ``defect data''~$\D$ (Definition~\ref{def:d-dim-defect-data}). 
Then in Section~\ref{subsec:ndTQFTs} we define $n$-dimensional defect TQFTs, show how they are themselves the objects of a symmetric monoidal category, and discuss the example of ``Euler theories''. 
Finally in Sections~\ref{subsec:point-defects-from} and~\ref{subsec:completewrtpointinsertions} 
we define two completions of a given defect TQFT:
passing to a maximal set of labels for point defects, and internalising the example of Euler theories.

\subsection{Stratified bordisms}
\label{subsec:strabord}

Defects are geometrically realised as a system of submanifolds of a bordism.
To specify the types of allowed neighbourhoods for the submanifolds  in a defect TQFT, we rely on the concept of stratifications, as discussed in \cite{CMS}
	or in \cite{Pflaum} (where the term ``decomposed space'' is used).

%

By an $n$-dimensional \textsl{stratified manifold}\footnote{%
In the literature it is common to denote by a stratified manifold a topological space that is a 
manifold outside of certain singularities that are located on the strata and satisfy certain regularity conditions. 
These conditions  specify the type of allowed singularity and the allowed adjacency conditions for all strata. 
The stratified manifolds we will consider are however such that the total space is still a topological manifold. 

So the regularity conditions are only needed to specify adjacency conditions, and we note that what we call a ``stratified manifold'' in the present paper is only the non-singular subclass of what is often called by that name in the literature. 
}
we mean an $n$-dimensional 
	topological 
manifold~$M$ (with empty boundary), together with a filtration $M = F_n \supset F_{n-1} \supset \dots \supset F_0 \supset F_{-1} = \emptyset$ such that for all $i\in \{ 0,1,\dots,n \}$, $M_i:= F_i\setminus F_{i-1}$ is a smooth $i$-dimensional submanifold of~$M$. 
The connected components $M_i^\alpha$ of $M_i$ are called \textsl{$i$-strata}, and we ask the sets of $i$-strata to be finite. 
Furthermore, for all strata $M_i^\alpha, M_j^\beta$ with $i < j$ and $M_i^\alpha \cap \overline M_j^\beta \neq \emptyset$, we demand $M_i^\alpha \subset \overline M_j^\beta$. 

Every triangulated manifold (cf.~Section~\ref{subsec:oritriaPachmo}) is a stratified manifold. 
Here, $F_{i}$ consist of the union of all (closed) $j$-simplices for $j\leqslant i$, and $M_{i}$ consists of the open $i$-simplices.  
The Poincar\'{e} dual of a triangulated manifold is typically not triangulated, but it continues to be stratified. 

We are interested in \textsl{oriented stratified $n$-manifolds $M$}, where the manifold~$M$ is itself oriented. Furthermore, each $i$-stratum with $i<n$ is equipped with a choice of orientation, while the orientation of each $n$-stratum is taken to be the one induced from $M$.
	\textsl{Morphisms between oriented stratified manifolds} are continuous maps that restrict to smooth orientation preserving maps between the strata, see e.\,g.~\cite[Ch.\,1]{Pflaum}. 

A \textsl{stratified manifold $M$ with boundary $\partial M$} is an $n$-manifold $M$ with boundary together with a filtration $M = F_n \supset F_{n-1} \supset \dots \supset F_0 \supset F_{-1} = \emptyset$ as above, such that the interior of $M$ is a stratified manifold and each $i$-stratum is a submanifold whose boundary is empty or lies in the boundary of $M$, intersecting $\partial M$ transversely. 
We also view $\partial M$ a stratified manifold with the stratification induced from~$M$. 
In case $M$ is oriented we equip $\partial M$ and all its strata with the induced orientations from the strata of $M$.
For more details we refer to \cite[Sect.\,2.1]{CMS}

\medskip

\noindent
\textbf{Convention. }
{}From now on we make the assumption that, unless specified otherwise, all manifolds and all stratified manifolds we consider are compact 
and oriented, possibly with boundary, 
	while all maps between them are continuous and their restrictions to strata are smooth and orientation preserving.
 
\medskip
 
As an example we consider a stratified 3-manifold~$M$ where a small part of the stratification is shown here:
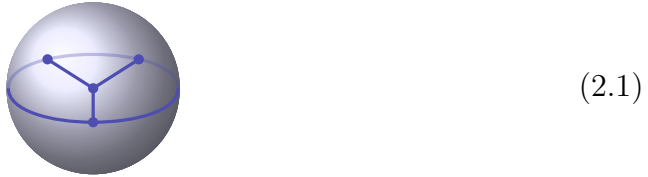
\be
\label{eq:3ballstrat}
\tikzzbox{\begin{tikzpicture}[very thick,scale=1.2,color=gray!60!blue, baseline=-0.1cm]
\fill[ball color=blue!10!white] (0,0) circle (0.95 cm);
\coordinate (v1) at (0.5,0.32);
\coordinate (v2) at (-0.5,0.32);
\coordinate (v3) at (0,-0.375);
\coordinate (v4) at (0,0);
\fill (v4) circle (1.6pt) node[gray!60!blue, opacity=0.6] {};
\fill (v1) circle (1.6pt) node[gray!60!blue, opacity=0.1] {};
\fill (v2) circle (1.6pt) node[gray!60!blue, opacity=0.1] {};
\fill (v3) circle (1.6pt) node[gray!60!blue, opacity=0.1] {};
\draw[color=gray!60!blue, opacity=0.32]	(-0.93,0) .. controls +(0,0.5) and +(0,0.5) .. (0.93,0);
\draw[color=gray!60!blue]	(-0.93,0) .. controls +(0,-0.5) and +(0,-0.5) .. (0.93,0);
\draw[color=gray!60!blue, very thick] (v1) -- (v4);
\draw[color=gray!60!blue, very thick] (v2) -- (v4);
\draw[color=gray!60!blue, very thick] (v3) -- (v4);
\end{tikzpicture}}
\ee
In this case there are two 3-strata (the two half-balls), five 2-strata (the two hemispheres and the three triangle-shaped regions), six 1-strata, and four 0-strata shown. 
Other $2$-strata that are not shown might for  example meet the equator from outside.
Keeping the outer stratification in $M$ we obtain another stratification of $M$ by exchanging the ball~\eqref{eq:3ballstrat} with just one 2-stratum, three 1-strata, and three 0-strata: 
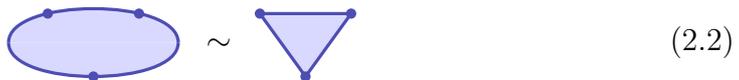
\be
\label{eq:triastrat}
\tikzzbox{\begin{tikzpicture}[very thick,scale=1.2,color=gray!60!blue, baseline=-0.1cm]
\coordinate (v1) at (0.5,0.32);
\coordinate (v2) at (-0.5,0.32);
\coordinate (v3) at (0,-0.375);
\fill [blue!15,opacity=1] (-0.93,0) .. controls +(0,0.5) and +(0,0.5) .. (0.93,0);
\fill [blue!15,opacity=1] (-0.93,0) .. controls +(0,-0.5) and +(0,-0.5) .. (0.93,0);
\fill (v1) circle (1.6pt) node[gray!60!blue, opacity=0.1] {};
\fill (v2) circle (1.6pt) node[gray!60!blue, opacity=0.1] {};
\fill (v3) circle (1.6pt) node[gray!60!blue, opacity=0.1] {};
\draw[color=gray!60!blue]	(-0.93,0) .. controls +(0,0.5) and +(0,0.5) .. (0.93,0);
\draw[color=gray!60!blue]	(-0.93,0) .. controls +(0,-0.5) and +(0,-0.5) .. (0.93,0);
\end{tikzpicture}}
\;\;
\sim 
\;
\tikzzbox{\begin{tikzpicture}[very thick,scale=1.2,color=gray!60!blue, baseline=-0.1cm]
\coordinate (v1) at (0.5,0.32);
\coordinate (v2) at (-0.5,0.32);
\coordinate (v3) at (0,-0.375);
\fill [blue!15,opacity=1] (v1) -- (v2) -- (v3);
\fill (v1) circle (1.6pt) node[gray!60!blue, opacity=0.1] {};
\fill (v2) circle (1.6pt) node[gray!60!blue, opacity=0.1] {};
\fill (v3) circle (1.6pt) node[gray!60!blue, opacity=0.1] {};
\draw[color=gray!60!blue, very thick] (v1) -- (v2);
\draw[color=gray!60!blue, very thick] (v2) -- (v3);
\draw[color=gray!60!blue, very thick] (v1) -- (v3);
\end{tikzpicture}}
\ee
Later in Section~\ref{subsubsec:sodareod}, local changes of stratifications which exchange~\eqref{eq:3ballstrat} and~\eqref{eq:triastrat} (called ``bubble moves'') will be important for us. 

\medskip

For any $n\in \Z_+$ one may now consider the category $\Bordstrat_n$ of \textsl{stratified bordisms}. 
An object in $\Bordstrat_n$ is a closed $(n-1)$-dimensional stratified manifold~$\Sigma$ 
	(oriented and compact
by the above convention).
A morphism $\Sigma \to \Sigma'$ is an equivalence class of $n$-dimensional stratified manifolds~$M$ with 
	parametrised
boundary.
	We describe the parametrisation and the equivalence relation in turn.
	The boundary parametrisation is a germ (in $\eps>0$) of orientation preserving embeddings $\iota \colon (\Sigma \times [0,\eps)) \sqcup (\Sigma' \times (-\eps,0]) \to M$ (which are in particular continuous maps whose restrictions to strata are smooth and orientation preserving by our convention) which map $(\Sigma \times \{0\})^{\text{rev}} \sqcup (\Sigma' \times \{0\})$ onto $\partial M$; we will use $\iota$ to denote both the germ and a representative map. The operation $(-)^{\text{rev}}$ reverses the orientation of all strata. Two such stratified manifolds $(M,\iota)$ and $(\widetilde M, \widetilde\iota)$ are equivalent if there is an isomorphism $f\colon M\to \widetilde M$ such that $f \circ \iota  = \widetilde\iota$ on $(\Sigma \times [0,\delta)) \sqcup (\Sigma' \times (-\delta,0])$ for some small enough $\delta>0$.
Composition 
	in $\Bordstrat_n$ is defined by choosing representatives, gluing along the boundary parametrisation, and then taking the bordism class. 
Thus composition 
is well-defined, associative and unital. 
Further standard arguments show that $\Bordstrat_n$ has a natural structure of a symmetric monoidal category.

\subsection{Defect bordisms}
\label{subsec:defectbords} 
 
\subsubsection{Undecorated defect bordisms}
\label{subsubsec:undecodefbord}

We think of defects in a defect TQFT as ``combinatorial'' in nature, meaning that for a defect confined to a given stratum~$Y$ it is only the distribution of strata in the immediate surroundings of~$Y$ that matters. 
	We 
will impose regularity conditions by requiring the existence of certain local neighbourhoods (detailed below) for all strata. 

Like in the case of a smooth manifold we shall first specify the type of open stratified manifolds that we want to allow as local neighbourhoods and then define a defect bordism via an atlas of charts taking values in these neighbourhoods. 
The definition is inductive on the dimension~$n$, and stratified manifolds whose underlying manifolds are standard $n$-spheres feature prominently in the induction step: 

\begin{definition}
\label{def:Bordndef}
For all $n \in \Z_+$ we define three related structures recursively:
\begin{itemize}
\item 
the sets $\Nbh_n$ of \textsl{local neighbourhoods for $n$-dimensional defect bordisms},
\item 
the symmetric monoidal \textsl{category of $n$-dimensional defect bordisms $\Borddefn{n}$}, 
\item 
the \textsl{set of defect $n$-spheres} $\Sphere_{n}$.
\end{itemize}
For $n=1$ the above data is fixed as follows.
\begin{itemize}
\item 
The set $\Nbh_1$ consists of three open stratified $1$-manifolds: 
the oriented interval $(-1,1)$, oriented from $-1$ to $1$, and the interval $(-1,1)$ with the same orientation and 
an oriented 0-stratum at~$0$ which is either oriented~$+$ or~$-$: 
\be
\label{eq:N1elements}
\Nbh_1 = 
\Big\{
\tikzzbox{\begin{tikzpicture}[thick,scale=2.321,color=blue!50!black, baseline=0cm, >=stealth, rotate=-90
]
\coordinate (v') at (0,0);
\coordinate (v) at (0,1);
\draw[string=gray!60!blue, very thick] (v') -- (v);
\end{tikzpicture}}
\, , \quad
\tikzzbox{\begin{tikzpicture}[thick,scale=2.321,color=blue!50!black, baseline=0cm, >=stealth, rotate=-90
]
\coordinate (v') at (0,0);
\coordinate (d) at (0,0.5);
\coordinate (v) at (0,1);
\draw[string=gray!60!blue, very thick] (v') -- (d);
\draw[string=gray!60!blue, very thick] (d) -- (v);
\fill[color=red!90!black] (d) circle (0.9pt) node[above] (0up) {{\footnotesize $+$}};
\end{tikzpicture}}
\, , \quad
\tikzzbox{\begin{tikzpicture}[thick,scale=2.321,color=blue!50!black, baseline=0cm, >=stealth, rotate=-90
]
\coordinate (v') at (0,0);
\coordinate (d) at (0,0.5);
\coordinate (v) at (0,1);
\draw[string=gray!60!blue, very thick] (v') -- (d);
\draw[string=gray!60!blue, very thick] (d) -- (v);
\fill[color=red!90!black] (d) circle (0.9pt) node[above] (0up) {{\footnotesize $-$}};
\end{tikzpicture}}
\Big\}
\, .
\ee
\item 
$\Borddefn{1} := \Bordstrat_1$.
\item
$
\Sphere_{1} := 
\Big\{
	S  \,\Big|\, 
	\text{$S$ is a stratified $1$-manifold with underlying manifold $S^1$}
\Big\}
$.
\end{itemize}
Now assume that $\Nbh_{n}$, $\Borddefn{n}$ and $\Sphere_n$ are defined for a given $n\geqslant 1$.
\begin{itemize}
\item 
The set  $\Nbh_{n+1}$ consists of all open stratified $(n+1)$-manifolds of two types. 

One type is $X \times (-1,1)$ for $X \in \Nbh_n$ with orientations induced from the orientation of $X$ taken together with the standard orientation of $(-1,1)$. 
Each $j$-stratum in $X$ produces a $(j+1)$-stratum in  $X \times (-1,1)$.\footnote{%
The convention for the orientation is chosen such that the orientation on $X$ is the induced boundary orientation of $X \times [-1,1)$.} 
For $p \in X$ a $0$-stratum with orientation~$-$, the orientation of the 1-stratum $\{p\} \times (-1,1)$ is obtained by reversing the orientation of the interval.

The other type of element in $\Nbh_{n+1}$ is an open cone 
\be
\label{eq:ConeSigma}
C(\Sigma)=(\Sigma \times [0,1))/(\Sigma\times\{0\}) \, , 
\ee
where~$\Sigma$ is an element in $\Sphere_{n}$. 
We identify $C(\Sigma)$ with the open $(n+1)$-ball $B^{n+1} \subset \R^{n+1}$, with $0 \in \R^{n+1}$ as cone point.
It has a natural structure of a stratified manifold with the cone point as $0$-stratum and with $j$-strata~$Y$ on $\Sigma$ inducing $(j+1)$-strata $Y \times (0,1)$ of $C(\Sigma)$.
Each $\Sigma \in \Sphere_{n}$ yields two elements in $\Nbh_{n+1}$, one for either orientation of the $0$-stratum at the cone point of~\eqref{eq:ConeSigma}.
\item 
The symmetric monoidal category $\Borddefn{n+1}$ has as objects closed stratified $n$-manifolds equipped with a compatible system of charts mapping an open neighbourhood of each point 
	isomorphically 
(as oriented stratified manifolds) to an element in $\Nbh_{n}$.
The morphisms of $\Borddefn{n+1}$ are those morphisms of $\Bordstratn{n+1}$ whose representatives have open neighbourhoods around interior points that are isomorphic 
	to 
an element of $\Nbh_{n+1}$.
\item 
The set $\Sphere_{n+1}$ consists of all those stratified $(n+1)$-manifolds~$S$ whose underlying manifold is the standard $(n+1)$-sphere $S^{n+1}$, 
	and such that the isomorphism 
class of~$S$ defines a morphism $[S] \colon \emptyset \to \emptyset$ in $\Borddefn{n+1}$. 
\end{itemize}
\end{definition}

Note that as a consequence of the recursive definition we have the following equivalent ways to think about defect $n$-spheres: 
\begin{align}
\Sphere_{n} &= 
\Big\{
	S  \,\Big|\, 
	\text{$S$ is a stratified $n$-manifold with	underlying}
\nonumber \\[-.8em]
& \hspace{3em}
	 \text{ manifold $S^n$ whose class is a morphism $[S] \colon \emptyset \to \emptyset$ in $\Borddefn{n}$}
\Big\}
\nonumber \\
&=
\Big\{
	S \in \Borddefn{n+1} \,\Big|\, \text{ the underlying manifold of $S$ is } S^n 
\Big\}
\, . 
\label{eq:two-descri-of-sphere_n}
\end{align}

In the definition of the bordisms~$M$ in $\Borddefn{n+1}$ above it is enough to consider open neighbourhoods only of interior points. 
Points in $\partial M$ then have compatible neighbourhoods coming from the parametrisation of $\partial M$ with respect to source or target objects. 

\begin{remark}
Why is this the correct definition of local neighbourhoods for our purposes? 
Ultimately we are motivated by examples and the attitude that topological defects should be combinatorial objects. 
Still, there might be situations where one is led to considering different types of local neighbourhoods.
At least our definition is closed under taking iterated cones: 
for $X \in \Nbh_{n-k}$ we have that the interior of $C^k X$ is 
	isomorphic 
to an element of $\Nbh_n$, as the cone of an $i$-ball is topologically an $(i+1)$-ball.
The resulting $n$-ball can also be realised as a neighbourhood of~$0$ in $CX \times (-1,1)^{k-1}$.
\end{remark}

To illustrate Definition~\ref{def:Bordndef} we now go through the iteration for $n \in \{1,2,3\}$:

\begin{example}
\begin{enumerate}
\item
The sets 
 $\Nbh_{1}$, $\Sphere_{1}$ and the category $\Borddefn{1}$ are directly given in the definition. 
An example of a morphism from $\emptyset$ to $\emptyset$ (which is thus also an element in $\Sphere_1$) is
\be\label{eq:starlikecircle}
\tikzzbox{\begin{tikzpicture}[very thick,scale=0.7,color=blue!50!black, baseline=-0.1cm]
\draw[string=gray!60!blue, very thick, >=stealth] (0,0) circle (1.75);
\draw[color=gray!60!blue, postaction={decorate}, decoration={markings,mark=at position .66 with {\arrow[draw=gray!60!blue]{>}}}, very thick, >=stealth] (1.75,0) arc [start angle=0, delta angle=90, radius=1.75];
\draw[color=gray!60!blue, postaction={decorate}, decoration={markings,mark=at position .9 with {\arrow[draw=gray!60!blue]{>}}}, very thick, >=stealth] (1.75,0) arc [start angle=0, delta angle=340, radius=1.75];
%
\fill[color=red!90!black] (0:1.75) circle (3.5pt) node[right] {{\footnotesize $-$}};
\fill[color=red!90!black] (120:1.75) circle (3.5pt) node[above] {{\footnotesize $-$}};
\fill[color=red!90!black] (240:1.75) circle (3.5pt) node[below] {{\footnotesize $+$}};
\end{tikzpicture}}
\, .
\ee
\item 
The set $\Nbh_{2}$ consists of open neighbourhoods of two types. 
The first type are cylinders over the elements~\eqref{eq:N1elements} in $\Nbh_{1}$, resulting in the three elements of $\Nbh_2$:
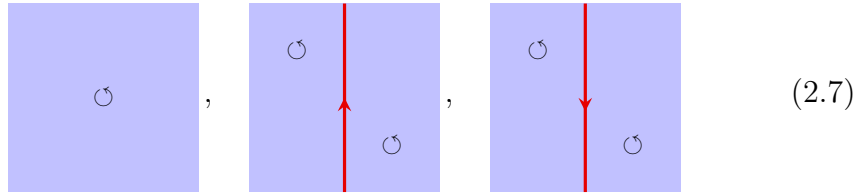
\be\label{eq:2dNbh}
\tikzzbox{\begin{tikzpicture}[thick,scale=2.5,color=black, baseline=1.2cm, >=stealth
]
\fill [blue!60,opacity=0.4] (0,0) -- (0,1) -- (1,1) -- (1,0);
%
\draw[line width=1] (0.5,0.5) node[line width=0pt] (alpha) {{\footnotesize $\circlearrowleft$}};
\end{tikzpicture}}
\, , \quad
\tikzzbox{\begin{tikzpicture}[thick,scale=2.5,color=black, baseline=1.2cm, >=stealth]
\fill [blue!60,opacity=0.4] (0,0) -- (0,1) -- (1,1) -- (1,0);
%
%
\draw[line width=1] (0.75,0.25) node[line width=0pt] (alpha) {{\footnotesize $\circlearrowleft$}};
\draw[line width=1] (0.25,0.75) node[line width=0pt] (alpha) {{\footnotesize $\circlearrowleft$}};
%
\draw[
	color=red!90!black, 
	very thick,
	 >=stealth, 
	postaction={decorate}, decoration={markings,mark=at position .5 with {\arrow[draw]{>}}}
	] 
 (0.5,0) --  (0.5,1);
\end{tikzpicture}}
\, , \quad
\tikzzbox{\begin{tikzpicture}[thick,scale=2.5,color=black, baseline=1.2cm, >=stealth]
\fill [blue!60,opacity=0.4] (0,0) -- (0,1) -- (1,1) -- (1,0);
%
%
\draw[line width=1] (0.75,0.25) node[line width=0pt] (alpha) {{\footnotesize $\circlearrowleft$}};
\draw[line width=1] (0.25,0.75) node[line width=0pt] (alpha) {{\footnotesize $\circlearrowleft$}};
%
\draw[
	color=red!90!black, 
	very thick,
	 >=stealth, 
	postaction={decorate}, decoration={markings,mark=at position .5 with {\arrow[draw]{<}}}
	] 
 (0.5,0) --  (0.5,1);
\end{tikzpicture}}
\ee

The second type of allowed neighbourhood is a cone over a circle in $\Sphere_{1}$. 
Hence there are infinitely many elements of $\Nbh_2$ of this second type. 
Taking for example the circle \eqref{eq:starlikecircle} and the two possible orientations of the cone point we arrive at:
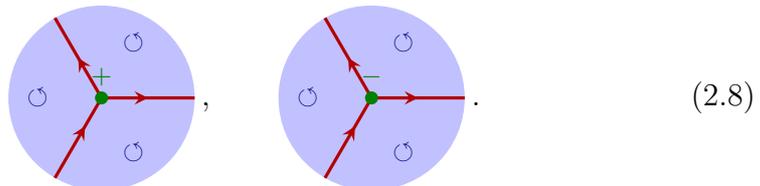
\be\label{eq:starlikedisk}
\tikzzbox{\begin{tikzpicture}[very thick,scale=0.7,color=blue!50!black, baseline=-0.1cm]
\fill [blue!60,opacity=0.4] (0,0) circle (1.75);
\draw[
	color=red!70!black, 
	>=stealth,
	decoration={markings, mark=at position 0.5 with {\arrow{>}},
					}, postaction={decorate}
	] 
 (0,0) -- (0:1.75);
 \draw[
	color=red!70!black, 
	>=stealth,
	decoration={markings, mark=at position 0.5 with {\arrow{>}},
					}, postaction={decorate} 
	] 
 (0,0) -- (120:1.75);
  \draw[
	color=red!70!black, 
	>=stealth,
	decoration={markings, mark=at position 0.5 with {\arrow{<}},
					}, postaction={decorate}
	] 
 (0,0) -- (240:1.75); 
%
\fill[color=green!50!black] (0,0) circle (3.5pt) node[above] {{\footnotesize $+$}};
%
\foreach \x in {1,...,3}
	\draw[line width=1] (120*\x - 60:1.2) node[line width=0pt] (Xbottom) {{\footnotesize $\circlearrowleft$}};
\end{tikzpicture}}
\, , \qquad
\tikzzbox{\begin{tikzpicture}[very thick,scale=0.7,color=blue!50!black, baseline=-0.1cm]
\fill [blue!60,opacity=0.4] (0,0) circle (1.75);
\draw[
	color=red!70!black, 
	>=stealth,
	decoration={markings, mark=at position 0.5 with {\arrow{>}},
					}, postaction={decorate}
	] 
 (0,0) -- (0:1.75);
 \draw[
	color=red!70!black, 
	>=stealth,
	decoration={markings, mark=at position 0.5 with {\arrow{>}},
					}, postaction={decorate} 
	]
 (0,0) -- (120:1.75);
  \draw[
	color=red!70!black, 
	>=stealth,
	decoration={markings, mark=at position 0.5 with {\arrow{<}},
					}, postaction={decorate}
	] 
 (0,0) -- (240:1.75); 
%
\fill[color=green!50!black] (0,0) circle (3.5pt) node[above] {{\footnotesize $-$}};
%
\foreach \x in {1,...,3}
	\draw[line width=1] (120*\x - 60:1.2) node[line width=0pt] (Xbottom) {{\footnotesize $\circlearrowleft$}};
\end{tikzpicture}}
\, .
\ee

The category $\Borddefn{2}$ has stratified circles like \eqref{eq:starlikecircle} as objects. 
Morphisms of $\Borddefn{2}$ are those morphisms of $\Bordstratn{2}$ where each point has a local neighbourhood in $\Nbh_{2}$.%
\footnote{This situation illustrates our choice of the definition of stratified manifolds and their morphisms as given in Section~\ref{subsec:strabord}, rather than requiring a smooth structure on the total space and smooth maps that respects the filtration: 
A smooth map has a differential at the 0-stratum in \eqref{eq:starlikedisk}, and since the differential is linear, such smooth maps cannot relate local situations with arbitrary angles between the 1-strata meeting at the 0-stratum.}
An example of an element of $\Sphere_2$ is 
\vspace{-0.7cm}
\be
\label{eq:defsphere}
\tikzzbox{\begin{tikzpicture}[very thick,scale=1.2,color=green!60!black=-0.1cm, >=stealth, baseline=0]
\fill[ball color=white!95!blue] (0,0) circle (0.95 cm);
\coordinate (v1) at (-0.4,-0.6);
\coordinate (v2) at (0.4,-0.6);
\coordinate (v3) at (0.4,0.6);
\coordinate (v4) at (-0.4,0.6);
\draw[color=red!80!black, very thick, rounded corners=0.5mm, postaction={decorate}, decoration={markings,mark=at position .5 with {\arrow[draw=red!80!black]{>}}}] 
	(v2) .. controls +(0,-0.25) and +(0,-0.25) .. (v1);
\draw[color=red!80!black, very thick, rounded corners=0.5mm, postaction={decorate}, decoration={markings,mark=at position .62 with {\arrow[draw=red!80!black]{>}}}] 
	(v4) .. controls +(0,0.15) and +(0,0.15) .. (v3);
\draw[color=red!80!black, very thick, rounded corners=0.5mm, postaction={decorate}, decoration={markings,mark=at position .5 with {\arrow[draw=red!80!black]{>}}}] 
	(v4) .. controls +(0.25,-0.1) and +(-0.05,0.5) .. (v2);
\draw[color=red!80!black, very thick, rounded corners=0.5mm, postaction={decorate}, decoration={markings,mark=at position .58 with {\arrow[draw=red!80!black]{>}}}] 
	(v3) .. controls +(-0.9,0.99) and +(-0.75,0.4) .. (v1);
\draw[color=red!80!black, very thick, rounded corners=0.5mm, postaction={decorate}, decoration={markings,mark=at position .5 with {\arrow[draw=red!80!black]{>}}}] 
	(v1) .. controls +(-0.15,0.5) and +(-0.15,-0.5) .. (v4);
\draw[color=red!80!black, very thick, rounded corners=0.5mm, postaction={decorate}, decoration={markings,mark=at position .5 with {\arrow[draw=red!80!black]{>}}}] 
	(v3) .. controls +(0.25,-0.5) and +(0.25,0.5) .. (v2);
\fill (v1) circle (1.6pt) node[black, opacity=0.6, right, font=\tiny] { $+$ };
\fill (v2) circle (1.6pt) node[black, opacity=0.6, right, font=\tiny] { $+$ };
\fill (v3) circle (1.6pt) node[black, opacity=0.6, right, font=\tiny] { $-$ };
\fill (v4) circle (1.6pt) node[black, opacity=0.6, right, font=\tiny] { $-$ };
\end{tikzpicture}}
\ee
where the orientation of the 2-strata is counterclockwise on the front side of the sphere. 
\item 
Next we describe the set $\Nbh_{3}$. 
By taking a cylinder over the first cylinder in~\eqref{eq:2dNbh} we obtain the empty 3-ball (drawn as a cube)
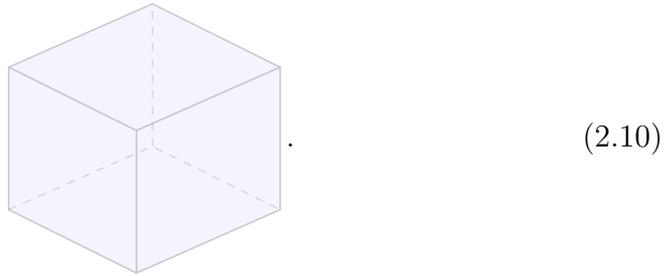
\be
\tikzzbox{\begin{tikzpicture}[thick,scale=2.1,color=blue!50!black, baseline=0.0cm, >=stealth, 
				style={x={(-0.9cm,-0.4cm)},y={(0.8cm,-0.4cm)},z={(0cm,0.9cm)}}]
\draw[
	 color=gray,            
	 opacity=0.3, 
	 semithick,
	 dashed
	 ] 
	 (1,0,0) -- (0,0,0) -- (0,1,0)
	 (0,0,0) -- (0,0,1);
\coordinate (left0) at (0, 0.45, 0);
\coordinate (left1) at (0, 0.45, 1);
\coordinate (right0) at (1, 0.45, 0);
\coordinate (right1) at (1, 0.45, 1);
%
\fill [blue!20,opacity=0.2] (1,0,0) -- (1,1,0) -- (0,1,0) -- (0,1,1) -- (0,0,1) -- (1,0,1);
%
\draw[
	 color=gray, 
	 opacity=0.4, 
	 semithick
	 ] 
	 (0,1,1) -- (0,1,0) -- (1,1,0) -- (1,1,1) -- (0,1,1) -- (0,0,1) -- (1,0,1) -- (1,0,0) -- (1,1,0)
	 (1,0,1) -- (1,1,1);
\end{tikzpicture}}
\, . 
\ee
This ball has the induced standard orientation from~$\R^3$, as do all 3-strata of all other elements in~$\mathcal N_3$. 
For the other two discs in~\eqref{eq:2dNbh} we get (the shading emphasises the opposite orientation)
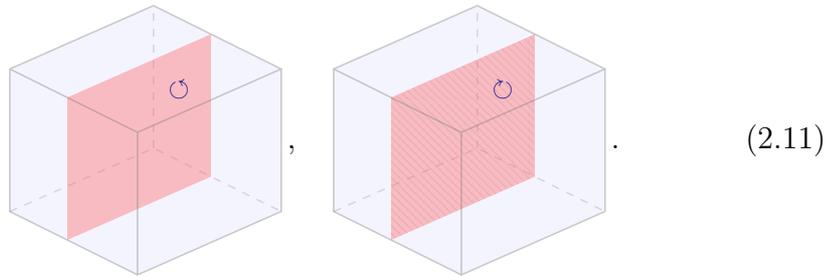
\be
\tikzzbox{\begin{tikzpicture}[thick,scale=2.1,color=blue!50!black, baseline=0.0cm, >=stealth, 
				style={x={(-0.9cm,-0.4cm)},y={(0.8cm,-0.4cm)},z={(0cm,0.9cm)}}]
\draw[
	 color=gray,            
	 opacity=0.3, 
	 semithick,
	 dashed
	 ] 
	 (1,0,0) -- (0,0,0) -- (0,1,0)
	 (0,0,0) -- (0,0,1);
\coordinate (left0) at (0, 0.45, 0);
\coordinate (left1) at (0, 0.45, 1);
\coordinate (right0) at (1, 0.45, 0);
\coordinate (right1) at (1, 0.45, 1);
%
\fill [blue!20,opacity=0.2] (1,0,0) -- (1,1,0) -- (0,1,0) -- (0,1,1) -- (0,0,1) -- (1,0,1);
\fill [red!60,opacity=0.4](left0) -- (left1) -- (right1) -- (right0);
\draw[line width=1] (0, 0.2, 0.5) node[line width=0pt] (beta) {{\footnotesize  $\circlearrowleft$}};
%
\draw[
	 color=gray, 
	 opacity=0.4, 
	 semithick
	 ] 
	 (0,1,1) -- (0,1,0) -- (1,1,0) -- (1,1,1) -- (0,1,1) -- (0,0,1) -- (1,0,1) -- (1,0,0) -- (1,1,0)
	 (1,0,1) -- (1,1,1);
\end{tikzpicture}}
\, , \quad
\tikzzbox{\begin{tikzpicture}[thick,scale=2.1,color=blue!50!black, baseline=0.0cm, >=stealth, 
				style={x={(-0.9cm,-0.4cm)},y={(0.8cm,-0.4cm)},z={(0cm,0.9cm)}}]
\draw[
	 color=gray, 
	 opacity=0.3, 
	 semithick,
	 dashed
	 ] 
	 (1,0,0) -- (0,0,0) -- (0,1,0)
	 (0,0,0) -- (0,0,1);
\coordinate (left0) at (0, 0.45, 0);
\coordinate (left1) at (0, 0.45, 1);
\coordinate (right0) at (1, 0.45, 0);
\coordinate (right1) at (1, 0.45, 1);
%
\fill [blue!20,opacity=0.2] (1,0,0) -- (1,1,0) -- (0,1,0) -- (0,1,1) -- (0,0,1) -- (1,0,1);
\fill [pattern=north west lines, opacity=0.3] (left0) -- (left1) -- (right1) -- (right0);
\fill [red!60,opacity=0.4](left0) -- (left1) -- (right1) -- (right0);
\draw[line width=1] (0, 0.2, 0.5) node[line width=0pt] (beta) {{\footnotesize  $\circlearrowright$}}; 
%
\draw[
	 color=gray, 
	 opacity=0.4, 
	 semithick
	 ] 
	 (0,1,1) -- (0,1,0) -- (1,1,0) -- (1,1,1) -- (0,1,1) -- (0,0,1) -- (1,0,1) -- (1,0,0) -- (1,1,0)
	 (1,0,1) -- (1,1,1);
%
\end{tikzpicture}}
\, .
\ee

Taking cylinders over the cones  in $\Nbh_{2}$ we obtain 
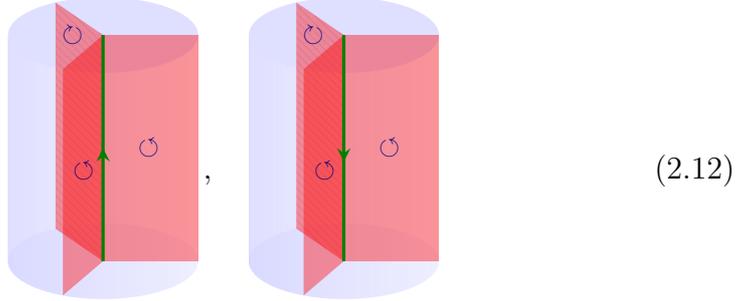
\be\label{eq:starlikecylinder}
\tikzzbox{\begin{tikzpicture}[very thick,scale=1.0,color=blue!50!black, baseline=-1.9cm]
\fill [blue!15,
      opacity=0.5, 
      left color=blue!15, 
      right color=white] 
      (-1.25,0) -- (-1.25,-3) arc (180:360:1.25 and 0.5) -- (1.25,0) arc (0:180:1.25 and -0.5);
\fill [blue!35,opacity=0.1] (-1.25,-3) arc (180:360:1.25 and 0.5) -- (1.25,-3) arc (0:180:1.25 and 0.5);
\fill [blue!25,opacity=0.5] (-1.25,0) arc (180:360:1.25 and 0.5) -- (1.25,0) arc (0:180:1.25 and 0.5);
%
\fill [red,opacity=0.4] (0,0) -- (0,-3) -- (1.25,-3) -- (1.25,0);
\fill [pattern=north west lines, opacity=0.3] (0,0) -- ($(0,0)+(120:1.25 and 0.5)$) -- ($(0,-3)+(120:1.25 and 0.5)$) -- (0,-3) -- (0,0);
\fill [red,opacity=0.4] (0,0) -- ($(0,0)+(120:1.25 and 0.5)$) -- ($(0,-3)+(120:1.25 and 0.5)$) -- (0,-3) -- (0,0);
\fill [red,opacity=0.4] (0,0) -- ($(0,0)+(245:1.25 and 0.5)$) -- ($(0,-3)+(245:1.25 and 0.5)$) -- (0,-3) -- (0,0);
 \draw[
	color=green!50!black, 
	>=stealth,
	decoration={markings, mark=at position 0.5 with {\arrow{>}},
					}, postaction={decorate}
	] 
(0,-3) -- (0,0);
\draw[line width=1] (0.6, -1.5) node[line width=0pt] (beta) {{\footnotesize  $\circlearrowleft$}};
\draw[line width=1]  ($(0,-1.8)+(245:0.6 and 0)$)  node[line width=0pt] (beta) {{\footnotesize  $\circlearrowleft$}};
\draw[line width=1]  ($(0,0)+(120:0.8 and 0)$)  node[line width=0pt] (beta) {{\footnotesize  $\circlearrowright$}};
\end{tikzpicture}}
\, , \quad 
\tikzzbox{\begin{tikzpicture}[very thick,scale=1.0,color=blue!50!black, baseline=-1.9cm]
\fill [blue!15,
      opacity=0.5, 
      left color=blue!15, 
      right color=white] 
      (-1.25,0) -- (-1.25,-3) arc (180:360:1.25 and 0.5) -- (1.25,0) arc (0:180:1.25 and -0.5);
\fill [blue!35,opacity=0.1] (-1.25,-3) arc (180:360:1.25 and 0.5) -- (1.25,-3) arc (0:180:1.25 and 0.5);
\fill [blue!25,opacity=0.5] (-1.25,0) arc (180:360:1.25 and 0.5) -- (1.25,0) arc (0:180:1.25 and 0.5);
%
\fill [red,opacity=0.4] (0,0) -- (0,-3) -- (1.25,-3) -- (1.25,0);
\fill [pattern=north west lines, opacity=0.3] (0,0) -- ($(0,0)+(120:1.25 and 0.5)$) -- ($(0,-3)+(120:1.25 and 0.5)$) -- (0,-3) -- (0,0);
\fill [red,opacity=0.4] (0,0) -- ($(0,0)+(120:1.25 and 0.5)$) -- ($(0,-3)+(120:1.25 and 0.5)$) -- (0,-3) -- (0,0);
\fill [red,opacity=0.4] (0,0) -- ($(0,0)+(245:1.25 and 0.5)$) -- ($(0,-3)+(245:1.25 and 0.5)$) -- (0,-3) -- (0,0);
 \draw[
	color=green!50!black, 
	>=stealth,
	decoration={markings, mark=at position 0.5 with {\arrow{<}},
					}, postaction={decorate}
	] 
(0,-3) -- (0,0);
\draw[line width=1] (0.6, -1.5) node[line width=0pt] (beta) {{\footnotesize  $\circlearrowleft$}};
\draw[line width=1]  ($(0,-1.8)+(245:0.6 and 0)$)  node[line width=0pt] (beta) {{\footnotesize  $\circlearrowleft$}};
\draw[line width=1]  ($(0,0)+(120:0.8 and 0)$)  node[line width=0pt] (beta) {{\footnotesize  $\circlearrowright$}};
\end{tikzpicture}}
\ee
for the two discs in \eqref{eq:starlikedisk}, respectively. 

Finally, we have to consider cones over spheres in $\Sphere_2$. 
For example, the cones over \eqref{eq:defsphere} (with the two possible orientations for the interior cone point) are the open stratified 3-balls
\vspace{-0.7cm}
\be
\tikzzbox{\begin{tikzpicture}[very thick,scale=1.2,color=green!60!black=-0.1cm, >=stealth, baseline=0]
\fill[ball color=white!95!blue] (0,0) circle (0.95 cm);
\coordinate (v1) at (-0.4,-0.6);
\coordinate (v2) at (0.4,-0.6);
\coordinate (v3) at (0.4,0.6);
\coordinate (v4) at (-0.4,0.6);
%
\fill [red!80!black,opacity=0.3] (0,0) -- (v2) .. controls +(0,-0.25) and +(0,-0.25) .. (v1);
\fill [red!80!black,opacity=0.3] (0,0) -- (v4) .. controls +(0,0.15) and +(0,0.15) .. (v3);
\fill [red!80!black,opacity=0.3] (0,0) -- (v4) .. controls +(0.25,-0.1) and +(-0.05,0.5) .. (v2);
\fill [red!80!black,opacity=0.3] (0,0) -- (v3) .. controls +(-0.9,0.99) and +(-0.75,0.4) .. (v1);
\fill [red!80!black,opacity=0.3] (0,0) -- (v1) .. controls +(-0.15,0.5) and +(-0.15,-0.5) .. (v4);
\fill [red!80!black,opacity=0.3] (0,0) -- (v3) .. controls +(0.25,-0.5) and +(0.25,0.5) .. (v2);
%
\draw[thick, opacity=0.6, postaction={decorate}, decoration={markings,mark=at position .6 with {\arrow[draw=green!60!black]{<}}}] (0,0) -- (v1);
\draw[thick, opacity=0.6, postaction={decorate}, decoration={markings,mark=at position .6 with {\arrow[draw=green!60!black]{<}}}] (0,0) -- (v2);
\draw[thick, opacity=0.6, postaction={decorate}, decoration={markings,mark=at position .7 with {\arrow[draw=green!60!black]{>}}}] (0,0) -- (v3);
\draw[thick, opacity=0.6, postaction={decorate}, decoration={markings,mark=at position .5 with {\arrow[draw=green!60!black]{>}}}] (0,0) -- (v4);
\draw[color=red!80!black, very thin, rounded corners=0.5mm, postaction={decorate}, decoration={markings,mark=at position .5 with {\arrow[draw=red!80!black]{>}}}] 
	(v2) .. controls +(0,-0.25) and +(0,-0.25) .. (v1);
\draw[color=red!80!black, very thin, rounded corners=0.5mm, postaction={decorate}, decoration={markings,mark=at position .62 with {\arrow[draw=red!80!black]{>}}}] 
	(v4) .. controls +(0,0.15) and +(0,0.15) .. (v3);
\draw[color=red!80!black, very thin, rounded corners=0.5mm, postaction={decorate}, decoration={markings,mark=at position .5 with {\arrow[draw=red!80!black]{>}}}] 
	(v4) .. controls +(0.25,-0.1) and +(-0.05,0.5) .. (v2);
\draw[color=red!80!black, very thin, rounded corners=0.5mm, postaction={decorate}, decoration={markings,mark=at position .58 with {\arrow[draw=red!80!black]{>}}}] 
	(v3) .. controls +(-0.9,0.99) and +(-0.75,0.4) .. (v1);
\draw[color=red!80!black, very thin, rounded corners=0.5mm, postaction={decorate}, decoration={markings,mark=at position .5 with {\arrow[draw=red!80!black]{>}}}] 
	(v1) .. controls +(-0.15,0.5) and +(-0.15,-0.5) .. (v4);
\draw[color=red!80!black, very thin, rounded corners=0.5mm, postaction={decorate}, decoration={markings,mark=at position .5 with {\arrow[draw=red!80!black]{>}}}] 
	(v3) .. controls +(0.25,-0.5) and +(0.25,0.5) .. (v2);
\fill[magenta!10!black] (0,0) circle (1.6pt) node[black, opacity=0.6, right, font=\tiny, left] {$+$};
\end{tikzpicture}}
\, , \quad 
\tikzzbox{\begin{tikzpicture}[very thick,scale=1.2,color=green!60!black=-0.1cm, >=stealth, baseline=0]
\fill[ball color=white!95!blue] (0,0) circle (0.95 cm);
\coordinate (v1) at (-0.4,-0.6);
\coordinate (v2) at (0.4,-0.6);
\coordinate (v3) at (0.4,0.6);
\coordinate (v4) at (-0.4,0.6);
%
\fill [red!80!black,opacity=0.3] (0,0) -- (v2) .. controls +(0,-0.25) and +(0,-0.25) .. (v1);
\fill [red!80!black,opacity=0.3] (0,0) -- (v4) .. controls +(0,0.15) and +(0,0.15) .. (v3);
\fill [red!80!black,opacity=0.3] (0,0) -- (v4) .. controls +(0.25,-0.1) and +(-0.05,0.5) .. (v2);
\fill [red!80!black,opacity=0.3] (0,0) -- (v3) .. controls +(-0.9,0.99) and +(-0.75,0.4) .. (v1);
\fill [red!80!black,opacity=0.3] (0,0) -- (v1) .. controls +(-0.15,0.5) and +(-0.15,-0.5) .. (v4);
\fill [red!80!black,opacity=0.3] (0,0) -- (v3) .. controls +(0.25,-0.5) and +(0.25,0.5) .. (v2);
%
\draw[thick, opacity=0.6, postaction={decorate}, decoration={markings,mark=at position .6 with {\arrow[draw=green!60!black]{<}}}] (0,0) -- (v1);
\draw[thick, opacity=0.6, postaction={decorate}, decoration={markings,mark=at position .6 with {\arrow[draw=green!60!black]{<}}}] (0,0) -- (v2);
\draw[thick, opacity=0.6, postaction={decorate}, decoration={markings,mark=at position .7 with {\arrow[draw=green!60!black]{>}}}] (0,0) -- (v3);
\draw[thick, opacity=0.6, postaction={decorate}, decoration={markings,mark=at position .5 with {\arrow[draw=green!60!black]{>}}}] (0,0) -- (v4);
\draw[color=red!80!black, very thin, rounded corners=0.5mm, postaction={decorate}, decoration={markings,mark=at position .5 with {\arrow[draw=red!80!black]{>}}}] 
	(v2) .. controls +(0,-0.25) and +(0,-0.25) .. (v1);
\draw[color=red!80!black, very thin, rounded corners=0.5mm, postaction={decorate}, decoration={markings,mark=at position .62 with {\arrow[draw=red!80!black]{>}}}] 
	(v4) .. controls +(0,0.15) and +(0,0.15) .. (v3);
\draw[color=red!80!black, very thin, rounded corners=0.5mm, postaction={decorate}, decoration={markings,mark=at position .5 with {\arrow[draw=red!80!black]{>}}}] 
	(v4) .. controls +(0.25,-0.1) and +(-0.05,0.5) .. (v2);
\draw[color=red!80!black, very thin, rounded corners=0.5mm, postaction={decorate}, decoration={markings,mark=at position .58 with {\arrow[draw=red!80!black]{>}}}] 
	(v3) .. controls +(-0.9,0.99) and +(-0.75,0.4) .. (v1);
\draw[color=red!80!black, very thin, rounded corners=0.5mm, postaction={decorate}, decoration={markings,mark=at position .5 with {\arrow[draw=red!80!black]{>}}}] 
	(v1) .. controls +(-0.15,0.5) and +(-0.15,-0.5) .. (v4);
\draw[color=red!80!black, very thin, rounded corners=0.5mm, postaction={decorate}, decoration={markings,mark=at position .5 with {\arrow[draw=red!80!black]{>}}}] 
	(v3) .. controls +(0.25,-0.5) and +(0.25,0.5) .. (v2);
\fill[magenta!10!black] (0,0) circle (1.6pt) node[black, opacity=0.6, right, font=\tiny, left] {$-$};
\end{tikzpicture}}
\, . 
\ee
\end{enumerate}
\end{example}

\subsubsection{Decorated defect bordisms}

Having specified the allowed geometric configurations for defects, we next turn to the sets of labels that  defects carry. 
Our definition of ``defect data''~$\D$ in~$n$ dimensions and of the decorated bordism category $\Borddefn{n}(\D)$ is again inductive.

Roughly, these definitions amount to the following. 
The defect labels for $i$-strata~$Y$ consist of a set $D_{i}$ together with a map $f_i$ that specifies the allowed labels for the adjacent strata. 
A local neighbourhood of~$Y$ has the form 
$B^{n-i} \times (-1,1)^{i}$, and the relevant information about the adjacent strata is encoded in a decorated stratification of the sphere $S^{n-i-1}=\partial B^{n-i}$. 
More precisely, we must consider the stratified sphere only up to isomorphism. 
This leads to an equivalence class (with respect to isomorphisms of stratified manifolds) of decorated spheres, the set of which we will denote $[\Sphere_{n-i-1}(\partial^{i+1}\D)]$, and the map that specifies the allowed configuration for labels $D_{i}$  is a map $f_{i}\colon D_{i} \rightarrow [\Sphere_{n-i-1}(\partial^{i+1}\D)]$. 

The precise definition is as follows:

\begin{definition}
\label{def:d-dim-defect-data}
For all $n \in \N$ a set of \textsl{$n$-dimensional defect data} 
\be\label{eq:d-dim-defect-datum_def}
\D^{n} 
	\equiv 
\big( D^{n}_{n}, D^{n}_{n-1},\ldots, D^{n}_{0}; \, f^{n}_{n-1},f^{n}_{n-2},\ldots,f^{n}_{0} \big)
\ee
consists of sets $D^{n}_k$ of \textsl{$k$-dimensional defect labels} together with the \textsl{$k$-dimensional adjacency maps} $f^n_k$ out of $D^{n}_k  \times \{\pm\}$, whose targets will be defined inductively below.
Namely, for all $n \in \N$ we define two related structures recursively:
\begin{itemize}
\item 
$\mathcal D^n$, the \textsl{class  
of all $n$-dimensional defect data}, together with a \textsl{boundary map} $\partial \colon \mathcal D^{n} \to \mathcal D^{n-1}$ (for $n \geqslant 1$), and
\item 
for each $\D^{n} \in \mathcal D^n$ and $n\geqslant 1$ the symmetric monoidal category $\Borddefn{n}(\D^{n})$ of 
\textsl{$n$-dimensional decorated defect bordisms} (while for $n=0$ we define a set $\Borddefn{0}(\D^{0})$).
\end{itemize}
We start the induction with $n=0$:
\begin{itemize}
\item 
$\mathcal{D}^{0} := \{ (D^{0}_{0}) \,|\, \text{$D^{0}_{0}$ is a non-empty set} \}$,
\item
for $\D^{0} \in \mathcal{D}^{0}$,
$\Borddefn{0}(\D^{0})$ is the set whose elements are isomorphism classes of finite sets of oriented points labelled with elements in  $D^{0}_{0}$.
We take 
\be
[\Sphere_{0}(\D^{0})] \subset \Borddefn{0}(\D^{0})
\ee 
to be the subset consisting of classes whose underlying point-set is an oriented $0$-sphere (i.\,e.~two points with opposite orientation).
\end{itemize}
Now assume that $\mathcal D^{n}$, $\partial \colon \mathcal D^{n} \to \mathcal D^{n-1}$ (for $n \geqslant 1$) and $\Borddefn{n}(\D^{n})$ 
have been defined up to and including a given $n\geqslant 0$.
For $n \geqslant 1$ we set (compare to \eqref{eq:two-descri-of-sphere_n})
\be
[\Sphere_{n}(\D^{n})] :=
\Big\{ [M] \in \Hom_{\Borddefn{n}(\D^{n})}(\emptyset,\emptyset)
\, \Big| \, \text{$M$ has underlying manifold $S^n$} \Big\} \, .
\ee
Then:
\begin{itemize}
\item
The set $\mathcal{D}^{n+1}$ consists of sets
\be\label{eq:n+1-dim-defect-datum_def}
\D^{n+1} 
	\equiv 
\big( D^{n+1}_{n+1}, D^{n+1}_{n},\ldots, D^{n+1}_{0}; \, f^{n+1}_{n},f^{n+1}_{n-1},\ldots,f^{n+1}_{0} \big)
\ee
subject to two conditions:
\begin{itemize}
\item
$\partial( \D^{n+1} ) := \big( D^{n+1}_{n+1}, D^{n+1}_{n},\ldots, D^{n+1}_{1}; \, f^{n+1}_{n},f^{n+1}_{n-1},\ldots,f^{n+1}_{1} \big)$ is an element of $\mathcal{D}^n$, i.\,e.~omitting $D^{n+1}_0$ and $f^{n+1}_0$ from $\D^{n+1}$ gives a set of $n$-dimensional defect data. 
We thus obtain a map $\partial \colon \mathcal D^{n+1} \to \mathcal D^{n}$.
\item
$f^{n+1}_{0}$ is a map $D^{n+1}_{0} \times \{\pm\} \rightarrow [\Sphere_{n}(\partial\D^{n+1})]$ such that
\be\label{eq:def-bord-def_fn+1_0-dual}
f^{n+1}_{0}(\phi,-) 
= 
\big(f^{n+1}_{0}(\phi,+) \big)^{\text{rev}}
\ee
where for any $S \in [\Sphere_{n}(\D^{n})]$, $S^{\text{rev}}$ denotes the bordism class with reversed orientation for all strata.
\end{itemize}
Iterating the boundary map we have that, for $0 \leqslant j \leqslant m$, 
\be
\partial^{j}\D^{m} := \big( D^{m}_{m},\ldots, D^{m}_{j}; \, f^{m}_{m-1},\ldots,f^{m}_{j} \big)
~\in~ \mathcal{D}^{m-j} \, ,
\ee
and that source and target of the adjacency map $f^{m}_j$ are, for $j<m$,
\be
\label{eq:adjacency-map-fmj}
f^{m}_{j}\colon D^{m}_{j} \times \{\pm\} \lra [\Sphere_{m-j-1}(\partial^{j+1}\D^{m})] \, ,
\ee
which satisfies the duality condition 
\be\label{eq:adjacency-map-duality}
f^{m}_{j}(\phi,-) 
= 
\big( f^{m}_{j}(\phi_,+) \big)^{\text{rev}}
\ee
for all $\phi \in D_j^{m}$. 
\item
For $\D^{n+1} \in \mathcal{D}^{n+1}$, the category $\Borddefn{n+1}(\D^{n+1})$ is defined as follows:
\begin{itemize}
\item
Objects are those of $\Borddefn{n+1}$ together with a decoration by $\partial \D^{n+1}$
in such a way that their classes define bordisms $\emptyset \to \emptyset$ in $\Borddefn{n}(\partial \D^{n+1})$.
\item                           
Morphisms are the morphisms $Y$ of $\Borddefn{n+1}$ where each $i$-stratum $Y_i^\alpha$ is labelled by an element of $D^{n+1}_i$, subject to two requirements: 
(i)~If the stratum meets the boundary of the bordism, it restricts to the decorations on the boundary.
(ii)~Let the stratum $Y_i^\alpha$ be labelled by $\phi \in D^{n+1}_i$ and let $\varepsilon = +$ for $i>0$\footnote{One may equally well choose $\varepsilon=-$ for $i>0$ as the open sets $U$ below resulting from $\varepsilon=+$ and $\varepsilon=-$ will be related by 
	an isomorphism 
which preserves the orientation of all strata.}
and let $\varepsilon \in\{ \pm\}$ 
be the orientation of $Y_i^\alpha$ for $i=0$.
Let $X$ be a representative $(n-i)$-sphere of the class $f^{n+1}_i(\phi,\varepsilon) \in [\Sphere_{n-i}(\partial^{i+1}\D^{n+1})]$. 
Let $U$ be the stratified manifold given by the interior of $C(X) \times [-1,1]^i$, equipped with the induced labelling by $\D^{n+1}$ (and $\{0\} \times (-1,1)^i$ labelled by $\phi$).
Then each interior point of $Y_i^\alpha$ has a local neighbourhood 
	isomorphic 
to~$U$ in a way compatible with the labelling.
\end{itemize}
\end{itemize}
\end{definition}

Usually the dimension~$n$ of a set of defect data $\D^n$ will be clear from the context, in which case we will simply write $\D, D_i, f_i$ for $\D^n, D_i^n, f_i^n$, respectively. 
Also, according to~\eqref{eq:adjacency-map-duality} it suffices to know the values $f^{m}_{j}(\phi_,+)$ of the 
adjacency maps, so we will sometimes abbreviate
\be
	 f^{m}_{j}(\phi) := f^{m}_{j}(\phi_,+) \, .
\ee
The signs in~\eqref{eq:def-bord-def_fn+1_0-dual} and~\eqref{eq:adjacency-map-fmj} in particular ensure that a label $x \in D_{i}$ can occur both on the in- and outgoing boundary of a bordism. For $i=1$, if a 1-stratum meets the boundary with both endpoints, the neighbourhood of one endpoint is determined by $\varepsilon=+$, of the other by $\varepsilon=-$.
\begin{example}
\label{ex:defectdatan12}
Let us spell out what the induction in Definition~\ref{def:d-dim-defect-data} amounts to for $n=1$ and $n=2$:
\begin{enumerate}
\item
A set of $1$-dimensional defect data $\D$ consist of two sets $D_{0}$ and $D_{1}$. 
Elements in the set of 
	sphere 
classes $[\Sphere_{0}(\partial \D)]$ consist of two oppositely oriented points decorated each with $D_{1}$.
Thus $f_{0}$ is a map 
\be
\label{eq:2}
f_{0} \colon D_{0} \times \{ \pm \} \lra D_{1} \times D_{1} \, . 
\ee
Since reversing the orientation of the two points of the 0-sphere lies in the same class as keeping the orientation of the two points but exchanging their labels, condition \eqref{eq:def-bord-def_fn+1_0-dual} states that if $f_1(x,+) = (a,b)$ then $f_0(x,-) = (b,a)$. 
	
It can be helpful to think of $x \in D_0$ as a ``morphism'' with source $a$ and target $b$, that is, we define maps $s,t \colon D_0 \to D_1$ such that $f_1(x,+) = (s(x),t(x))$ (see Section~\ref{sec:highercatfor} for more on this point of view).

$\Borddefn{1}(\D)$ has as objects oriented points decorated with $D_{1}$.
Morphisms are stratified 1-manifolds with 1-strata decorated with $D_{1}$, compatible with the boundary decorations, together with 
$0$-strata $p$ with orientation~$\eps$ that are decorated with elements $\phi \in D_{0}$ such that for $\eps=+$, $s(\phi)$ is the decoration on the incoming 1-stratum at $p$ and $t(\phi)$ the decoration of the outgoing 1-stratum, while for $\eps=-$ this order is reversed: 
\be
\tikzzbox{\begin{tikzpicture}[thick,scale=2.321,color=blue!50!black, baseline=0cm, >=stealth, rotate=-90
]
\coordinate (v') at (0,0);
\coordinate (d) at (0,0.5);
\coordinate (v) at (0,1);
\draw[string=gray!60!blue, very thick] (v') -- (d);
\draw[string=gray!60!blue, very thick] (d) -- (v);
\fill[color=red!90!black] (d) circle (0.9pt) node[above] (0up) {{\footnotesize $+$}};
\fill[color=red!90!black] (d) circle (0.9pt) node[below] (0up) {{\footnotesize $\phi$}};
\fill[color=gray!60!blue] (0,0.1) circle (0pt) node[below] (0up) {{\footnotesize $s(\phi)$}};
\fill[color=gray!60!blue] (0,0.9) circle (0pt) node[below] (0up) {{\footnotesize $t(\phi)$}};
\end{tikzpicture}}
\, , \quad
\tikzzbox{\begin{tikzpicture}[thick,scale=2.321,color=blue!50!black, baseline=0cm, >=stealth, rotate=-90
]
\coordinate (v') at (0,0);
\coordinate (d) at (0,0.5);
\coordinate (v) at (0,1);
\draw[string=gray!60!blue, very thick] (v') -- (d);
\draw[string=gray!60!blue, very thick] (d) -- (v);
\fill[color=red!90!black] (d) circle (0.9pt) node[above] (0up) {{\footnotesize $-$}};
\fill[color=red!90!black] (d) circle (0.9pt) node[below] (0up) {{\footnotesize $\phi$}};
\fill[color=gray!60!blue] (0,0.1) circle (0pt) node[below] (0up) {{\footnotesize $t(\phi)$}};
\fill[color=gray!60!blue] (0,0.9) circle (0pt) node[below] (0up) {{\footnotesize $s(\phi)$}};
\end{tikzpicture}}
\ee
\item
A set of $2$-dimensional defect data $\D$ consists of three sets $D_{0},D_{1},D_{2}$ together with maps $f_{0}, f_1$ as follows.
We demand that $\partial \D=(D_{2},D_{1};f_{1})$ forms a set of 1-dimensional defect data, i.\,e.~the map
\be
\label{eq:f1in2d}
f_{1} \colon D_{1} \times \{\pm\} \lra D_{2} \times D_{2} 
\ee
is defined as in point~(i) above.

The set $[\Sphere_{1}(\partial \D)]$ is already defined by the $1$-dimensional case: 
it consists of (isotopy classes of) oriented stratified decorated circles where each $1$-stratum is decorated 
with elements in $D_{2}$, and each 0-stratum with elements in $D_{1}$, such that source and target of the decorations match with the map $f_{1}$. 
Possibly there are no $0$-strata, in which case the circle is decorated just with an element in $D_{2}$.

By definition, the map 
\be
\label{eq:f0in2d}
f_{0} \colon D_{0} \times \{\pm \} \lra [\Sphere_{1}(\partial \D)] 
\ee
is subject to the duality condition \eqref{eq:def-bord-def_fn+1_0-dual}.
The elements of $[\Sphere_{1}(\partial \D)]$ are spheres considered up to isomorphisms of stratified manifolds, so in particular up to rotations. 
Describing the stratified circles combinatorially and restricting to positively oriented points we recover the junction map~$j$ of \cite[Eq.\,(2.2)]{dkr1107.0495}: $f_0(x,+) = j(x)$, where
\be
j \colon D_0 \lra D_2 \sqcup \bigsqcup_{m\in\Z_{+}} \big( (D_1 \times \{ \pm \}) \times_{D_{2}} \dots \times_{D_{2}} (D_1 \times \{ \pm \}) \big)/C_m 
\ee
determines the neighbourhood of a $D_0$-decorated 0-stratum (with the first factor $D_{2}$ corresponding to the case of no 1-strata 
in the neighbourhood). 
The condition in the product on the right is that source and target of adjacent $D_{1}$-terms must agree, including the first and last term. 
$C_{m}$ denotes the cyclic group with~$m$ elements, which acts naturally on the product of~$m$ elements. 
\end{enumerate}

We do not describe the case $n=3$ in detail, since a combinatorial description of defect 2-spheres 
seems impracticable. 
Note however that the map~\eqref{eq:f0in2d} also appears as the folding map $f^{3}_{1}$ for defect 1-strata in 3-dimensional defect bordisms \cite[Def. 2.6]{CMS}, just as \eqref{eq:f1in2d} is a reincarnation of \eqref{eq:2}.
\end{example}
 
For any $i$-dimensional defect label $x\in D_i$ in~$\D$ we can and do choose a representative $(n-i-1)$-sphere $S_x$ such that the adjacency map~$f_i$ sends~$x$ to the class of $S_x$, that is, $f_i(x) = [S_x]$. 
{}From~$S_x$ we produce a decorated stratified $n$-ball by first taking the cone and then iterated cylinders. 
More precisely, we define the $n$-dimensional open \textsl{$x$-defect ball} as 
\be
\label{eq:Bxball}
	 B_x := { C (S_x) \times (-1,1)^i }  \, \cap \, B^n  \ .
\ee
For later use we note that its closure is
\be
\label{eq:Bxball-closed}
	\overline B_x = \cc{ C (S_x) \times (-1,1)^i }  \, \cap \, \overline B^n \, ,
\ee
and we denote its boundary by
\be
\label{eq:SigmaxisdelBx}
	\Sigma_x = \partial \overline B_x \, .
\ee
The open ball $B_x$ describes what the surroundings of an $x$-labelled $i$-stratum look like in $\Bordd[n]$, and it will be important to us in Sections~\ref{subsec:point-defects-from} and~\ref{subsec:completewrtpointinsertions}. 

\subsubsection{Symmetric monoidal category of defect data}

We now turn the sets of defect data $\mathcal{D}^n$ into symmetric monoidal categories. 
We start by adding maps of defect data as morphisms:

\begin{definition}
\label{definition:map-def-data}
A \textsl{morphism of $n$-dimensional defect data} $h=(h_{n}, \ldots, h_{0})\colon \D \rightarrow \D'$ is a collection of maps $h_{i} \colon D_{i} \rightarrow D_{i}'$ such that for all $i \in \{0, \ldots, n\}$ the squares
\be
\begin{tikzpicture}[
			     baseline=(current bounding box.base), 
			     descr/.style={fill=white,inner sep=3.5pt}, 
			     normal line/.style={->}
			     ] 
\matrix (m) [matrix of math nodes, row sep=3.5em, column sep=4.0em, text height=1.5ex, text depth=0.1ex] {%
D_{i}  &  {[\Sphere_{n-i-1}(\partial^{i+1} \D)]}
\\
D_{i}'  &  {[\Sphere_{n-i-1}(\partial^{i+1}\D')]}
\\
};
\path[font=\footnotesize] (m-1-1) edge[->] node[above] {$f_{i}$} (m-1-2);
\path[font=\footnotesize] (m-1-1) edge[->] node[left] {$h_{i}$} (m-2-1);
\path[font=\footnotesize] (m-1-2) edge[->] node[right] {$h_{*}$} (m-2-2);
\path[font=\footnotesize] (m-2-1) edge[->] node[below] {$f_{i}'$} (m-2-2);
\end{tikzpicture}
\ee
commute, where $h_*\colon [\Sphere_{n-i-1}(\partial^{i+1} \D)] \rightarrow [\Sphere_{n-i-1}(\partial^{i+1} \D')]$ 
is the map of decorated spheres induced by the maps~$h_i$.
\end{definition}

Together with the above notion of morphism, the obvious identity morphisms and composition, the sets of defect data $\mathcal{D}^n$ form a category for all $n \in \Z_+$. 

\medskip

The monoidal structure on $\mathcal D^n$ is slightly less straightforward, as illustrated by the following example. 
Let $\D, \D' \in \mathcal D^n$ and pick defect labels $a \in D_{n-2}$ and $a' \in D_{n-2}'$ of codimension-two defects. 
Suppose that $f^\D_{n-2}(a)$ is a 1-sphere with two 0-strata labelled by $x,y \in D_{n-1}$ for $x \neq y$. 
Similarly, $f^{\D'}_{n-2}(a')$ is a 1-sphere with two 0-strata labelled by $x',y' \in D_{n-1}'$.
How many defect conditions should $a$ and $a'$ give rise to in the to-be-constructed tensor product $\D\otimes\D'$? 
One might guess that there should be one defect condition labelled by the pair $(a,a')$. However, there are now two distinct choices for what the neighbourhood of $(a,a')$ should look like. One is the 1-sphere with 0-strata labelled $(x,x')$ and $(y,y')$, and the other is the 1-sphere with 0-strata labelled $(x,y')$ and $(y,x')$. 
Thus in the tensor product $\D\otimes\D'$, we need two distinct defect labels corresponding to the pair $(a,a')$ which have to be distinguished by the adjacency maps $f^{\D\otimes\D'}_i$.

This example shows that the tensor product of defect data is more complicated than just the Cartesian product of the sets of defect labels. 
Instead, one has to take into account automorphisms of the defect sphere describing the surrounding arrangement of higher-dimensional defect strata. 
We now describe this construction in detail. 

For $\D, \D' \in  \mathcal D^n$ we define their tensor product $\D \otimes \D' \in \mathcal D^n$ to have component sets of the form
\be\label{eq:DDi}
(\D \otimes \D')_i = 
\Big\{ 
\big[ x,x', S \stackrel{\varphi}{\lra} U \stackrel{\;\,\varphi'}{\longleftarrow} S' \big]
\Big\} 
\ee
for all $i \in \{ 0,1,\dots,n \}$. 
Here $x\in D_i$ and $x'\in D'_i$ are defect labels, $S\in f^\D_i(x)$ and $S' \in f^{\D'}_i(x')$ are decorated defect spheres, $U \in \Sphere_{n-i-1}$ is an undecorated defect sphere, and $\varphi \colon S\to U$, $\varphi' \colon S' \to U$ are 
	isomorphisms 
of undecorated defect spheres. 
By definition, two such tuples 
\be
\big(x,x', S \stackrel{\varphi}{\lra} U \stackrel{\;\,\varphi'}{\longleftarrow} S'\big)
\, , \quad 
\big(x,x', \widetilde S \stackrel{\;\,\widetilde \varphi}{\lra} \widetilde U \stackrel{\;\,\widetilde \varphi'}{\longleftarrow} \widetilde S'\big)
\ee
represent the same equivalence class $[x,x', S \stackrel{\varphi}{\to} U \stackrel{\;\;\varphi'}{\leftarrow} S']$ if there are isomorphisms of decorated defect spheres $\sigma \colon S \to \widetilde S$, $\sigma' \colon S' \to \widetilde S'$ and an isomorphism of undecorated defect spheres $\nu \colon U \to \widetilde U$ such that the diagram 
\be
\begin{tikzpicture}[
			     baseline=(current bounding box.base), 
			     descr/.style={fill=white,inner sep=3.5pt}, 
			     normal line/.style={->}
			     ] 
\matrix (m) [matrix of math nodes, row sep=3.5em, column sep=4.0em, text height=0.5ex, text depth=0.1ex] {%
&&& \widetilde S && \widetilde S'
\\
S & & S' & & {\widetilde U} &
\\
& U &&&&
\\
};
\path[font=\footnotesize] (m-2-1) edge[->] node[above] {$\sigma$} (m-1-4);
\path[font=\footnotesize] (m-2-1) edge[->] node[above] {$\varphi$} (m-3-2);
\path[font=\footnotesize] (m-2-3) edge[->] node[above] {$\varphi'$} (m-3-2);
\path[font=\footnotesize] (m-1-4) edge[->] node[above, near start] {$\widetilde \varphi$} (m-2-5);
\path[font=\footnotesize] (m-1-6) edge[->] node[below, near start] {$\widetilde \varphi'$} (m-2-5);
\path[font=\footnotesize] (m-2-3) edge[-, commutative diagrams/crossing over] node[] {} (m-1-6);
\path[font=\footnotesize] (m-2-3) edge[->] node[above, near start] {$\sigma'$} (m-1-6);
\path[font=\footnotesize] (m-3-2) edge[->] node[below] {$\nu$} (m-2-5);
\end{tikzpicture}
\ee
commutes. 
Furthermore, we define the adjacency maps $f^{\D \otimes \D'}_i$ by setting 
\be
f^{\D \otimes \D'}_i
\Big(\big[x,x', S \stackrel{\varphi}{\lra} U \stackrel{\;\,\varphi'}{\longleftarrow} S'\big] \Big)
\,
\stackrel{\text{def}}{=}
\,
[U_{\varphi,\varphi'}] \, \in \, 
\Sphere_{n-i-1}(\partial^{i+1}(\D \otimes \D')) 
\ee
where $U_{\varphi,\varphi'}$ has underlying stratified manifold~$U$, and the decoration of the strata $U_j^\alpha$ is obtained from the pair $\varphi,\varphi'$.

We illustrate the above definition by spelling it out for the motivating example above: 
Let $U \subset \R^2$ be the standard circle with two $0$-strata at $(0,+1)$ and $(0,-1)$. 
For the defect circle~$S$ with $f^\D_{n-2}(a) = [S]$ we take~$U$ with the labels $x,y \in D_{n-1}$ at $(0,+1)$ and $(0,-1)$, respectively, while~$S'$ with $f^{\D'}_{n-2}(a') = [S']$ is~$U$ with $(0,+1)$, $(0,-1)$ decorated by $x', y' \in D'_{n-1}$. 
Then in $\D \otimes \D'$ we have two inequivalent elements $( a,a', S \stackrel{\id}{\to} U \stackrel{\;\,\id}{\leftarrow} S')$ and $( a,a', S \stackrel{\id}{\to} U \stackrel{\;\,\pi}{\leftarrow} S')$, where $\pi$ is rotation by $\pi$.

\medskip

Given two morphisms $h \colon \D \to \Ee$ and $g \colon \D' \to \Ee'$ in $\mathcal D^n$, we define their tensor product $h\otimes g \colon \D \otimes \D' \to \Ee \otimes \Ee'$ by setting for all $i\in \{ 0,1,\dots,n \}$: 
\be\label{eq:htensorgi}
(h\otimes g)_i 
\Big( 
\big[ x,x', S \stackrel{\varphi}{\lra} U \stackrel{\;\,\varphi'}{\longleftarrow} S' \big]
\Big)
\,
\stackrel{\text{def}}{=}
\,
\Big[ h_i(x) , g_i(x'), h_*(S) \stackrel{\varphi}{\lra} U \stackrel{\;\,\varphi'}{\longleftarrow} g_*(S') \Big]
. 
\ee
Here the induced maps $h_*$ and $g_*$ are as in Definition~\ref{definition:map-def-data}, and it follows that $(h\otimes g)_* \circ f^{\D\otimes\D'}_i 
= f^{\Ee\otimes\Ee'}_i \circ (h\otimes g)_i$, so that $h\otimes g$ really is a morphism in~$\mathcal D^n$. 

The monoidal unit in $\mathcal D^n$ is $\D^{\id}$ whose sets
\be
\label{eq:Didi}
D^{\id}_{i} := [\Sphere_{n-i-1}]
\ee
consist of all defect spheres for all $i \in \{0,1,\dots,n \}$, and whose adjacency maps are identities, $f^{\D^{\id}}_i([S]) = [S]$. 

Finally we equip $\mathcal{D}^n$ with the symmetric braiding given by the exchange of factors.

\begin{lemma}\label{lem:Dn-sym-mon-cat}
$\mathcal D^n$ is a symmetric monoidal category. 
\end{lemma}

\begin{proof}
We show that for the tensor product given in~\eqref{eq:DDi}, \eqref{eq:htensorgi} and~\eqref{eq:Didi} there are natural unitors $\D \otimes \D^{\id} \cong \D \cong \D^{\id} \otimes \D$; the associator is straightforward to determine, and the pentagon and hexagon are then clear. 

For any $\D \in \mathcal D^n$, the isomorphism $\rho_\D \colon \D \otimes \D^{\id} \to \D$ has components which project $[x,[\Sigma], S \stackrel{\varphi}{\to} U \stackrel{\varphi'}{\leftarrow} \Sigma] \in (\D \otimes \D^{\id})_i$ to $x \in D_i$. 
To give the action of $\rho^{-1}_\D$, we choose a representative $\widetilde S \in f_i(x)$, and we write $\widetilde S_\circ \in \Sphere_{n-i-1}$ for~$\widetilde S$ with its $\D$-decoration discarded. 
Then
\be
(\rho^{-1}_\D)_i(x) 
\;
\stackrel{\text{def}}{=} 
\;
\Big[ x ,[\widetilde S_\circ], \widetilde S \stackrel{\id}{\lra} \widetilde S_\circ \stackrel{\;\,\id}{\longleftarrow} \widetilde S_\circ \Big]
.
\ee
Note that the right-hand side is independent of the choice of representative $\widetilde S$.

Clearly we have $\rho_\D \circ \rho^{-1}_\D = 1_\D$. 
To see that also $\rho^{-1}_\D \circ \rho_\D$ equals $1_{\D\otimes\D^{\id}}$, simply act with the former on $(x,[\Sigma], S \stackrel{\varphi}{\to} U \stackrel{\;\,\varphi'}{\leftarrow} \Sigma)$ and note that the following diagram commutes: 
\be
\begin{tikzpicture}[
			     baseline=(current bounding box.base), 
			     descr/.style={fill=white,inner sep=3.5pt}, 
			     normal line/.style={->}
			     ] 
\matrix (m) [matrix of math nodes, row sep=3.5em, column sep=4.0em, text height=0.5ex, text depth=0.1ex] {%
&&& S && S_\circ
\\
S & & \Sigma & & {S_\circ} &
\\
& U &&&&
\\
};
\path[font=\footnotesize] (m-2-1) edge[->] node[above] {$\id$} (m-1-4);
\path[font=\footnotesize] (m-2-1) edge[->] node[above] {$\varphi$} (m-3-2);
\path[font=\footnotesize] (m-2-3) edge[->] node[above] {$\varphi'$} (m-3-2);
\path[font=\footnotesize] (m-1-4) edge[->] node[above, near start] {$\id$} (m-2-5);
\path[font=\footnotesize] (m-1-6) edge[->] node[below, near start] {$\id$} (m-2-5);
\path[font=\footnotesize] (m-2-3) edge[-, commutative diagrams/crossing over] node[] {} (m-1-6);
\path[font=\footnotesize] (m-2-3) edge[->] node[above, near start] {$\varphi^{-1} \circ \varphi'$} (m-1-6);
\path[font=\footnotesize] (m-3-2) edge[->] node[below] {$\varphi^{-1}$} (m-2-5);
\end{tikzpicture}
\ee
\end{proof}

\subsection[$n$-dimensional defect TQFTs]{$\boldsymbol{n}$-dimensional defect TQFTs}
\label{subsec:ndTQFTs}

We will now define defect TQFTs in arbitrary dimension, describe maps between them, and discuss their tensor products. 
Unless specified otherwise, in this section~$n$ is any positive integer, and~$\D$ is a
set of $n$-dimensional defect data.

\subsubsection{Defect TQFTs and their morphisms}

\begin{definition}
\label{def:defeTQFT}
 An \textsl{$n$-dimensional defect TQFT} with defect data~$\D$ is a symmetric monoidal functor 
\be
  \label{eq:def-TQFT}
  \zz\colon \Borddefn{n}(\D) \longrightarrow \Vectk. 
\ee
\end{definition}

The most basic example of a defect TQFT for fixed defect data~$\D$ is the \textsl{identity TQFT} 
\be
\label{eq:idTQFT}
\unit_{\D} \colon \Bordd[n] \lra \Vectk
\ee
which by definition maps all objects to $\Bbbk$, and all morphisms to $\id_\Bbbk$.
Closed $n$-dimensional TQFTs, i.\,e.~symmetric monoidal functors $\Bord_{n} \rightarrow \Vectk$, provide another class 
	of 
examples by viewing them as functors on $\Bord_n^{\text{def}}(\D^*)$, where the set of defect data $\D^*$ consisting of the singleton set $D^*_{n}=\{ \ast \}$ and $D^*_{j}=\emptyset$ for all $j \in \{0,1, \ldots, n-1\}$.

Another class of examples of defect TQFTs comes about by ``compactifying'' higher-dimensional defect TQFTs. 
Indeed, thanks to the prominent role of the boundary map~$\partial$ in Definition~\ref{def:d-dim-defect-data} of defect data, the following observation is immediate: 

\begin{lemma}
\label{lemma:compactification}
Let $\zz$ be an $n$-dimensional defect TQFT with defect data $\D$, and let~$M$ 
	be 
a closed $k$-manifold with $k < n$. 
Then the \textsl{compactification of $\zz$ along $M$} is an $(n-k)$-dimensional defect TQFT $\zz_{M}\colon \Borddefn{n-k}(\partial^{k}\D)\rightarrow \Vectk$ with 
\be
  \label{eq:compactification}
  \zz_{M}(N)=\zz(N \times M)
\ee
on both objects and morphisms~$N$ in $\Borddefn{n-k}(\partial^{k}\D)$; here the decorated stratification of~$N$ naturally induces the decorated stratification of $N \times M$.
\end{lemma}

\medskip

Defect TQFTs in dimension~$n$ form the objects of a symmetric monoidal category $\Deftqft_{n}$ to which we turn next. 
We start by defining the morphisms between two defect TQFTs $\zz \colon \Bordd[n] \to \Vectk$ and $\zz' \colon \Bord_n^{\text{def}}(\D')\to \Vectk$. If $\D=\D'$ one can in particular consider monoidal natural transformations $\zz \Rightarrow \zz'$.
In general, one has to take into account maps of defect data $h\colon \D \to \D'$ as in Definition~\ref{definition:map-def-data}, which induce symmetric monoidal functors 
\be
  h_{*} \colon \Borddefn{n}(\D) \lra \Borddefn{n}(\D')
\ee
by applying the component maps $h_{i}$ to the decoration of the strata of the objects and morphisms in $\Borddefn{n}(\D)$.
Hence we define a \textsl{morphism} from $\zz$ to $\zz'$ to be a pair $(h, \varphi)$ where $h\colon \D \rightarrow \D'$ is a map of defect data and $\varphi\colon \zz \Rightarrow \zz' \circ h_{*}$ is a monoidal natural transformation: 
\be
\begin{tikzpicture}[
			     baseline=(current bounding box.base), 
			     descr/.style={fill=white,inner sep=3.5pt}, 
			     normal line/.style={->}
			     ] 
\matrix (m) [matrix of math nodes, row sep=4.5em, column sep=4.0em, text height=1.5ex, text depth=0.1ex] {%
\Borddefn{n}(\D) & 
\\
\Borddefn{n}(\D') & \Vectk 
\\
};
\path[font=\footnotesize] (m-1-1) edge[->] node[above] {$\zz$} (m-2-2);
\path[font=\footnotesize] (m-2-1) edge[->] node[below] {$\zz'$} (m-2-2);
\path[font=\footnotesize] (m-1-1) edge[->] node[left] {$h_{*}$} (m-2-1);
\fill[color=black, font=\footnotesize] (-0.52,-0.1) circle (0pt) node (0up) {  $ \varphi  $ };
\fill[color=black] (-0.35,-0.3) circle (0pt) node (0up) {  $ \rotatebox[origin=c]{230}{$\implies$}  $ };
\end{tikzpicture}
\ee 
In the same way as for closed TQFTs without defects one proves that~$\varphi$ in the above definition is necessarily invertible (see e.\,g.~\cite{CRlecturenotes}); note however that this does not imply $\Deftqft_{n}$ is a groupoid as $h_*$ need not have an inverse. 

The composition of morphisms $(h, \varphi) \colon \zz \to \zz'$ and $(g, \psi) \colon \zz' \to \zz''$ is defined by composing~$h_*$ with~$g_*$ and pasting~$\varphi$ and~$\psi$ together, 
\be
\begin{tikzpicture}[
			     baseline=(current bounding box.base), 
			     descr/.style={fill=white,inner sep=3.5pt}, 
			     normal line/.style={->}
			     ] 
\matrix (m) [matrix of math nodes, row sep=4.5em, column sep=4.0em, text height=1.5ex, text depth=0.1ex] {%
\Borddefn{n}(\D) & & 
\\
\Borddefn{n}(\D') & \Vectk & 
\\
\Borddefn{n}(\D'') & & \Vectk \, .
\\
};
\path[font=\footnotesize] (m-1-1) edge[->] node[above] {$\zz$} (m-2-2);
\path[font=\footnotesize] (m-2-1) edge[->] node[below] {$\zz'$} (m-2-2);
\path[font=\footnotesize] (m-3-1) edge[->] node[below] {$\zz''$} (m-3-3);
\path[font=\footnotesize] (m-1-1) edge[->] node[left] {$h_{*}$} (m-2-1);
\path[font=\footnotesize] (m-2-1) edge[->] node[left] {$g_{*}$} (m-3-1);
\path[font=\footnotesize] (m-2-2) edge[commutative diagrams/equal] node[left] {} (m-3-3);
\fill[color=black, font=\footnotesize] (-1.85,0.95) circle (0pt) node (0up) {  $ \varphi  $ };
\fill[color=black] (-1.65,0.75) circle (0pt) node (0up) {  $ \rotatebox[origin=c]{230}{$\implies$}  $ };
\fill[color=black, font=\footnotesize] (-0.27,-1.35) circle (0pt) node (0up) {  $ \psi  $ };
\fill[color=black] (-0.65,-1.25) circle (0pt) node (0up) {  $ \rotatebox[origin=c]{230}{$\implies$}  $ };
\end{tikzpicture}
\ee
The identity morphism on $\zz \colon \Bordd[n] \to \Vectk$ is $(\id_\D, \id_\zz)$. Altogether, we have a category $\Deftqft_{n}$.

\subsubsection{Symmetric monoidal category of defect TQFTs}

To describe the monoidal structure on $\Deftqft_{n}$, we have to produce symmetric monoidal functors on $\Bord_n^{\text{def}}(\D \otimes \D')$ from defect TQFTs $\zz \colon \Bordd[n] \to \Vectk$ and $\zz' \colon \Bord_n^{\text{def}}(\D') \to \Vectk$, where $\D\otimes\D'$ is the tensor product in~$\mathcal D^n$ defined in Section~\ref{subsec:defectbords}. 
For an object or morphism~$M$ in $\Bord_n^{\text{def}}(\D \otimes \D')$
 we write $p_1(M)$ for the object or morphism in $\Bordd[n]$ which is~$M$ but with the decorations from~$\D'$ forgotten. 
Similarly, we write $p_2(M)$ for the object or morphism in $\Bord_n^{\text{def}}(\D')$ which is obtained by discarding the $\D$-decoration. 
This gives us a symmetric monoidal functor 
\be
P \colon \Bord_n^{\text{def}}(\D \otimes \D') \lra \Bordd[n] \times \Bord_n^{\text{def}}(\D') 
\, , \quad 
M \lmt \big( p_1(M), p_2(M) \big) 
\, . 
\ee
With this notation we have a natural notion of tensor product of defect TQFTs: 

\begin{definition}
\label{definition:tensor-prod-defect-TQFT}
Let $\zz$ and $\zz'$ be two $n$-dimensional defect TQFTs with defect data $\D$ and $\D'$, respectively. 
Their \textsl{tensor product} is the defect TQFT 
\begin{align}
\zz \otimes \zz' \colon \Bord_n^{\text{def}}(\D \otimes \D') 
& \stackrel{P}{\lra} \Bordd[n] \times \Bord_n^{\text{def}}(\D')  
\nonumber \\
& \stackrel{\zz \times \zz'}{\lra} \Vectk \times \Vectk 
\nonumber \\
& \stackrel{\otimes_\Bbbk}{\lra} \Vectk \, . 
\end{align}
\end{definition}

It follows that we have 
\be
 ( \zz  \otimes \zz') (M)= \zz(p_{1}(M)) \otimes_\Bbbk \zz'(p_{2}(M))
\ee
for objects and morphisms $M$ in $\Borddefn{n}(\D \otimes \D')$. 
The unit for the tensor product is the identity defect TQFT 
\be\label{eq:unit-defect-TQFT}
\unit := \unit_{\D^{\id}}
\ee 
which is the special case of~\eqref{eq:idTQFT} where $\D = \D^{\id}$ is the monoidal unit of~$\mathcal D^n$ with $D^{\id}_{i} = [\Sphere_{n-i-1}]$, cf.~\eqref{eq:Didi}. 
This means that for each local neighbourhood there exists precisely one defect label 
in $\D^{\id}$, that is, the functor
\be\label{eq:identity-datum=unlabelled}
	\Bord_n^{\text{def}}(\D^{\id}) \xrightarrow{\text{forget}} \Bord_n^{\text{def}}
\ee
is an equivalence.

To describe the tensor product also on morphisms in $\Deftqft_{n}$, let us consider defect TQFTs 
\begin{align}
\zz \colon \Bordd[n] 
& \lra \Vectk 
\, , 
& \zz' \colon \Bord_n^{\text{def}}(\D') \lra \Vectk 
\, , 
\nonumber \\
\mathcal Y \colon \Bord_n^{\text{def}}(\Ee) 
& \lra \Vectk 
\, ,  
& \mathcal Y' \colon \Bord_n^{\text{def}}(\Ee') \lra \Vectk
\, . 
\end{align}
For two morphisms $(g,\psi) \colon \mathcal Y \to \mathcal Y'$ and $(h,\varphi) \colon \zz \to \zz'$, their tensor product $(g,\psi) \otimes (h,\varphi) \colon \mathcal Y \otimes \zz \to \mathcal Y' \otimes \zz'$ is defined to be the commutative diagram 
\be
\label{eq:morphtensor}
\begin{tikzpicture}[
			     baseline=(current bounding box.base), 
			     descr/.style={fill=white,inner sep=3.5pt}, 
			     normal line/.style={->}
			     ] 
\matrix (m) [matrix of math nodes, row sep=4.5em, column sep=2.5em, text height=1.5ex, text depth=0.1ex] {%
\!\!\! \Borddefn{n}(\Ee \otimes \D) &&&
\\
\!\!\! & \!\!\!\!\!\!\Borddefn{n}(\Ee) \times \Borddefn{n}(\D) && 
\\
\!\!\! & \!\!\!\!\!\!\Borddefn{n}(\Ee') \times \Borddefn{n}(\D') & \Vectk \times \Vectk & 
\\
\!\!\! \Borddefn{n}(\Ee' \otimes \D') &&& \Vectk . 
\\
};
\path[font=\footnotesize] (m-1-1) edge[->] node[left] {$(g\otimes h)_*$} (m-4-1);
\path[font=\footnotesize] (m-1-1) edge[->] node[above] {$P$} (m-2-2);
\path[font=\footnotesize] (m-4-1) edge[->] node[below] {$P$} (m-3-2);
\path[font=\footnotesize] (m-2-2) edge[->] node[left] {$g_* \times h_*$} (m-3-2);
\path[font=\footnotesize] (m-4-1) edge[->] node[below] {$\mathcal Y' \otimes \zz'$} (m-4-4);
\path[font=\footnotesize] (m-2-2) edge[->] node[above] {$\;\;\;\;\;\;\; \mathcal Y \times \zz$} (m-3-3);
\path[font=\footnotesize] (m-3-2) edge[->] node[below] {$\mathcal Y' \times \zz'$} (m-3-3);
\path[font=\footnotesize] (m-1-1) edge[->, out=0, in=90] node[above, sloped] {$\mathcal Y \otimes \zz$} (m-4-4);
\path[font=\footnotesize] (m-3-3) edge[->] node[above] {$\otimes_\Bbbk$} (m-4-4);
\fill[color=black, font=\footnotesize] (0.3,-0.1) circle (0pt) node (0up) {  $ \psi \times \varphi $ };
\fill[color=black] (0.8,-0.3) circle (0pt) node (0up) {  $ \rotatebox[origin=c]{235}{$\implies$}  $ };
\end{tikzpicture}
\ee
It is straightforward to verify that the identity defect TQFT~$\unit$ is the monoidal unit, and again unitors, associator and symmetric braiding are given by the obvious choice. 
To summarise: 

\begin{proposition}
\label{prop:TQFTissymmon}
The category $\Deftqft_{n}$ of $n$-dimensional defect TQFTs is symmetric monoidal. 
\end{proposition}

\medskip

There is an equivalence relation on the set of objects in $\Deftqft_{n}$ which will be relevant for us in Sections~\ref{subsec:point-defects-from} and~\ref{subsec:completewrtpointinsertions}: 

\begin{definition}
\label{def:ZZsequiv}
Two TQFTs $\zz \colon \Bordd[n] \to \Vectk$ and $\zz' \colon \Bord_n^{\text{def}}(\D') \to \Vectk$ are \textsl{equivalent}, $\zz \sim \zz'$, if there are morphisms $(h,\varphi) \colon \zz \rightleftarrows \zz' :\!\!(g,\psi)$ between them. 
\end{definition}

This equivalence relation is compatible with the monoidal structure on $\Deftqft_{n}$ in the sense that if $\zz \sim \zz'$ and $\mathcal Y \sim \mathcal Y'$, then also $\mathcal Y \otimes \zz \sim \mathcal Y' \otimes \zz'$. 
The morphisms witnessing the latter equivalence are constructed from those of the former equivalences and the tensor product defined in~\eqref{eq:morphtensor}. 

\begin{remark}
\begin{enumerate}
\item
If $\zz, \zz'$ are equivalent as in Definition~\ref{def:ZZsequiv} then there are isomorphisms of symmetric monoidal functors $\zz \cong \zz' \circ h_*$ and $\zz' \cong \zz \circ g_*$ and thus also  $\zz \cong \zz \circ g_* \circ h_*$ and $\zz' \cong \zz' \circ h_* \circ g_*$. 
In this sense equivalent TQFTs determine one another and thus deserve to be called equivalent. 
\item
If $\zz$ and $\zz'$ are isomorphic in $\Deftqft_{n}$ then they are clearly also equivalent in the sense of Definition~\ref{def:ZZsequiv}. The converse is not true: a morphism $(h,\phi) \colon \zz \to \zz'$ is an isomorphism if and only if $h$ is invertible and in Sections~\ref{subsec:point-defects-from} and~\ref{subsec:completewrtpointinsertions} we will encounter examples of equivalent defect TQFTs with non-isomorphic sets of defect labels.
\item
A natural mechanism to produce additional equivalences in a 1-category is to add a layer of 2-morphisms, thereby turning it into a 2-category, and then passing to its homotopy 1-category. 
In $\Deftqft_{n}$, for example, one could add a single 2-(iso)morphism between any two 1-morphisms. 
In the homotopy category, all 1-morphisms then 
lie in one equivalence class, and this produces Definition~\ref{def:ZZsequiv}. 
Finding a more natural 2-categorical structure on defect TQFTs is a problem for future research.
\end{enumerate}
\end{remark}

An \textsl{invertible $n$-dimensional defect TQFT} is an invertible object~$\zz$ in the monoidal category $\Deftqft_{n}$. 
This means that $\zz$ is invertible if and only if there exists an $n$-dimensional defect TQFT $\zz'$ and an isomorphism of defect TQFTs $\zz \otimes \zz' \cong \unit$.
A particular class of invertible defect TQFTs in any dimension~$n$ are ``Euler defect TQFTs'': 

\begin{example}
\label{example:Eulerx}
Let us for the moment restrict the discussion to the field $\Bbbk=\C$ and later return to general fields.
Recall that for a triangulated topological manifold~$M$ its Euler characteristic $\chi(M)$ can be computed as the alternating sum of the number of $k$-simplices over all dimensions $k \leqslant \dim M$.
{}From this characterisation it is immediate that if $M \circ_{\Sigma} N$ is the result of gluing two triangulated manifolds~$M$ and~$N$ along a common boundary $\Sigma$, then
\begin{equation}
  \chi(M \circ_{\Sigma}N)=\chi(M)+ \chi(N) - \chi(\Sigma) \, .
\end{equation}
Following~\cite{Quinnlectures}, for any $n\in \Z_+$, $\kappa,\lambda \in \C$, 
one thus obtains a closed TQFT $\zz_{(\kappa,\lambda)} \colon \Bord_n \to \Vect_\C$ which maps all objects to~$\C$, and on a bordism~$M$ with boundary $(\partial_{\text{in}} M)^{\text{rev}} \cup \partial_{\text{out}}M$ we set 
\be
\zz_{(\kappa,\lambda)}(M) =  \exp\Big\{ \kappa \big(\chi(M)-\lambda \cdot \chi(\partial_{\text{in}}M) - (1-\lambda) \cdot \chi( \partial_{\text{out}}M) 
\big)\Big\} 
\, .
\ee 
These exponents ensure that $\zz_{(\kappa,\lambda)}(M \circ_{\Sigma} N)=\zz_{(\kappa,\lambda)}(M) \cdot \zz_{(\kappa,\lambda)}(N)$, and one can verify that $\zz_{(\kappa,\lambda)} \cong \zz_{(\kappa,\lambda')}$ for all $\lambda, \lambda' \in \C$.
Clearly, $\zz_{(\kappa,\lambda)}$ is an invertible closed TQFT with inverse $\zz_{(-\kappa,\lambda)}$.

For later convenience we prefer the symmetric choice $\lambda=\frac{1}{2}$, as this attaches the same weight to in- and outgoing boundaries and reduces the risk of confusion. To shorten notation and avoid lots of factors of two, we will use the following rescaled version of the Euler character:
\be
\label{eq:chireldef}
\chirel(M) := 2\chi(M) - \chi(\partial M)  \, ,
\ee 
such that $\zz_{(\kappa, \frac{1}{2})}(M)= \text{e}^{\kappa \cdot \chirel(M)/2}$. 
We will also return to the case of general fields $\Bbbk$ by replacing $\E^\kappa$ in the above expression by $\Psi^2$ for some $\Psi \in \Bbbk^\times$. Altogether, we define the closed \textsl{Euler theory} to be the TQFT
\be
\label{eq:closedEuler}
\zz^{\text{eu}}_\Psi \colon \Bord_n \lra \Vectk 
\, , \quad
\zz^{\text{eu}}_\Psi (M) = \Psi^{\chirel(M)} \, . 
\ee

We can lift this example to an invertible $n$-dimensional defect TQFT 
\be\label{eq:Euler-defect-TQFT}
\zz^{\text{eu}}_{\Psi} \colon 
	\Bord_n^{\text{def}}(\D^{\text{eu}}) 
\lra \Vectk
\ee 
for any tuple $\Psi=(\psi_{1}, \ldots,\psi_{n}) \in (\Bbbk^{\times} )^{n}$ as follows. 
The defect data 
	$\D^{\text{eu}}$ 
for $\zz^{\text{eu}}_{\Psi}$ is the one for the monoidal unit~$\unit$  as in~\eqref{eq:Didi} (and hence by \eqref{eq:identity-datum=unlabelled} we can also think of $\zz^{\text{eu}}_{\Psi}$ as an unlabelled defect TQFT, i.\,e.\ a symmetric monoidal functor $\Bord_n^{\text{def}} \to \Vectk$).

For all objects~$\Sigma$ we set $\zz^{\text{eu}}_{\Psi}(\Sigma) = \Bbbk$. 
For a morphism~$M$ in $\Bord_n^{\text{def}}(\D^{\text{eu}})$ as before we write $M_{j}^{\alpha_{j}}$ for its $j$-strata, and we recall that $\partial M_{j}^{\alpha_{j}} = \partial M \cap M_{j}^{\alpha_{j}}$. 
Then we define the \textsl{Euler defect TQFT} on~$M$ as the natural generalisation of~\eqref{eq:closedEuler}, by assigning a \textsl{weight} $\psi_j^{\chirel({M}_{j}^{\alpha_{j}})}$ to every stratum $M_{j}^{\alpha_{j}}$ of dimension $j \geqslant 1$: 
\be
\label{eq:def-n-Euler}
\zz^{\text{eu}}_{\Psi}(M) = \prod_{j=1}^n \prod_{\alpha_{j}}  \psi_j^{\chirel({M}_{j}^{\alpha_{j}})} \, .
\ee
\end{example}

\subsection{Point defects from states on spheres}
\label{subsec:point-defects-from}

Suppose we are given a defect TQFT $\zz \colon \Bordd[n] \to \Vectk$ with some set $D_0$ of labels for point defects. If one ``regularises'' a point defect by replacing it by a small sphere around that point, one can interpret
certain states in the vector space associated to that sphere as being located on the point. 
To allow for such an interpretation, the states need to satisfy an invariance condition given below. 

In this section we define when a defect TQFT has a ``complete set of point defects'' (which we will call $D_0$-complete) and we will show that every defect TQFT factors through a $D_0$-complete one. 
For a $D_0$-complete theory we describe a multiplication of point defects in which one replaces two neighbouring point defects by a ``fused'' point defect.

\medskip

We first need to define the notion of an ``invariant state'' in the state space of a defect sphere. 
Let $\Sigma \in \Sphere_{n-1}(\partial \D)$ be given. Write $\overline B := C(\Sigma) \subset \R^n$ 
for the closed unit ball given by the cone over $\Sigma$, as a decorated stratified manifold but without a label for its central 0-stratum.
Note that $\partial \overline B = \Sigma$.
Similarly, for $\varepsilon>0$ write $B_{1+\eps}$ for the open ball of radius $1+\eps$ which is  given by the inner of the (now slightly larger) cone over $\Sigma$. By construction, $\overline B \subset B_{1+\eps}$ as decorated stratified manifolds.

Define the set of 
embeddings
\be
	\mathrm{Emb}(\overline B)
\ee
of $\overline B$ into itself as follows: 
	an element of $\mathrm{Emb}(\overline B)$ is a germ (in $\eps>0$) of maps $f\colon B_{1+\eps} \to \overline B$ of decorated stratified manifolds which are isomorphisms onto their images; as for boundary parametrisations we will write $f$ both for the germ and for a representative map. Composition of maps is independent of representatives and turns $\mathrm{Emb}(\overline B)$ into a non-unital semigroup.

Let $B$ be the inner of $\overline B$. Given an element $f \in \mathrm{Emb}(\overline B)$ define a bordism $H_f \colon \Sigma \to \Sigma$ in $\Bordd[n]$ as
\be
	H_f := \overline B \setminus f(B) \, ,
\ee
where the outer boundary is the outgoing boundary and is parametrised by $\Sigma$ via the identity map
	(with its canonical germ on $\Sigma \times (1-\eps,1]$, where the second coordinate is the radius),
while the inner boundary is the ingoing boundary and is parametrised by $\Sigma$ via 
	the germ $f$ restricted to $\Sigma \times [1,1+\eps)$.
For $f,g \in \mathrm{Emb}(\overline B)$, composition of bordisms is compatible with the semigroup structure on $\mathrm{Emb}(\overline B)$,
\be
	H_f \circ H_g = H_{f \circ g} \, 
\ee
in $\Bordd[n]$.
The \textsl{subspace of invariant states in $\zz(\Sigma)$} is defined to be
\be
\label{eq:subsp-inv-sta}
\mathcal{Y}_\Sigma
:=
\big\{ \, \psi \in \zz(\Sigma) \,\big|\,
\zz(H_f)(\psi) = \psi
\text{ for all } f \in \mathrm{Emb}(\overline B) \,\big\} \, .
\ee
One reason to introduce $\mathcal{Y}_\Sigma$ is that point defects give rise to such invariant states, as proved in the next lemma.
Conversely, if one tries to describe point defects by cutting out small balls and assigning states to the resulting boundary spheres, one has to make sure that the result is independent of the chosen boundary parametrisation. 
We will see in the proof of Proposition~\ref{prop:point-defect-completed} how this is ensured by the above invariance condition.

\begin{lemma}\label{lem:points-are-invariant}
Let $x \in D_0$ be such that $f_0(x) = [\Sigma]$ and let $\overline B_x$ be the closed cone $C(\Sigma)$ as above, but with central point labelled $x$. Then $\overline B_x$ is a bordism $\emptyset \to \Sigma$ and we have\footnote{Since $\overline B_x$ is a morphism from $\emptyset$ to $\Sigma$, $\zz(\overline B_x)$ is a linear map from $\Bbbk \to \zz(\Sigma)$. 
We will identify linear maps $\Bbbk \to \zz(\Sigma)$ with $\zz(\Sigma)$ by evaluating on $1 \in \Bbbk$. 
This is to avoid writing lots of ``$(1)$'', e.\,g.\ \eqref{eq:points-are-invariant} would read $\zz(\overline B_x)(1) \in \mathcal{Y}_\Sigma$.}
\be\label{eq:points-are-invariant}
	\zz(\overline B_x) \in \mathcal{Y}_\Sigma \, .
\ee
\end{lemma}

\begin{proof}
We have $\zz(H_f) \circ \zz(\overline B_x) = \zz(H_f \circ \overline B_x)$. It is therefore enough to show that $H_f \circ \overline B_x$ is 
	isomorphic 
to $\overline B_x$ as a decorated stratified manifold with parametrised boundary. But this is clear from the construction: 
the 
	isomorphism 
$\varphi \colon H_f \circ \overline B_x \to \overline B_x$ is given by $\varphi|_{H_f} = \id \colon H_f \to H_f \subset \overline B_x$ and $\varphi|_{B_x} = f \colon B_x \to f(B_x) \subset \overline B_x$.
\end{proof}

After these preparations we can introduce the notion of $D_0$-completeness. For $\Sigma \in \Sphere_{n-1}(\partial \D)$ define the subset
\be
	D_0(\Sigma) := \big\{ x \in D_0 \,\big|\, f_0(x) = [\Sigma] \,\big\} \, ,
\ee
that is, $D_0(\Sigma)$ contains 
all point defect labels that can be assigned to the 0-stratum in the cone $C(\Sigma)$. By Lemma~\ref{lem:points-are-invariant} we obtain a map
\be\label{eq:point-def-stat-corr}
	Y_\Sigma \colon D_0(\Sigma) \lra \mathcal{Y}_\Sigma
	\, , \quad
	x \longmapsto \zz(\overline B_x) \, ,
\ee
which assigns to a point defect label 
the corresponding state on the surrounding sphere. 
This map can be thought of as a variant of the state-field correspondence one has in conformal field theories.

\begin{definition}
\label{def:D0complete}
A defect TQFT $\zz \colon \Bordd[n] \to \Vectk$ is called \textsl{$D_0$-complete} if
for all $\Sigma \in \Sphere_{n-1}(\partial \D)$,
the map \eqref{eq:point-def-stat-corr} is a bijection of sets.
\end{definition}

For a $D_0$-complete TQFT, the spaces $D_0(\Sigma)$ inherit the structure of a $\Bbbk$-vector space via the bijection \eqref{eq:point-def-stat-corr}, and we will use this vector space structure for $D_0(\Sigma)$ below.

The next proposition shows that working with $D_0$-complete theories is not a restriction, as every defect TQFT factors through such a theory.

\begin{proposition}\label{prop:point-defect-completed}
For a given defect TQFT $\zz \colon \Bordd[n] \to \Vectk$
there exist defect data $\D^\bullet$, a map of defect data $h \colon \D \to \D^\bullet$, and a defect TQFT $\zz^\bullet \colon \Borddefn{n}(\D^\bullet) \to \Vectk$ such that $\zz^\bullet$ is $D_0$-complete and
\be
	\zz = \zz^\bullet \circ h_* \ .
\ee
\end{proposition}

\begin{proof}
We start by introducing the point-completed defect data $\D^\bullet$. 
It will differ from $\D$ only in the label set and adjacency map for 0-strata. 
Choose a subset 
\be
\label{eq:schickS}
\mathcal{S} \subset \Sphere_{n-1}(\partial \D)
\ee
of representatives of the classes in $[\Sphere_{n-1}(\partial \D)]$, that is, for each $S \in [\Sphere_{n-1}(\partial \D)]$ there exists a unique $\Sigma \in \mathcal{S}$ such that $S = [\Sigma]$.
An element of $D_0^\bullet$ is a pair consisting of a defect sphere $\Sigma \in \mathcal{S}$ 
and a state $\psi \in \mathcal{Y}_\Sigma$, while $f_0^\bullet$ simply forgets the state:
\be
D_0^\bullet
:=
\big\{ \, (\Sigma,\psi) \,\big|\, \Sigma \in \mathcal{S} , \, \psi \in \mathcal{Y}_\Sigma \,\big\} \, ,
\ee
and
\be
f^\bullet_0 \colon 
	D^\bullet_0 \times \{+\}
\lra [\Sphere_{n-1}(\partial \D)]
\,  , \quad
(\Sigma,\psi) \lmt [\Sigma] \, ,
\ee
while $f^\bullet_0$ is determined on $D^\bullet_0 \times \{-\}$ by \eqref{eq:adjacency-map-duality}.
Finally, $\D^\bullet$ is given by
\be
	\D^\bullet :=
	(D_n, \dots ,D_1,D^\bullet_0 ;
	f_{n-1},\dots,f_1,f_0^\bullet) \, .
\ee
By Lemma~\ref{lem:points-are-invariant}, there is a canonical map of defect data $h \colon \D \to \D^\bullet$ (recall Definition~\ref{definition:map-def-data}) which is given by the identity on $D_i$ for $i \in \{1,\dots,n \}$ and by 
\be\label{eq:D-bullet-compare-h0}
	h_0 \colon D_0 \lra D^\bullet_0
\, , \quad
	x \lmt Y_\Sigma(x) \, ,
\ee
where $\Sigma \in \mathcal{S}$ is such that $f_0(x,+)=[\Sigma]$, i.\,e.\ $x \in D_0(\Sigma)$, and $Y_\Sigma$ was defined in \eqref{eq:point-def-stat-corr}.

We can now describe the $D_0$-complete defect TQFT
\be
\zz^\bullet \colon \Borddefn{n}(\D^\bullet) \lra \Vectk
\ee 
obtained from $\zz$ by extending from $\D$ to $\D^\bullet$. 
Note that the objects of $\Borddefn{n}(\D^\bullet)$ are the same as those of $\Borddefn{n}(\D)$, and so on objects~$X$ we set
\be\label{eq:Z-bullet-obj}
	\zz^\bullet(X) := \zz(X) \, .
\ee

Let 
\be
M \colon X \lra Y
\ee
be a bordism in $\Borddefn{n}(\D^\bullet)$. 
Let $p \in M_0$ be a $0$-stratum of~$M$ and let $(\Sigma_p,\psi_p) \in D_0^\bullet$ be its label. 
Write $\overline B_p$ for the closed cone $C(\Sigma_p)$ with central point labelled $(\Sigma_p,\psi_p)$.
Fix a 
	an isomorphism-onto-its-image 
$f_p \colon \overline B_p \to M$ of decorated stratified manifolds, such that $f_p(0)=p$. 
Such a local neighbourhood exists by the definition of defect bordisms (embed a slightly larger open ball $B_{p,1+\varepsilon}$ as above and restrict to $\overline B_p$).
Repeat this procedure for all $p \in M_0$. 
By restricting the maps~$f_p$ to balls of smaller radii if necessary, we may assume that all images are disjoint. 

We can now define a new bordism
\be\label{eq:M0-fp-def}
	M\big( (f_p)_{p \in M_0} \big) \colon~
	X \sqcup \bigsqcup_{p \in M_0} \Sigma_{p} \longrightarrow Y
\ee
as follows. 
As manifold, $M\big( (f_p)_{p \in M_0} \big) = M \setminus \bigsqcup_{p \in M_0} f_p(B_{p})$. The new boundary component arising from cutting out the open ball around $p$ is parametrised by $f_p$, restricted to $\Sigma_{p} = \partial \overline B_{p}$. 
We make the ansatz
\be\label{eq:Z-bullet-morph}
	\zz^\bullet(M) \colon \zz(X) \lra \zz(Y)
	\, , \quad
	u \lmt
	\zz\big( M\big( (f_p)_{p \in M_0} \big) \big)\Big(u \otimes
	\bigotimes_{p \in M_0} \psi_p \Big) \, .
\ee
The proof of the proposition is complete once we show the following

\medskip

\noindent
\textsl{Claim:}
$\zz^\bullet(M)$ in \eqref{eq:Z-bullet-morph} is independent of the choice of the $f_p$.

\medskip

It is enough to consider the case that in \eqref{eq:M0-fp-def}, $M_0$ consists of a single point, the general case follows from gluing. Let thus $M_0 = \{p\}$ and let $f_p, g_p \colon \overline B_p \to M$ be two choices of local neighbourhood. 

Choose open subsets $U,V \subset \overline B_p$ such that $\varphi := g_p^{-1} \circ f_p|_{U} \colon U \to V$ is 
	an isomorphism. 
Pick
$h \in \mathrm{Emb}(\overline  B_p
)$ such that $\mathrm{im}(h) \subset U$. Then $h' := \varphi \circ h$ is equally an element of $\mathrm{Emb}(\overline B_p)$. 
With the thus-constructed maps we have the following identities of bordisms:
\be
M(f_p) \circ (1_X \sqcup H_h) = M(f_p \circ h) = M(g_p \circ h') =
M(g_p) \circ (1_X \sqcup H_{h'}) \, .
\ee
Let now $\psi \in \mathcal{Y}_{\Sigma_p}$. 
Then
\begin{align}
&\zz\big( M(f_p) \big)\big(u \otimes \psi \big)
=
\zz\big( M(f_p) \circ (1_X \sqcup H_h \big))\big(u \otimes \psi \big)
\nonumber \\ &
=
\zz\big( M(g_p) \circ (1_X \sqcup H_{h'} ) \big) \big(u \otimes \psi \big)
=
\zz\big( M(g_p) \big)\big(u \otimes \psi \big) \, ,
\end{align}
where in the first and last step 
we used that $\psi \in \mathcal{Y}_{\Sigma_p}$, so that it is left invariant by $\zz(H_h)$ and $\zz(H_{h'})$.
This proves the claim.
\end{proof}

\begin{remark}
\begin{enumerate}
\item
To justify the name ``completion'', we note that
if $\zz$ is already $D_0$-complete, then $\zz^\bullet$ is isomorphic to $\zz$ in $\Deftqft_{n}$, and so in particular $\zz^\bullet \sim \zz$. 

To see this, recall that by construction $D^\bullet_0(\Sigma) = \mathcal{Y}_\Sigma$ and that \eqref{eq:D-bullet-compare-h0} states that $h_0|_{D_0(\Sigma)} = Y_\Sigma$. If $\zz$ is $D_0$-complete, by definition the latter is an isomorphism, so that $h \colon \D \to \D^\bullet$ is an isomorphism.
\item
Note that the construction of $\D^\bullet$ and $\zz^\bullet$ does not depend on the zero component $D_0$ of $\D$ at all. 
Furthermore, by Proposition~\ref{prop:point-defect-completed} a defect TQFT for $\D$ factors through one for $\D^\bullet$, that is, every defect TQFT factors through a $D_0$-complete one. 
In this sense the label set $D_0$ is superfluous in the description of defect TQFTs.
\end{enumerate}
\end{remark}

\medskip

Let $\zz \colon \Bordd[n] \to \Vectk$ be a $D_0$-complete  defect TQFT. We will now describe how 
 point defects on strata of dimension 
$\geqslant 1$ carry an algebra structure. 
This algebra structure 
describes the fusion of point defects and 
will be instrumental in the next section.

Fix an $i \in \{1,\dots,n\}$ and a defect label $x \in D_i$. 
Recall the open and closed $x$-defect balls $B_x$ and $\overline B_x$ from \eqref{eq:Bxball} and \eqref{eq:Bxball-closed}, and that 
by definition of the defect bordisms, every point in an $i$-stratum labelled by $x$ has a neighbourhood 
	isomorphic 
to $B_x$. 
As before
we abbreviate $\Sigma_x = \partial \overline B_x$ and
\be
\label{eq:Ax-defect-alg-def}
	A_x := D_0(\Sigma_x) \, .
\ee
Recall that for $D_0$-complete theories, as we assume here, $A_x$ is a $\Bbbk$-vector space.

For $u,v \in A_x$ let 
\be
M_x(u,v) \colon \emptyset \lra \Sigma_x
\ee 
be the bordism in $\Borddefn{n}(\D)$ obtained from $\overline B_x$ by adding two point defects $p,q$ on the $i$-stratum of $\overline B_x$, that is, on the central fibre $(\{0\} \times [-1,1]^i) \cap \overline B_x$. The 0-strata of $M_x(u,v)$ are then $p,q$, and we label $p$ by $u$ and $q$ by $v$. 

Suppose $\overline B_x$ contains a $j$-stratum $(\overline B_x)_j^\alpha$ with $j\neq i$. 
Note that necessarily $j > i$. 
Let $y \in D_j$ be the defect label of $(\overline B_x)_j^\alpha$. For $w \in A_y$ let 
\be
N^\alpha_{y,x}(w) \colon  \emptyset \lra \Sigma_x
\ee
be the bordism in $\Borddefn{n}(\D)$ obtained from $\overline B_x$ by adding one point defect on $(\overline B_x)_j^\alpha$ which is labelled $w$.

\begin{proposition}
\label{prop:point-defect-algebras}
Let $x \in D_i$ for $i \geqslant 1$.
\begin{enumerate}
\item
$A_x$ with multiplication
\be\label{eq:Ax-product-def}
	m_x \colon A_x \otimes A_x \lra A_x
	\, , \quad
	(u,v) \longmapsto Y_{\Sigma_x}^{-1}\big( \zz(M_x(u,v))\big) \, ,
\ee
and unit $1_{A_x} := Y_{\Sigma_x}^{-1}\big(\zz(\overline B_x)\big)$ 
is a unital associative algebra. 
$A_x$ is commutative for $i>1$. 
\item
The map 
\be
	b_{y,x}^\alpha \colon A_y \lra A_x 
	\, , \quad
	w \lmt Y_{\Sigma_x}^{-1}\big(\zz(N^\alpha_{x,y}(w))\big) \, ,
\ee
is an algebra homomorphism.
\end{enumerate}
\end{proposition}

We call $A_x$ the \textsl{algebra of point insertions on~$x$}.
For $i=n-1$ the above map $A_y \to A_x$ can be thought of as a variant of the familiar bulk-boundary map in open-closed TQFTs, to which it reduces if the defect label $x\in D_{n-1}$ is in fact a ``boundary condition''. 

\medskip

The proof of Proposition~\ref{prop:point-defect-algebras} relies on the next lemma, which allows us to move point defects on connected components without changing the value of the functor. We need to prepare a bit of notation to state it.

Let $M \colon X \to Y$ be a bordism in $\Borddefn{n}(\D)$ and let $p_1,\dots,p_m \in M$ be distinct points such that $p_k \in M^{\alpha_k}_{i_k}$ for some $i_k$-stratum $M^{\alpha_k}_{i_k}$ of dimension $i_k >0$ with defect label $x_k \in D_{i_k}$. Let $u_1,\dots,u_m \in D_0$ be point defect labels such that $u_k \in A_{x_k}$.
Write
\be\label{eq:M-points-to-0strata}
	M(p_1,\dots,p_m;u_1,\dots,u_m) \colon X \lra Y
\ee
for the bordism obtained from $M$ by 
declaring $p_1,\dots,p_m$ to be new 0-strata, 
such that $p_k$ is labelled by $u_k \in D_0$.
We have the following lemma, which holds also for non-$D_0$-complete defect TQFTs.

\begin{lemma}\label{lem:move-point-insertions}
Let $p_1,\dots,p_m \in M$ and $q_1,\dots,q_m \in M$ be two choices of $m$ distinct points such that $p_k,q_k$ lie in the same $i_k$-stratum $M^{\alpha_k}_{i_k}$. 
For a 1-dimensional stratum we demand in addition that all $q_k$ on that stratum occur in the same order along that stratum as the $p_k$.
Then
\be
	\zz\big(M(p_1,\dots,p_m;u_1,\dots,u_m)\big)
	=
	\zz\big(M(q_1,\dots,q_m;u_1,\dots,u_m)\big) \, .
\ee
\end{lemma}

\begin{proof}
It is enough to show that we can replace $p_1$ by $q_1$, i.\,e.\ that 
\be
\zz\big(M(p_1,p_2,\dots,p_m;u_1,\dots,u_m)\big) = \zz\big(M(q_1,p_2,\dots,p_m;u_1,\dots,u_m)\big) \, .
\ee
Pick a smooth path $\gamma \colon [0,1] \to M$ such that $\gamma(0)=p_1$, $\gamma(1)=q_1$ and such that $\gamma$ lies entirely in the stratum $M^{\alpha_1}_{i_1}$ and does not intersect any of the other points $p_2,\dots,p_m$. For each point $\gamma(t)$ of the path there is 
	an isomorphism-onto-its-image 
$g_t \colon \overline B_{x_1} \to M$ such that $g_t(0) = \gamma(t)$, and such that $\mathrm{im}(g_t)$ does not contain any of $p_2,\dots,p_m$. 
Pick a finite collection $0=t_0,t_1,\dots,t_N=1$ such that the images $\mathrm{im}(g_{t_j})$ cover $\gamma$.

It is now enough to show that $p_1$ can be moved to a point $r$ in the intersection of 
$\mathrm{im}(g_{t_0})$ and $\mathrm{im}(g_{t_1})$. 
Repeating this procedures allows one to move $p_1$ to $q_1$.

Both $p_1$ and $r$ lie in the image of $g_{t_0}$. 
Let $\tilde p_1$ and $\tilde r$ be their pre-images in $\overline B_{x_1}$. 
There is 
	an isormophism~$\phi$ 
of the decorated stratified manifold $\overline B_{x_1}$ which is the identity in some neighbourhood of the boundary of $\overline B_{x_1}$ and which maps $\tilde p_1$ and $\tilde r$. 
Using~$\phi$ we obtain 
	an isomorphism 
of decorated stratified manifolds $M\to M$ which is the identity outside $\mathrm{im}(g_{t_0})$ and which equals $g_{t_0} \circ \phi \circ g_{t_0}^{-1}$ on the image. 
This 
	isomorphism 
maps $p_1$ to $r$, showing that $\zz\big(M(p_1,p_2,\dots,p_m;u_1,\dots,u_m)\big) = \zz\big(M(r,p_2,\dots,p_m;u_1,\dots,u_m)\big)$.
\end{proof}

A different way of stating the above lemma is that $\zz\big(M(p_1,\dots,p_m;u_1,\dots,u_m)\big)$ depends on $p_1,\dots,p_m$ only up to homotopy in the configuration space of~$m$ ordered distinct points, where during homotopies all points must remain in their respective strata.

\begin{proof}[Proof of Proposition \ref{prop:point-defect-algebras}]
Part (i): Commutativity for $i>1$ is immediate from Lemma~\ref{lem:move-point-insertions}. 
For unitality, pick a chart $g \colon \overline B_x \to B_x$ around $p$ such that $g(0) = p$ and $q \notin \mathrm{im}(g)$. Let $W$ be the bordism obtained from $M_x(u,v)$ by cutting out $g(B_x)$ and parametrising the boundary by $g|_{\partial \overline B_x}$. Recall that $\partial\overline B_x = \Sigma_x$. By definition of $Y_{\Sigma_x}$ 
(see \eqref{eq:point-def-stat-corr}) 
we have
\be
	\zz(M_x(u,v)) = \zz(W)(Y_{\Sigma_x}(u)) ~\in~ \zz(\Sigma_x) \, .
\ee
Denote by $M_x(v)$ the bordism obtained from $\overline B_x$ by adding a $v$-labelled 0-stratum at the point $q$.
For $u = 1_{A_x}$ one obtains
\begin{align}
	\zz(M_x(1_{A_x},v)) 
		&= \zz(W)(Y_{\Sigma_x}(1_{A_x})) 
	= \zz(W) \circ \zz(\overline B_x)
\nonumber \\ &
	= \zz(M_x(v)) 
	= Y_{\Sigma_x}(v) \, ,
\end{align}
where in the last equality we used once more Lemma~\ref{lem:move-point-insertions}.

The verification of associativity works along the same lines. One shows that $m_x(m_x(u,v),w)$ and $m_x(u,m_x(v,w))$ are equal to $\zz$ evaluated on the bordism given by $\overline B_x$ with three additional 0-strata $p,q,r$ labelled $u,v,w$ respectively. We omit the details.

\medskip

\noindent
Part (ii): We need to show $b^\alpha_{y,x}(m_y(u,v)) = m_x( b^\alpha_{y,x}(u), b^\alpha_{y,x}(v) )$.
Writing out both sides as a single bordism gives $\overline B_x$ with additional 0-strata $p,q$ inserted on $M_j^\alpha$ and labelled $u,v$. Applying Lemma~\ref{lem:move-point-insertions} gives the result, we again skip the details.
\end{proof}

Considering the algebras of point defects provides another justification for our notion of equivalence of defect TQFTs from Definition \ref{def:ZZsequiv}.

\begin{lemma}
\label{lemma:equiv-iso-Alg}
Let $\zz \colon \Borddefn{n}(\D) \to \Vectk$ and $\zz' \colon \Borddefn{n}(\D') \to \Vectk$ be $D_0$-complete defect TQFTs. Suppose $\zz \sim \zz'$ via $(h,\varphi) \colon \zz \rightleftarrows \zz' :\!\!(g,\psi)$. 
Then for every $x \in D_{i}$ the map $h_0|_{A_x} \colon A_x \to A_{h_i(x)}$ is an isomorphism of algebras.
\end{lemma}

\begin{proof}
Recall the definitions~\eqref{eq:Bxball-closed} and~\eqref{eq:SigmaxisdelBx}. 
Evaluating the naturality square of $\varphi \colon \zz \to \zz' \circ h_{*}$ for the bordism $\overline{B}_x \colon \emptyset \to \Sigma_x$ shows that the diagram 
\begin{equation}
\begin{tikzpicture}[
			     baseline=(current bounding box.base), 
			     descr/.style={fill=white,inner sep=3.5pt}, 
			     normal line/.style={->}
			     ] 
\matrix (m) [matrix of math nodes, row sep=3.5em, column sep=3.5em, text height=1.5ex, text depth=0.1ex] {%
D_{0}(\Sigma_{x})  &D_{0}'(\Sigma_{h_i(x)})  \\
\zz(\Sigma_{x}) & \zz'(\Sigma_{h_i(x)}) 
\\
};
\path[font=\footnotesize] (m-1-1) edge[->] node[above] {$h_{0}$} (m-1-2);
\path[font=\footnotesize] (m-1-1) edge[->] node[left] {$\zz(\overline{B}_{x})$} (m-2-1);
\path[font=\footnotesize] (m-1-2) edge[->] node[right]  {$\zz'(\overline{B}_{h_i(x)})$} (m-2-2);
\path[font=\footnotesize] (m-2-1) edge[->] node[below] {$\varphi_{\Sigma_{x}}$} (m-2-2);
\end{tikzpicture}
\end{equation}
commutes. 
Since $\zz$  is $D_{0}$-complete, 
the image of $\zz(\overline{B}_{x})$ equals the subspace $\mathcal{Y}_{\Sigma_{x}}$ of invariant states. By Lemma~\ref{lem:points-are-invariant}, $\varphi_{\Sigma_{x}}$ therefore
induces a map 
$\varphi_{\mathcal{Y}} \colon \mathcal{Y}_{\Sigma_{x}} \to \mathcal{Y}_{\Sigma_{h_i(x)}}$ such that the diagram 
\begin{equation}
\begin{tikzpicture}[
			     baseline=(current bounding box.base), 
			     descr/.style={fill=white,inner sep=3.5pt}, 
			     normal line/.style={->}
			     ] 
\matrix (m) [matrix of math nodes, row sep=3.5em, column sep=3.5em, text height=1.5ex, text depth=0.1ex] {%
D_{0}(\Sigma_{x})  &D_{0}'(\Sigma_{h_i(x)})  \\
\mathcal{Y}_{\Sigma_{x}} & \mathcal{Y}_{\Sigma_{h_i(x)}} 
\\
};
\path[font=\footnotesize] (m-1-1) edge[->] node[above] {$h_{0}$} (m-1-2);
\path[font=\footnotesize] (m-1-1) edge[->] node[left] {$Y_{\Sigma_{x}}$} (m-2-1);
\path[font=\footnotesize] (m-1-2) edge[->] node[right] {$Y_{\Sigma_{h_i(x)}}$} (m-2-2);
\path[font=\footnotesize] (m-2-1) edge[->] node[below] {$\varphi_{\mathcal{Y}}$} (m-2-2);
\end{tikzpicture}
\end{equation}
commutes.
Since the bottom path of the square consists of three (linear) isomorphisms, also $h_0$ restricted to $D_0(\Sigma_x) = A_x$ is a linear isomorphism.
To show that $h_0|_{A_x}$ is an algebra isomorphism, we compute, for $u,v \in A_x$,
\begin{align}
	m'_{h_i(x)}(h_0(u),h_0(v))
	&=
	Y_{\Sigma_{h_i(x)}}^{-1} \circ \zz'(M_{h_i(x)}(h_0(u),h_0(v)))
\nonumber \\
	&=
	\big( \varphi_{\mathcal{Y}} \circ Y_{\Sigma_{x}} \circ h_0^{-1}\big)^{-1}
	\circ \varphi_{\mathcal{Y}} \circ \zz(M_{x}(u,v))
\nonumber \\
	&=
	h_0 \circ m_x(u,v) \, ,
\end{align}
where in the second step we used that $\varphi$ is natural.
\end{proof}

\subsection{Euler-completing defect TQFTs}
\label{subsec:completewrtpointinsertions}

In the Euler defect TQFTs discussed in Example~\ref{example:Eulerx}, each stratum contributed a weight calculated from its Euler character. 
Tensoring an arbitrary defect TQFT~$\zz$ with an Euler theory allows one to attach such weights to defect strata for~$\zz$ as well. 
One may ask if one can enlarge the set of defect labels of $\zz$ in such a way that assigning different weights just amounts to choosing a different defect label.
This is indeed possible by ``internalising'' the construction of Example~\ref{example:Eulerx}. 
As an additional bonus, this internal version allows for weights that are not just scalars, but arbitrary invertible point defects. 

Given a $D_0$-complete defect TQFT $\zz \colon \Borddefn{n}(\D) \to \Vectk$ 
(recall Definition~\ref{def:D0complete}) we will 
define Euler-completed defect data $\D^{\euc}$ together with an injection $\ieuc \colon \D \to \D^{\euc}$, 
as well as the Euler-completed TQFT 
\be
\zz^{\euc} \colon \Borddefn{n}(\D^{\euc}) \lra \Vectk \, . 
\ee 
The Euler completion has the following properties, which justify its name:
\begin{enumerate}
\item $\zz$ factors through $\zz^{\euc}$ as $\zz = \zz^{\euc} \circ \ieuc_*$. 
\item $\zz^\euc$ is equivalent to $(\zz^\euc)^\euc$ in the sense of Definition~\ref{def:ZZsequiv}.
\item $\zz^\euc \otimes \zz^{\text{eu}}_\Psi$ is equivalent to $\zz^\euc$.
\end{enumerate}

\medskip

Let thus $\zz\colon \Borddefn{n}(\D) \to \Vectk$ be a $D_0$-complete defect TQFT. 
Recall from Proposition~\ref{prop:point-defect-algebras} that the set of point defect labels on an $i$-stratum decorated with $x \in D_{i}$ naturally acquires the structure of an algebra, denoted by $A_{x}$ in \eqref{eq:Ax-defect-alg-def}. 
We write $A_{x}^{\times}$ for the set of invertible elements in this algebra.

We would now like to say that the 
new sets $D^{\euc}_i$ of defect labels 
for $i$-strata consist of pairs $(x,\phi)$, where $x \in D_i$ and where $\phi \in A_x^\times$ describes the modified weight given to $x$-labelled $i$-strata. However, this is not quite enough, as one has to keep track of the thus-extended defect labels in the defect sphere $f_i(x)$.
The correct definition below makes use of the notation
\be
	\mathrm{Strat}(M)
	=
	\bigcup_{i=0}^n \bigcup_{\alpha_i} \big\{ M_i^{\alpha_i} \big\} 
\ee
to denote the set of strata 
$M_i^{\alpha_i}$ 
of a bordism $M$.

\begin{definition}
\label{def:EulercompleteD}
Let $\D \in \mathcal D^n$. 
The \textsl{Euler-completed defect data} $\D^{\euc} \in \mathcal D^n$ is given as follows. 
\begin{itemize}
\item
The label sets $D^{\euc}_{i}$, $i \in \{1, \ldots, n\}$, consist of triples
\be
\label{eq:1}
	(x,\phi,\Psi) \, ,
\ee
where $x \in D_i$, $\phi \in A_x^{\times}$ and $\Psi = (\psi_S)_{S \in \mathrm{Strat}(f_i(x))}$. 
Here, $\psi_S \in A_y^\times$ where $y$ is the label of the stratum $S$ of the defect sphere $f_i(x)$.
The set $D^\euc_0$ consists of pairs $(x,\Psi)$ with $x \in D_0$ and $\Psi$ as above.
\item 
The value $f_i^\euc(x,\phi,\Psi)$ of the adjacency map $f_i^\euc$ is given by the defect sphere $f_i(x)$, except that a $j$-stratum $S \in \mathrm{Strat}(f_i(x))$ with label $y \in D_j$ is decorated by $(y,\psi_S,\tilde \Psi)$, where $\tilde\Psi$ is determined from~$\Psi$ by the defect labels and weights adjacent to~$S$. 
For $i=0$, we set $f_{0}^{\euc}(x, \Psi)=f_{0}(x)$ with decorations from $\Psi$.
\end{itemize}
\end{definition}

It is clear that $\D^{\euc}$ is again a set of defect data. 
We can realise~$\D$ as a retract of $\D^\euc$ via the injection $\ieuc$ and surjection $\peuc$ of defect data: 
for $i>0$ we have 
\begin{align}
	\ieuc_i \colon D_i &\lra D^\euc_i
	\, , &  \quad
	\peuc_i \colon D^\euc_i &\lra D_i \, , 
	\nonumber \\
	x &\longmapsto (x,1_{A_x},\Psi_{\unit}) \, , 
	&
	(x,\phi,\Psi)  &\longmapsto  x
	\, ,
	\label{eq:h-compl-g-compl-def}
\end{align}
where $\Psi_{\unit}$ assigns to each stratum of the defect sphere $f_i(x)$ the unit $1_{A_y}$ in the corresponding algebra $A_y$; for $i=0$ we set $\ieuc_0(x) = (x,\Psi_{\unit})$ and $\peuc_0(x,\Psi) = x$. 

We now describe the (non-functorial) map ``insertion of point weights'' 
\be 
W \colon \Borddefn{n}(\D^{\euc}) \lra \Borddefn{n}(\D) 
\ee
on objects and morphisms of $\Borddefn{n}(\D^{\euc})$. 
To indicate the decoration with weights, we use the notation 
$\Sigma^{\euc}$ and $M^{\euc}$ for objects and morphisms of  $\Borddefn{n}(\D^{\euc})$, respectively.
We write $\mathrm{Strat}_{>0}(M)$ for the set of all strata of $M$ of dimension $>0$. 
\begin{itemize}
\item 
On objects, $W = \peuc_*$, that is, $W(\Sigma^\euc)$ agrees with $\Sigma^\euc$ as stratified manifold, but labels $(x,\Psi)$ and 
$(x,\phi,\Psi)$ 
are replaced by just $x$.
\item 
On morphisms, $W$ maps the class $[M^{\euc}]$ to the class $[M\big((p_S),(\phi_S^{\chirel(S)})\big)]$, where the latter is defined as follows.
Let $M$ represent the $\D$-decorated bordism $\peuc_*([M^\euc])$. 
For each $S \in \mathrm{Strat}_{>0}(M)$, $p_S$ is a choice of 
one point on~$S$, $\phi_S$ is the corresponding component of the defect label $(x_S,\phi_S,\Psi_S)$ that~$S$ carries in $M^\euc$, and the product $\phi_S^{\chirel(S)}$ is computed in $A_{x_S}$.
The defect bordism $M\big((p_S),(\phi_S^{\chirel(S)})\big)$ is defined as in \eqref{eq:M-points-to-0strata}: 
each $p_S$ becomes an additional 0-stratum labelled by $\phi_S^{\chirel(S)}$ 
	(with the symmetric Euler characteristic $\chirel(S)$ introduced in~\eqref{eq:chireldef}). 
By Lemma~\ref{lem:move-point-insertions} the class  $[M\big((p_S),(\phi_S^{\chirel(S)})\big)]$ is independent of the choice of the $p_{S}$.
 \end{itemize}

Even though $W$ itself is not a functor (composition of bordisms may result in more than one additional 0-stratum on a given $j$-stratum), we have:

\begin{lemma}
$\zz \circ W \colon \Borddefn{n}(\D^{\euc}) \to \Vectk$ is a defect TQFT.
\end{lemma}

\begin{proof}
It is easy to see that $W$ is strictly monoidal and symmetric (these conditions can be formulated without $W$ being compatible with composition), so that it is enough to verify that $\zz \circ W$ is functorial. 

Since all strata $S$ in the cylinder $X = \Sigma^\euc \times [0,1]$ over an object $\Sigma^\euc$ have $\chirel(S)=0$, the additional 0-strata 
in $W(X)$
are labelled by $\phi_S^0 = 1$. Thus $\zz(W(X)) = \id$, as required.

Given two composable bordisms $M^\euc$ and $N^\euc$ we need to show $\zz(W(M^\euc)) \circ \zz(W(N^\euc)) = \zz(W(M^\euc \circ N^\euc))$. This will follow directly from additivity of $\chirel$ under gluing.
Indeed, suppose two strata $S \in \mathrm{Strat}_{>0}(M^\euc)$ and $T \in \mathrm{Strat}_{>0}(N^\euc)$ with common label $(x,\phi,\Psi)$ get glued together, resulting in the stratum $U$ of $M^\euc \circ N^\euc$. 
Then $\chirel(S) + \chirel(T) = \chirel(U)$.

Let $p$ (labelled $\phi^{\chirel(S)}$) and $q$ (labelled $\phi^{\chirel(T)}$) be the additional 0-strata in $S$ and $T$, respectively, that are added by $W$. By Lemma~\ref{lem:move-point-insertions} and by the definition of the product of $A_x$ in \eqref{eq:Ax-product-def}, under $\zz$ the two point defects can be replaced by a single point defect anywhere on $U$, which is labelled by the product $\phi^{\chirel(S)} \phi^{\chirel(T)} = \phi^{\chirel(U)}$. But the latter is the 0-stratum added on $U$ in $W(M^\euc \circ N^\euc)$.
\end{proof}

The above lemma allows us to define:

\begin{definition}
\label{def:Eulercompletion}
Let $\zz \colon \Borddefn{n}(\D) \to \Vectk$ 
be a $D_{0}$-complete defect TQFT. 
The defect TQFT 
\be
\zz^{\euc} := \zz \circ W \colon \Borddefn{n}(\D^{\euc}) \lra \Vectk 
\ee
is called the \textsl{Euler completion of $\zz$}.
\end{definition}

It is immediate from the definition that
\be\label{eq:z-factors-zeuc}
	\zz^\euc \circ \ieuc_* = \zz \, .
\ee
The Euler completion is also compatible with our notion of equivalence of defect TQFTs:

\begin{lemma}
\label{lemma:euler-compl-defect}
If $\zz, \zz'$ are equivalent $D_{0}$-complete defect TQFTs, $\zz \sim \zz'$, then also $\zz^{\euc} \sim \zz'^{\euc}$. 
\end{lemma}

\begin{proof}
Let $(h,\varphi) \colon \zz \rightleftarrows \zz' :\!\!(g,\psi)$  be 
maps witnessing the 
equivalence between $\zz$ and $\zz'$. 
The maps of defect data $h$ and $g$ extend to maps of Euler-completed defect data as follows. We define $h^{\euc}\colon \D^{\euc} \to \D'^{\euc}$ by mapping 
$(x, \phi, \Psi) \in D_{i}^{\euc}$ to $(h(x), h(\phi), h(\Psi))$ where $h$ on $A_{x}$ is defined in Lemma~\ref{lemma:equiv-iso-Alg}, and we do not display indices on~$h$ for convenience. 
Since it is shown there that $h \colon A_{x} \rightarrow A_{h(x)}$ is an algebra isomorphism, it follows that $h(\phi) \in A_{h(x)}^{\times}$. 
Proceeding componentwise defines also $h(\Psi)$. 
As the symmetric Euler characteristics remain the same under $h_{*}^{\euc} \colon
\Borddefn{n}(\D^{\euc}) \to \Borddefn{n}(\D'^{\euc})$, 
it follows that  $\varphi$
induces a natural transformation $\varphi^{\euc} \colon \zz^{\euc} \Rightarrow \zz'^{\euc} \circ h_{*}^{\euc}$. 
The same argument applied to $(g, \psi)$ concludes the proof. 
\end{proof}

In the next two lemmas we investigate the structure of the point defects of $\zz^{\euc}$, working towards the equivalence $(\zz^{\euc})^{\euc} \sim \zz^{\euc}$.

\begin{lemma} 
$\zz^{\euc}$ is $D_{0}$-complete.
\end{lemma}

\begin{proof} We proceed in two steps. 

\textsl{Step 1:} We start with showing that 
\be
\mathcal{Y}_{\Sigma^{\euc}}
=
\big\{ \, \psi \in \zz^{\euc}(\Sigma^{\euc}) \,\big|\,
\zz^\euc(H_f^{\euc})(\psi) 
= \psi \text{ for all } f \in \mathrm{Emb}(\overline B) \, \big\}
\ee
as in Section \ref{subsec:point-defects-from} for $\Sigma^{\euc} \in \Sphere_{n-1}(\partial \D^{\euc})$, 
is equal to $\mathcal{Y}_{\Sigma}$.

By definition, $\zz^{\euc}(\Sigma^{\euc})=\zz(\Sigma)$ and we are left with showing that 
$\chirel(S)=0$ for all $(j+1)$-strata~$S$ of $H_{f}^{\euc}$, for all~$j$.  
Indeed, then $\zz^{\euc}(H_{f}^{\euc})=\zz(H_{f})$ and thus $\mathcal{Y}_{\Sigma^{\euc}}=\mathcal{Y}_{\Sigma}$. 

Fix a $j$-stratum~$T$ and of~$\Sigma$ and $f \in \mathrm{Emb}(\overline B)$. 
There there is a unique $j$-stratum~$T'$ of~$\Sigma$ such that~$f$ restricted to the cone $C(T)$ is a map $\overline{C(T)} \to C(T')$. 
By construction we have $\overline{B}=H_{f} \circ f(\overline{B})$ as stratified bordisms. 
The symmetric Euler characteristics are related by $\chirel(C(T')) =\chirel(S) + \chirel(C(T))$, where $S := \overline{C(T')} \setminus f(C(T))$.
Since $C(T)$ is contractible we have $\chi(C(T)) = 0$ and thus  $\chirel(CT)= 2\chi(C(T)) - \chi(\partial C(T))=-\chi(T) $ and analogously $\chirel(C(T'))=-\chi(T')$. 
Thus we conclude that $\chirel(S)=\chi(T)-\chi(T')$ and we are done with the first step once we established the following result.

\medskip

\noindent
\textsl{Claim:}
The Euler characteristics of $T$ and $T'$ are equal. 

\medskip

To establish this, we consider the smooth embedding  
\begin{equation}
  \label{eq:4}
f_{T} \colon T \times \R \lra   T' \times \R \, . 
\end{equation}
which is obtained from the above map $\overline{C(T)} \to C(T')$ by removing the cone points and then composing with a map which inflates the interval $(0,1)$ to the real line~$\R$ (with $\infty$ corresponding to the cone point, and $-\infty$ corresponding to the boundary of~$\overline B$). 
We will show that $f_{T}$ is a homotopy equivalence, thus proving the claim.

The map $f_{T}$ induces a map $\pi_{k}(f_{T}) \colon \pi_{k}(T \times \R) \rightarrow \pi_{k}(T' \times \R)$ between homotopy groups for all $k\in \Z_+$. 
Then by the Whitehead theorem, $f_{T}$ is a homotopy equivalence if $\pi_{k}(f_{T})$ is an isomorphism for all~$k$. 

Since we can find an $\eps > 0$ such that a ball $B_{\eps}$ of radius~$\eps$ is contained in $f(B)$, there exists an $r \in \R$ such that $T' \times (r, \infty) \subset f(T \times \R)$. 
Choosing some  $z>r$ guarantees that around $T' \times \{z\}$ also a collar lies in $f(T \times \R)$. 
Pick a basepoint $f(p)$ on $T' \times \{z\}$.

To show surjectivity of $\pi_{k}(f_{T})$, let $b\colon S^{k} \rightarrow T' \times \R$ be a based map. 
It is homotopic to a map $b_{1}\colon S^{k} \rightarrow T' \times \{z\}$ by $b_{t}(x)=(b_{T'}(x), b_{\R}(x)+ t(z-g_{\R}(x))$, for $t \in [0,1]$, where we write $b(x)=(b_{T'}(x), b_{\R}(x)) \in T' \times \R$. 
Clearly, $[b_{1}]$ is in the image of $\pi_{k}(f_{T})$.

To show injectivity, consider now a continuous map $g\colon S^{k} \rightarrow T \times \R$ such that there is a homotopy in $T' \times \R$ from $f_{T}\circ g$ to the constant map. 
Pick such a homotopy~$h$. 
Since~$h$ will in general leave the image $f(T \times \R)$, we will ``shift it above~$z$'', where~$z$ is as in the previous paragraph. 
By continuity of~$f$, there exists an $R \in \R$ such that $f(T \times (R, \infty)) \subset T'  \times (z, \infty)$.
	First we apply the homotopy $g_{t}(x)= (g_{T}(x),g_{\R}(x)+ t(R-g_{\R}(x))$ from $g$ to a map~$g_{1}$. 
Then $f_T(g_{1}(x)) \in T' \times (z, \infty)$ for all $x \in S^k$, and $f_T\circ g_1$ is homotopic to $f_{T} \circ g$, but it does in general not preserve the basepoint. 
However, since~$T'$ is connected, $f_T\circ g_1$ is homotopic in $T' \times (z, \infty)$ to a based map~$g'$ that is also in the image of~$f_{T}$. 
Composing the homotopies we obtain a homotopy $h' \colon S^{k} \times [0,1] \rightarrow T' \times \R $ from~$g'$ to the constant map to $f(p)$. 
By compactness of $S^{k}$ and by using the collar of $T' \times \{z\}$ which is still in $f(T \times \R)$, there exists an interval $[a,b] \subset [0,1]$, such that $h'(x,t) \in T' \times (r,\infty)$ for $t \in [0,1] \setminus [a,b]$. 
Consider a smooth function $\rho \colon [0,1] \rightarrow [0,1]$ with $\rho(0)=0=\rho(1)$ and $\rho|_{[a,b]}=1$, and set 
$\widetilde{h}(x,t)=(h'_{T'}(x,t), h'_{\R}(x,t) + \rho(t)(z-h'_\R(x,t)))$. 
By construction, $\widetilde{h}$ is a homotopy in $f(T \times \R)$ to the constant map, and injectivity of $f_{T}$ follows. 

We conclude that $\zz^{\euc}(H_{f}^{\euc})=\zz(H_{f})$, and thus $\mathcal{Y}_{\Sigma^{\euc}}=\mathcal{Y}_{\Sigma}$. 

\textsl{Step 2:}
Next we construct the inverse to the map $Y_{\Sigma^{\euc}} \colon D_{0}^{\euc}(\Sigma^{\euc}) \to 
\mathcal{Y}_{\Sigma^{\euc}}$.  
Let $(x,\Psi) \in D_{0}^{\euc}$ and consider the ball  $\overline{B}=\overline{C(\Sigma^{\euc})}$ whose
cone point is decorated with $(x,\Psi)$. 
By definition, $Y_{\Sigma^{\euc}}(x,\Psi)=\zz^{\euc}(\overline{B})$. 
After application of~$W$, a stratum $T \in \mathrm{Strat}_{>0}(\overline B)$ contains a point defect with label $\psi_{\partial T}^{\chirel(T)}$, where $\partial T = T \cap \Sigma^{\euc}$ is the boundary of the stratum $T$. 
We define a defect bordism $X_{\Sigma}$ as follows:
First take the cylinder $\Sigma \times I$, then insert on each stratum $S \times I$ a defect point with decoration $\psi_{S}^{-\chirel(T)}$, with $T$ the stratum in $\overline B$ bounded by $S$. This yields 
$X_{\Sigma}\in \Borddefn{n}(\D)$.
Applying $\zz(X_{\Sigma})$  to an invariant vector in $\mathcal{Y}_{\Sigma^{\euc}}$ yields again an invariant vector.
The linear maps $\zz^{\euc}(\overline B)$ and $\zz(X_{\Sigma})$ can be composed and by construction the diagram 
\begin{equation}
\begin{tikzpicture}[
			     baseline=(current bounding box.base), 
			     descr/.style={fill=white,inner sep=3.5pt}, 
			     normal line/.style={->}
			     ] 
\matrix (m) [matrix of math nodes, row sep=3.5em, column sep=3.5em, text height=1.5ex, text depth=0.1ex] {%
D_{0}^{\euc}(\Sigma^{\euc})  &  \mathcal{Y}_{\Sigma^{\euc}}    &   \mathcal{Y}_{\Sigma^{\euc}}\\
D_{0}(\Sigma)  &  & \mathcal{Y}_{\Sigma}
\\
};
\path[font=\footnotesize] (m-1-1) edge[->] node[above] {$Y_{\Sigma^{\euc}}$} (m-1-2);
\path[font=\footnotesize] (m-1-2) edge[->] node[above] {$\zz(X_{\Sigma})$} (m-1-3);
\path[font=\footnotesize] (m-1-1) edge[->] node[left] {$\id$} (m-2-1);
\path[font=\footnotesize] (m-1-3) edge[->] node[right] {$\id$} (m-2-3);
\path[font=\footnotesize] (m-2-1) edge[->] node[below] {$Y_{\Sigma}$} (m-2-3);
\end{tikzpicture}
\end{equation}
commutes.
By assumption, $Y_{\Sigma}$ is an isomorphism and $\zz(X_{\Sigma})$ is an isomorphism by construction, and so we conclude that $Y_{\Sigma^{\euc}}$ is an isomorphism as well.
\end{proof}

We now turn to the algebra of point defects. 
For $(x,\phi,\Psi) \in D_i^\euc$ we have so far two point-set isomorphisms between $A_{(x,\phi,\Psi)}$ and $A_{x}$. However, neither the isomorphism $Y_{\Sigma}^{-1} \circ Y_{\Sigma^{\euc}}$ nor the map $\peuc \colon D_{0}^\euc(\Sigma^{\euc}) \to D_{0}(\Sigma)$ from \eqref{eq:h-compl-g-compl-def}
are algebra isomorphisms. 
Instead we have:

\begin{lemma}\label{lem:sigma-alg-iso}
For $i>0$, let $(x, \phi,\Psi) \in D_{i}^{\euc}$ be an $i$-dimensional defect label. 
Consider the map
\begin{equation}
  \label{eq:3}
 \sigma_{(x,\phi,\Psi)} \colon A_{(x,\phi,\Psi)}\lra   A_{x}
 \, , \quad
 a \longmapsto 
 \phi^{E_{i}} \cdot \peuc(a) 
\cdot \phi^{E_{i}} \, ,
\end{equation}
where $E_{i}:= \chi(S^{i-1})-1 =  (-1)^{i-1}$, and the multiplication on the right-hand side takes place in $A_{x}$.
The map $\sigma_{(x,\phi,\Psi)}$ is an algebra isomorphism.
\end{lemma}

\begin{proof}
We first compute the change in the symmetric Euler characteristic when removing a point from a manifold. 
{}From
covering an $n$-manifold $X$ by  $X= X \setminus \{p\} \cup B$, $B$ a ball around $p \in X$, one sees that  the conventional Euler characteristic~$\chi$ satisfies $\chi(X)=\chi(X\setminus \{p\})+ \chi(B) - \chi(B\setminus \{p\})$.  Thus, using $\chi(B)=1$ and $\chi(B \setminus\{p\})=\chi(S^{n-1})$, we obtain $\chi(X \setminus \{p\})= \chi(X)-1 +\chi(S^{n-1})$. 
Hence we have 
	(recalling~\eqref{eq:chireldef})
\be
\label{eq:chisymm-cut-pt}
\chirel(X\setminus\{p\})=\chirel(X) + 2\chi(S^{n-1})-2=\chirel(X)+2E_{n} \, .
\ee

For simplicity we assume first that $i \geq 2$.
The multiplication $m_{(x,\phi,\Psi)}$ of $a,b \in A_{(x,\phi,\Psi)}$ is defined by $m_{(x,\phi,\Psi)}(a,b)= Y_{\Sigma^{\euc}}^{-1}(\zz^{\euc}(\overline{B}(p,q;a,b)))$. 
Write $a = (a',\Psi')$ and $b =(b',\Psi')$, so that $\peuc(a)=a'$, etc.
We compute further, omitting the location of the 0-strata from the notation 
	and denoting by $M$ the $i$-stratum on which the point defects $a,b$ are inserted,
\begin{align}
  \zz^{\euc} \big(\overline{B}(a,b) \big) &=\zz\big(\overline{B}(a',b',\phi^{\chirel(M)+4E_{i}}, \dots)\big) 
  \nonumber\\
&=\zz \big(\overline{B}(a'\cdot b', \phi^{\chirel(M)+2E_{i}},\phi^{2E_{i}}, \dots) \big) 
\end{align}
where $c = (a' \cdot  \phi^{2E_{i}} \cdot  b',\Psi') \in D^\euc_0(\Sigma^\euc)$, and where ``$\dots$'' indicates the point defect Euler weights inserted in $\overline B$ by $W$ that are not located on the $i$-stratum $M$.
Thus $m_{(x,\phi,\Psi)}(a;b) = ( a' \cdot \phi^{2E_{i}} \cdot  b' , \Psi')$.
{}From this expression the claimed isomorphism follows straightforwardly. 
In the case $i=1$ it follows along the same lines (but taking more care when inserting Euler weights disconnects a stratum) that the multiplication takes the same form $m_{(x,\phi,\Psi)}(a;b) = ( a' \cdot \phi^{2E_{i}} \cdot  b' , \Psi')$.
\end{proof}

Euler weights on line defects play a special role since inserting them disconnects the line defect. To obtain the equivalence  $\zz^{\euc\euc} \sim  \zz^\euc$ of defect TQFTs, they need to be treated separately. 
The next lemma shows that they do not add anything new to a given defect TQFT, i.\,e. we show that it suffices to consider Euler weights just for defects of dimension $i \geq 2$. 
To this end we consider the truncation $\D^{\euctwo}$ of $\D^{\euc}$ with defect labels $D_{i}^{\euctwo}=D_{i}^{\euc}$ for $i \geq 2$, and with $D_{1}^{\euctwo}= \{ (x, \Psi) \,|\, x \in D_{1}, \Psi = (\psi_S)_{S \in \mathrm{Strat}(f_1(x))}  \}$. 
The set $D^\euc_0$ consists of pairs $(x,\Psi)$ with $x \in D_0$ and~$\Psi$ as before, where adjacent line defects have no Euler weights.
$\D^{\euctwo}$ is a set of defect data, and we have a map of defect data $k \colon \D^{\euctwo} \to \D^{\euc}$ which is given by the identity on $D_{i}^{\euctwo}$ for $i \geq 2$, and by choosing the Euler weight $1_{x} \in A_{x}$ for $i=1$. 
We  thus obtain a defect TQFT $\zz^{\euctwo}= \zz^{\euc} \circ k_{*} \colon \Borddefn{n}(\D^{\euctwo})\to \Vectk$. 
	
\begin{lemma}
\label{lem:euctwo-equ}
Let $\zz\colon\Borddefn{n}(\D)\to \Vectk$ be a $D_{0}$-complete defect TQFT. 
Then the defect TQFTs $\zz^{\euc}$ and $\zz^{\euctwo}$ are equivalent. 
\end{lemma}

\begin{proof}
We already have that by definition $\zz^{\euctwo}=\zz^{\euc} \circ k_*$. Below we construct a map of defect data $l\colon \D^{\euc} \to \D^{\euctwo}$ such that $\zz^{\euc}=\zz^{\euctwo} \circ l_{*}$. This establishes the equivalence (with identity natural isomorphisms).
	 
The maps $l_{i}$ for $i \geq 2$ are again taken to be the identity. 
For $i=1$ we define $l_1$ to map $(a, \phi, \Psi) \in \D_{1}^{\euc}$ to $(a, \Psi) \in D_{1}^{\euctwo}$, i.\,e.~it forgets the Euler weights on line defects. 
Those will be absorbed in the point defects $D_{0}$ by the map $l_{0}$: 

Using the natural decomposition of $\mathrm{Strat}(M)$ according to the dimension of the strata, we can write a  point defect as $(a, \Psi)=(a,((\psi_{M_{1}^{\alpha_{1}}})_{\alpha_{1}}, (\psi_{M_{2}^{\alpha_{2}}})_{\alpha_{2}}, \ldots)) \in \D_{0}^{\euc}$, where in particular $(\psi_{M_{1}^{\alpha_{1}}})_{\alpha_{1}}$ indicates the tuple of Euler-weights on the adjacent line defects. 
We define
\be
l_{0}\big((a, \Psi)\big) 
=
\mathcal{Y}^{-1}_{\Sigma} \circ \zz \big(\overline{B}_{a}\big((\psi_{M_{1}^{\alpha_{1}}})_{\alpha_{1}} \big) \big) \, ,
\ee
where $\overline{B}_{a}$ is the defect ball for $a$ and in $\overline{B}_{a}((\psi_{M_{1}^{\alpha_{1}}})_{\alpha_{1}})$ we understand that for each $\alpha_1$, the Euler weight $\psi_{M_{1}^{\alpha_{1}}}$ is inserted on the 1-stratum $M_{1}^{\alpha_{1}}$. 

We claim that  $\zz^{\euc}$ is equal to $\zz^{\euctwo} \circ l_{*}$. 
To this end we consider $\zz^{\euc}(M)=\zz(M\big((p_S),(\phi_S^{\chirel(S)})\big))$ for a bordism $M \in \Borddefn{n}(\D^{\euc})$.
For 1-stratum~$S$ in~$M$ the symmetric Euler characteristic can either be 
$\chirel(S)=0$, if~$S$ is a closed circle, or if~$S$ has two boundary points on $\partial M$; 
or we can have $\chirel(S)=1$, if~$S$ has one endpoint on an inner 0-stratum of~$M$ and one endpoint in $\partial M$; 
finally, $\chirel(S)=2$ if both endpoints of~$S$ are an inner 0-stratum of~$M$ (which could be the same for both ends). 
In this last case we replace the single point defect with Euler weight $\phi_{S}^{2}$ by two point defects on $S$, each labelled with $\phi_S$. This produces a new bordism~$M'$ but does not affect the value of the TQFT: $\zz^{\euc}(M)=\zz(M')$. 
We can now choose small balls around all inner 0-strata of~$M'$ which contain exactly one 0-stratum labelled with an Euler weight for each adjacent 1-stratum. 
Using functoriality of~$\zz$ we can first evaluate these small balls and then the remaining bordism to obtain $\zz^{\euc}(M)$. 
But doing so is by definition of $l_{*}$ the same as evaluating $\zz^{\euctwo}(l_{*}(M))$.
This concludes the proof. 
\end{proof}

\begin{proposition} 
If $\zz \colon \Borddefn{n}(\D) \to \Vectk$ is a $D_{0}$-complete defect TQFT then $\zz^{\euc\euc}$ and $\zz^\euc$ are equivalent, $\zz^{\euc\euc}\sim  \zz^\euc$. 
\end{proposition}

\begin{proof}
We show that $(\zz^{\euctwo})^{\euctwo} \sim \zz^{\euctwo}$, the statement then follows using Lemmas~\ref{lem:euctwo-equ}
and~\ref{lemma:euler-compl-defect} in the series of equivalences $\zz^{\euc\euc} \sim  (\zz^{\euctwo})^{\euc} \sim (\zz^{\euctwo})^{\euctwo} \sim \zz^{\euctwo} \sim \zz^\euc$. 

{}From \eqref{eq:z-factors-zeuc} we already have a map $h^{\euc\euc} \colon \D^\euc \to \D^{\euc\euc}$ such that $\zz^{\euc\euc} \circ h^{\euc\euc}_* =  \zz^\euc$. 
By restriction we obtain 
a map $h^{\euctwo\euctwo} \colon \D^{\euctwo} \to \D^{\euctwo\euctwo}$ with $\zz^{\euctwo\euctwo} \circ h^{\euctwo\euctwo}_* =  \zz^{\euctwo}$. 
Below we provide a degreewise surjective map of defect data $t \colon \D^{\euctwo \euctwo} \rightarrow \D^{\euctwo}$ such that  $\zz^{\euctwo \euctwo} (M)=\zz^{\euctwo}(t_*(M))$ for all bordisms~$M$ in $\Bord_n^{\text{def}}(\D^{\euctwo \euctwo})$.
Since both defect TQFTs agree on objects by definition, this will imply $\zz^{\euctwo\euctwo} = \zz^{\euctwo} \circ t_*$, and hence prove the proposition.

The component maps of $t$ for $i>0$ are
\be
\label{eq:varphiGen}
t_i \colon D_i^{\euctwo \euctwo} \lra D_i^{\euctwo} 
\, , \quad
\big((x, \phi, \Psi),\phi',\Psi' \big) \lmt \big(x,\phi \cdot \sigma_{(x, \phi,\Psi)}(\phi') , \Psi'' \big)
\ee
where $\sigma_{(x, \phi, \Psi)}$ is the algebra isomorphism \eqref{eq:3}, and $\Psi''$ is obtained from $\Psi'$
by applying the map \eqref{eq:varphiGen} in each dimension. More precisely, starting from $i=n$, where $\Psi'$ is not present in $D_n^{\euctwo \euctwo}$, the tuple $\Psi''$ is determined by applying $t_{i}$ inductively, passing from~$i$ to $i-1$. 
Finally, for $i=0$, $t_0 \colon D_0^{\euctwo \euctwo} \to D_0^{\euctwo}$ maps $((x, \Psi),\Psi')$ to $(x,\Psi'')$.
	
We have to show that for every bordism~$M$ in $\Bord_n^{\text{def}}(\D^{\euctwo \euctwo})$, which we schematically write $(M,((x,\phi,\Psi),\phi',\Phi'))$ to indicate its decoration, we have 
\be
  \label{eq:ZppZp}
  \zz^{\euctwo \euctwo} \big(M, \big((x,\phi,\Psi),\phi',\Phi'\big) \big) 
  = 
  \zz^{\euctwo} \big(M, t\big((x,\phi,\Psi),\phi',\Phi'\big)\big)\, .
\ee
To prove \eqref{eq:ZppZp}, we evaluate its left-hand side in two steps. Consider a stratum $M_{i}$, $i \geq 2$, that is decorated by $((x,\phi,\Psi),\phi',\Phi') \in \D_{i}^{\euctwo \euctwo}$. 
Below, by abuse of notation we write $\zz^{\euctwo \euctwo}(M_{i})$ instead of $\zz^{\euctwo \euctwo}(M)$ to emphasise the changes in the strata of $M$.
By definition, $\zz^{\euctwo \euctwo}(M_{i})= \zz^{\euctwo}(M_{i}((\phi')^{\chirel(M_{i})_{\euc}}))$, where the notation $(\phi')^{\chirel(M_{i})_{\euc}}$ indicates that the power is computed in the algebra $A_{(x,\phi,\Psi)}$. 
Computing further we obtain  
\begin{align}
  \zz^{\euctwo \euctwo}(M_{i}) & = \zz \big(M_{i}\big((\phi')^{\chirel(M_{i})_{\euc}}, \phi^{\chirel(M_{i})+2E_{i}} \big)\big)  
  \nonumber
  \\
 & = \zz \big(M_{i}((\phi')^{\chirel(M_{i})_{\euc}}, \phi^{\chirel(M_{i})}, \phi^{2E_{i}}) \big) \, .
\end{align}
By Lemma~\ref{lem:sigma-alg-iso} this last expression is equal to 
\begin{align}
  \zz\big(M_{i} \big( \sigma_{(x,\phi,\Psi)}((\phi')^{\chirel(M_{i})_{\euc}}), \phi^{\chirel(M_{i})}) \big) &=
\zz\big(M_{i} \big( (\sigma_{(x,\phi,\Psi)}(\phi'))^{\chirel(M_{i})}, \phi^{\chirel(M_{i})}) \big)
\nonumber \\
&=\zz\big(M_{i} \big( (\sigma_{(x,\phi,\Psi)}(\phi')\cdot \phi)^{\chirel(M_{i})}) \big) \, , 
\end{align}
where we used in the second step that $\sigma_{(x,\phi,\Psi)}$ is an algebra isomorphism. Now we are done, since the last expression is by definition $\zz^{\euctwo} (M, t((x,\phi,\Psi),\phi',\Phi'))$ and this argument holds for all strata $M_{i}$.
\end{proof}

We remark that the argument in the proof relies on commutativity of $A_x$ and does not work directly for $i=1$. This is the reason for the detour via $ \zz^{\euctwo}$.

Finally, let us come back to point (iii) mentioned in the introductory paragraph of the present subsection.

\begin{lemma}
$\zz^\euc \otimes \zz^{\text{eu}}_\Psi \sim \zz^\euc$ for any defect TQFT $\zz \colon \Borddefn{n}(\D) \to \Vectk$ and all $\Psi = (\psi_1,\dots,\psi_n) \in (\Bbbk^{\times})^n$ as in \eqref{eq:Euler-defect-TQFT}.
 \end{lemma}

\begin{proof}
Recall that $\zz^{\text{eu}}_\Psi$ has defect data $\D^{\mathrm{id}}$, but that it differs from the identity defect TQFT by its action on bordisms as given in \eqref{eq:def-n-Euler}. 
	We will show that $\zz^\euc \otimes \zz^{\text{eu}}_\Psi$ and $\zz^\euc$ are in fact isomorphic, not just equivalent.
    
First we give an isomorphism of defect data $f \colon \D^\euc \otimes \D^{\mathrm{id}} \to \D^\euc$ (which is not the standard unitor) such that $\zz^\euc \otimes \zz^{\text{eu}}_\Psi = \zz^\euc \circ f_*$. 
We set~$f$ to be a composition of the unitor $\D^\euc \otimes \D^{\mathrm{id}} \to \D^\euc$ from Lemma~\ref{lem:Dn-sym-mon-cat} with an isomorphism $\tilde f\colon \D^\euc \to \D^\euc$. 
The latter simply maps a defect label $(x,\phi,\Phi) \in D_j^\euc$ to $(x,\psi_j \phi,\Phi')$, using the linear structure of $A_{x}$, where in $\Phi'$ all weights are multiplied by the scalar $\psi_k$ of the corresponding dimension.
This modification absorbs the factor \eqref{eq:def-n-Euler} into the evaluation of $\zz^\euc = \zz \circ W$ on $f_*(M)$, for any object or bordism~$M$. 
We conclude that $\zz^\euc \otimes \zz^{\text{eu}}_\Psi = (\zz^\euc \circ \tilde f_*) \otimes \unit_{\D^{\mathrm{id}}}$.
\end{proof}

We note that in Remark~\ref{rem:orbifolddataforanyZ} below we will spell out details of Euler weights for 3-dimensional defect TQFT in relation to the orbifold construction, to which we now turn.

\section{Orbifolds}
\label{sec:orbifolds}

In this section we construct orbifolds for defect TQFTs. 
We start with some preliminaries on oriented triangulations and Pachner moves in Section~\ref{subsec:oritriaPachmo}. 
Then we introduce the concept of an ``orbifold datum''~$\A$ for an $n$-dimensional defect TQFT $\zz \colon \Bordd[n] \rightarrow \Vectk$ for arbitrary $n \in \Z_+$ in Section~\ref{subsec:ndimorbis}, and we explain how to obtain the associated orbifold theory $\zz_\A \colon \Bord_n \to \Vectk$. 
The case $n=2$ had already been treated in the literature, and in Section~\ref{subsec:2dimorbis} we briefly discuss how it fits into our broader picture. 
Our main example is the 3-dimensional case which we study in detail in Section~\ref{subsec:3dimorbis}.

\subsection{Oriented triangulations and Pachner moves}
\label{subsec:oritriaPachmo}

Orbifolds are constructed by decorating the Poincar\'e dual of oriented triangulations of bordisms. 
As preparation for that, below we recall basic facts about triangulations of $n$-dimensional manifolds, Pachner moves between them, and the oriented versions thereof which we will need. 
For illustration and later use, we spell out the cases $n=2$ and $n=3$ is some detail. 

\medskip

Let $n\in\N$. 
By an \textsl{$n$-simplex}~$K$ we mean the convex hull of $n+1$ points $p_1, p_2, \ldots, p_{n+1} \in \R^{N}$ for some integer $N \geqslant n$, 
such that $\{ p_2-p_1,\, p_3-p_1, \ldots,\, p_{n+1}-p_1 \}$ is linearly independent. 
Hence
\be\label{eq:n-simplex-RN}
K = \left\{ \sum_{i=1}^{n+1} t_i p_i \;\Big|\; t_i \in \R_{\geqslant 0} \;\text{ and }\; \sum_{i=1}^{n+1} t_i = 1 \right\} 
\, . 
\ee
The \textsl{standard $n$-simplex $\Delta^n$} is the special case where $\{ p_1, \ldots,  p_{n+1}\}$ is the standard basis of $\R^{n+1}$. 
By a \textsl{simplicial complex} we mean a finite collection~$C$ of simplices such that (1) all faces of all simplices in~$C$ are elements of~$C$ as well, and (2) if $\sigma, \sigma' \in C$, then $\sigma \cap \sigma'$ is either empty or a face of both~$\sigma$ and~$\sigma'$. 

A \textsl{triangulation} of a topological manifold~$M$ is a simplicial complex~$C$ together with a 
homeomorphism 
$|C| \to M$, where $|C|$ is the geometric realisation (or polyhedron) of~$C$. 
A manifold together with a choice of triangulation is called a \textsl{triangulated manifold}.
 
A \textsl{smooth triangulation} of a smooth manifold $M$ is a triangulation $f \colon |C| \to M$ where for each simplex $S \subset C$, the restriction $f|_S$ is smooth and satisfies an extension property and a non-degeneracy condition on differentials, see \cite[Def.\,8.3]{Munkres} for details. 
	To compare 
different triangulations we need that given two smooth triangulations $f \colon |C| \to M$ and $f' \colon |C'| \to M$, there exists a $\delta$-approximation~$\widetilde f \colon |C'| \to M$ to~$f'$ such that $f^{-1} \circ \widetilde f \colon |C'| \to |C|$ is piecewise-linear \cite[Cor.\,10.13]{Munkres}. 

An \textsl{oriented $n$-simplex} is an $n$-simplex together with an equivalence class of total orders of its vertices. Two total orders are equivalent if they are related by an even permutation of vertices. 
Given an $n$-simplex $K$ as in \eqref{eq:n-simplex-RN} together with a total order $(p_1,p_2,\dots,p_{n+1})$ of its vertices $p_i$, 
one obtains an orientation of the $n$-dimensional tangent space (seen as a linear subspace of $\R^N$) at each point of~$K$ from the oriented basis $(p_2-p_1,\, p_3-p_1, \ldots,\, p_{n+1}-p_1)$.

By a \textsl{triangulation with total order} we mean a triangulation $|C|\to M$ of an $n$-manifold $M$ with a total order of the vertices of $C$. 
In a triangulation with total order, every $k$-simplex with $k \in \{ 1,2,\ldots,n\}$ is oriented by restricting the total order. 

\medskip

We will often work in the Poincar\'e dual picture for oriented $n$-manifolds. 
For $k \in \{0,1,\dots,n-1\}$, the orientation of a $k$-stratum in the dual of a triangulation is induced by that of the corresponding $(n-k)$-simplex by the rule that together (first the $k$-stratum, then the $(n-k)$-simplex) they produce the orientation of the underlying manifold~$M$. 
We take the orientation of an $n$-stratum to be that of the underlying manifold~$M$.
As a consequence a vertex in the Poincar\'e dual stratification is oriented positively if the orientation of the  corresponding $n$-simplex agrees with the orientation of the manifold, and negatively otherwise. 

It is convenient to describe the total order on the vertices by a ``height function'', that is, an injective function $h \colon C_0 \to \R$ from the set of vertices of the simplicial complex $C$ to $\R$. 
Let us illustrate the orientation of a complex and its dual in the cases $n=2$ and $n=3$. 

\begin{example}
\label{ex:oritria}
\begin{enumerate}
\item
For $n=2$, an $n$-simplex obtains a total order by assigning numbers $h(p_i)$ to the vertices $p_1,p_2,p_3$. 
Then the edges are oriented towards the ``higher'' vertices, and the face is oriented such that the induced boundary orientation agrees with two of the edge orientations. 
For example
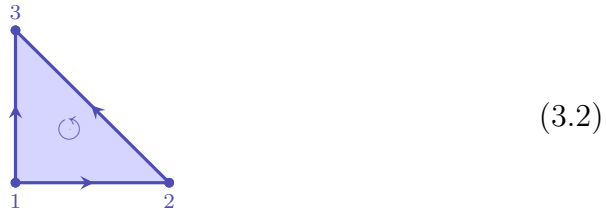
\be\label{eq:2simplex}
\tikzzbox{\begin{tikzpicture}[thick,scale=2.021,color=gray!60!blue, baseline=0.75cm, >=stealth]
\coordinate (v1) at (0,0);
\coordinate (v2) at (1,0);
\coordinate (v3) at (0,1);
\fill [blue!30,opacity=0.545] (v1) -- (v2) -- (v3);
\fill[color=gray!60!blue] (v1) circle (0.9pt) node[above] (0up) {};
\fill[color=gray!60!blue] (v2) circle (0.9pt) node[above] (0up) {};
\fill[color=gray!60!blue] (v3) circle (0.9pt) node[above] (0up) {};
\draw[string=gray!60!blue, very thick] (v1) -- (v2);
\draw[string=gray!60!blue, very thick] (v1) -- (v3);
\draw[string=gray!60!blue, very thick] (v2) -- (v3);
\fill[color=gray!60!blue] (0.35,0.35) circle (0pt) node (0up) {$\circlearrowleft$};
\fill[color=gray!60!blue] (v1) circle (0.9pt) node[below] (0up) {{\scriptsize$1$}};
\fill[color=gray!60!blue] (v2) circle (0.9pt) node[below] (0up) {{\scriptsize$2$}};
\fill[color=gray!60!blue] (v3) circle (0.9pt) node[above] (0up) {{\scriptsize$3$}};
\end{tikzpicture}}
\ee
is a 2-simplex with total order from a height function which takes values $1$, $2$, and $3$, respectively.

If we take the standard orientation of $M=\R^2$, i.\,e.~the ordered basis $((\begin{smallmatrix} 1 \\ 0 \end{smallmatrix}), (\begin{smallmatrix} 0 \\ 1 \end{smallmatrix}))$, then according to our rule the oriented Poincar\'e dual of \eqref{eq:2simplex} is 
\be
\tikzzbox
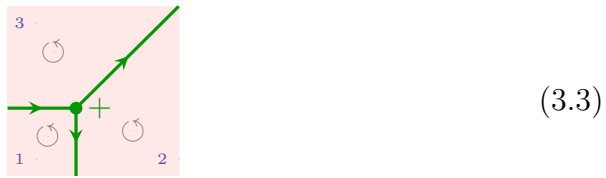
{\begin{tikzpicture}[thick,scale=1.5,color=gray!60!blue, baseline=0.5cm, >=stealth]
\coordinate (a1) at (0.35,0.35);
\coordinate (e1) at (-0.25,0.35);
\coordinate (e2) at (0.35,-0.25);
\coordinate (e3) at (1.25,1.25);
\fill [red!15,opacity=0.6] (-0.25,-0.25) -- (1.25,-0.25) -- (1.25, 1.25) -- (-0.25,1.25);
\draw[string=green!60!black, very thick] (e1) -- (a1);
\draw[string=green!60!black, very thick] (a1) -- (e2);
\draw[string=green!60!black, very thick] (a1) -- (e3);
\fill[color=green!60!black] (a1) circle (1.6pt) node[right] (0up) {$+$};
\fill[color=black, opacity=0.5] (0.1,0.1) circle (0pt) node (0up) {$\circlearrowleft$};
\fill[color=black, opacity=0.5] (0.85,0.15) circle (0pt) node (0up) {$\circlearrowleft$};
\fill[color=black, opacity=0.5] (0.15,0.85) circle (0pt) node (0up) {$\circlearrowleft$};
\fill[color=gray!60!blue] (0,-0.1) circle (0pt) node[left] (0up) {{\tiny$1$}};
\fill[color=gray!60!blue] (1.25,-0.1) circle (0pt) node[left] (0up) {{\tiny$2$}};
\fill[color=gray!60!blue] (0,1.1) circle (0pt) node[left] (0up) {{\tiny$3$}};
\end{tikzpicture}}
\ee
where we also show the heights of the vertices corresponding to the 2-strata. 
The 0-stratum is oriented by $+$ since the vertex order in \eqref{eq:2simplex} induces the same orientation as that of~$M$.
\item
For $n=3$, an example of a 3-simplex (or tetrahedron) with total order is
\be\label{eq:3simplex}
\tikzzbox{\begin{tikzpicture}[thick,scale=2.321,color=gray!60!blue, baseline=-0.3cm, >=stealth, 
				style={x={(-0.6cm,-0.4cm)},y={(1cm,-0.2cm)},z={(0cm,0.9cm)}}]
\coordinate (v1) at (1,0,0);
\coordinate (v2) at (1,1,0);
\coordinate (v3) at (0,0,0);
\coordinate (v4) at (0.25,0.1,0.75);
\fill [blue!20,opacity=0.545] (v1) -- (v2) -- (v3);
\fill [blue!20,opacity=0.545] (v4) -- (v2) -- (v3);
\fill [blue!20,opacity=0.545] (v1) -- (v4) -- (v3);
\fill[color=gray!60!blue] (v3) circle (0.9pt) node[left] (0up) {{\scriptsize$3$}};
\fill[color=gray!60!blue] (v1) circle (0.9pt) node[below] (0up) {{\scriptsize$1$}};
\fill[color=gray!60!blue] (v2) circle (0.9pt) node[below] (0up) {{\scriptsize$2$}};
\fill[color=gray!60!blue] (v4) circle (0.9pt) node[above] (0up) {{\scriptsize$4$}};
\draw[string=gray!60!blue, very thick] (v1) -- (v3);
\draw[string=gray!60!blue, very thick] (v2) -- (v3);
\draw[string=gray!60!blue, very thick] (v3) -- (v4);
\fill [blue!20,opacity=0.545] (v1) -- (v2) -- (v4);
\draw[string=gray!60!blue, very thick] (v2) -- (v4);
\draw[string=gray!60!blue, very thick] (v1) -- (v4);
\draw[string=gray!60!blue, very thick] (v1) -- (v2);
\end{tikzpicture}}
\ee
where the orientation of the interior agrees with that induced from the standard orientation of $\R^3$, and the orientations of the edges and faces are induced as in the 2-dimensional case. 
Thus for this orientation of $\R^3$ the Poincar\'e dual of \eqref{eq:3simplex} is 
\be
\tikzzbox{\begin{tikzpicture}[thick,scale=2.321,color=blue!50!black, baseline=0.0cm, >=stealth, 
				style={x={(-0.6cm,-0.4cm)},y={(1cm,-0.2cm)},z={(0cm,0.9cm)}}]
	\pgfmathsetmacro{\yy}{0.2}
\coordinate (P) at (0.5, \yy, 0);
\coordinate (R) at (0.625, 0.5 + \yy/2, 0);
\coordinate (L) at (0.5, 0, 0);
\coordinate (R1) at (0.25, 1, 0);
\coordinate (R2) at (0.5, 1, 0);
\coordinate (R3) at (0.75, 1, 0);
\coordinate (Pt) at (0.5, \yy, 1);
\coordinate (Rt) at (0.375, 0.5 + \yy/2, 1);
\coordinate (Lt) at (0.5, 0, 1);
\coordinate (R1t) at (0.25, 1, 1);
\coordinate (R2t) at (0.5, 1, 1);
\coordinate (R3t) at (0.75, 1, 1);
\coordinate (alpha) at (0.5, 0.5, 0.5);
%
\draw[string=green!60!black, very thick] (alpha) -- (Rt);
\fill [red!50,opacity=0.545] (L) -- (P) -- (alpha) -- (Pt) -- (Lt);
\fill [red!50,opacity=0.545] (Pt) -- (Rt) -- (alpha);
\fill [red!50,opacity=0.545] (Rt) -- (R1t) -- (R1) -- (P) -- (alpha);
\fill [red!50,opacity=0.545] (Rt) -- (R2t) -- (R2) -- (R) -- (alpha);
\draw[string=green!60!black, very thick] (alpha) -- (Rt);
\fill[color=blue!60!black] (0.5,0.59,0.94) circle (0pt) node[left] (0up) { {\scriptsize$\circlearrowright$} };
\fill[color=green!60!black] (alpha) circle (1.2pt) node[right] (0up) {{\scriptsize$+$} };
\fill [red!50,opacity=0.545] (Pt) -- (R3t) -- (R3) -- (R) -- (alpha);
\fill [red!50,opacity=0.545] (P) -- (R) -- (alpha);
%
\draw[string=green!60!black, very thick] (P) -- (alpha);
\draw[string=green!60!black, very thick] (R) -- (alpha);
\draw[string=green!60!black, very thick] (alpha) -- (Pt);
\fill[color=gray!60!blue] (0.65,0.48,0) circle (0pt) node[left] (0up) {{\tiny$1$}};
\fill[color=gray!60!blue] (0.58,1.05,0) circle (0pt) node[left] (0up) {{\tiny$2$}};
\fill[color=gray!60!blue] (0.33,1.05,0) circle (0pt) node[left] (0up) {{\tiny$3$}};
\fill[color=gray!60!blue] (0.25,0.48,1) circle (0pt) node[left] (0up) {{\tiny$4$}};
%
%
\fill[color=blue!60!black] (0.5,0.2,0.15) circle (0pt) node[left] (0up) { {\scriptsize$\circlearrowright$} };
\fill[color=blue!60!black] (0.5,0.5,0.09) circle (0pt) node[left] (0up) { {\scriptsize$\circlearrowright$} };
\fill[color=blue!60!black] (0.5,0.89,0.0) circle (0pt) node[left] (0up) { {\scriptsize$\circlearrowright$} };
\fill[color=blue!60!black] (0.5,1.04,0.13) circle (0pt) node[left] (0up) { {\scriptsize$\circlearrowright$} };
\fill[color=blue!60!black] (0.5,1.19,0.28) circle (0pt) node[left] (0up) { {\scriptsize$\circlearrowright$} };
%
\draw [black,opacity=1, very thin] (Pt) -- (Lt) -- (L) -- (P);
\draw [black,opacity=1, very thin] (Pt) -- (Rt);
\draw [black,opacity=1, very thin] (Rt) -- (R1t) -- (R1) -- (P);
\draw [black,opacity=1, very thin] (Rt) -- (R2t) -- (R2) -- (R);
\draw [black,opacity=1, very thin] (Pt) -- (R3t) -- (R3) -- (R);
\draw [black,opacity=1, very thin] (P) -- (R);
\end{tikzpicture}}
\, . 
\ee
The orientation of the 2- and 1-strata in the dual are deduced with the right-hand rule: 
\be
	\!\!\!
\tikzzbox{\begin{tikzpicture}[thick,scale=2.321,color=blue!50!black, baseline=0.0cm, >=stealth, 
				style={x={(-0.6cm,-0.4cm)},y={(1cm,-0.2cm)},z={(0cm,0.9cm)}}]
	\pgfmathsetmacro{\yy}{0.2}
\coordinate (v1) at (1, 0.5, 0);
\coordinate (v2) at (0, 0.5, 0);
\coordinate (v3) at (0, 0.5, 1);
\coordinate (v4) at (1, 0.5, 1);
\fill[color=gray!60!blue] (0.5,1,0.5) circle (0.6pt) node[above] (0up) {{\tiny$1$}};
\fill[color=gray!60!blue] (0.5,0,0.5) circle (0.6pt) node[above] (0up) {{\tiny$2$}};
\draw[color=gray!60!blue, densely dotted] (0.5,0,0.5) -- (0.5,0.5,0.5);
\fill [red!50,opacity=0.545] (v1) -- (v2) -- (v3) -- (v4);
\fill[color=gray!60!blue] (0.2,0.5,0.7) circle (0pt) node[above] (0up) {$\circlearrowright$};
\draw[color=gray!60!blue, densely dotted, postaction={decorate}, decoration={markings,mark=at position .7 with {\arrow[draw=gray!60!blue]{>}}}] (0.5,1,0.5) -- (0.5,0.5,0.5);
\end{tikzpicture}}
\, , \qquad
\tikzzbox{\begin{tikzpicture}[thick,scale=2.321,color=blue!50!black, baseline=0.0cm, >=stealth, 
				style={x={(-0.6cm,-0.4cm)},y={(1cm,-0.2cm)},z={(0cm,0.9cm)}}]
	\pgfmathsetmacro{\yy}{0.2}
\coordinate (T) at (0.5, 0.4, 0);
\coordinate (L) at (0.5, 0, 0);
\coordinate (R1) at (0.3, 1, 0);
\coordinate (R2) at (0.7, 1, 0);
\coordinate (1T) at (0.5, 0.4, 1);
\coordinate (1L) at (0.5, 0, 1);
\coordinate (1R1) at (0.3, 1, );
\coordinate (1R2) at (0.7, 1, );
%
\coordinate (p3) at (0.1, 0.1, 0.5);
\coordinate (p2) at (0.5, 0.95, 0.5);
\coordinate (p1) at (0.9, 0.1, 0.5);
\draw[color=gray!60!blue, densely dotted, postaction={decorate}, decoration={markings,mark=at position .6 with {\arrow[draw=gray!60!blue]{>}}}] (p1) -- (p3);
\draw[color=gray!60!blue, densely dotted, postaction={decorate}, decoration={markings,mark=at position .5 with {\arrow[draw=gray!60!blue]{>}}}] (p2) -- (p3);
\draw[color=gray!60!blue, densely dotted, postaction={decorate}, decoration={markings,mark=at position .7 with {\arrow[draw=gray!60!blue]{>}}}] (p1) -- (p2);
\fill[color=gray!60!blue] (p1) circle (0.6pt) node[above] (0up) {{\tiny$1$}};
\fill[color=gray!60!blue] (p3) circle (0.6pt) node[above] (0up) {{\tiny$3$}};
%
\fill[color=blue!60!black] (0.5,0.25,0.15) circle (0pt) node[left] (0up) { {\scriptsize$\circlearrowright$} };
\fill [red!50,opacity=0.545] (L) -- (T) -- (1T) -- (1L);
\fill [red!50,opacity=0.545] (R1) -- (T) -- (1T) -- (1R1);
\fill [red!50,opacity=0.545] (R2) -- (T) -- (1T) -- (1R2);
\fill[color=blue!60!black] (0.5,0.25,0.15) circle (0pt) node[left] (0up) { {\scriptsize$\circlearrowright$} };
\fill[color=blue!60!black] (0.15,0.95,0.07) circle (0pt) node[left] (0up) { {\scriptsize$\circlearrowright$} };
\fill[color=blue!60!black] (0.55,0.95,0.1) circle (0pt) node[left] (0up) { {\scriptsize$\circlearrowright$} };
%
\draw[string=green!60!black, very thick] (T) -- (1T);
\fill[color=gray!60!blue] (p2) circle (0.6pt) node[above] (0up) {{\tiny$2$}};
%
\draw [black,opacity=1, very thin] (1T) -- (1L) -- (L) -- (T);
\draw [black,opacity=1, very thin] (1T) -- (1R1) -- (R1) -- (T);
\draw [black,opacity=1, very thin] (1T) -- (1R2) -- (R2) -- (T);
\end{tikzpicture}}
\begin{tikzpicture}[
			     baseline=(current bounding box.base), 
			     descr/.style={fill=white,inner sep=3.5pt}, 
			     normal line/.style={->}
			     ] 
\matrix (m) [matrix of math nodes, row sep=3.5em, column sep=3.3em, text height=1.1ex, text depth=0.1ex] {%
{}
&
{}
\\
};
\path[font=\footnotesize] (m-1-1) edge[<->] node[auto] { {\scriptsize rotate}} (m-1-2);
\end{tikzpicture}
\tikzzbox{\begin{tikzpicture}[thick,scale=2.321,color=blue!50!black, baseline=0.0cm, >=stealth, 
				style={x={(-0.6cm,-0.4cm)},y={(1cm,-0.2cm)},z={(0cm,0.9cm)}}]
	\pgfmathsetmacro{\yy}{0.2}
\coordinate (T) at (0.5, 0.4, 0);
\coordinate (L) at (0.5, 0, 0);
\coordinate (R1) at (0.3, 1, 0);
\coordinate (R2) at (0.7, 1, 0);
\coordinate (1T) at (0.5, 0.4, 1);
\coordinate (1L) at (0.5, 0, 1);
\coordinate (1R1) at (0.3, 1, );
\coordinate (1R2) at (0.7, 1, );
%
\coordinate (p3) at (0.1, 0.1, 0.5);
\coordinate (p2) at (0.5, 0.95, 0.5);
\coordinate (p1) at (0.9, 0.1, 0.5);
\draw[color=gray!60!blue, densely dotted, postaction={decorate}, decoration={markings,mark=at position .6 with {\arrow[draw=gray!60!blue]{<}}}] (p1) -- (p3);
\draw[color=gray!60!blue, densely dotted, postaction={decorate}, decoration={markings,mark=at position .5 with {\arrow[draw=gray!60!blue]{<}}}] (p2) -- (p3);
\draw[color=gray!60!blue, densely dotted, postaction={decorate}, decoration={markings,mark=at position .7 with {\arrow[draw=gray!60!blue]{<}}}] (p1) -- (p2);
\fill[color=gray!60!blue] (p1) circle (0.6pt) node[above] (0up) {{\tiny$3$}};
\fill[color=gray!60!blue] (p3) circle (0.6pt) node[above] (0up) {{\tiny$1$}};
%
\fill[color=blue!60!black] (0.5,0.25,0.15) circle (0pt) node[left] (0up) { {\scriptsize$\circlearrowleft$} };
\fill [pattern=north west lines, opacity=0.4] (L) -- (T) -- (1T) -- (1L);
\fill [pattern=north west lines, opacity=0.4] (R1) -- (T) -- (1T) -- (1R1);
\fill [pattern=north west lines, opacity=0.4] (R2) -- (T) -- (1T) -- (1R2);
\fill [red!50,opacity=0.545] (L) -- (T) -- (1T) -- (1L);
\fill [red!50,opacity=0.545] (R1) -- (T) -- (1T) -- (1R1);
\fill [red!50,opacity=0.545] (R2) -- (T) -- (1T) -- (1R2);
\fill[color=blue!60!black] (0.5,0.25,0.15) circle (0pt) node[left] (0up) { {\scriptsize$\circlearrowleft$} };
\fill[color=blue!60!black] (0.15,0.95,0.07) circle (0pt) node[left] (0up) { {\scriptsize$\circlearrowleft$} };
\fill[color=blue!60!black] (0.55,0.95,0.1) circle (0pt) node[left] (0up) { {\scriptsize$\circlearrowleft$} };
%
\draw[string=green!60!black, very thick] (1T) -- (T);
\fill[color=gray!60!blue] (p2) circle (0.6pt) node[above] (0up) {{\tiny$2$}};
%
%
\draw [black,opacity=1, very thin] (1T) -- (1L) -- (L) -- (T);
\draw [black,opacity=1, very thin] (1T) -- (1R1) -- (R1) -- (T);
\draw [black,opacity=1, very thin] (1T) -- (1R2) -- (R2) -- (T);
\end{tikzpicture}}
\, . 
\ee
Here and below we indicate opposite orientations of 2-strata by a stripy pattern. 
\end{enumerate}
\end{example}

We note that the oriented  2- and 3-simplices with total order 
\be
\tikzzbox{\begin{tikzpicture}[thick,scale=2.021,color=gray!60!blue, baseline=0.75cm, >=stealth]
\coordinate (v3) at (0,0);
\coordinate (v2) at (1,0);
\coordinate (v1) at (0,1);
\fill [pattern=north west lines, opacity=0.4] (v1) -- (v2) -- (v3);
\fill [blue!30,opacity=0.545] (v1) -- (v2) -- (v3);
\fill[color=gray!60!blue] (v3) circle (0.9pt) node[above] (0up) {};
\fill[color=gray!60!blue] (v2) circle (0.9pt) node[above] (0up) {};
\fill[color=gray!60!blue] (v1) circle (0.9pt) node[above] (0up) {};
\draw[string=gray!60!blue, very thick] (v1) -- (v2);
\draw[string=gray!60!blue, very thick] (v1) -- (v3);
\draw[string=gray!60!blue, very thick] (v2) -- (v3);
\fill[color=gray!60!blue] (0.35,0.35) circle (0pt) node (0up) {$\circlearrowright$};
\fill[color=gray!60!blue] (v1) circle (0.9pt) node[above] (0up) {{\scriptsize$1$}};
\fill[color=gray!60!blue] (v2) circle (0.9pt) node[below] (0up) {{\scriptsize$2$}};
\fill[color=gray!60!blue] (v3) circle (0.9pt) node[below] (0up) {{\scriptsize$3$}};
\end{tikzpicture}}
\, , 
\qquad
\tikzzbox{\begin{tikzpicture}[thick,scale=2.321,color=gray!60!blue, baseline=-0.3cm, >=stealth, 
				style={x={(-0.6cm,-0.4cm)},y={(1cm,-0.2cm)},z={(0cm,0.9cm)}}]
\coordinate (v1) at (1,0,0);
\coordinate (v2) at (1,1,0);
\coordinate (v3) at (0,0,0);
\coordinate (v4) at (0.25,0.1,0.75);
\fill [pattern=north west lines, opacity=0.4] (v1) -- (v2) -- (v3);
\fill [blue!20,opacity=0.545] (v1) -- (v2) -- (v3);
\fill [blue!20,opacity=0.545] (v4) -- (v2) -- (v3);
\fill [blue!20,opacity=0.545] (v1) -- (v4) -- (v3);
\fill[color=gray!60!blue] (v3) circle (0.9pt) node[left] (0up) {{\scriptsize$2$}};
\fill[color=gray!60!blue] (v1) circle (0.9pt) node[below] (0up) {{\scriptsize$1$}};
\fill[color=gray!60!blue] (v2) circle (0.9pt) node[below] (0up) {{\scriptsize$3$}};
\fill[color=gray!60!blue] (v4) circle (0.9pt) node[above] (0up) {{\scriptsize$4$}};
\draw[string=gray!60!blue, very thick] (v1) -- (v3);
\draw[string=gray!60!blue, very thick] (v3) -- (v2);
\draw[string=gray!60!blue, very thick] (v3) -- (v4);
\fill [blue!20,opacity=0.545] (v1) -- (v2) -- (v4);
\draw[string=gray!60!blue, very thick] (v2) -- (v4);
\draw[string=gray!60!blue, very thick] (v1) -- (v4);
\draw[string=gray!60!blue, very thick] (v1) -- (v2);
\end{tikzpicture}}
\ee
are not equivalent to \eqref{eq:2simplex} and \eqref{eq:3simplex} for $n=2$ and $n=3$, respectively. 
This is a special case of part (ii) of the following lemma, which will be important in determining the amount of defect data necessary as an input for the orbifold construction.

\begin{lemma}
\label{lem:rotasimp}
Let $K$ be an $n$-simplex in $\R^N$ as in \eqref{eq:n-simplex-RN}, with $1 \leqslant n \leqslant N$. 
\begin{enumerate}
\item
If $n \leqslant N-1$, any two total vertex orders of $K$ are related by an orientation preserving affine linear automorphism of $\R^N$ which maps $K$ to $K$. 
\item
If $n=N$, two total orders are related by an automorphism as in (i) iff the two total orders induce the same orientation on $K$. 
\item
In both cases, the restriction of the automorphism to $K$ is unique.
\end{enumerate}
\end{lemma}

\begin{proof}
For part (ii), let $(p_1,p_2,\dots,p_{n+1})$ be the ordered set of vertices of $K$ with respect to the first order, and let $(p_{\sigma(1)},\dots,p_{\sigma(n+1)})$ for $\sigma \in S_{n+1}$ be the second order. 
	There is a unique affine linear automorphism $F$ of $\R^N$ such that $F(p_i) = p_{\sigma(i)}$, for $i \in \{1,\dots,n+1\}$. 
	Then~$F$ maps~$K$ to~$K$ by definition, and it is orientation preserving iff $\sigma$ is even.\footnote{To see this, first assume that~$\sigma$ has a fixed point, say $\sigma(i_0)=i_0$. Write $F(x) = A(x-p_{i_0}) + p_{i_0}$ for a linear map $A$. Abbreviating $v_i := p_i-p_{i_0}$ we see that $A(v_i)=v_{\sigma(i)}$, showing that $\det(A)>0$ iff $\sigma$ is even.
	For $n \geq 2$, if $\sigma$ has no fixed point it can be written as the composition of two permutations, each of which has a fixed point.
	For $n=1$ the statement is clear.}
This proves part (ii). 

To get part (i), it is enough to add to the above argument one orientation preserving affine automorphism of $\R^N$ which reverses the orientation of $K$. 
To do so, first extend $(p_1,p_2,\dots,p_{n+1})$ to an affine basis of $\R^N$, that is, pick $p_{n+2},\dots,p_{N+1}$ such that $(p_2-p_{1},p_3-p_{1},\dots,p_{N+1}-p_{1})$ is a basis of $\R^N$.
The affine linear automorphism which exchanges $p_1$ and $p_2$ and keeps $p_3,\dots,p_{N+1}$ fixed maps $K$ to $K$ and is orientation reversing on $K$ and on $\R^N$. Compose this automorphism with an affine linear reflection along a hyperplane containing $K$. The resulting affine automorphism maps $K$ to $K$, reverses the orientation of $K$, but is orientation preserving on $\R^N$.

Part (iii) is clear.
\end{proof}

\medskip

We now turn to Pachner moves. 
For an $n$-dimensional triangulated manifold~$M$ we consider a subcomplex $K \subset M$ that is isomorphic to a collection of $n$-simplices $F \subset \partial \Delta^{n+1}$ of the standard $(n+1)$-simplex; we write $\varphi\colon K \to F$ for the isomorphism. 
Then a \textsl{Pachner move} by definition is the replacement 
\be
M \lmt \big(M\setminus K\big) \cup_{\varphi|_{\partial K}} \big(\partial \Delta^{n+1} \setminus 
\overset{\circ}{F} 
\big)
\, .
\ee
Put differently, `it glues the other side of $\Delta^{n+1}$ into~$M$ instead of~$K$'. 
Since $\Delta^{n+1}$ has only finitely many (namely $n+2$) faces, there are only finitely many such moves. 

Pachner's theorem \cite{Pachpaper} states that if two triangulated PL manifolds are PL isomorphic, then there exists a finite sequence of Pachner moves from one triangulation to the other. 
For fixed~$n$ and $k  \in \{ 1,2,\dots,n+1 \}$, we will refer to the Pachner move which replaces~$k $ faces $F \subset \partial \Delta^{n+1}$ by the other $(n+2-k)$ faces in $\partial \Delta^{n+1}$ as a \textsl{$k $-$(n+2-k )$ Pachner move}.  

Since any two smooth triangulations of a smooth $n$-manifold are PL isomorphic 
(up to an arbitrarily small $\delta$-approximation), 
Pachner's result directly holds in our setting -- if orientations are discarded. 
To relate triangulations with total order, we have to consider Pachner moves for all possible orders on the vertices of simplices. 
We will refer to Pachner moves between triangulations with total order as \textsl{oriented Pachner moves}. 

\begin{proposition}
For $n\in \Z_+$ and any two finite smooth triangulations with total order of a smooth $n$-manifold, there is a finite sequence of oriented Pachner moves
	which takes both to a common refinement with total order, up to an arbitrarily small $\delta$-approximation.
\end{proposition}

\begin{proof}
Given Pachner's theorem, it is sufficient to show that the total order of a given triangulation can be changed to any other total order with finitely many oriented Pachner moves. 
For this it is enough to see how the height of any given vertex in the triangulation can be changed. 

Let $f \colon |C| \to M$ be a triangulation with total order, and let $h \colon C_0 \to \R$ be its height function. 
For a vertex $v \in C_0$, we want to change the value $h(v) =: \gamma$ to~$\gamma'$ for any $\gamma' \in \R \setminus \{ h(\nu) \,|\, \nu \in C_0 \}$. 
Denote by $h' \colon C_0 \to \R$ the height function which only differs from $h$ in $v$, where it takes the value $\gamma'$.

Let $S_v \subset M$ be the image under $f$ of the star of $v$ and let $p = f(v)$. 
Choose a diffeomorphism $\phi$ of $M$ which is the identity on the complement of $S_v$ and which maps $p$ to $\tilde p \neq p$ (by construction of $\phi$, we must have $\tilde p \in S_v$). Let the triangulation $\widetilde f \colon |C| \to M$ be given by $\widetilde f = \phi \circ f$. Note that $f(v) \neq \widetilde f(v)$, and that in fact $f(v) \notin \widetilde f(C_0)$. We can find \cite[Cor.\,10.13]{Munkres} a $\delta$-approximation $f' \colon |C'| \to M$ to~$\widetilde f$ such that $f'^{-1} \circ f \colon |C| \to |C'|$ is a PL isomorphism. For $\delta$ small enough we maintain the property $f(v) \notin f'(C'_0)$.

Thanks to Pachner's theorem there is a finite sequence of Pachner moves from the triangulation $f\colon |C| \to M$ to $f'\colon |C'| \to M$. 
We turn this into a sequence of oriented Pachner moves by applying the rule that whenever a new vertex is created at the position of one of the vertices in $C_0$, it must have the corresponding height from~$C$ 
(while the choice of heights at new vertices is arbitrary). 
Running the same argument (with the same choices for new heights, etc.) but starting from~$h'$ instead of~$h$, 
we arrive again at $C'$, equipped with the same height function as before. Indeed, the only difference between $h$ and $h'$ was in the vertex $v$ which is no longer present in $C'$.
\end{proof}

\begin{example}
We illustrate oriented Pachner moves and their Poincar\'{e} dual stratifications for $n=2$ and $n=3$. 
\begin{enumerate}
\item
For $n=2$, the $(n+1)$-simplex has four faces, and accordingly there are two types of Pachner moves: 
those replacing two triangles sharing a single edge by the `other' two faces of the tetrahedron $\Delta^3$, and those replacing one triangle by the `other' three faces of $\Delta^3$ meeting at a single vertex -- as well as the inverse operations. 
These are the \textsl{2-2} and \textsl{1-3 moves}, respectively, and locally they look as follows: 
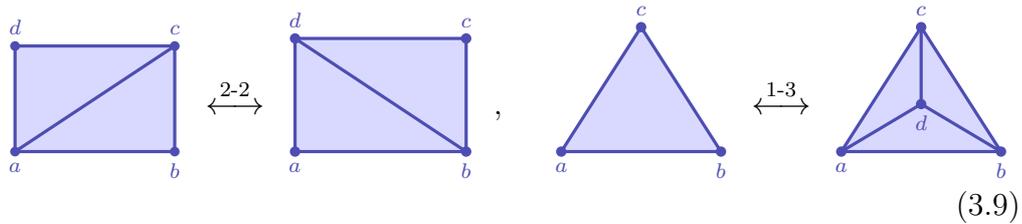
\be\label{eq:P2213}
\!
\tikzzbox{
}
\ee 
where $a,b,d,c \in \R$ are pairwise distinct. 
Hence for fixed heights $a,b,c$, the unoriented 1-3 Pachner move yields four inequivalent oriented Pachner moves depending on the relative value of~$d$.
The Poincar\'{e} dual moves are
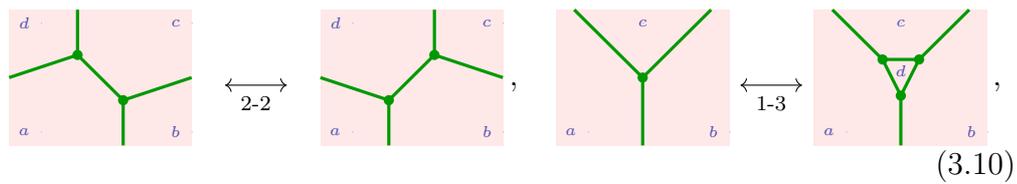
\be
\label{eq:Poincaredual2213}
\tikzzbox{%
}
\, , 
\ee
with the induced orientations as explained in Example~\ref{ex:oritria}(i). 

\item
For $n=3$, there are again only two types of Pachner moves and their inverses. 
The \textsl{2-3 move} is applied to two tetrahedra sharing precisely one face, and replaces them with the `other' tetrahedra on the boundary of $\Delta^4$, i.\,e.~with three tetrahedra that share a single edge: 
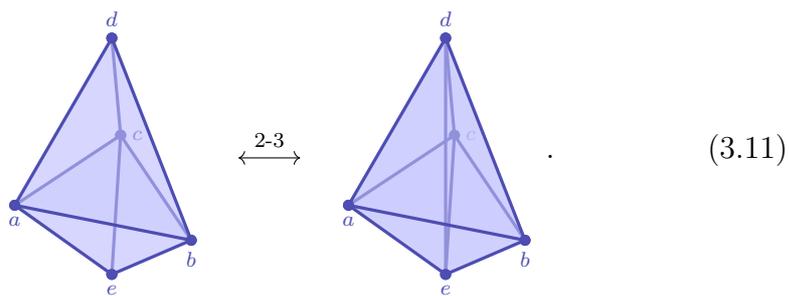
\be\label{eq:P23}
\tikzzbox{%
}
\, . 
\ee
The remaining operation is the \textsl{1-4 move} which replaces one tetrahedron by four tetrahedra meeting at a single vertex: 
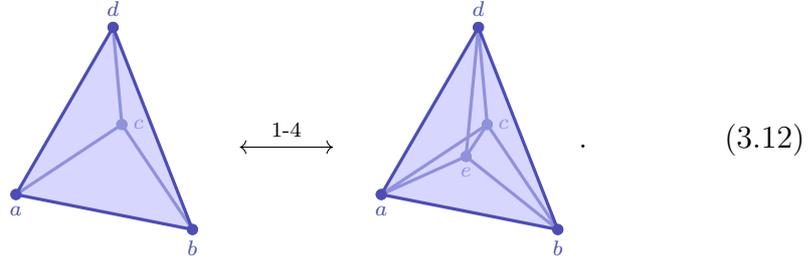
\be\label{eq:P14}
\tikzzbox{%
}
\, . 
\ee
There are five inequivalent oriented 1-4 moves for fixed $a,b,c,d$, depending on the relative value of the height~$e$. 
The Poincar\'{e} dual move to~\eqref{eq:P23} is
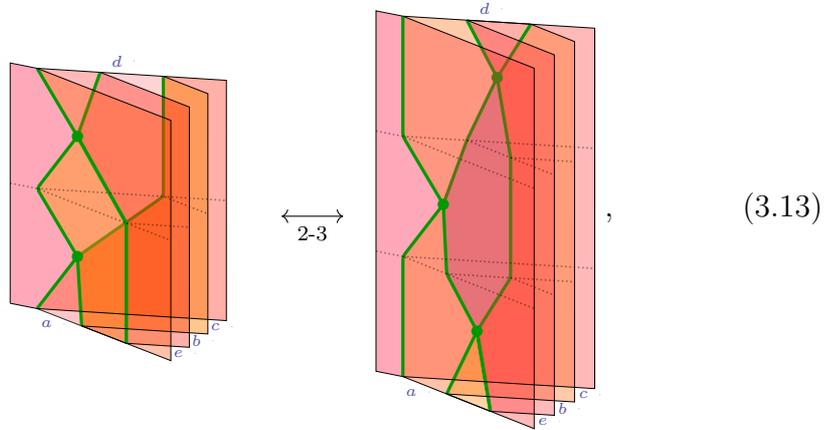
\be
\tikzzbox{%
}
\, , 
\label{eq:23dual}
\ee
with the induced orientations as explained in Example~\ref{ex:oritria}(ii). 
We will not have need in this article to work with the move Poincar\'{e} dual to~\eqref{eq:P14} directly. 
\end{enumerate}
\end{example}

\subsection[Orbifolds of $n$-dimensional TQFTs]{Orbifolds of $\boldsymbol{n}$-dimensional TQFTs}
\label{subsec:ndimorbis}

For any $n\in \Z_+$, we fix an $n$-dimensional defect TQFT
\be
\zz\colon  \Bordd[n] \lra \Vectk \, . 
\ee
We want to use~$\zz$ to construct a new closed $n$-dimensional TQFT, the ``orbifold theory''
\be
\zz_\orb \colon \Bord_n \lra \Vectk
\ee
which depends on a choice of ``orbifold datum''~$\orb$. 
In this section we will first define what an orbifold datum~$\orb$ for a given defect TQFT~$\zz$ is, and then construct the associated orbifold theory $\zz_\orb$.\footnote{%
It turns out that it is often useful to take $\zz$ in the orbifold construction described in this section to be the Euler completion of some other TQFT $\zz'$, that is $\zz = (\zz')^\euc$, cf.~Section~\ref{subsec:completewrtpointinsertions} and Section~\ref{subsec:3dimorbis}.
Indeed, we will see in \cite{CRS3} that certain examples of orbifold data only become available after this completion.
}

\medskip

Recall from Lemma~\ref{lem:rotasimp} that for $k  < n$ any two oriented $k $-simplices can be rotated into one another, while there are precisely two oriented $n$-simplices up to rotation. 
Hence from the perspective of a \textsl{topological} QFT, there are only two oriented $n$-simplices and only one oriented $k $-simplex for every $k  \in \{1,\dots,n-1\}$. 
Thus with $j:=n-k$, the $j$-strata of the Poincar\'{e} dual of a triangulation in~$n$ dimensions can be decorated with the following data: 

\begin{definition}
\label{def:orbidatan}
An \textsl{orbifold datum}~$\orb$ for a defect TQFT $\zz\colon  \Bordd[n] \to \Vectk$ is a choice of 
\begin{itemize}
\item
an element $\orb_j \in D_j$ for all $j\in\{1,2,\dots,n\}$, 
\item 
two elements $\orb_0^+, \orb_0^- \in D_0$, 
\end{itemize}
subject to the following constraints (using the notation of Section~\ref{subsec:defectbords}): 
\begin{enumerate}
\item
\textsl{Compatibility: }
For $j\in \{1,2,\dots,n-1\}$, the representatives of the class $f_j(\orb_j) \in [\text{Sphere}_{n-j-1}^{\text{def}}(\partial^{j+1} \D)]$ are homeomorphic (as stratified topological manifolds) to the Poincar\'e dual (in $n-j-1$ dimensions) of the boundary of an $(n-j)$-simplex. 
The $i$-strata in $f_j(\orb_j)$ are decorated by $\A_{i+j+1}$ for all $i \in \{ 0,\dots,n-j-1\}$. 
Similarly, the classes $f_0(\orb_0^+), f_0(\orb_0^-) \in [\text{Sphere}_{n-1}^{\text{def}}(\partial \D)]$ have representatives which are homeomorphic to the $(n-1)$-dimensional Poincar\'e dual of the boundary of the two inequivalent oriented $n$-simplices, 
respectively, and both $f_0(\orb_0^+), f_0(\orb_0^-)$ have their $i$-strata decorated by $\A_{i+1}$ for all $i \in \{ 0,\dots,n-1\}$. 
\item
\textsl{Invariance: } 
Let $j\in \{1,\dots,n+1\}$. 
For any oriented $j$-$(n+2-j)$ Pachner move, consider its Poincar\'{e} dual move between $n$-balls~$B$ and~$B'$, viewed as stratified bordisms in $\Bord_n^{\text{strat}}$ with~$\emptyset$ as their common source. 
$B,B'$ also have a common target, and there is a unique way to make them into defect bordisms by decorating their $k$-strata with $\orb_k$ for all $k\in\{0,1,\dots,n\}$. 
Then we demand the equality of vectors 
\be\label{eq:def-orb-ii-Z=Z}
\zz(B) = \zz(B') \, . 
\ee
\end{enumerate}
\end{definition}

In Sections~\ref{subsec:2dimorbis} and~\ref{subsec:3dimorbis} we will spell out both conditions for $n=2$ and $n=3$ in detail. 
For now we note that condition (i) of Definition~\ref{def:orbidatan} ensures that, as a stratified manifold, the Poincar\'{e} dual of a triangulation with total vertex order has a unique decoration by the data in~$\orb$, and thus
that condition (ii) can be consistently stated. 
In particular, 
\be
\big( 
\{\orb_n\}, \{\orb_{n-1}\}, \dots, \{\orb_1\}, \{ \orb_0^+, \orb_0^-\}; \, f_{n-1}, f_{n-2}, \dots, f_0 
\big)
\ee
is a set of defect data in the sense of Definition~\ref{def:d-dim-defect-data}. 
Condition (ii) itself implies that for any bordism~$M$ in $\Bord_n$, if we decorate the Poincar\'{e} dual of any oriented triangulation~$t$ of~$M$ with the defect labels from the orbifold datum~$\orb$, then the evaluation on the resulting defect bordism in $\Bordd[n]$ does not depend 
on the choice of~$t$ (after a limit construction to obtain source and target).
This makes the orbifold theory $\zz_\orb$ well-defined, which is the content of Constructions~\ref{constr:ZAM}--\ref{constr:ZAMonmorphisms} and Definition and Theorem~\ref{thmdef:orbifoldtheory}. 

\begin{remark}
We denote orbifold data by~$\orb$ as they are to be thought of as certain types of algebras: 
The element $\orb_{n-1}$, which labels $(n-1)$-strata that are Poincar\'{e} dual to edges, appears as the underlying ``space'' of the algebra; 
its ``multiplication'' is provided by the element $\orb_{n-2}$, which decorates $(n-2)$-strata that are Poincar\'{e} dual to a 2-simplex whose three edges (two ``ingoing'' and one ``outgoing'' for $n>2$, and both options for $n=2$) correspond to the $(n-1)$-strata associated with $\orb_{n-1}$. 
This will be made precise for $n\in \{2,3\}$ in the categorical formulation of Section~\ref{sec:highercatfor}. 
For example, for $n=2$, $\orb_1$ is literally an algebra 
(and coalgebra) with multiplication~$\orb_0^+$ (and comultiplication~$\orb_0^-$), while for $n=3$, $\orb_2$ is a monoidal category -- i.\,e.~an algebra in a higher category -- with tensor product~$\orb_1$; see also Sections~\ref{subsec:2dimorbis} and~\ref{subsec:3dimorbis}. 
\end{remark}

We begin by defining $\zz_\orb(M)$ for a closed oriented $n$-manifold~$M$, i.\,e.~a morphism $\emptyset \to \emptyset$ in $\Bord_n$. 
This is a special case of Construction~\ref{constr:ZAMonmorphisms} below, but it serves as a good warm-up and it highlights the core part of the full construction: 

\begin{construction}[Evaluation of $\zz_\orb$ on closed manifolds~$M$]
\label{constr:ZAM}
{}
\hspace{2cm}\phantom{dsdsd}
Let $\orb$ be an orbifold datum for a defect TQFT $\zz \colon \Bordd[n] \to \Vectk$. 
The number $\zz_\orb(M)$ is constructed as follows: 
\begin{enumerate}
\item
Choose a triangulation~$t$ with total order of~$M$, and denote the Poincar\'{e} dual stratification by $M^t$.%
\footnote{%
As for any Poincar\'{e} dual of a triangulation, 
the stratification $M^t$ of $M$ is only unique up to isotopy, but 
this ambiguity will be rendered inconsequential in point (iii) below.}
\item 
Decorate $M^t$ with the orbifold datum~$\orb$ 
to obtain a bordism $M^{t,\orb}$ in $\Bordd[n]$. 
More precisely, decorate 
	\begin{itemize}
	\item
	every $j$-stratum with $\orb_j$ for all $j \in \{1,2,\dots,n\}$, 
	\item
	every 0-stratum with either $\orb_0^+$ or $\orb_0^-$, as dictated by the orientation of the 0-stratum. 
	\end{itemize}
\item
Apply~$\zz$: 
\be
\zz_\orb(M) \, \stackrel{\text{def}}{=} \, \zz \big( M^{t, \orb} \big) \, . 
\ee 
\end{enumerate}
\end{construction}

By Theorem and Definition~\ref{thmdef:orbifoldtheory} below, $\zz_\orb(M)$ is independent of the choice of  triangulation with total order.

\medskip

After the special bordisms $M\colon \emptyset \to \emptyset$, we will now define the functor $\zz_\orb$ on all of $\Bord_n$. 
We will obtain $\zz_\orb$ as a limit construction, which is well-established in the literature on state sum constructions (see e.\,g.~\cite{TurVir, BarWes,LawrenceIntro, bk1004.1533}). 
In fact in \cite{CRS3} we will show that for $n=3$ models of Turaev-Viro type are the special case of $\zz_\orb$ where~$\zz$ is taken to be ``the trivial defect TQFT'', and the orbifold datum~$\orb$ is extracted from a spherical fusion category. 

\begin{construction}[Evaluation of $\zz_\orb$ on objects]
\label{constr:ZAMonobjects}
{}
\hspace{2cm}\phantom{dsdsdsdsdsdsdsds}
Let $\orb$ be an orbifold datum for a defect TQFT $\zz$, 
and let $\Sigma \in \Bord_n$. 
We define $\zz_\orb(\Sigma) \in \Vectk$ as follows: 
\begin{enumerate}
\item
For every triangulation~$\tau$ with total order of~$\Sigma$, denote the Poincar\'{e} dual stratification by $\Sigma^{\tau}$. 
Decorate $\Sigma^{\tau}$ with the orbifold datum~$\orb$. 
More precisely, decorate every $j$-stratum of $\Sigma^{\tau}$ with $\orb_{j+1}$. 
This makes~$\Sigma$ into an object $\Sigma^{\tau, \orb} \in \Bordd[n]$, for every triangulation $\tau$.\footnote{Here again the stratification $\Sigma^\tau$ of $\Sigma$ is only unique up to isotopy. However, as opposed to Construction~\ref{constr:ZAM}, evaluating $\zz$ on the object $\Sigma^{\tau,\A}$ now may depend on the choice of stratification in the isotopy class. But since different choices lead to isomorphic spaces $\zz(\Sigma^{\tau,\A})$, the ambiguity disappears in the limit construction. We will use this to just speak of ``the Poincar\'e dual'' rather than of ``a choice of Poincar\'e dual''.}	
\item 
Let $\tau,\tau'$ be triangulations as in (i).
Consider the cylinder $C_\Sigma = \Sigma \times [0,1]$ viewed as a bordism $\Sigma \to \Sigma$ in $\Bord_n$. 
Choose an oriented triangulation~$t$ of $C_\Sigma$ extending the triangulations $\tau$ and $\tau'$ on the ingoing and outgoing boundaries, respectively. 
Decorate the Poincar\'{e} dual 
$C_\Sigma^t$ with the orbifold datum~$\orb$ (analogously to Construction~\ref{constr:ZAM}) to obtain a morphism 
\be\label{eq:constr-obj-Ct}
C_\Sigma^{t,\orb}\colon \Sigma^{\tau, \orb} \lra \Sigma^{\tau', \orb}
\ee
in $\Bordd[n]$. 
\item 
Note that $\zz(C_\Sigma^{t,\orb})$ is independent of the choice of $t$. We define $\zz_\orb(\Sigma)$ to be the limit of $\zz$ applied to \eqref{eq:constr-obj-Ct} over all $\tau$, i.\,e.~$\zz_\orb(\Sigma)$ is the universal cone 
\be\label{eq:Zorb-state-space-cone}
\begin{tikzpicture}[
			     baseline=(current bounding box.base), 
			     descr/.style={fill=white,inner sep=3.5pt}, 
			     normal line/.style={->}
			     ] 
\matrix (m) [matrix of math nodes, row sep=3.5em, column sep=2.0em, text height=1.5ex, text depth=0.1ex] {%
& \zz_\orb(\Sigma) & 
\\
\zz\big(\Sigma^{\tau,\orb}\big) && \zz\big(\Sigma^{\tau',\orb}\big)
\\
};
\path[font=\footnotesize] (m-1-2) edge[->] node[below] {} (m-2-1);
\path[font=\footnotesize] (m-1-2) edge[->] node[below] {} (m-2-3);
\path[font=\footnotesize] (m-2-1) edge[->] node[below] {$ \zz\big(C_\Sigma^{t,\orb}\big) $} (m-2-3);
\end{tikzpicture}
\ee
for all triangulations $\tau, \tau'$.

To compute $\zz_\orb(\Sigma)$ explicitly, note that for  $\tau = \tau'$ the linear map $\zz(C_\Sigma^{t,\orb})$ is an idempotent. One may take $\zz_\orb(\Sigma) = \textrm{im}\zz(C_\Sigma^{\tilde t,\orb})$ for a fixed choice of triangulation $\tilde\tau$ and $C_\Sigma^{\tilde t,\orb} \colon \Sigma^{\tilde \tau,\A} \to \Sigma^{\tilde \tau,\A}$.
The diagonal arrows in \eqref{eq:Zorb-state-space-cone} are then
given by evaluating $\zz$ on $C_\Sigma^{t,\orb}$ for appropriate $t$.
\end{enumerate}
\end{construction}

\begin{construction}[Action of $\zz_\orb$ on morphisms]
\label{constr:ZAMonmorphisms}
{}
\hspace{2cm}\phantom{dsdsdsdsdsdsdsds}
Let $\orb$ be an orbifold datum for a defect TQFT $\zz$, and let $M \colon \Sigma_1 \to \Sigma_2$ be a morphism in $\Bord_n$. 
We define $M_\cup$ to be~$M$ viewed as a bordism $\emptyset \to \Sigma  := \Sigma_1^{\text{rev}} \sqcup \Sigma_2$.
\begin{enumerate}
\item
For any fixed oriented triangulation~$\tau'$ of~$\Sigma$, choose an oriented triangulation~$t'$ of $M_\cup$ 
extending the triangulation $\tau'$ on the boundary. 
\item 
Decorate the Poincar\'{e} dual stratification $M_\cup^{t'}$ with the orbifold datum~$\orb$ 
 to produce a morphism $M_\cup^{t',\orb}\colon \emptyset \to \Sigma^{\tau',\orb}$ in $\Bordd[n]$. 
\item 
Repeat steps (i) and (ii) for every triangulation $\tau''$ of~$\Sigma$ to produce a morphism $M_\cup^{t'',\orb}\colon \emptyset \to \Sigma^{\tau'',\orb}$ in $\Bordd[n]$. 
\item 
Note that the defining properties of~$\orb$ (ensuring triangulation invariance) imply that the diagrams
\be
\begin{tikzpicture}[
			     baseline=(current bounding box.base), 
			     descr/.style={fill=white,inner sep=3.5pt}, 
			     normal line/.style={->}
			     ] 
\matrix (m) [matrix of math nodes, row sep=3.5em, column sep=2.0em, text height=1.5ex, text depth=0.1ex] {%
& \zz(\emptyset) & 
\\
\zz\big(\Sigma^{\tau',\orb}\big) && \zz\big(\Sigma^{\tau'',\orb}\big)
\\
};
\path[font=\footnotesize] (m-1-2) edge[->] node[above, sloped] {$ \zz(M_{\cup}^{t',\orb}) $} (m-2-1);
\path[font=\footnotesize] (m-1-2) edge[->] node[above, sloped] {$ \zz(M_{\cup}^{t'',\orb}) $} (m-2-3);
\path[font=\footnotesize] (m-2-1) edge[->] node[below] {$ \zz\big(C_\Sigma^{t,\orb}\big) $} (m-2-3);
\end{tikzpicture}
\ee
commute, where the cylinders $C_\Sigma^{t,\A}$ are as in Construction~\ref{constr:ZAMonobjects}. 
\item
Since $\zz(\emptyset) = \Bbbk$ and because $\zz_\orb(\Sigma)$ was defined to be the universal cone, we have that every face in the associated diagram
\be\label{eq:def-on-morph-via-arrow-into-cone}
\begin{tikzpicture}[
			     baseline=(current bounding box.base), 
			     descr/.style={fill=white,inner sep=3.5pt}, 
			     normal line/.style={->}
			     ] 
\matrix (m) [matrix of math nodes, row sep=3.5em, column sep=2.0em, text height=1.5ex, text depth=0.1ex] {%
& \Bbbk & 
\\
& \zz_\orb(\Sigma) & 
\\
\zz\big(\Sigma^{\tau',\A}\big) && \zz\big(\Sigma^{\tau'',\A}\big)
\\
};
\path[font=\footnotesize] (m-1-2) edge[->,dotted] node[right] {$ \exists! $} (m-2-2);
\path[font=\footnotesize] (m-1-2) edge[->, out = -155, in=90] node[above, sloped] {$ \zz(M_{\cup}^{t',\A}) $} (m-3-1);
\path[font=\footnotesize] (m-1-2) edge[->, out = -25, in=90] node[above, sloped] {$ \zz(M_{\cup}^{t'',\A}) $} (m-3-3);
\path[font=\footnotesize] (m-2-2) edge[->] node[below] {} (m-3-1);
\path[font=\footnotesize] (m-2-2) edge[->] node[below] {} (m-3-3);
\path[font=\footnotesize] (m-3-1) edge[->] node[below] {$ \zz\big(C_\Sigma^{t,\A}\big) $} (m-3-3);
\end{tikzpicture}
\ee
commutes. 
Let us write $v_M \in \zz_\orb(\Sigma)$ for the image of $1\in \Bbbk$ under the above unique map $\Bbbk \to \zz_\orb(\Sigma)$. 
Then we define the linear map
\be
\zz_\orb(M) \colon \zz_\orb(\Sigma_1) \lra \zz_\orb(\Sigma_2)
\ee
to be the image of $v_M$ under the canonical isomorphism $\Hom_\Bbbk(\Bbbk, \zz_\orb(\Sigma)) = \Hom_\Bbbk(\Bbbk, \zz_\orb(\Sigma_1^{\text{rev}}) \otimes_\Bbbk \zz_\orb(\Sigma_2)) \cong \Hom_\Bbbk(\zz_\orb(\Sigma_1), \zz_\orb(\Sigma_2))$.

To compute $v_M$ (and hence $\zz_\orb(M)$) explicitly, use the definition of $\zz_\orb(\Sigma)$ in part (iii) of Construction~\ref{constr:ZAMonobjects} as the image of $C_\Sigma^{\tilde t,\orb} \colon \Sigma^{\tilde \tau,\A} \to \Sigma^{\tilde \tau,\A}$ under~$\zz$. 
Let $\hat t$ be an extension of $\tilde \tau$ to $M_\cup$.
The universal arrow in \eqref{eq:def-on-morph-via-arrow-into-cone} is then given by
$\zz(M_{\cup}^{\hat t,\A})$. Thus for this choice of $\zz_\orb(\Sigma)$ we find
\be
	v_M = \zz \big(M_{\cup}^{\hat t,\A} \big) \, .
\ee	
\end{enumerate}
\end{construction}

\medskip

In summary, we have established the following: 

\begin{theoremdefinition}
\label{thmdef:orbifoldtheory}
Let $\zz\colon  \Bordd[n] \to \Vectk$ be a defect TQFT and let~$\orb$ be an orbifold datum for~$\zz$. 
Then the output of Constructions~\ref{constr:ZAMonobjects} and~\ref{constr:ZAMonmorphisms} is a closed TQFT $\zz_\orb \colon \Bord_n \to \Vectk$, which we call the \textsl{$\orb$-orbifold theory}.
\end{theoremdefinition}

In Sections~\ref{subsec:2dimorbis} and~\ref{subsec:3dimorbis} we will discuss the cases $n=2$ and especially $n=3$ in detail. 

\begin{remark}
\label{rem:symmontarget}
For any symmetric 
monoidal category~$\mathcal C$, 
an \textsl{$n$-dimensional defect TQFT valued in~$\mathcal C$} is a symmetric monoidal functor $\zz\colon  \Bordd[n] \to \mathcal C$. 
Apart from the existence of limits, the above orbifold construction does not depend on special properties of the case $\mathcal C = \Vectk$. 
\end{remark}

\subsection[Orbifolds of $2$-dimensional TQFTs]{Orbifolds of $\boldsymbol{2}$-dimensional TQFTs}
\label{subsec:2dimorbis}

Let us briefly consider the case of a 2-dimensional defect TQFT $\zz\colon  \Bordd[2] \to \Vectk$ (see also \cite{ffrs0909.5013,cr1210.6363} and Section~\ref{subsec:specorbdat2} below).
According to Definition~\ref{def:orbidatan}, an orbifold datum~$\orb$ for~$\zz$ is a list of elements $\orb_j \in D_j$ for $j \in \{ 1,2 \}$ as well as $\orb_0^+,\orb_0^- \in D_0$, such that in particular
\be
\label{eq:A2A1A0}
\begin{tikzpicture}[thick,scale=1.52,color=black, baseline=0.8cm]
\fill [color=orange!12, draw=white] 
(-1,0) -- (1,0) -- (1,1) -- (-1,1);
\draw[line width=1] (0,0.5) node[line width=0pt] (sch) {{\scriptsize $\mathcal A_2$}};
\end{tikzpicture}
\, , \quad
\begin{tikzpicture}[thick,scale=1.52,color=black, baseline=0.8cm]
\fill [color=orange!12, draw=white] 
(-1,0) -- (1,0) -- (1,1) -- (-1,1);
%
\draw[
	color=green!50!black, 
	very thick,
	 >=stealth, 
	 opacity=1.0, 
	postaction={decorate}, decoration={markings,mark=at position .5 with {\arrow[draw]{>}}}
	] 
	(0,0) -- (0,1);
%
\draw[line width=1] (-0.5,0.5) node[line width=0pt] (sch) {{\scriptsize $\mathcal A_2$}};
\draw[line width=1] (0.5,0.5) node[line width=0pt] (sch) {{\scriptsize $\mathcal A_2$}};
\draw[line width=1] (0.18,0.15) node[line width=0pt] (sch) {{\scriptsize $\mathcal A_1$}};
\end{tikzpicture}
\, , \quad
\begin{tikzpicture}[very thick,scale=1.52,color=black, baseline=0.8cm]
\fill [color=orange!12, draw=white] 
(-1,0) -- (1,0) -- (1,1) -- (-1,1);
%
\draw[
	color=green!50!black, 
	 >=stealth, 
	postaction={decorate}, decoration={markings,mark=at position .5 with {\fill circle (2pt);}}, 
	decoration={markings,mark=at position .25 with {\arrow[draw]{>}}},
	decoration={markings,mark=at position .792 with {\arrow[draw]{<}}}
	] 
	(-0.35,0) .. controls +(0,0.7) and +(0,0.7) .. (0.35,0);
\draw[
	color=green!50!black, 
	 >=stealth, 
	postaction={decorate}, 
	decoration={markings,mark=at position .7 with {\arrow[draw]{>}}}
	] 
	(0,0.5) -- (0,1); 
%
\draw[line width=1] (-0.75,0.75) node[line width=0pt] (sch) {{\scriptsize $\mathcal A_2$}};
\draw[line width=1] (0.75,0.75) node[line width=0pt] (sch) {{\scriptsize $\mathcal A_2$}};
\draw[line width=1] (0,0.18) node[line width=0pt] (sch) {{\scriptsize $\mathcal A_2$}};
\draw[line width=1] (-0.49,0.15) node[line width=0pt] (sch) {{\scriptsize $\mathcal A_1$}};
\draw[line width=1] (0.5,0.15) node[line width=0pt] (sch) {{\scriptsize $\mathcal A_1$}};
\draw[line width=1] (0.17,0.85) node[line width=0pt] (sch) {{\scriptsize $\mathcal A_1$}};
\draw[line width=1] (-0.16,0.64) node[line width=0pt] (sch) {{\scriptsize $\mathcal A_0^+$}};
\draw[line width=1] (0,0.39) node[line width=0pt] (sch) {{\scriptsize $+$}};
\end{tikzpicture}
\, , \quad
\begin{tikzpicture}[very thick,scale=1.52,color=black, baseline=0.8cm]
\fill [color=orange!12, draw=white] 
(-1,0) -- (1,0) -- (1,1) -- (-1,1);
%
\draw[
	color=green!50!black, 
	 >=stealth, 
	postaction={decorate}, decoration={markings,mark=at position .5 with {\fill circle (2pt);}},
	decoration={markings,mark=at position .25 with {\arrow[draw]{<}}},
	decoration={markings,mark=at position .792 with {\arrow[draw]{>}}}
	] 
	(-0.35,1) .. controls +(0,-0.7) and +(0,-0.7) .. (0.35,1);
\draw[
	color=green!50!black, 
	 >=stealth, 
	postaction={decorate}, 
	decoration={markings,mark=at position .7 with {\arrow[draw]{<}}}
	] 
	(0,0.5) -- (0,0); 
%
\draw[line width=1] (-0.75,0.25) node[line width=0pt] (sch) {{\scriptsize $\mathcal A_2$}};
\draw[line width=1] (0.75,0.25) node[line width=0pt] (sch) {{\scriptsize $\mathcal A_2$}};
\draw[line width=1] (0,0.78) node[line width=0pt] (sch) {{\scriptsize $\mathcal A_2$}};
\draw[line width=1] (-0.48,0.85) node[line width=0pt] (sch) {{\scriptsize $\mathcal A_1$}};
\draw[line width=1] (0.49,0.85) node[line width=0pt] (sch) {{\scriptsize $\mathcal A_1$}};
\draw[line width=1] (0.17,0.15) node[line width=0pt] (sch) {{\scriptsize $\mathcal A_1$}};
\draw[line width=1] (-0.15,0.36) node[line width=0pt] (sch) {{\scriptsize $\orb_0^-$}};
\draw[line width=1] (0,0.56) node[line width=0pt] (sch) {{\scriptsize $-$}};
\end{tikzpicture}
\ee
are local patches of bordisms in $\Bordd[2]$. 
For any bordism~$M$ in $\Bord_2$ together with a choice of triangulation~$t$ with total order, we can decorate the Poincar\'{e} dual with~$\orb$ to obtain a bordism $M^{t,\orb}$ in $\Bordd[2]$.
Each inner point of $M^{t,\orb}$ now has a neighbourhood 
	isomorphic 
to one of the patches shown in \eqref{eq:A2A1A0}.
The constraints on~$\orb$ imply that evaluation of~$\zz$ on any $\A$-decorated bordism is invariant under the Poincar\'e dual oriented Pachner moves~\eqref{eq:Poincaredual2213}. 

A ``special'' type of solution to the constraints on the data in~\eqref{eq:A2A1A0} to be an orbifold datum for~$\zz$ has a well-studied purely algebraic description. 
Indeed, there is a natural monoidal category $\B_\zz(a,a)$ associated to every element $a \in D_2$ (see \cite{dkr1107.0495}),
and we may rewrite~\eqref{eq:A2A1A0} as
\begin{align}
& * := \orb_2 
\, , \quad
A := \orb_1 \in \B_\zz(*,*) 
\, , \quad
\nonumber
\\
&
\mu := \zz\Bigg(
\begin{tikzpicture}[very thick, scale=0.45,color=green!50!black, baseline, >=stealth]
\fill [color=orange!12, draw=white]  (0,0) circle (2);
\draw[very thick, color=red!80!black] (0,0) circle (2);
\fill (-45:2) circle (4.0pt) node[left] {};
\fill (225:2) circle (4.0pt) node[left] {};
\fill (90:2) circle (4.0pt) node[left] {};
\draw[
	color=green!50!black, 
	 >=stealth, 
	postaction={decorate}, decoration={markings,mark=at position .5 with {\fill circle (2pt);}}, 
	decoration={markings,mark=at position .21 with {\arrow[draw]{>}}},
	decoration={markings,mark=at position .85 with {\arrow[draw]{<}}}
	] 
	(-45:2) .. controls +(0,2) and +(0,2) .. (225:2);
\draw[
	color=green!50!black, 
	 >=stealth, 
	postaction={decorate}, 
	decoration={markings,mark=at position .7 with {\arrow[draw]{>}}}
	] 
	(0,0) -- (90:2); 
\draw[line width=1, color=black] (0,-0.45) node[line width=0pt] (sch) {{\scriptsize $\mathcal A_0^+$}};
\end{tikzpicture}
\Bigg)
\colon A \otimes A \lra A
\, ,
\nonumber
\\
&
\Delta := \zz\Bigg(
\begin{tikzpicture}[very thick, scale=0.45,color=green!50!black, baseline, >=stealth, rotate=180]
\fill [color=orange!12, draw=white]  (0,0) circle (2);
\draw[very thick, color=red!80!black] (0,0) circle (2);
\fill (-45:2) circle (4.0pt) node[left] {};
\fill (225:2) circle (4.0pt) node[left] {};
\fill (90:2) circle (4.0pt) node[left] {};
\draw[
	color=green!50!black, 
	 >=stealth, 
	postaction={decorate}, decoration={markings,mark=at position .5 with {\fill circle (2pt);}}, 
	decoration={markings,mark=at position .21 with {\arrow[draw]{<}}},
	decoration={markings,mark=at position .85 with {\arrow[draw]{>}}}
	] 
	(-45:2) .. controls +(0,2) and +(0,2) .. (225:2);
\draw[
	color=green!50!black, 
	 >=stealth, 
	postaction={decorate}, 
	decoration={markings,mark=at position .7 with {\arrow[draw]{<}}}
	] 
	(0,0) -- (90:2); 
\draw[line width=1, color=black] (0,-0.45) node[line width=0pt] (sch) {{\scriptsize $\orb_0^-$}};
\end{tikzpicture}
\Bigg)
\colon A \lra A \otimes A 
\, . 
\label{eq:starAmuDelta}
\end{align}
It was shown in \cite[Prop.\,3.4]{cr1210.6363} that a sufficient condition for~\eqref{eq:A2A1A0} to form an orbifold datum is that~\eqref{eq:starAmuDelta} together with 
\be
\eps := \zz\Bigg(
\begin{tikzpicture}[very thick, scale=0.45,color=green!50!black, baseline, >=stealth]
\fill [color=orange!12, draw=white]  (0,0) circle (2);
\draw[very thick, color=red!80!black] (0,0) circle (2);
\fill (90:2) circle (4.0pt) node[left] {};
\draw[postaction={decorate}, decoration={markings,mark=at position .51 with {\arrow[draw=green!50!black]{>}}}] (0,0.8) -- (90:2);
\draw[-dot-] (0,0.8) .. controls +(0,-0.5) and +(0,-0.5) .. (-0.75,0.8);
\draw[color=green!50!black] (-0.75,0.8) .. controls +(0,0.5) and +(0,0.5) .. (-1.5,0.8);
\draw[color=green!50!black] (-0.375,-0.4) .. controls +(0,-0.75) and +(0,-0.75) .. (-1.5,-0.4);
\draw[postaction={decorate}, decoration={markings,mark=at position .51 with {\arrow[draw=green!50!black]{<}}}] (-1.5,-0.4) -- (-1.5,0.8);
\draw (-0.375,-0.4) -- (-0.375,0.4);
\fill[black] (0.25,0.25) circle (0pt) node {{\scriptsize$\orb_0^-$}};
\end{tikzpicture}
\Bigg)
\colon \one \lra A
\, , \quad
\eta := \zz\Bigg(
\begin{tikzpicture}[very thick, scale=0.45,color=green!50!black, baseline, rotate=180, >=stealth]
\fill [color=orange!12, draw=white]  (0,0) circle (2);
\draw[very thick, color=red!80!black] (0,0) circle (2);
\fill (90:2) circle (4.0pt) node[left] {};
\draw[postaction={decorate}, decoration={markings,mark=at position .51 with {\arrow[draw=green!50!black]{<}}}]  (0,0.8) -- (90:2);
\draw[-dot-] (0,0.8) .. controls +(0,-0.5) and +(0,-0.5) .. (0.75,0.8);
\draw[color=green!50!black] (0.75,0.8) .. controls +(0,0.5) and +(0,0.5) .. (1.5,0.8);
\draw[color=green!50!black] (0.375,-0.4) .. controls +(0,-0.75) and +(0,-0.75) .. (1.5,-0.4);
\draw[postaction={decorate}, decoration={markings,mark=at position .51 with {\arrow[draw=green!50!black]{>}}}] (1.5,-0.4) -- (1.5,0.8);
\draw (0.375,-0.4) -- (0.375,0.4);
\fill[black] (-0.25,0.15) circle (0pt) node {{\scriptsize${\mathcal A}_0$}};
\end{tikzpicture}
\Bigg)
\colon A \lra \one 
\ee
form a $\Delta$-separable symmetric Frobenius algebra in $\B_\zz(*,*)$ (whose definition we will recall in Section~\ref{subsec:specorbdat2}).
The condition of $\Delta$-separability means that $\mu \circ \Delta = \id$; in terms of line defects, this amounts to leaving out a bubble (cf.~the first identity below in~\eqref{eq:DeltasepsymFrob}). 
While not itself Poincar\'e dual to a triangulation (instead the dual cell complex has triangles glued to each other along two edges), this ``bubble-move'' implies the 3-1 Pachner move. 

Quite generally, if invariance under the $(n{+}1)$-$1$-Pachner move -- the only move that changes the number of simplices -- in Definition~\ref{def:orbidatan} is replaced by invariance under appropriate ``bubble moves'' we refer to such an orbifold datum as ``special''.
Hence we define a \textsl{2-dimensional special orbifold datum} for~$\zz$ to be a $\Delta$-separable symmetric Frobenius algebra in~$\B_\zz$. 
(Note that $\Delta$-separable Frobenius algebras are closely related to special Frobenius algebras \cite[Def.\,2.3]{fs0106050}, which is another reason for the term ``special orbifold datum''.)

\subsection[Orbifolds of $3$-dimensional TQFTs]{Orbifolds of $\boldsymbol{3}$-dimensional TQFTs}
\label{subsec:3dimorbis}

Being able to treat 3-dimensional orbifolds was our main motivation for developing the general formalism in Section~\ref{subsec:ndimorbis}.
Important earlier work on the significance of defects in 3-dimensional TQFT includes studies of examples in
	Rozansky-Witten theory \cite{KRS, KR0909.3643}, 
	Dijkgraaf-Witten theory \cite{fpsv1404.6646}, 
   	Chern-Simons theory \cite{ks1012.0911}, 
    	and a general analysis of defects in Reshitikhin-Turaev TQFTs \cite{fsv1203.4568}.
There is a close connection between 3-dimensional TQFTs and $(2{+}1)$-dimensional topological phases of matter.  
In the latter context, defects, group symmetries and orbifolding have been investigated e.\,g.\ in~\cite{KK1104.5047, BJQ, fs1310.1329,BBCW1410.4540, CGPW}. 
A general 3-categorical algebraic framework to accommodate the structure of defects in 3-dimensional TQFTs is developed in \cite{BMS, CMS}, and in Section~\ref{subsec:specorbdat3} we will place the present approach to orbifolds in this framework.

\medskip

Fix a 3-dimensional defect TQFT $\zz\colon  \Bordd[3] \to \Vectk$. 
According to Definition~\ref{def:orbidatan}, an orbifold datum~$\orb$ for~$\zz$ is a list of elements $\orb_j \in D_j$ for $j \in \{ 1,2,3 \}$ as well as $\orb_0^+,\orb_0^- \in D_0$, such that (with all $\orb_2$-decorated 2-strata oriented
	by
the blackboard framing)
\be
\!\!\!\!\!\!
\tikzzbox{\begin{tikzpicture}[thick,scale=2.321,color=blue!50!black, baseline=0.0cm, >=stealth, 
				style={x={(-0.6cm,-0.4cm)},y={(1cm,-0.2cm)},z={(0cm,0.9cm)}}]
	\pgfmathsetmacro{\yy}{0.2}
\coordinate (v1) at (1, 0.5, 0);
\coordinate (v2) at (0, 0.5, 0);
\coordinate (v3) at (0, 0.5, 1);
\coordinate (v4) at (1, 0.5, 1);
\fill [red!50,opacity=0.545] (v1) -- (v2) -- (v3) -- (v4);
\fill[color=blue!60!black] (0.5,0.46,0) circle (0pt) node[left] (0up) { {\scriptsize$\mathcal A_2$} };
%
\fill[color=black] (0.45,0.75,0.01) circle (0pt) node[left] (0up) { {\scriptsize$\mathcal A_3$} };
\fill[color=black] (0.5,0.45,0.99) circle (0pt) node[left] (0up) { {\scriptsize$\mathcal A_3$} };
%
%
\end{tikzpicture}}
\, , \qquad
\tikzzbox{\begin{tikzpicture}[thick,scale=2.321,color=blue!50!black, baseline=0.0cm, >=stealth, 
				style={x={(-0.6cm,-0.4cm)},y={(1cm,-0.2cm)},z={(0cm,0.9cm)}}]
	\pgfmathsetmacro{\yy}{0.2}
\coordinate (T) at (0.5, 0.4, 0);
\coordinate (L) at (0.5, 0, 0);
\coordinate (R1) at (0.3, 1, 0);
\coordinate (R2) at (0.7, 1, 0);
\coordinate (1T) at (0.5, 0.4, 1);
\coordinate (1L) at (0.5, 0, 1);
\coordinate (1R1) at (0.3, 1, );
\coordinate (1R2) at (0.7, 1, );
%
\coordinate (p3) at (0.1, 0.1, 0.5);
\coordinate (p2) at (0.5, 0.95, 0.5);
\coordinate (p1) at (0.9, 0.1, 0.5);
%
\fill[color=blue!60!black] (0.5,0.25,0.15) circle (0pt) node[left] (0up) { {\scriptsize$\mathcal A_2$} };
\fill [red!50,opacity=0.545] (L) -- (T) -- (1T) -- (1L);
\fill [red!50,opacity=0.545] (R1) -- (T) -- (1T) -- (1R1);
\fill [red!50,opacity=0.545] (R2) -- (T) -- (1T) -- (1R2);
\fill[color=blue!60!black] (0.5,0.25,0.15) circle (0pt) node[left] (0up) { {\scriptsize$\mathcal A_2$} };
\fill[color=blue!60!black] (0.15,0.95,0.07) circle (0pt) node[left] (0up) { {\scriptsize$\mathcal A_2$} };
\fill[color=blue!60!black] (0.55,0.95,0.1) circle (0pt) node[left] (0up) { {\scriptsize$\mathcal A_2$} };
%
\draw[string=green!60!black, very thick] (T) -- (1T);
\fill[color=green!60!black] (0.5,0.43,0.5) circle (0pt) node[left] (0up) { {\scriptsize$\mathcal A_1$} };
%
%
\fill[color=black] (0.5,1.14,0.04) circle (0pt) node[left] (0up) { {\scriptsize$\mathcal A_3$} };
\fill[color=black] (0.7,0.5,0.05) circle (0pt) node[left] (0up) { {\scriptsize$\mathcal A_3$} };
\fill[color=black] (0.3,0.5,1.02) circle (0pt) node[left] (0up) { {\scriptsize$\mathcal A_3$} };
%
\draw [black,opacity=1, very thin] (1T) -- (1L) -- (L) -- (T);
\draw [black,opacity=1, very thin] (1T) -- (1R1) -- (R1) -- (T);
\draw [black,opacity=1, very thin] (1T) -- (1R2) -- (R2) -- (T);
\end{tikzpicture}}
\, , \qquad
\tikzzbox{\begin{tikzpicture}[thick,scale=2.321,color=blue!50!black, baseline=0.0cm, >=stealth, 
				style={x={(-0.6cm,-0.4cm)},y={(1cm,-0.2cm)},z={(0cm,0.9cm)}}]
	\pgfmathsetmacro{\yy}{0.2}
\coordinate (P) at (0.5, \yy, 0);
\coordinate (R) at (0.625, 0.5 + \yy/2, 0);
\coordinate (L) at (0.5, 0, 0);
\coordinate (R1) at (0.25, 1, 0);
\coordinate (R2) at (0.5, 1, 0);
\coordinate (R3) at (0.75, 1, 0);
\coordinate (Pt) at (0.5, \yy, 1);
\coordinate (Rt) at (0.375, 0.5 + \yy/2, 1);
\coordinate (Lt) at (0.5, 0, 1);
\coordinate (R1t) at (0.25, 1, 1);
\coordinate (R2t) at (0.5, 1, 1);
\coordinate (R3t) at (0.75, 1, 1);
\coordinate (alpha) at (0.5, 0.5, 0.5);
%
\draw[string=green!60!black, very thick] (alpha) -- (Rt);
\fill [red!50,opacity=0.545] (L) -- (P) -- (alpha) -- (Pt) -- (Lt);
\fill [red!50,opacity=0.545] (Pt) -- (Rt) -- (alpha);
\fill [red!50,opacity=0.545] (Rt) -- (R1t) -- (R1) -- (P) -- (alpha);
\fill [red!50,opacity=0.545] (Rt) -- (R2t) -- (R2) -- (R) -- (alpha);
\draw[string=green!60!black, very thick] (alpha) -- (Rt);
\fill[color=blue!60!black] (0.5,0.59,0.94) circle (0pt) node[left] (0up) { {\scriptsize$\mathcal A_2$} };
\fill[color=green!60!black] (0.5,0.77,0.77) circle (0pt) node[left] (0up) { {\scriptsize$\mathcal A_1$} };
\fill [red!50,opacity=0.545] (Pt) -- (R3t) -- (R3) -- (R) -- (alpha);
\fill [red!50,opacity=0.545] (P) -- (R) -- (alpha);
%
\draw[string=green!60!black, very thick] (P) -- (alpha);
\draw[string=green!60!black, very thick] (R) -- (alpha);
\draw[string=green!60!black, very thick] (alpha) -- (Pt);
%
\fill[color=green!60!black] (alpha) circle (1.2pt) node[left] (0up) { {\scriptsize$\mathcal A_0^+$} };
\fill[color=green!60!black] (alpha) circle (0pt) node[right] (0up) { {\scriptsize$+$} };
\fill[color=blue!60!black] (0.5,0.25,0.15) circle (0pt) node[left] (0up) { {\scriptsize$\mathcal A_2$} };
\fill[color=blue!60!black] (0.5,0.5,0.09) circle (0pt) node[left] (0up) { {\scriptsize$\mathcal A_2$} };
\fill[color=blue!60!black] (0.5,0.89,0.0) circle (0pt) node[left] (0up) { {\scriptsize$\mathcal A_2$} };
\fill[color=blue!60!black] (0.5,1.04,0.13) circle (0pt) node[left] (0up) { {\scriptsize$\mathcal A_2$} };
\fill[color=blue!60!black] (0.5,1.19,0.28) circle (0pt) node[left] (0up) { {\scriptsize$\mathcal A_2$} };
\fill[color=green!60!black] (0.5,0.35,0.24) circle (0pt) node[left] (0up) { {\scriptsize$T$} };
\fill[color=green!60!black] (0.5,0.72,0.21) circle (0pt) node[left] (0up) { {\scriptsize$T$} };
\fill[color=green!60!black] (0.5,0.4,0.71) circle (0pt) node[left] (0up) { {\scriptsize$T$} };
%
\draw [black,opacity=1, very thin] (Pt) -- (Lt) -- (L) -- (P);
\draw [black,opacity=1, very thin] (Pt) -- (Rt);
\draw [black,opacity=1, very thin] (Rt) -- (R1t) -- (R1) -- (P);
\draw [black,opacity=1, very thin] (Rt) -- (R2t) -- (R2) -- (R);
\draw [black,opacity=1, very thin] (Pt) -- (R3t) -- (R3) -- (R);
\draw [black,opacity=1, very thin] (P) -- (R);
\fill[color=black] (0.5,1.12,-0.04) circle (0pt) node[left] (0up) { {\scriptsize$\mathcal A_3$} };
\fill[color=black] (0.5,1.26,0.1) circle (0pt) node[left] (0up) { {\scriptsize$\mathcal A_3$} };
\fill[color=black] (0.7,0.5,0.05) circle (0pt) node[left] (0up) { {\scriptsize$\mathcal A_3$} };
\fill[color=black] (0.3,0.5,1.05) circle (0pt) node[left] (0up) { {\scriptsize$\mathcal A_3$} };
\end{tikzpicture}}
\, , \qquad
\tikzzbox{\begin{tikzpicture}[thick,scale=2.321,color=blue!50!black, baseline=0.0cm, >=stealth, 
				style={x={(-0.6cm,-0.4cm)},y={(1cm,-0.2cm)},z={(0cm,0.9cm)}}]
	\pgfmathsetmacro{\yy}{0.2}
\coordinate (P) at (0.5, \yy, 0);
\coordinate (R) at (0.375, 0.5 + \yy/2, 0);
\coordinate (L) at (0.5, 0, 0);
\coordinate (R1) at (0.25, 1, 0);
\coordinate (R2) at (0.5, 1, 0);
\coordinate (R3) at (0.75, 1, 0);
\coordinate (Pt) at (0.5, \yy, 1);
\coordinate (Rt) at (0.625, 0.5 + \yy/2, 1);
\coordinate (Lt) at (0.5, 0, 1);
\coordinate (R1t) at (0.25, 1, 1);
\coordinate (R2t) at (0.5, 1, 1);
\coordinate (R3t) at (0.75, 1, 1);
\coordinate (alpha) at (0.5, 0.5, 0.5);
%
\draw[string=green!60!black, very thick] (alpha) -- (Rt);
\fill [red!50,opacity=0.545] (L) -- (P) -- (alpha) -- (Pt) -- (Lt);
\fill [red!50,opacity=0.545] (Pt) -- (R1t) -- (R1) -- (R) -- (alpha);
\fill [red!50,opacity=0.545] (Rt) -- (R2t) -- (R2) -- (R) -- (alpha);
\fill [red!50,opacity=0.545] (Pt) -- (Rt) -- (alpha);
\fill [red!50,opacity=0.545] (P) -- (R) -- (alpha);
\fill[color=green!60!black] (0.5,0.82,0.3) circle (0pt) node[left] (0up) { {\scriptsize$\mathcal A_1$} };
\fill[color=blue!60!black] (0.5,0.55,0.19) circle (0pt) node[left] (0up) { {\scriptsize$\mathcal A_2$} };
\draw[string=green!60!black, very thick] (R) -- (alpha);
\fill [red!50,opacity=0.545] (Rt) -- (R3t) -- (R3) -- (P) -- (alpha);
\draw[string=green!60!black, very thick] (alpha) -- (Rt);
%
\draw[string=green!60!black, very thick] (P) -- (alpha);
\draw[string=green!60!black, very thick] (alpha) -- (Pt);
%
\fill[color=green!60!black] (alpha) circle (1.2pt) node[left] (0up) { {\scriptsize$\orb_0^-$} };
\fill[color=green!60!black] (alpha) circle (0pt) node[right] (0up) { {\scriptsize$-$} };
\fill[color=blue!60!black] (0.5,0.25,0.15) circle (0pt) node[left] (0up) { {\scriptsize$\mathcal A_2$} };
\fill[color=blue!60!black] (0.5,0.89,0.0) circle (0pt) node[left] (0up) { {\scriptsize$\mathcal A_2$} };
\fill[color=blue!60!black] (0.5,1.04,0.13) circle (0pt) node[left] (0up) { {\scriptsize$\mathcal A_2$} };
\fill[color=blue!60!black] (0.5,1.19,0.28) circle (0pt) node[left] (0up) { {\scriptsize$\mathcal A_2$} };
\fill[color=blue!60!black] (0.5,0.55,0.86) circle (0pt) node[left] (0up) { {\scriptsize$\mathcal A_2$} };
\fill[color=green!60!black] (0.5,0.35,0.24) circle (0pt) node[left] (0up) { {\scriptsize$\mathcal A_1$} };
\fill[color=green!60!black] (0.5,0.4,0.71) circle (0pt) node[left] (0up) { {\scriptsize$\mathcal A_1$} };
\fill[color=green!60!black] (0.5,0.73,0.72) circle (0pt) node[left] (0up) { {\scriptsize$\mathcal A_1$} };
%
\draw [black,opacity=1, very thin] (Pt) -- (Lt) -- (L) -- (P) ;
\draw [black,opacity=1, very thin] (Pt) -- (R1t) -- (R1) -- (R);
\draw [black,opacity=1, very thin] (Rt) -- (R2t) -- (R2) -- (R);
\draw [black,opacity=1, very thin] (Pt) -- (Rt);
\draw [black,opacity=1, very thin] (P) -- (R);
\draw [black,opacity=1, very thin] (Rt) -- (R3t) -- (R3) -- (P);
\fill[color=black] (0.5,1.12,-0.04) circle (0pt) node[left] (0up) { {\scriptsize$\mathcal A_3$} };
\fill[color=black] (0.5,1.26,0.1) circle (0pt) node[left] (0up) { {\scriptsize$\mathcal A_3$} };
\fill[color=black] (0.7,0.5,0.05) circle (0pt) node[left] (0up) { {\scriptsize$\mathcal A_3$} };
\fill[color=black] (0.3,0.5,1.05) circle (0pt) node[left] (0up) { {\scriptsize$\mathcal A_3$} };
\end{tikzpicture}}
\ee
are local patches of bordisms in $\Bordd[3]$. 
	For
any bordism~$M$ in $\Bord_3$ together with a choice of triangulation~$t$, we can decorate the Poincar\'{e} dual with~$\orb$ to obtain a bordism $M^{t,\orb}$ in $\Bordd[3]$.
By the constraints on~$\orb$, $\zz$ is invariant under the Poincar\'e duals of the oriented Pachner moves~\eqref{eq:P23} and~\eqref{eq:P14}. 
In particular, $\zz(M^{t,\orb})$ is independent of the choice of triangulation~$t$. 

Up to rotations there are precisely 20 inequivalent oriented 2-3 moves and ten inequivalent 1-4 moves: 
\be
\tikzzbox{\begin{tikzpicture}[thick,scale=2.321,color=gray!60!blue, baseline=-0.3cm, >=stealth, 
				style={x={(-0.6cm,-0.4cm)},y={(1cm,-0.2cm)},z={(0cm,0.9cm)}}]
\coordinate (v1) at (1,0,0);
\coordinate (v2) at (1,1,0);
\coordinate (v3) at (0,0,0);
\coordinate (v4) at (0.25,0.1,0.75);
\coordinate (v0) at (0.25,0.1,-0.75);
\fill [blue!20,opacity=0.545] (v1) -- (v2) -- (v3);
\fill [blue!20,opacity=0.545] (v4) -- (v2) -- (v3);
\fill [blue!20,opacity=0.545] (v1) -- (v4) -- (v3);
\fill[color=gray!60!blue] (v3) circle (0.9pt) node[right] (0up) {{\scriptsize$c$}};
\fill[color=gray!60!blue] (v1) circle (0.9pt) node[below] (0up) {{\scriptsize$a$}};
\fill[color=gray!60!blue] (v2) circle (0.9pt) node[below] (0up) {{\scriptsize$b$}};
\fill[color=gray!60!blue] (v4) circle (0.9pt) node[above] (0up) {{\scriptsize$d$}};
\fill[color=gray!60!blue] (v0) circle (0.9pt) node[below] (0up) {{\scriptsize$e$}};
\fill[color=gray!60!blue] (v3) circle (0.9pt) node[left] (0up) {};
\fill [blue!20,opacity=0.545] (v0) -- (v1) -- (v3);
\fill [blue!20,opacity=0.545] (v0) -- (v2) -- (v3);
\draw[color=gray!60!blue, very thick] (v1) -- (v3);
\draw[color=gray!60!blue, very thick] (v2) -- (v3);
\draw[color=gray!60!blue, very thick] (v3) -- (v4);
%
%
\draw[color=gray!60!blue, very thick] (v0) -- (v3);
\fill [blue!20,opacity=0.545] (v0) -- (v1) -- (v2);
\draw[color=gray!60!blue, very thick] (v0) -- (v1);
\draw[color=gray!60!blue, very thick] (v0) -- (v2);
\fill [blue!20,opacity=0.545] (v1) -- (v2) -- (v4);
\draw[color=gray!60!blue, very thick] (v2) -- (v4);
\draw[color=gray!60!blue, very thick] (v1) -- (v4);
\draw[color=gray!60!blue, very thick] (v1) -- (v2);
\fill[color=gray!60!blue] (v0) circle (0.9pt) node[below] (0up) {};
\fill[color=gray!60!blue] (v1) circle (0.9pt) node[below] (0up) {};
\fill[color=gray!60!blue] (v2) circle (0.9pt) node[below] (0up) {};
\fill[color=gray!60!blue] (v4) circle (0.9pt) node[above] (0up) {};
\end{tikzpicture}}
%
\!\!\!\!
\begin{tikzpicture}[
			     baseline=(current bounding box.base), 
			     descr/.style={fill=white,inner sep=3.5pt}, 
			     normal line/.style={->}
			     ] 
\matrix (m) [matrix of math nodes, row sep=3.5em, column sep=2.0em, text height=1.1ex, text depth=0.1ex] {%
{}
&
{}
\\
};
\path[font=\footnotesize] (m-1-1) edge[<->] node[auto] { {\scriptsize 2-3}} (m-1-2);
\end{tikzpicture}
%
\!\!\!\!
\tikzzbox{\begin{tikzpicture}[thick,scale=2.321,color=gray!60!blue, baseline=-0.3cm, >=stealth, 
				style={x={(-0.6cm,-0.4cm)},y={(1cm,-0.2cm)},z={(0cm,0.9cm)}}]
\coordinate (v1) at (1,0,0);
\coordinate (v2) at (1,1,0);
\coordinate (v3) at (0,0,0);
\coordinate (v4) at (0.25,0.1,0.75);
\coordinate (v0) at (0.25,0.1,-0.75);
\fill [blue!20,opacity=0.545] (v0) -- (v1) -- (v4);
\fill [blue!20,opacity=0.545] (v0) -- (v2) -- (v4);
\fill [blue!20,opacity=0.545] (v0) -- (v3) -- (v4);
\fill[color=gray!60!blue] (v3) circle (0.9pt) node[right] (0up) {{\scriptsize$c$}};
\fill[color=gray!60!blue] (v1) circle (0.9pt) node[below] (0up) {{\scriptsize$a$}};
\fill[color=gray!60!blue] (v2) circle (0.9pt) node[below] (0up) {{\scriptsize$b$}};
\fill[color=gray!60!blue] (v4) circle (0.9pt) node[above] (0up) {{\scriptsize$d$}};
\fill[color=gray!60!blue] (v0) circle (0.9pt) node[below] (0up) {{\scriptsize$e$}};
\fill [blue!20,opacity=0.545] (v1) -- (v2) -- (v3);
\fill [blue!20,opacity=0.545] (v4) -- (v2) -- (v3);
\fill [blue!20,opacity=0.545] (v1) -- (v4) -- (v3);
\fill[color=gray!60!blue] (v3) circle (0.9pt) node[left] (0up) {};
\fill [blue!20,opacity=0.545] (v0) -- (v1) -- (v3);
\fill [blue!20,opacity=0.545] (v0) -- (v2) -- (v3);
\draw[color=gray!60!blue, very thick] (v1) -- (v3);
\draw[color=gray!60!blue, very thick] (v2) -- (v3);
\draw[color=gray!60!blue, very thick] (v3) -- (v4);
\draw[color=gray!60!blue, very thick] (v0) -- (v3);
\draw[color=gray!60!blue, very thick] (v0) -- (v4);
\fill [blue!20,opacity=0.545] (v0) -- (v1) -- (v2);
\draw[color=gray!60!blue, very thick] (v0) -- (v1);
\draw[color=gray!60!blue, very thick] (v0) -- (v2);
\fill [blue!20,opacity=0.545] (v1) -- (v2) -- (v4);
\draw[color=gray!60!blue, very thick] (v2) -- (v4);
\draw[color=gray!60!blue, very thick] (v1) -- (v4);
\draw[color=gray!60!blue, very thick] (v1) -- (v2);
\fill[color=gray!60!blue] (v0) circle (0.9pt) node[below] (0up) {};
\fill[color=gray!60!blue] (v1) circle (0.9pt) node[below] (0up) {};
\fill[color=gray!60!blue] (v2) circle (0.9pt) node[below] (0up) {};
\fill[color=gray!60!blue] (v4) circle (0.9pt) node[above] (0up) {};
\end{tikzpicture}}
\, , \qquad
\tikzzbox{\begin{tikzpicture}[thick,scale=2.321,color=gray!60!blue, baseline=-0.3cm, >=stealth, 
				style={x={(-0.6cm,-0.4cm)},y={(1cm,-0.2cm)},z={(0cm,0.9cm)}}]
\coordinate (v1) at (1,0,0);
\coordinate (v2) at (1,1,0);
\coordinate (v3) at (0,0,0);
\coordinate (v4) at (0.25,0.1,0.75);
%
\fill [blue!20,opacity=0.545] (v1) -- (v2) -- (v3);
\fill [blue!20,opacity=0.545] (v4) -- (v2) -- (v3);
\fill [blue!20,opacity=0.545] (v1) -- (v4) -- (v3);
\fill[color=gray!60!blue] (v3) circle (0.9pt) node[right] (0up) {{\scriptsize$c$}};
\fill[color=gray!60!blue] (v1) circle (0.9pt) node[below] (0up) {{\scriptsize$a$}};
\fill[color=gray!60!blue] (v2) circle (0.9pt) node[below] (0up) {{\scriptsize$b$}};
\fill[color=gray!60!blue] (v4) circle (0.9pt) node[above] (0up) {{\scriptsize$d$}};
\fill[color=gray!60!blue] (v3) circle (0.9pt) node[left] (0up) {};
\draw[color=gray!60!blue, very thick] (v1) -- (v3);
\draw[color=gray!60!blue, very thick] (v2) -- (v3);
\draw[color=gray!60!blue, very thick] (v3) -- (v4);
\fill [blue!20,opacity=0.545] (v1) -- (v2) -- (v4);
\draw[color=gray!60!blue, very thick] (v2) -- (v4);
\draw[color=gray!60!blue, very thick] (v1) -- (v4);
\draw[color=gray!60!blue, very thick] (v1) -- (v2);
\fill[color=gray!60!blue] (v1) circle (0.9pt) node[below] (0up) {};
\fill[color=gray!60!blue] (v2) circle (0.9pt) node[below] (0up) {};
\fill[color=gray!60!blue] (v4) circle (0.9pt) node[above] (0up) {};
\end{tikzpicture}}
%
\!\!\!\!
\begin{tikzpicture}[
			     baseline=(current bounding box.base), 
			     descr/.style={fill=white,inner sep=3.5pt}, 
			     normal line/.style={->}
			     ] 
\matrix (m) [matrix of math nodes, row sep=3.5em, column sep=2.0em, text height=1.1ex, text depth=0.1ex] {%
{} 
&
{}
\\
};
\path[font=\footnotesize] (m-1-1) edge[<->] node[auto] { {\scriptsize 1-4}} (m-1-2);
\end{tikzpicture}
\!\!\!\!
%
\tikzzbox{\begin{tikzpicture}[thick,scale=2.321,color=gray!60!blue, baseline=-0.3cm, >=stealth, 
				style={x={(-0.6cm,-0.4cm)},y={(1cm,-0.2cm)},z={(0cm,0.9cm)}}]
\coordinate (v1) at (1,0,0);
\coordinate (v2) at (1,1,0);
\coordinate (v3) at (0,0,0);
\coordinate (v4) at (0.25,0.1,0.75);
\coordinate (v5) at (0.7,0.3,0.175);
\fill [blue!20,opacity=0.545] (v1) -- (v2) -- (v3);
\fill [blue!20,opacity=0.545] (v4) -- (v2) -- (v3);
\fill [blue!20,opacity=0.545] (v1) -- (v4) -- (v3);
\fill[color=gray!60!blue] (v3) circle (0.9pt) node[left] (0up) {};
\fill[color=gray!60!blue] (v3) circle (0.9pt) node[right] (0up) {{\scriptsize$c$}};
\fill[color=gray!60!blue] (v1) circle (0.9pt) node[below] (0up) {{\scriptsize$a$}};
\fill[color=gray!60!blue] (v2) circle (0.9pt) node[below] (0up) {{\scriptsize$b$}};
\fill[color=gray!60!blue] (v4) circle (0.9pt) node[above] (0up) {{\scriptsize$d$}};
\fill[color=gray!60!blue] (v5) circle (0.9pt) node[below] (0up) {{\scriptsize$e$}};
\draw[color=gray!60!blue, very thick] (v1) -- (v5);
\draw[color=gray!60!blue, very thick] (v2) -- (v5);
\draw[color=gray!60!blue, very thick] (v3) -- (v5);
\draw[color=gray!60!blue, very thick] (v4) -- (v5);
\fill[color=gray!60!blue] (v5) circle (0.9pt) node[below] (0up) {};
\draw[color=gray!60!blue, very thick] (v1) -- (v3);
\draw[color=gray!60!blue, very thick] (v2) -- (v3);
\draw[color=gray!60!blue, very thick] (v3) -- (v4);
\fill [blue!20,opacity=0.545] (v1) -- (v2) -- (v4);
\draw[color=gray!60!blue, very thick] (v2) -- (v4);
\draw[color=gray!60!blue, very thick] (v1) -- (v4);
\draw[color=gray!60!blue, very thick] (v1) -- (v2);
\fill[color=gray!60!blue] (v1) circle (0.9pt) node[below] (0up) {};
\fill[color=gray!60!blue] (v2) circle (0.9pt) node[below] (0up) {};
\fill[color=gray!60!blue] (v4) circle (0.9pt) node[above] (0up) {};
\end{tikzpicture}}
\ee
Indeed, for each of the ten inequivalent ways to assign heights $d,e$ to the top and bottom vertices in the 2-3 move, there are up to rotations two ways to assign the remaining heights $a,b,c$ (``clockwise'' and ``counterclockwise''). 
Similarly, for each of the two inequivalent oriented tetrahedra on the left-hand side of the 1-4 move (recall Lemma~\ref{lem:rotasimp}(ii)), there are five possibilities to assign a height~$e$ to the new vertex on the right-hand side.

\subsubsection{Special orbifold data}

In this section we will describe a class of special orbifold data for which we demand only ten constraints, which fall into three classes.
We will show that these ten constraints imply the above 30 constraints. 
To reflect the fact that we are dealing with special orbifold data, we will use special notation:  
we shall from now on write
\begin{align*}
 * & \text{ \; for } \orb_3 
\, , \quad \\
\A & \text{ \; (as in \textsl{a}lgebra) for } \orb_2 
\, , \quad \\
T & \text{ \; (as in \textsl{t}ensor) for } \orb_1
\, , \quad\\ 
\al \in \zz \big( {S}_{\A,T}^2 \big), \; \alb \in \zz \big( ({S}_{\A,T}^2)^{\text{rev}} \big) & \text{ \; (as in \textsl{a}ssociator) instead of } \orb_0^+, \orb_0^-
\, ,
\end{align*}
where ${S}_{\A,T}^2, ({S}_{\A,T}^2)^{\text{rev}} \in \Bordd[3]$ are
\vspace{-0.5cm}
\be\label{eq:ATspheres}
{S}_{\A,T}^2 \; := 
\tikzzbox{\begin{tikzpicture}[very thick,scale=1.2,color=green!60!black=-0.1cm, >=stealth, baseline=0]
\fill[ball color=white!95!blue] (0,0) circle (0.95 cm);
\coordinate (v1) at (-0.4,-0.6);
\coordinate (v2) at (0.4,-0.6);
\coordinate (v3) at (0.4,0.6);
\coordinate (v4) at (-0.4,0.6);
\draw[color=red!80!black, very thick, rounded corners=0.5mm, postaction={decorate}, decoration={markings,mark=at position .5 with {\arrow[draw=red!80!black]{>}}}] 
	(v2) .. controls +(0,-0.25) and +(0,-0.25) .. (v1);
\draw[color=red!80!black, very thick, rounded corners=0.5mm, postaction={decorate}, decoration={markings,mark=at position .62 with {\arrow[draw=red!80!black]{>}}}] 
	(v4) .. controls +(0,0.15) and +(0,0.15) .. (v3);
\draw[color=red!80!black, very thick, rounded corners=0.5mm, postaction={decorate}, decoration={markings,mark=at position .5 with {\arrow[draw=red!80!black]{>}}}] 
	(v4) .. controls +(0.05,-0.5) and +(-0.05,0.5) .. (v2);
\draw[color=red!80!black, very thick, rounded corners=0.5mm, postaction={decorate}, decoration={markings,mark=at position .58 with {\arrow[draw=red!80!black]{>}}}] 
	(v3) .. controls +(-0.9,0.99) and +(-0.75,0.4) .. (v1);
\draw[color=red!80!black, very thick, rounded corners=0.5mm, postaction={decorate}, decoration={markings,mark=at position .5 with {\arrow[draw=red!80!black]{>}}}] 
	(v1) .. controls +(-0.15,0.5) and +(-0.15,-0.5) .. (v4);
\draw[color=red!80!black, very thick, rounded corners=0.5mm, postaction={decorate}, decoration={markings,mark=at position .5 with {\arrow[draw=red!80!black]{>}}}] 
	(v3) .. controls +(0.25,-0.5) and +(0.25,0.5) .. (v2);
\fill (v1) circle (1.6pt) node[black, opacity=0.6, right, font=\tiny] {${T}$};
\fill (v2) circle (1.6pt) node[black, opacity=0.6, right, font=\tiny] {${T}$};
\fill (v3) circle (1.6pt) node[black, opacity=0.6, right, font=\tiny] {${T}$};
\fill (v4) circle (1.6pt) node[black, opacity=0.6, right, font=\tiny] {${T}$};
\fill (-0.8,0.37) circle (0pt) node[color=red!80!black, font=\tiny] {${\mathcal A}$};
\fill (-0.4,0.07) circle (0pt) node[color=red!80!black, font=\tiny] {${\mathcal A}$};
\fill (-0.08,-0.1) circle (0pt) node[color=red!80!black, font=\tiny] {${\mathcal A}$};
\fill (0.72,0.07) circle (0pt) node[color=red!80!black, font=\tiny] {${\mathcal A}$};
\fill (0.08,-0.62) circle (0pt) node[color=red!80!black, font=\tiny] {${\mathcal A}$};
\fill (0.06,0.59) circle (0pt) node[color=red!80!black, font=\tiny] {${\mathcal A}$};
\fill (0.14,0.28) circle (0pt) node[color=gray, font=\tiny] {${*}$};
\fill (-0.14,-0.34) circle (0pt) node[color=gray, font=\tiny] {${*}$};
\fill (-0.58,0.3) circle (0pt) node[color=gray, font=\tiny] {${*}$};
\fill (0.76,-0.18) circle (0pt) node[color=gray, font=\tiny] {${*}$};
\end{tikzpicture}}
\; , \quad ({S}_{\A,T}^2)^{\text{rev}} 
\;
= 
\;\;
\tikzzbox{\begin{tikzpicture}[very thick,scale=1.2,color=green!60!black=-0.1cm, >=stealth, baseline=0]
\fill[ball color=white!95!blue] (0,0) circle (0.95 cm);
\coordinate (v1) at (-0.4,-0.6);
\coordinate (v2) at (0.4,-0.6);
\coordinate (v3) at (0.4,0.6);
\coordinate (v4) at (-0.4,0.6);
\draw[color=red!80!black, very thick, rounded corners=0.5mm, postaction={decorate}, decoration={markings,mark=at position .5 with {\arrow[draw=red!80!black]{>}}}] 
	(v2) .. controls +(0,-0.25) and +(0,-0.25) .. (v1);
\draw[color=red!80!black, very thick, rounded corners=0.5mm, postaction={decorate}, decoration={markings,mark=at position .62 with {\arrow[draw=red!80!black]{>}}}] 
	(v4) .. controls +(0,0.15) and +(0,0.15) .. (v3);
\draw[color=red!80!black, very thick, rounded corners=0.5mm, postaction={decorate}, decoration={markings,mark=at position .58 with {\arrow[draw=red!80!black]{>}}}] 
	(v3) .. controls +(0.05,-0.5) and +(-0.05,0.5) .. (v1);
\draw[color=red!80!black, very thick, rounded corners=0.5mm, postaction={decorate}, decoration={markings,mark=at position .56 with {\arrow[draw=red!80!black]{>}}}] 
	(v1) .. controls +(-0.25,0.5) and +(-0.25,-0.5) .. (v4);
\draw[color=red!80!black, very thick, rounded corners=0.5mm, postaction={decorate}, decoration={markings,mark=at position .5 with {\arrow[draw=red!80!black]{>}}}] 
	(v3) .. controls +(0.15,-0.5) and +(0.15,0.5) .. (v2);
\draw[color=red!80!black, very thick, rounded corners=0.5mm, postaction={decorate}, decoration={markings,mark=at position .68 with {\arrow[draw=red!80!black]{>}}}] 
	(v4) .. controls +(0.9,0.99) and +(0.75,0.4) .. (v2);
\fill (v1) circle (1.6pt) node[black, opacity=0.6, left, font=\tiny] {${T}$};
\fill (v2) circle (1.6pt) node[black, opacity=0.6, left, font=\tiny] {${T}$};
\fill (v3) circle (1.6pt) node[black, opacity=0.6, left, font=\tiny] {${T}$};
\fill (v4) circle (1.6pt) node[black, opacity=0.6, left, font=\tiny] {${T}$};
\fill (0.85,0.16) circle (0pt) node[color=red!80!black, font=\tiny] {${\mathcal A}$};
\fill (0.39,0.07) circle (0pt) node[color=red!80!black, font=\tiny] {${\mathcal A}$};
\fill (-0.06,0.13) circle (0pt) node[color=red!80!black, font=\tiny] {${\mathcal A}$};
\fill (-0.72,0.07) circle (0pt) node[color=red!80!black, font=\tiny] {${\mathcal A}$};
\fill (0.01,-0.64) circle (0pt) node[color=red!80!black, font=\tiny] {${\mathcal A}$};
\fill (0.01,0.59) circle (0pt) node[color=red!80!black, font=\tiny] {${\mathcal A}$};
\fill (0.14,-0.28) circle (0pt) node[color=gray, font=\tiny] {${*}$};
\fill (-0.14,0.34) circle (0pt) node[color=gray, font=\tiny] {${*}$};
\fill (0.58,0.3) circle (0pt) node[color=gray, font=\tiny] {${*}$};
\fill (-0.76,-0.18) circle (0pt) node[color=gray, font=\tiny] {${*}$};
\end{tikzpicture}}
\!\! . 
\ee

\medskip

We will now discuss and motivate the ten constraints imposed on the special orbifold datum $\A \equiv (*,\A,T,\al,\alb)$. 
Then in Definition~\ref{def:orbidata} we shall present the concise characterisation. 

The first constraint on~$\A$ is a single 2-3 move, namely the one involving only vertices of type~$\al$ (and none of type~$\alb$): 
\be\label{eq:constraint1}
\text{under } \zz\colon  
\quad
\tikzzbox{\begin{tikzpicture}[thick,scale=1.77,color=blue!50!black, baseline=0.8cm, >=stealth,
				style={x={(-0.55cm,-0.4cm)},y={(1cm,-0.2cm)},z={(0cm,0.9cm)}}]
	\pgfmathsetmacro{\yy}{0.2}
\coordinate (P) at (0.5, \yy, 0);
\coordinate (R) at (0.625, 0.5 + \yy/2, 0);
\coordinate (L) at (0.5, 0, 0);
\coordinate (R1) at (0.25, 1, 0);
\coordinate (R2) at (0.5, 1, 0);
\coordinate (R3) at (0.75, 1, 0);
\coordinate (N1) at ($(R1) + (R2) - (R)$);
\coordinate (N2) at ($(R) + 2*(R2) - 2*(R)$);
\coordinate (N3) at ($(R3) + (R2) - (R)$);
\coordinate (N4) at ($2*(R3) - (R)$);
\coordinate (Lo) at (0.5, 0, 2);
\coordinate (Po) at (0.5, \yy, 2);
\coordinate (O1) at ($(N1) + (0,0,2)$);
\coordinate (O2) at ($(N2) + (0,0,2)$);
\coordinate (O3) at ($(N3) + (0,0,2)$);
\coordinate (O4) at ($(N4) + (0,0,2)$);
\coordinate (A2) at ($(O2) - (R3) + (R)$);
\coordinate (A3) at ($(O3) - 2*(R3) + 2*(R)$);
\coordinate (A4) at ($(O4) - 3*(R3) + 3*(R)$);
\coordinate (Pt) at (0.5, \yy, 1);
\coordinate (Rt) at ($(A2) + (0,0,-1)$);
\coordinate (Lt) at (0.5, 0, 1);
\coordinate (M1) at ($(N1) + (0,0,1)$);
\coordinate (M2) at ($(N2) + (0,0,1)$);
\coordinate (M3) at ($(N3) + (0,0,1)$);
\coordinate (M4) at ($(N4) + (0,0,1)$);
%
\coordinate (alphabottom) at (0.5, 0.5, 0.5);
\coordinate (alphatop) at (0.5, 0.5, 1.5);
%
\fill [red!50,opacity=0.445] (P) -- (alphabottom) -- (R);
\fill [red!60,opacity=0.445] (R) -- (alphabottom) -- ($(R3) + (0,0,1)$) -- (R3);
\fill [red!80!orange,opacity=0.345] (P) -- (alphabottom) -- (Rt) -- (A2) -- (O1) -- (N1);
\fill [red!50,opacity=0.345] (alphabottom) -- (Rt) -- (A2) -- (A3) -- (alphatop) -- (Pt);
\fill [red!80,opacity=0.345] (alphatop) -- (A4) -- (A3);
\fill[color=blue!60!black] (0.3, 0.79, 1.89) circle (0pt) node (0up) { {\scriptsize$\mathcal A$} };
\fill [red!50!yellow,opacity=0.445] (R) -- (alphabottom) -- (Rt) -- (A2) -- (O2) -- (N2);
\draw[color=green!60!black, very thick, rounded corners=0.5mm, postaction={decorate}, decoration={markings,mark=at position .17 with {\arrow[draw=green!60!black]{>}}}] (alphabottom) -- (Rt) -- (A2);
\fill[color=blue!60!black] (0.6, 0.5, 2) circle (0pt) node (0up) { {\scriptsize$\mathcal A$} };
\fill [orange!90!magenta,opacity=0.545] (R) -- (alphabottom) -- (Pt) -- (alphatop) -- ($(R3) + (0,0,1)$)-- (R3);
\fill [red!80!orange,opacity=0.345] (R3) -- ($(R3) + (0,0,1)$) -- (alphatop) -- (A3) -- (O3) -- (N3);
\draw[string=green!60!black, very thick] (alphatop) -- (A3); 	
\fill [red!70!yellow,opacity=0.345] (R3) -- ($(R3) + (0,0,1)$) -- (alphatop) -- (A4) -- (O4) -- (N4);
\fill [red!80!magenta,opacity=0.345] (L) -- (P) -- (alphabottom) -- (Pt) -- (alphatop) -- (Po) -- (Lo);
%
\draw[string=green!60!black, very thick] (P) -- (alphabottom);
\draw[string=green!60!black, very thick] (R) -- (alphabottom);
\draw[string=green!60!black, very thick] (R3) -- ($(R3) + (0,0,1)$);
\draw[string=green!60!black, very thick] ($(R3) + (0,0,1)$) -- (alphatop);
\draw[string=green!60!black, very thick] (alphatop) -- (A4);
\draw[color=green!60!black, very thick, rounded corners=0.5mm, postaction={decorate}, decoration={markings,mark=at position .33 with {\arrow[draw=green!60!black]{>}}}] (alphabottom) -- (Pt) -- (alphatop);
\draw[color=green!60!black, very thick, rounded corners=0.5mm] (R3) --  ($(R3) + (0,0,1)$) -- (alphatop) ;
%
\fill[color=green!60!black] (alphabottom) circle (1.2pt) node[left] (0up) { {\scriptsize$\alpha$} };
\fill[color=green!60!black] (alphatop) circle (1.2pt) node[left] (0up) { {\scriptsize$\alpha$} };
%
\fill[color=blue!60!black] (0.5,0.25,0.15) circle (0pt) node[left] (0up) { {\scriptsize$\mathcal A$} };
\fill[color=blue!60!black] (0.5,0.55,0.11) circle (0pt) node[left] (0up) { {\scriptsize$\mathcal A$} };
\fill[color=blue!60!black] (0.5,0.85,0.04) circle (0pt) node[left] (0up) { {\scriptsize$\mathcal A$} };
\fill[color=blue!60!black] (0.57,1.29,-0.02) circle (0pt) node[left] (0up) { {\scriptsize$\mathcal A$} };
\fill[color=blue!60!black] (0.57,1.44,0.08) circle (0pt) node[left] (0up) { {\scriptsize$\mathcal A$} };
\fill[color=blue!60!black] (0.57,1.59,0.22) circle (0pt) node[left] (0up) { {\scriptsize$\mathcal A$} };
\fill[color=blue!60!black] (0.57,1.74,0.365) circle (0pt) node[left] (0up) { {\scriptsize$\mathcal A$} };
%
%
\draw [black,opacity=0.4141, densely dotted, semithick] (Lt) -- (Pt) -- (0.75, 1, 1) -- (M4);
\draw [black,opacity=0.4141, densely dotted, semithick] (0.75, 1, 1) -- (M3);
\draw [black,opacity=0.4141, densely dotted, semithick] (Pt) -- (Rt) -- (M2);
\draw [black,opacity=0.4141, densely dotted, semithick] (Rt) -- (M1);
%
\draw [black!80!white,opacity=1, thin] (R) -- (P);
\draw [black,opacity=1, very thin] (A2) -- (A3);
\draw [black,opacity=1, very thin] (A4) -- (A3);
\draw [black,opacity=1, very thin] (A2) -- (O2) -- (N2) -- (R);
\draw [black,opacity=1, very thin] (A3) -- (O3) -- (N3) -- (R3);
\draw [black,opacity=1, very thin] (A4) -- (O4) -- (N4) -- (R3);
\draw [black,opacity=1, very thin] (R3) -- (R);
\draw [black,opacity=1, very thin] (L) -- (P);
\draw [black,opacity=1, very thin] (Po) -- (Lo) -- (L);
\draw [black,opacity=1, very thin] (R3) -- (R);
\draw [black,opacity=1, very thin] (A2) -- (O1) -- (N1) -- (P);
\end{tikzpicture}}
%
=
%
\tikzzbox{\begin{tikzpicture}[thick,scale=1.77,color=green!60!black, baseline=1.7cm, >=stealth, 
				style={x={(-0.6cm,-0.4cm)},y={(1cm,-0.2cm)},z={(0cm,0.9cm)}}]
	\pgfmathsetmacro{\yy}{0.2}
\coordinate (P) at (0.5, \yy, 0);
\coordinate (R) at (0.625, 0.5 + \yy/2, 0);
\coordinate (L) at (0.5, 0, 0);
\coordinate (R1) at (0.25, 1, 0);
\coordinate (R2) at (0.5, 1, 0);
\coordinate (R3) at (0.75, 1, 0);
\coordinate (N1) at ($(R1) + (R2) - (R)$);
\coordinate (N2) at ($(R) + 2*(R2) - 2*(R)$);
\coordinate (N3) at ($(R3) + (R2) - (R)$);
\coordinate (N4) at ($2*(R3) - (R)$);
\coordinate (1P) at ($(P) + (0,0,1)$);
\coordinate (1L) at (0.5, 0, 1);
\coordinate (1N1) at ($(N1) + (0,0,1)$);
\coordinate (1N2) at ($(N2) + (0,0,1)$);
\coordinate (1N3) at ($(N3) + (0,0,1)$);
\coordinate (1N4) at ($(N4) + (0,0,1)$);
\coordinate (1R) at ($(R) + (0,0,1)$);
\coordinate (1RR) at ($(R3) + (-0.25,0,1)$);
\coordinate (2P) at ($(P) + (0,0,2)$);
\coordinate (2L) at (0.5, 0, 2);
\coordinate (2N1) at ($(N1) + (0,0,2)$);
\coordinate (2N2) at ($(N2) + (0,0,2)$);
\coordinate (2N3) at ($(N3) + (0,0,2)$);
\coordinate (2N4) at ($(N4) + (0,0,2)$);
\coordinate (2RR) at ($(1RR) + (0,0,1)$);
\coordinate (2R) at ($(2N3) - 2*(2N3) + 2*(2RR)$);
\coordinate (3P) at ($(P) + (0,0,3)$);
\coordinate (3L) at (0.5, 0, 3);
\coordinate (3N1) at ($(N1) + (0,0,3)$);
\coordinate (3N2) at ($(N2) + (0,0,3)$);
\coordinate (3N3) at ($(N3) + (0,0,3)$);
\coordinate (3N4) at ($(N4) + (0,0,3)$);
\coordinate (3RR) at ($(2RR) + (-0.25,0,1)$);
\coordinate (3R) at ($(2R) + (0,0,1)$);
%
\coordinate (alpha1) at (0.5, 0.75, 0.5);
\coordinate (alpha2) at (0.5, 0.5, 1.5);
\coordinate (alpha3) at (0.25, 0.75, 2.5);
%
\coordinate (psi) at (0.5, 0.75, 1.5);
%
\fill [red!80!magenta,opacity=0.345] (P) -- (1P) -- (alpha2) -- (2P) -- (3P) -- (3L) -- (L);
\fill [red!80,opacity=0.345] (P) -- (1P) -- (alpha2) -- (2R) -- (alpha3) -- (3RR) -- (3N1) -- (N1);
\fill [red!80!orange,opacity=0.345] (alpha3) -- (3RR) -- (3R);
\fill [red!20!orange,opacity=0.345] (R) -- (alpha1) -- (1RR) -- (2RR) -- (alpha3) -- (3RR) -- (3N2) -- (N2);
\fill[color=blue!60!black] (0.3, 0.79, 2.89) circle (0pt) node (0up) { {\scriptsize$\mathcal A$} };
\draw[string=green!60!black, very thick] (alpha3) -- (3RR); 	
\fill [red!80,opacity=0.345] (R3) -- (alpha1) -- (1RR) -- (2RR) -- (alpha3) -- (3R) -- (3N3) -- (N3);
\fill [red!50!yellow,opacity=0.345] (R3) -- (alpha1) -- (R);
\fill [orange!60!blue,opacity=0.345] (alpha1) -- (1R) -- (alpha2) -- (2R) -- (alpha3) -- (2RR) -- (1RR);
\draw[color=green!60!black, very thick, rounded corners=0.5mm, postaction={decorate}, decoration={markings,mark=at position .34 with {\arrow[draw=green!60!black]{>}}}] (R3) -- (alpha1) -- (1RR) -- (2RR) -- (alpha3);
\fill [red!20!orange,opacity=0.345] (alpha2) -- (2P) -- (3P) -- (3R) -- (alpha3) -- (2R);
\fill[color=blue!60!black] (0.6, 0.5, 3) circle (0pt) node (0up) { {\scriptsize$\mathcal A$} };
\draw[string=green!60!black, very thick] (alpha3) -- (3R); 		
\fill[color=green!60!black] (alpha3) circle (1.2pt) node[left] (0up) { {\scriptsize$\alpha$} }; 	
\draw[color=green!60!black, very thick, rounded corners=0.5mm, postaction={decorate}, decoration={markings,mark=at position .3 with {\arrow[draw=green!60!black]{>}}}] (alpha2) -- (2R) -- (alpha3);
\fill[color=blue!60!black] (psi) circle (0pt) node (0up) { {\scriptsize$\mathcal A$} };
%
\fill [red!80,opacity=0.345] (R3) -- (alpha1) -- (1R) -- (alpha2) -- (2P) -- (3P) -- (3N4) -- (N4);
\fill [red!70!yellow,opacity=0.345] (P) -- (1P) -- (alpha2) -- (1R) -- (alpha1) -- (R);
\draw[color=green!60!black, very thick, rounded corners=0.5mm, postaction={decorate}, decoration={markings,mark=at position .3 with {\arrow[draw=green!60!black]{>}}}] (P) -- (1P) -- (alpha2);
\draw[color=green!60!black, very thick, rounded corners=0.5mm, postaction={decorate}, decoration={markings,mark=at position .3 with {\arrow[draw=green!60!black]{>}}}] (alpha1) -- (1R) -- (alpha2);
\draw[color=green!60!black, very thick, rounded corners=0.5mm, postaction={decorate}, decoration={markings,mark=at position .3 with {\arrow[draw=green!60!black]{>}}}] (alpha2) -- (2P) -- (3P);
\draw[string=green!60!black, very thick] (R) -- (alpha1); 		
\draw[string=green!60!black, very thick] (R3) -- (alpha1); 		
%
\fill[color=green!60!black] (alpha1) circle (1.2pt) node[left] (0up) { {\scriptsize$\alpha$} };
\fill[color=green!60!black] (alpha2) circle (1.2pt) node[left] (0up) { {\scriptsize$\alpha$} };
%
\fill[color=blue!60!black] (0.5,0.25,0.15) circle (0pt) node[left] (0up) { {\scriptsize$\mathcal A$} };
\fill[color=blue!60!black] (0.5,0.55,0.11) circle (0pt) node[left] (0up) { {\scriptsize$\mathcal A$} };
\fill[color=blue!60!black] (0.5,0.85,0.04) circle (0pt) node[left] (0up) { {\scriptsize$\mathcal A$} };
\fill[color=blue!60!black] (0.57,1.29,-0.02) circle (0pt) node[left] (0up) { {\scriptsize$\mathcal A$} };
\fill[color=blue!60!black] (0.57,1.44,0.08) circle (0pt) node[left] (0up) { {\scriptsize$\mathcal A$} };
\fill[color=blue!60!black] (0.57,1.59,0.22) circle (0pt) node[left] (0up) { {\scriptsize$\mathcal A$} };
\fill[color=blue!60!black] (0.57,1.74,0.365) circle (0pt) node[left] (0up) { {\scriptsize$\mathcal A$} };
%
%
\draw [black,opacity=0.4141, densely dotted, semithick] (2L) -- (2P) -- (2N4);
\draw [black,opacity=0.4141, densely dotted, semithick] (2P) -- (2N1);
\draw [black,opacity=0.4141, densely dotted, semithick] (2R) -- (2RR) -- (2N2);
\draw [black,opacity=0.4141, densely dotted, semithick] (2RR) -- (2N3);
\draw [black,opacity=0.4141, densely dotted, semithick] (1L) -- (1P) -- (1N4);
\draw [black,opacity=0.4141, densely dotted, semithick] (1R) -- (1RR) -- (1N3);
\draw [black,opacity=0.4141, densely dotted, semithick] (1RR) -- (1N2);
\draw [black,opacity=0.4141, densely dotted, semithick] (1P) -- (1N1);
%
\draw [black,opacity=1, very thin] (3P) -- (3L) -- (L) -- (P); 
\draw [black,opacity=1, very thin] (N1) -- (P); 
\draw [black,opacity=1, very thin] (3RR) -- (3N1) -- (N1) -- (3N1); 
\draw [black,opacity=1, very thin] (3RR) -- (3R); 
\draw [black,opacity=1, very thin] (3RR) -- (3N2) -- (N2) -- (R); 
\draw [black,opacity=1, very thin] (3R) -- (3N3) -- (N3) -- (R3); 
\draw [black,opacity=1, very thin] (R3) -- (R); 
\draw [black,opacity=1, very thin] (3P) -- (3RR);
\draw [black,opacity=1, very thin] (3P) -- (3N4) -- (N4) -- (R); 
\draw [black,opacity=1, very thin] (P) -- (R); 
\end{tikzpicture}}
\, . 
\ee
Put differently, the two stratified 3-balls containing the two sides of~\eqref{eq:constraint1} (viewed as bordisms from the empty set to their boundary spheres) evaluate identically under~$\zz$. 
In examples, this amounts to a pentagon condition on~$\al$. 

\medskip

The next type of constraint on the data $(*,\A, T, \al, \alb)$ demands that ``if two tetrahedra in the cell decomposition Poincar\'e dual to the stratification by defects share two faces, the tetrahedra may be replaced by just the other two faces''. 
We note immediately that in a triangulation any two tetrahedra can share at most one face, so applying this constraint leaves the realm of triangulations to one of more general stratifications. 
In Section~\ref{subsubsec:sodareod} we will show how transiently leaving the domain of triangulations in this way will guarantee invariance under all 20 oriented 2-3 moves. 

We want to make the above constraint precise. 
Let $a,b,c,d \in \R$ be pairwise distinct numbers, and consider the two tetrahedra 
\vspace{-0.5cm}
\be
\!\!
\tikzzbox
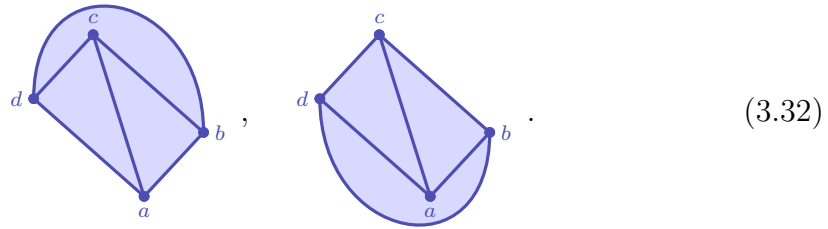
{\begin{tikzpicture}[thick,scale=2.234,color=gray!60!blue, baseline=-0.3cm, >=stealth, 
				style={x={(-0.6cm,-0.4cm)},y={(1cm,-0.2cm)},z={(0cm,0.9cm)}}]
\coordinate (v1) at (0,0,0);
\coordinate (v2) at (0,1,0);
\coordinate (v3) at (0,0.35,0.5);
\coordinate (v4) at (0,0.65,-0.5);
\fill [blue!15,opacity=1] (v1) -- (v3) -- (v4);
\fill [blue!15,opacity=1] (v2) -- (v3) -- (v4);
\fill [blue!15,opacity=1] (v1) -- (v2) -- (v2) .. controls +(0,0,0.95) and +(0,0,0.95) .. (v1);
\fill[color=gray!60!blue] (v1) circle (0.9pt) node[above] (0up) {};
\fill[color=gray!60!blue] (v2) circle (0.9pt) node[above] (0up) {};
\fill[color=gray!60!blue] (v3) circle (0.9pt) node[above] (0up) {};
\fill[color=gray!60!blue] (v4) circle (0.9pt) node[above] (0up) {};
\draw[color=gray!60!blue, very thick] (v3) -- (v4);
\draw[color=gray!60!blue, very thick] (v1) -- (v3);
\draw[color=gray!60!blue, very thick] (v2) -- (v3);
\draw[color=gray!60!blue, very thick] (v2) -- (v4);
\draw[color=gray!60!blue, very thick] (v1) -- (v4);
%
\fill[color=gray!60!blue] (v1) circle (0.9pt) node[left] (0up) {{\scriptsize$d$}};
\fill[color=gray!60!blue] (v2) circle (0.9pt) node[right] (0up) {{\scriptsize$b$}};
\fill[color=gray!60!blue] (v3) circle (0.9pt) node[above] (0up) {{\scriptsize$c$}};
\fill[color=gray!60!blue] (v4) circle (0.9pt) node[below] (0up) {{\scriptsize$a$}};
%
\draw[color=gray!60!blue, very thick] (v1) .. controls +(0,0,0.95) and +(0,0,0.95) .. (v2);
\end{tikzpicture}}
\, , \quad 
\tikzzbox{\begin{tikzpicture}[thick,scale=2.234,color=gray!60!blue, baseline=-0.3cm, >=stealth, 
				style={x={(-0.6cm,-0.4cm)},y={(1cm,-0.2cm)},z={(0cm,0.9cm)}}]
\coordinate (v1) at (0,0,0);
\coordinate (v2) at (0,1,0);
\coordinate (v3) at (0,0.35,0.5);
\coordinate (v4) at (0,0.65,-0.5);
\fill [blue!15,opacity=1] (v1) -- (v3) -- (v4);
\fill [blue!15,opacity=1] (v2) -- (v3) -- (v4);
\fill [blue!15,opacity=1] (v1) -- (v2) -- (v2) .. controls +(0,0,-0.95) and +(0,0,-0.95) .. (v1);
%
\fill[color=gray!60!blue] (v1) circle (0.9pt) node[above] (0up) {};
\fill[color=gray!60!blue] (v2) circle (0.9pt) node[above] (0up) {};
\fill[color=gray!60!blue] (v3) circle (0.9pt) node[above] (0up) {};
\fill[color=gray!60!blue] (v4) circle (0.9pt) node[above] (0up) {};
\draw[color=gray!60!blue, very thick] (v3) -- (v4);
\draw[color=gray!60!blue, very thick] (v1) -- (v3);
\draw[color=gray!60!blue, very thick] (v2) -- (v3);
\draw[color=gray!60!blue, very thick] (v2) -- (v4);
\draw[color=gray!60!blue, very thick] (v1) -- (v4);
%
\fill[color=gray!60!blue] (v1) circle (0.9pt) node[left] (0up) {{\scriptsize$d$}};
\fill[color=gray!60!blue] (v2) circle (0.9pt) node[right] (0up) {{\scriptsize$b$}};
\fill[color=gray!60!blue] (v3) circle (0.9pt) node[above] (0up) {{\scriptsize$c$}};
\fill[color=gray!60!blue] (v4) circle (0.9pt) node[below] (0up) {{\scriptsize$a$}};
\draw[color=gray!60!blue, very thick] (v1) .. controls +(0,0,-0.95) and +(0,0,-0.95) .. (v2);
%
\end{tikzpicture}}
\, . 
\ee
\vspace{-0.7cm}
\\
In these pictures we took a tetrahedron as in \eqref{eq:3simplex} and moved all of its vertices into the same 2-dimensional plane; in doing so we turned one of its edges into an arc to keep the tetrahedron from becoming degenerate.
We may either glue the two tetrahedra
along the two faces $(abc)$ and $(acd)$, or along $(abd)$ and $(bcd)$.
(The second gluing is easier to visualise if one moves the corresponding faces into the same plane, thereby turning the opposite straight edge into an arc.)
In either case we consider the move which replaces the two tetrahedra with the two non-shared faces (note that these moves decrease the number of 2- and 3-strata): 
\vspace{-0.7cm}
\be
\tikzzbox{
}
 . 
\vspace{-0.2cm}
\ee
Note that these two figures are just two different ways of visualising the same move, but rotated by $90^\circ$ relative to each other.

\begin{lemma}
\label{lem:aa}
For $\{ a,b,c,d\} = \{ 1,2,3,4 \}$ a move between oriented stratifications of the above type is up to rotation one of the 
following 
	six (the figures are rotated so that the highest vertex is the rightmost one): 
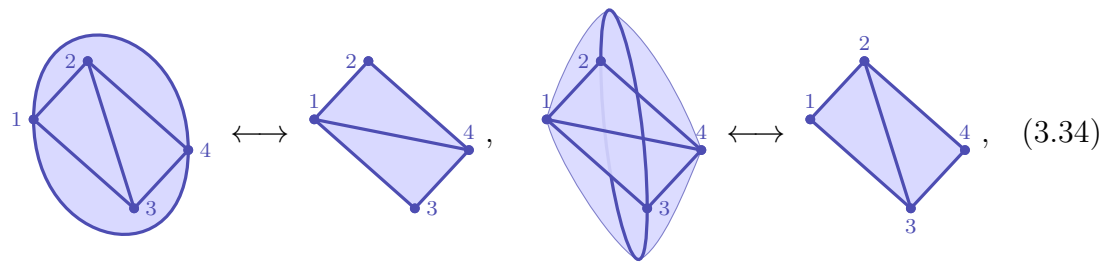
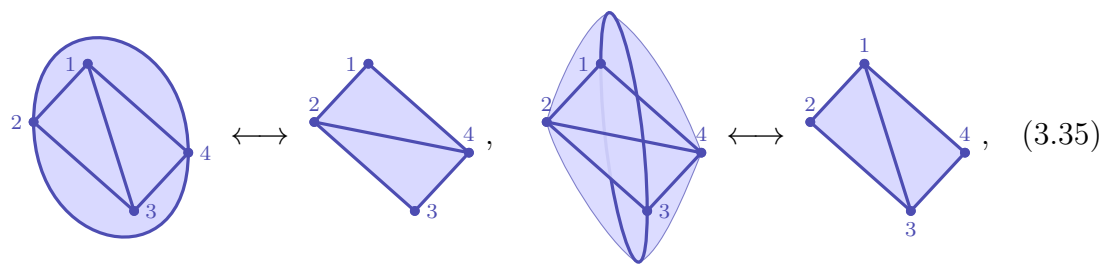
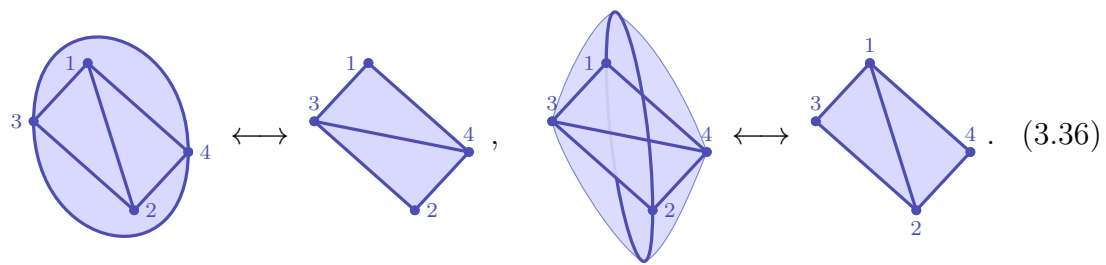
\vspace{-0.5cm}
\begin{align}
\label{eq:aas}
& \!\!\!\!\!
\tikzzbox{%
}
. 
\end{align}
\end{lemma}

\begin{proof}
Up to rotation, there are two choices where to place the highest number~4 on the vertices of the two tetrahedra: 
\vspace{-0.5cm}
\be\label{eq:twotetrahedra}
\tikzzbox{%
}
\, . 
\vspace{-0.5cm}
\ee
This amounts to the first and second column in the above list of moves.

Consider the left choice in~\eqref{eq:twotetrahedra}.
Up to rotation leaving the vertex with label~4 fixed there are two choices on which vertex to place the label~1.
One of these choices leaves two inequivalent options how to distribute the remaining labels~2 and~3, the other leaves only one choice. 
This leaves us with the moves in the left column of~\eqref{eq:aas}--\eqref{eq:aas3}. 

The right column follows analogously from the right stratification in~\eqref{eq:twotetrahedra}. 
\end{proof}

We now translate the moves of Lemma~\ref{lem:aa} into constraints on the data $(*, \A, T, \al, \alb)$. 
For this we pass to the stratifications Poincar\'{e} dual to the stratifications in \eqref{eq:aas}--\eqref{eq:aas3}, and demand that evaluation by~$\zz$ is invariant under these moves. 
Hence the constraints corresponding to these moves are that under~$\zz$: 
\begin{align}
& \!\!\!
\label{eq:alpha1}
\tikzzbox{
}
\end{align}
where the stripy patterns in~\eqref{eq:alpha2} again indicate
	2-strata oriented opposite to the paper plane.

\medskip

Finally, we have to ensure invariance under a type of bubble move: 
Our last constraint on $(*, \A, T, \al, \alb)$ is that under the functor~$\zz$: 
\be\label{eq:bubble2}
\tikzzbox{%
}
\ee
where in the first three pictures two $\mathcal A$-labelled hemispheres are glued together along a $T$-line. 
In the first picture, all 2-strata have the same orientation 
	as the paper plane,
in the second picture the rear hemisphere has opposite orientation, while in the third picture it is the front hemisphere. 

\medskip

Now we collect all of the above in a single notion. 
As in Example~\ref{ex:defectdatan12} we shall work with the \textsl{source} and \textsl{target maps} $s,t \colon D_2 \to D_3$, defined via $f_2(x) = (s(x), t(x))$, and the \textsl{folding map} $f := f_1$,
	where $f_j$ are the adjacency maps of the defect data~$\D$. 

\begin{definition}
\label{def:orbidata}
Given a 3-dimensional defect TQFT $\zz\colon  \Bordd[3] \to \Vectk$, a  \textsl{special orbifold datum} $\A \equiv (*, \A, T, \al, \alb)$ for~$\zz$ is a choice of 
\begin{itemize}
\item $* \in D_3$,
\item
$\mathcal A \in D_2$ with $s(\A) = t(\A) = *$, 
\item 
$T \in D_1$ with $f(T) = (\A,+) \times (\A,+) \times (\A,-)$, 
\item 
$\al \in \zz(S_{\A,T}^2)$ and $\alb \in \zz(\bar S_{\A,T}^2)$ as in~\eqref{eq:ATspheres}
\end{itemize}
such that the constraints~\eqref{eq:constraint1}, \eqref{eq:alpha1}--\eqref{eq:alpha3}, \eqref{eq:bubble2} are satisfied. 
\end{definition}

\begin{remark}
\label{rem:orbifolddataforanyZ}
Here we spell out what a special orbifold datum for the Euler completion $\zz^\euc \colon \Bord_n^{\text{def}}(\D^\euc) \to \Vectk$ of a TQFT $\zz\colon \Bordd[n] \to \Vectk$ means in terms of~$\zz$ and~$\D$ directly. 
Recall Definition~\ref{def:EulercompleteD} for the Euler-completed defect data~$\D^\euc$, and Definition~\ref{def:Eulercompletion} for~$\zz^\euc$. 
\begin{enumerate}
\item
By definition, to specify a special orbifold datum~$\A^\euc$ for~$\zz^\euc$ we first have to provide one element in each of $D_3^\euc$, $D_2^\euc$, $D_1^\euc$ as well as two elements $\al, \alb \in D_0^\euc$. 
Recall that for $i>0$, elements in $D_i^\euc$ are tuples of the form $(x,\phi,\Psi)$, where $x\in D_i$, $\phi$ is an invertible element in the algebra of point insertions on~$x$, i.\,e\ $\phi \in A_x^\times$ (cf.~\eqref{eq:Ax-defect-alg-def} and Proposition~\ref{prop:point-defect-algebras}), and~$\Psi$ is a tuple $(\psi_S)_{S \in \mathrm{Strat}(f_i(x))}$ with $\psi_S \in A_y^\times$ and~$y$ the label of the stratum~$S$ of $f_i(x)$. 
In the top dimension $i=3$ the tuple~$\Psi$ is empty, so the 3-dimensional label of a special orbifold datum~$\A^\euc$ for~$\zz^\euc$ is of the form 
\be
\label{eq:D3eucsod}
(*, \phi) \in D^\euc_3  
: \quad 
* \in D_3
\, , \quad 
\phi \in \zz \Big( 
\tikzzbox{\begin{tikzpicture}[very thick,scale=0.8,color=red!80!black, baseline=-0.1cm]
\fill[ball color=gray!40!white] (0,0) circle (0.95 cm);
\fill (0.05,0.05) circle (0pt) node[red!80!black] {};
\fill (0,0) circle (0pt) node[gray!50!black] {{\small$*$}};
\end{tikzpicture}}
 \Big)^{\!\times} . 
\ee
The 2-dimensional label 
	for~$\A^\euc$ 
hence must be of the form 
\be
\label{eq:D2eucsod}
\big( \A, \psi, (\phi,\phi) \big) \in D^\euc_2
: \quad 
\A \in D_2
\, , \quad
s(\A) = t(\A) = *
\, , \quad 
\psi \in \zz \Big( 
\tikzzbox{\begin{tikzpicture}[very thick,scale=0.8,color=red!80!black, baseline=-0.1cm]
\fill[ball color=gray!40!white] (0,0) circle (0.95 cm);
\fill (0.05,0.05) circle (0pt) node[red!80!black] {};
\draw[
	color=red!80!black, 
	very thick,
	opacity=1.0, 
	 >=stealth, 
	postaction={decorate}, decoration={markings,mark=at position .6 with {\arrow[draw]{>}}}
	] 
	(0,-0.95) .. controls +(-0.2,0.2) and +(-0.2,-0.2) .. (0,0.95);
\draw[color=red!80!black, opacity=0.12]	(0,-0.95) .. controls +(0.2,0.2) and +(0.2,-0.2) .. (0,0.95);%
\fill (-0.55,0) circle (0pt) node[gray!50!black] {{\small$*$}};
\fill (0.55,0) circle (0pt) node[gray!50!black] {{\small$*$}};
\fill (0.15,0.1) circle (0pt) node {{\small$\mathcal A$}};
\end{tikzpicture}}
\Big)^{\!\times}
\ee
as the Euler weights on both sides of an $(\A,\psi)$-labelled 2-stratum must be~$\phi$. 

Thanks to Lemma~\ref{lem:euctwo-equ} ($\zz^{\euctwo} \sim \zz^\euc$), point insertions on line defects need not to be considered. 
Hence a 1-dimensional label for our~$\A^\euc$ can always be taken to be $(T,\Psi)$, where~$T$ is as in Definition~\ref{def:orbidata}, and the tuple $\Psi = (\phi^{\times 3}, \psi^{\times 3})$ keeps track of the Euler weights inserted on the three 3-strata and the three 2-strata adjacent to the $T$-labelled 1-stratum: 
\be
\label{eq:D1eucsod}
\big( T, (\phi^{\times 3}, \psi^{\times 3}) \big) \in D^\euc_1
: \quad 
T \in D_1
\, , \quad 
f(T) = (\A,+) \times (\A,+) \times (\A,-)
\, . 
\ee

Finally, the 0-dimensional labels in the special orbifold datum~$\A^\euc$ for~$\zz^\euc$ features elements
\vspace{-0.5cm}
\be
\al 
\in \zz \Bigg( 
\!\!\!\!
\tikzzbox{
}
\!\!\!\!
\Bigg)^{\!\!\times}
\ee
so that we have 
\be
\label{eq:D0eucsod}
\big( \al, (\phi^{\times 4}, \psi^{\times 6}) \big)
\, , \; 
\big( \alb, (\phi^{\times 4}, \psi^{\times 6}) \big)
 \in D^\euc_0 
 \, . 
\ee
\item 
The constraints on a special orbifold datum of the form \eqref{eq:D3eucsod}--\eqref{eq:D1eucsod} and~\eqref{eq:D0eucsod} are that under~$\zz$: 
\begin{align}
&
\label{eq:32psi} 
\tikzzbox{%
}
\end{align}
where in~\eqref{eq:phipsi} there is an insertion of $\psi^2$ on each
	of the two hemispheres
ending on a $T$-line, and the enclosed 3-strata feature one $\phi$-insertion. 
Note that all of the above are identities between elements in vector spaces $\zz(B)$ for suitable decorated stratified 3-balls~$B$. 
\item
A datum $(*, \A, T, \al, \alb)$ for~$\zz$ that satisfies the constraints of Definition~\ref{def:orbidata} only ``up to normalisations'' can sometimes be adjusted to form a special orbifold datum for the completion $\zz^\euc$.  
This amounts to finding invertible elements $\phi, \psi$ as in~\eqref{eq:D3eucsod}, \eqref{eq:D2eucsod} which satisfy the constraints \eqref{eq:32psi}--\eqref{eq:phipsi}. If~$\al$ is invertible, such field insertions $\phi, \psi$ are ``unique for practical purposes'' in the sense that
their contribution to the orbifold theory $(\zz^\euc)_\A$ can be determined from the data $*,\A,T,\al$ alone as follows. 

By pre- and post-composing with~$\alpha^{-1}$ in~\eqref{eq:32psi} we find that~$\zz$ evaluated on
\be\label{eq:psicylinder}
\tikzzbox{\begin{tikzpicture}[thick,scale=2.321,color=blue!50!black, baseline=0.2cm, >=stealth, 
				style={x={(-0.6cm,-0.4cm)},y={(1cm,-0.2cm)},z={(0cm,0.9cm)}}]
	\pgfmathsetmacro{\yy}{0.2}
\coordinate (P) at (0.5, \yy, 0);
\coordinate (R) at (0.375, 0.5 + \yy/2, 0);
\coordinate (L) at (0.5, 0, 0);
\coordinate (R1) at (0.25, 1, 0);
\coordinate (R2) at (0.5, 1, 0);
\coordinate (R3) at (0.75, 1, 0);
\coordinate (Pt) at (0.5, \yy, 0.5);
\coordinate (Rt) at (0.625, 0.5 + \yy/2, 0.5);
\coordinate (Lt) at (0.5, 0, 0.5);
\coordinate (R1t) at (0.25, 1, 0.5);
\coordinate (R2t) at (0.5, 1, 0.5);
\coordinate (R3t) at (0.75, 1, 0.5);
\coordinate (aP) at (0.5, \yy, 0.5);
\coordinate (aR) at (0.625, 0.5 + \yy/2, 0.5);
\coordinate (aL) at (0.5, 0, 0.5);
\coordinate (aR1) at (0.25, 1, 0.5);
\coordinate (aR2) at (0.5, 1, 0.5);
\coordinate (aR3) at (0.75, 1, 0.5);
\coordinate (aPt) at (0.5, \yy, 1);
\coordinate (aRt) at (0.375, 0.5 + \yy/2, 1);
\coordinate (aLt) at (0.5, 0, 1);
\coordinate (aR1t) at (0.25, 1, 1);
\coordinate (aR2t) at (0.5, 1, 1);
\coordinate (aR3t) at (0.75, 1, 1);
\fill [red!50,opacity=0.545] (P) -- (aPt) -- (aLt) -- (L);5
\fill [red!50,opacity=0.545] (R) -- (aRt) -- (aR1t) -- (R1);
\fill [yellow!80!red,opacity=0.745] (R) -- (aRt) -- (aPt) -- (P);
\fill [orange!80!red,opacity=0.345] (R) -- (aRt) -- (aPt) -- (P);
\coordinate (psi) at (0.5, 0.4, 0.5);
\fill[color=green!20!black] (psi) circle (0.9pt) node[above] (0up) { {\scriptsize$\psi^2$} }; 	
\fill [red!50,opacity=0.545] (R) -- (aRt) -- (aR2t) -- (R2);
\fill[color=green!60!black] (0.375, 0.5 + \yy/2, 0.5) circle (0pt) node[left] (0up) { {\scriptsize$T$} };
\draw[color=green!60!black, very thick, rounded corners=0.5mm, postaction={decorate}, decoration={markings,mark=at position .51 with {\arrow[draw=green!60!black]{>}}}] (R) -- (aRt);
\fill[color=blue!60!black] (0.5,0.56,0.2) circle (0pt) node[left] (0up) { {\scriptsize$\mathcal A$} };
\fill [red!50,opacity=0.545] (P) -- (aPt) -- (aR3t) -- (R3);
\draw[string=green!60!black, very thick] (P) -- (aPt);
\fill[color=green!60!black] (0.5, \yy, 0.5) circle (0pt) node[left] (0up) { {\scriptsize$T$} };
\fill[color=blue!60!black] (0.5,0.2,0.15) circle (0pt) node[left] (0up) { {\scriptsize$\mathcal A$} };
\fill[color=blue!60!black] (0.5,0.89,0.0) circle (0pt) node[left] (0up) { {\scriptsize$\mathcal A$} };
\fill[color=blue!60!black] (0.5,1.04,0.13) circle (0pt) node[left] (0up) { {\scriptsize$\mathcal A$} };
\fill[color=blue!60!black] (0.5,1.19,0.28) circle (0pt) node[left] (0up) { {\scriptsize$\mathcal A$} };
\draw [black,opacity=1, very thin] (L) -- (aLt) -- (aPt) -- (aR3t) -- (R3) -- (P) -- cycle;
\draw [black,opacity=1, very thin] (aPt) -- (aR1t) -- (R1) -- (P);
\draw [black,opacity=1, very thin] (R) -- (R2) -- (aR2t) -- (aRt);
\end{tikzpicture}}
\ee
is given by~$\zz$ evaluated on a 3-ball containing only the known data $*,\A,T,\al,\al^{-1}$. 
Furthermore, in the construction of $(\zz^\euc)_\A$, the only local neighbourhood~$\psi$ will ever appear in is that of~\eqref{eq:psicylinder}. 
Hence knowing the action of~$\zz$ on the 3-ball containing it is sufficient. 

The insertion $\phi$ compensates the $T$-rimmed bubble in~\eqref{eq:phipsi}. 
In the examples we are considering in \cite{CRS3}, this condition amounts to the computation of a quantum dimension and determines~$\phi$ uniquely. 
\end{enumerate}
\end{remark}

\subsubsection{Special orbifold data are orbifold data}
\label{subsubsec:sodareod}

Let $\A \equiv (*, \A, T, \al, \alb)$ be a special orbifold datum for a defect TQFT $\zz\colon  \Bordd[3] \to \Vectk$. 
In a series of lemmas we will now show that~$\A$ really is an orbifold datum in the sense of Definition~\ref{def:orbidatan}. 
Hence the orbifold theory $\zz_\A$ of Theorem~\ref{thmdef:orbifoldtheory} is well-defined. 
We start with checking the remaining 19 oriented 2-3 moves: 

\begin{lemma}
\label{lem:ZAM23}
The invariance condition (ii) in Definition~\ref{def:orbidatan} holds for all oriented 2-3 Pachner moves.
\end{lemma}

\begin{proof}
	We need to show \eqref{eq:def-orb-ii-Z=Z}, i.\,e.\ that the Poincar\'e dual of the oriented 2-3 Pachner moves holds inside 3-balls after evaluating with $\zz$.
	
One of the defining conditions for special orbifold data~$\A$, namely the constraint~\eqref{eq:constraint1}, is precisely \eqref{eq:def-orb-ii-Z=Z} under the oriented 2-3 move
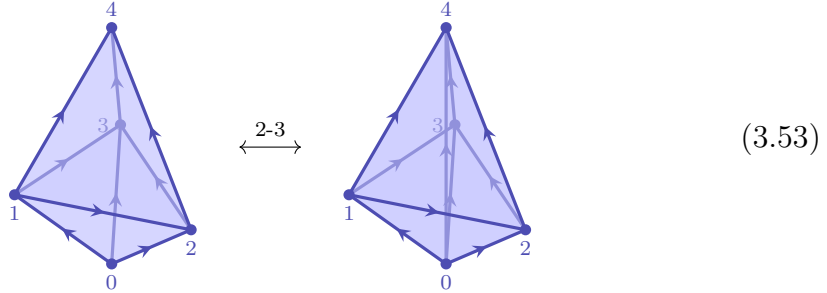
\be
\label{eq:ori23}
\tikzzbox{
}
\ee
which only involves tetrahedra of type~$\al$. 
To show invariance under the remaining 19 oriented 2-3 moves we use the conditions \eqref{eq:alpha1}--\eqref{eq:alpha3} repeatedly. 

For example, we may glue another tetrahedron to two faces on both sides of~\eqref{eq:ori23} by adding a new edge between the vertices labelled~4 and~0. 
The left-hand side then isotopically deforms into three tetrahedra joined along the edge (12): 
\vspace{-1.2cm}
\be
\tikzzbox{%
}
\, . 
\vspace{-0.8cm}
\ee
On the other hand, after gluing the new tetrahedron onto the right-hand side of~\eqref{eq:ori23}, we can first use condition~\eqref{eq:aas} and then rearrange isotopically:
\vspace{-1.0cm}
\be
\tikzzbox{%
}
\, . 
\vspace{-0.8cm}
\ee
Combined, we have established invariance under the 2-3 move
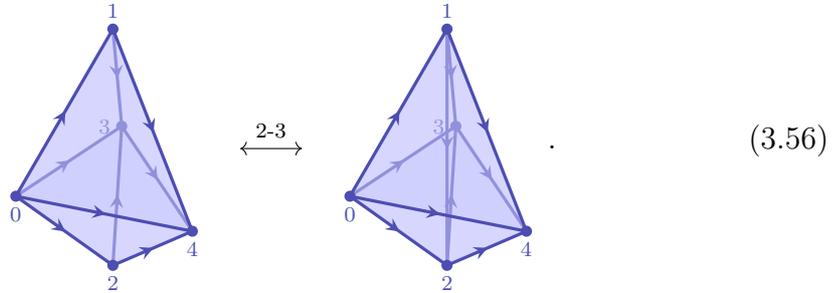
\be
\label{eq:new23}
\tikzzbox{%
}
\, . 
\ee

Repeating this argument with all three possible new edges on all incrementally established 2-3 moves produces all 20 inequivalent oriented 2-3 moves. 
More precisely, if we use the shorthand notation ${a \choose b}_\circlearrowleft$, ${a \choose b}_\circlearrowright$ for the moves
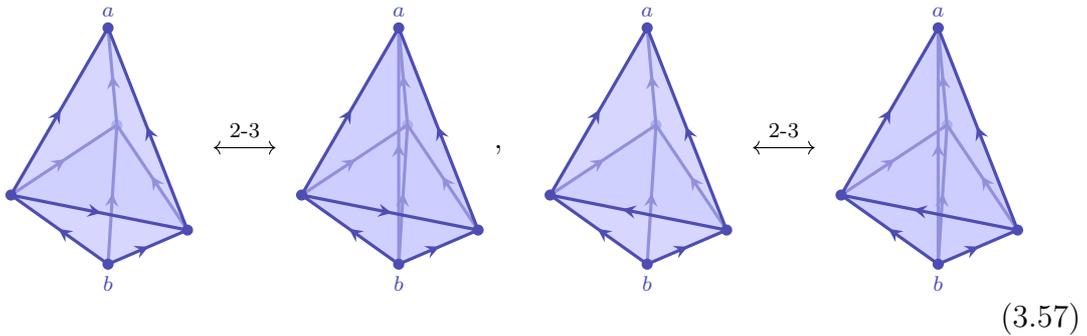
\be
\!\!\!\!\!\!
\tikzzbox{%
}
\ee
respectively, then we may abbreviate the above derivation of~\eqref{eq:new23} from~\eqref{eq:ori23} as ${4 \choose 0}_\circlearrowleft \to {1 \choose 2}_\circlearrowright$. 
By applying the moves of Lemma~\ref{lem:aa} to the other two shared edges in~\eqref{eq:ori23}, we obtain ${2 \choose 3}_\circlearrowright$ and ${1 \choose 3}_\circlearrowleft$ from ${4 \choose 0}_\circlearrowleft$. 
We summarise this as 
\be
\textstyle
{4 \choose 0}_\circlearrowleft \lra {1 \choose 2}_\circlearrowright, \, {2 \choose 3}_\circlearrowright, \, {1 \choose 3}_\circlearrowleft \, . 
\ee
Similarly, one finds 
\begin{align}
& \textstyle 
{1 \choose 2}_\circlearrowright \lra {3 \choose 4}_\circlearrowleft , \, {3 \choose 0}_\circlearrowleft
\, , \quad 
{2 \choose 3}_\circlearrowright \lra {1 \choose 0}_\circlearrowright , \, {1 \choose 4}_\circlearrowleft
\, , \quad 
{1 \choose 3}_\circlearrowleft \lra {2 \choose 0}_\circlearrowleft , \, {2 \choose 4}_\circlearrowright \, ,
\nonumber
\\
& \textstyle 
{1 \choose 0}_\circlearrowright \lra {2 \choose 4}_\circlearrowleft , \, {3 \choose 4}_\circlearrowright
\, , \quad 
{1 \choose 4}_\circlearrowleft \lra {3 \choose 0}_\circlearrowleft \, , 
\nonumber
\\
& \textstyle 
{2 \choose 0}_\circlearrowleft \lra {1 \choose 4}_\circlearrowright 
\, , \quad
{2 \choose 4}_\circlearrowright \lra {1 \choose 0}_\circlearrowleft \, , 
\nonumber
\\
& \textstyle 
{2 \choose 4}_\circlearrowleft \lra {1 \choose 3}_\circlearrowright 
\, , \quad 
{3 \choose 4}_\circlearrowright \lra {1 \choose 2}_\circlearrowleft \lra {4 \choose 0}_\circlearrowright
\, , \quad 
{1 \choose 4}_\circlearrowright \lra {2 \choose 3}_\circlearrowleft , \, {2 \choose 0}_\circlearrowright \, , 
\end{align}
establishing all 20 oriented 2-3 moves. 
\end{proof}

To verify invariance under all oriented 1-4 moves, we use an auxiliary lemma which we learned from \cite{Baezlecturenotes}. 
It features the \textsl{bubble moves}, which are operations on oriented stratified 3-dimensional manifolds which locally act by replacing the two stratifications~\eqref{eq:3ballstrat} and~\eqref{eq:triastrat}, for all possible orientations, 
\be\label{eq:bubble}
\tikzzbox{\begin{tikzpicture}[very thick,scale=1.2,color=gray!60!blue, baseline=-0.1cm]
\fill[ball color=blue!10!white] (0,0) circle (0.95 cm);
\coordinate (v1) at (0.5,0.32);
\coordinate (v2) at (-0.5,0.32);
\coordinate (v3) at (0,-0.375);
\coordinate (v4) at (0,0);
\fill (v4) circle (1.6pt) node[gray!60!blue, opacity=0.6] {};
\fill (v1) circle (1.6pt) node[gray!60!blue, opacity=0.1] {};
\fill (v2) circle (1.6pt) node[gray!60!blue, opacity=0.1] {};
\fill (v3) circle (1.6pt) node[gray!60!blue, opacity=0.1] {};
\draw[color=gray!60!blue, opacity=0.32]	(-0.93,0) .. controls +(0,0.5) and +(0,0.5) .. (0.93,0);
\draw[color=gray!60!blue]	(-0.93,0) .. controls +(0,-0.5) and +(0,-0.5) .. (0.93,0);
\draw[color=gray!60!blue, very thick] (v1) -- (v4);
\draw[color=gray!60!blue, very thick] (v2) -- (v4);
\draw[color=gray!60!blue, very thick] (v3) -- (v4);
\end{tikzpicture}}
\begin{tikzpicture}[
			     baseline=(current bounding box.base), 
			     descr/.style={fill=white,inner sep=3.5pt}, 
			     normal line/.style={->}
			     ] 
\matrix (m) [matrix of math nodes, row sep=3.5em, column sep=2.5em, text height=1.1ex, text depth=0.1ex] {%
{}
&
{}
\\
};
\path[font=\footnotesize] (m-1-1) edge[<->] node[above] { {\scriptsize bubble}} (m-1-2);
\end{tikzpicture}
\begin{tikzpicture}[very thick,scale=1.2,color=gray!60!blue, baseline=-0.1cm]
\coordinate (v1) at (0.5,0.32);
\coordinate (v2) at (-0.5,0.32);
\coordinate (v3) at (0,-0.375);
\fill [blue!15,opacity=1] (v1) -- (v2) -- (v3);
\fill (v1) circle (1.6pt) node[gray!60!blue, opacity=0.1] {};
\fill (v2) circle (1.6pt) node[gray!60!blue, opacity=0.1] {};
\fill (v3) circle (1.6pt) node[gray!60!blue, opacity=0.1] {};
\draw[color=gray!60!blue, very thick] (v1) -- (v2);
\draw[color=gray!60!blue, very thick] (v2) -- (v3);
\draw[color=gray!60!blue, very thick] (v1) -- (v3);
\end{tikzpicture}
\, .
\ee
These are Poincar\'{e} dual to 
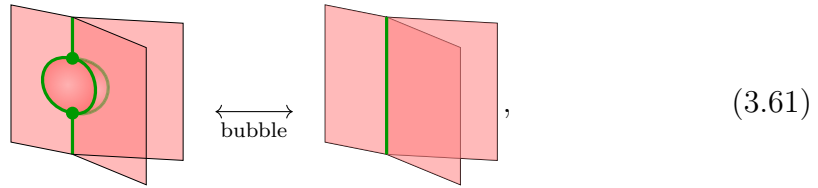
\be\label{eq:bubble}
\tikzzbox{\begin{tikzpicture}[thick,scale=2.021,color=blue!50!black, baseline=0.0cm, >=stealth, 
				style={x={(-0.6cm,-0.4cm)},y={(1cm,-0.2cm)},z={(0cm,0.9cm)}}]
	\pgfmathsetmacro{\yy}{0.2}
\coordinate (T) at (0.5, 0.4, 0);
\coordinate (L) at (0.5, 0, 0);
\coordinate (R1) at (0.3, 1, 0);
\coordinate (R2) at (0.7, 1, 0);
\coordinate (1T) at (0.5, 0.4, 1);
\coordinate (1L) at (0.5, 0, 1);
\coordinate (1R1) at (0.3, 1, 1);
\coordinate (1R2) at (0.7, 1, 1);
\coordinate (a) at (0.5, 0.4, 0.3);
\coordinate (1a) at (0.5, 0.4, 0.7);
%
\fill [red!50,opacity=0.545] (L) -- (T) -- (1T) -- (1L);
\fill [red!50,opacity=0.545] (R1) -- (T) -- (1T) -- (1R1);
\fill [red!50,opacity=0.545] (R2) -- (T) -- (1T) -- (1R2);
%
\draw[color=green!60!black, very thick] (T) -- (a);
\draw[color=green!60!black, very thick] (1a) -- (1T);
%
\fill[inner color=red!30!white,outer color=red!55!white, opacity=0.9, very thick] (1a) .. controls +(-0.1,0.28,0) and +(0,0.28,0) .. (a);
\draw[color=green!60!black, opacity=0.4,  very thick] (1a) .. controls +(-0.1,0.28,0) and +(0,0.28,0) .. (a);
\fill[inner color=red!30!white,outer color=red!55!white, very thick] (a) .. controls +(0,-0.26,0) and +(0,-0.26,0) .. (1a) -- (1a) .. controls +(0,0.2,0) and +(0,0.2,0) .. (a);
\draw[color=green!60!black, very thick] (a) .. controls +(0,-0.26,0) and +(0,-0.26,0) .. (1a);
\draw[color=green!60!black, very thick] (1a) .. controls +(0,0.2,0) and +(0,0.2,0) .. (a);
\fill[color=green!60!black] (a) circle (1.2pt) node[color=green!60!black, opacity=1, left, font=\footnotesize] {};
\fill[color=green!60!black] (1a) circle (1.2pt) node[color=green!60!black, opacity=1, left, font=\footnotesize] {};
%
%
\draw [black,opacity=1, very thin] (1T) -- (1L) -- (L) -- (T);
\draw [black,opacity=1, very thin] (1T) -- (1R1) -- (R1) -- (T);
\draw [black,opacity=1, very thin] (1T) -- (1R2) -- (R2) -- (T);
\end{tikzpicture}}
\begin{tikzpicture}[
			     baseline=(current bounding box.base), 
			     descr/.style={fill=white,inner sep=3.5pt}, 
			     normal line/.style={->}
			     ] 
\matrix (m) [matrix of math nodes, row sep=3.5em, column sep=2.5em, text height=1.1ex, text depth=0.1ex] {%
{}
&
{}
\\
};
\path[font=\footnotesize] (m-1-1) edge[<->] node[below] { {\scriptsize bubble}} (m-1-2);
\end{tikzpicture}
\tikzzbox{\begin{tikzpicture}[thick,scale=2.021,color=blue!50!black, baseline=0.0cm, >=stealth, 
				style={x={(-0.6cm,-0.4cm)},y={(1cm,-0.2cm)},z={(0cm,0.9cm)}}]
	\pgfmathsetmacro{\yy}{0.2}
\coordinate (T) at (0.5, 0.4, 0);
\coordinate (L) at (0.5, 0, 0);
\coordinate (R1) at (0.3, 1, 0);
\coordinate (R2) at (0.7, 1, 0);
\coordinate (1T) at (0.5, 0.4, 1);
\coordinate (1L) at (0.5, 0, 1);
\coordinate (1R1) at (0.3, 1, );
\coordinate (1R2) at (0.7, 1, );
%
\draw [black,opacity=1, very thin] (L) -- (T) -- (1T) -- (1L) -- cycle;
\draw [black,opacity=1, very thin] (R1) -- (T) -- (1T) -- (1R1) -- cycle;
\draw [black,opacity=1, very thin] (R2) -- (T) -- (1T) -- (1R2) -- cycle;
%
\fill [red!50,opacity=0.545] (L) -- (T) -- (1T) -- (1L);
\fill [red!50,opacity=0.545] (R1) -- (T) -- (1T) -- (1R1);
\fill [red!50,opacity=0.545] (R2) -- (T) -- (1T) -- (1R2);
%
\draw[color=green!60!black, very thick] (T) -- (1T);
\end{tikzpicture}}
\, ,
\ee
which we consider for all possible orientations that are consistent with $\A$-decorations. 

\begin{lemma}
\label{lem:bubbel14}
A 1-4 move is a concatenation of a bubble move and a 2-3 move. 
\end{lemma}

\begin{proof}
Given a tetrahedron, we pick one of its faces and apply the bubble move \eqref{eq:bubble} to it. 
The result are three tetrahedra, one pair of which shares three faces, and another pair shares a single face. 
To the latter pair we apply the 2-3 move \eqref{eq:P23}, producing a total of four tetrahedra meeting at a single vertex: 
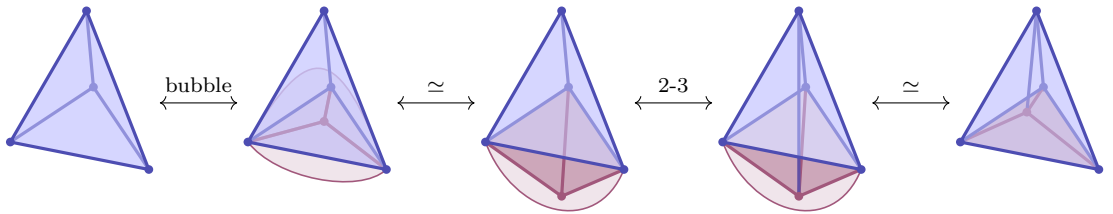
\vspace{1cm}
$$
\begin{tikzpicture}[
			     baseline=(current bounding box.base), 
			     descr/.style={fill=white,inner sep=3.5pt}, 
			     normal line/.style={->}
			     ] 
\matrix (m) [matrix of math nodes, row sep=3.5em, column sep=2.5em, text height=1.1ex, text depth=0.1ex] {%
\!\!\!\!\!
\tikzzbox{\begin{tikzpicture}[thick,scale=1.821,color=gray!60!blue, baseline=-0.3cm, >=stealth, 
				style={x={(-0.6cm,-0.4cm)},y={(1cm,-0.2cm)},z={(0cm,0.9cm)}}]
\coordinate (v1) at (1,0,0);
\coordinate (v2) at (1,1,0);
\coordinate (v3) at (0,0,0);
\coordinate (v4) at (0.25,0.1,0.75);
%
\fill [blue!20,opacity=0.545] (v1) -- (v2) -- (v3);
\fill [blue!20,opacity=0.545] (v4) -- (v2) -- (v3);
\fill [blue!20,opacity=0.545] (v1) -- (v4) -- (v3);
\fill[color=gray!60!blue] (v3) circle (0.9pt) node[left] (0up) {};
\draw[color=gray!60!blue, very thick] (v1) -- (v3);
\draw[color=gray!60!blue, very thick] (v2) -- (v3);
\draw[color=gray!60!blue, very thick] (v3) -- (v4);
\fill [blue!20,opacity=0.545] (v1) -- (v2) -- (v4);
\draw[color=gray!60!blue, very thick] (v2) -- (v4);
\draw[color=gray!60!blue, very thick] (v1) -- (v4);
\draw[color=gray!60!blue, very thick] (v1) -- (v2);
\fill[color=gray!60!blue] (v1) circle (0.9pt) node[below] (0up) {};
\fill[color=gray!60!blue] (v2) circle (0.9pt) node[below] (0up) {};
\fill[color=gray!60!blue] (v4) circle (0.9pt) node[above] (0up) {};

\end{tikzpicture}}
\!\!
&  
\!\!
\tikzzbox{\begin{tikzpicture}[thick,scale=1.821,color=gray!60!blue, baseline=-0.3cm, >=stealth, 
				style={x={(-0.6cm,-0.4cm)},y={(1cm,-0.2cm)},z={(0cm,0.9cm)}}]
\coordinate (v1) at (1,0,0);
\coordinate (v2) at (1,1,0);
\coordinate (v3) at (0,0,0);
\coordinate (v4) at (0.25,0.1,0.75);
%
\coordinate (m1) at (0.5,0.25,0);
\draw[color=magenta!60!black, very thick] (m1) -- (v1);
\draw[color=magenta!60!black, very thick] (m1) -- (v2);
\draw[color=magenta!60!black, very thick] (m1) -- (v3);
\fill[color=magenta!60!black] (m1) circle (0.9pt) node[left] (0up) {};
%
%
%
\fill [blue!20,opacity=0.545] (v1) -- (v2) -- (v3);
%
\coordinate (m1) at (0.5,0.25,0);

\fill[color=magenta!60!black, opacity=0.15, semithick] (v1) .. controls +(0,0.2,-0.25) and +(0,-0.2,-0.25) .. (v2);
\draw[color=magenta!60!black, semithick] (v1) .. controls +(0,0.2,-0.25) and +(0,-0.2,-0.25) .. (v2);
\fill[color=magenta!60!black, opacity=0.2, semithick] (v1) .. controls +(0,0.45,0.95) and +(0,-0.35,0.95) .. (v2);
\draw[color=magenta!60!black, semithick] (v1) .. controls +(0,0.45,0.95) and +(0,-0.35,0.95) .. (v2);
\draw[color=magenta!60!black, very thick] (m1) -- (v1);
\draw[color=magenta!60!black, very thick] (m1) -- (v2);
\draw[color=magenta!60!black, very thick] (m1) -- (v3);
\fill[color=magenta!60!black] (m1) circle (0.9pt) node[left] (0up) {};
%
\fill [blue!20,opacity=0.545] (v4) -- (v2) -- (v3);
\fill [blue!20,opacity=0.545] (v1) -- (v4) -- (v3);
\fill[color=gray!60!blue] (v3) circle (0.9pt) node[left] (0up) {};
\draw[color=gray!60!blue, very thick] (v1) -- (v3);
\draw[color=gray!60!blue, very thick] (v2) -- (v3);
\draw[color=gray!60!blue, very thick] (v3) -- (v4);
\fill [blue!20,opacity=0.545] (v1) -- (v2) -- (v4);
\draw[color=gray!60!blue, very thick] (v2) -- (v4);
\draw[color=gray!60!blue, very thick] (v1) -- (v4);
\draw[color=gray!60!blue, very thick] (v1) -- (v2);
\fill[color=gray!60!blue] (v1) circle (0.9pt) node[below] (0up) {};
\fill[color=gray!60!blue] (v2) circle (0.9pt) node[below] (0up) {};
\fill[color=gray!60!blue] (v4) circle (0.9pt) node[above] (0up) {};
\end{tikzpicture}}
\!\!
&  
\!\!
\tikzzbox{\begin{tikzpicture}[thick,scale=1.821,color=gray!60!blue, baseline=-0.3cm, >=stealth, 
				style={x={(-0.6cm,-0.4cm)},y={(1cm,-0.2cm)},z={(0cm,0.9cm)}}]
\coordinate (v1) at (1,0,0);
\coordinate (v2) at (1,1,0);
\coordinate (v3) at (0,0,0);
\coordinate (v4) at (0.25,0.1,0.75);
\coordinate (v0) at (0.25,0.1,-0.75);
\fill [magenta!60!black,opacity=0.245] (v1) -- (v2) -- (v3);
\fill [blue!20,opacity=0.545] (v4) -- (v2) -- (v3);
\fill [blue!20,opacity=0.545] (v1) -- (v4) -- (v3);
\draw[color=magenta!60!black, very thick] (v0) -- (v3);
\fill[color=gray!60!blue] (v3) circle (0.9pt) node[left] (0up) {};
%
\fill[color=magenta!60!black, opacity=0.15, semithick] (v1) .. controls +(0,0.2,-0.55) and +(0,-0.2,-0.6) .. (v2);
\draw[color=magenta!60!black, semithick] (v1) .. controls +(0,0.2,-0.55) and +(0,-0.2,-0.6) .. (v2);
%
%
%
\fill [magenta!60!black,opacity=0.245] (v0) -- (v1) -- (v3);
\fill [magenta!60!black,opacity=0.245] (v0) -- (v2) -- (v3);
\draw[color=gray!60!blue, very thick] (v1) -- (v3);
\draw[color=gray!60!blue, very thick] (v2) -- (v3);
\draw[color=gray!60!blue, very thick] (v3) -- (v4);
%
%
\fill [magenta!60!black,opacity=0.245] (v0) -- (v1) -- (v2);
\draw[color=magenta!60!black, very thick] (v0) -- (v1);
\draw[color=magenta!60!black, very thick] (v0) -- (v2);
\fill [blue!20,opacity=0.545] (v1) -- (v2) -- (v4);
\draw[color=gray!60!blue, very thick] (v2) -- (v4);
\draw[color=gray!60!blue, very thick] (v1) -- (v4);
\draw[color=gray!60!blue, very thick] (v1) -- (v2);
\fill[color=magenta!60!black] (v0) circle (0.9pt) node[below] (0up) {};
\fill[color=gray!60!blue] (v1) circle (0.9pt) node[below] (0up) {};
\fill[color=gray!60!blue] (v2) circle (0.9pt) node[below] (0up) {};
\fill[color=gray!60!blue] (v4) circle (0.9pt) node[above] (0up) {};
\end{tikzpicture}}
\!\!
&  
\!\!
\tikzzbox{\begin{tikzpicture}[thick,scale=1.821,color=gray!60!blue, baseline=-0.3cm, >=stealth, 
				style={x={(-0.6cm,-0.4cm)},y={(1cm,-0.2cm)},z={(0cm,0.9cm)}}]
\coordinate (v1) at (1,0,0);
\coordinate (v2) at (1,1,0);
\coordinate (v3) at (0,0,0);
\coordinate (v4) at (0.25,0.1,0.75);
\coordinate (v0) at (0.25,0.1,-0.75);
%
\fill [blue!20,opacity=0.545] (v4) -- (v2) -- (v3);
\fill [blue!20,opacity=0.545] (v1) -- (v4) -- (v3);
\draw[color=magenta!60!black, very thick] (v0) -- (v3);
\fill[color=gray!60!blue] (v3) circle (0.9pt) node[left] (0up) {};
%
\fill[color=magenta!60!black, opacity=0.15, semithick] (v1) .. controls +(0,0.2,-0.55) and +(0,-0.2,-0.6) .. (v2);
\draw[color=magenta!60!black, semithick] (v1) .. controls +(0,0.2,-0.55) and +(0,-0.2,-0.6) .. (v2);
%
%
%
\fill [magenta!60!black,opacity=0.245] (v0) -- (v1) -- (v3);
\fill [magenta!60!black,opacity=0.245] (v0) -- (v2) -- (v3);
\draw[color=gray!60!blue, very thick] (v1) -- (v3);
\draw[color=gray!60!blue, very thick] (v2) -- (v3);
\draw[color=gray!60!blue, very thick] (v3) -- (v4);
\draw[color=gray!60!blue, very thick] (v0) -- (v4);
%
%
\fill [magenta!60!black,opacity=0.245] (v0) -- (v1) -- (v2);
\draw[color=magenta!60!black, very thick] (v0) -- (v1);
\draw[color=magenta!60!black, very thick] (v0) -- (v2);
\fill [blue!20,opacity=0.545] (v1) -- (v2) -- (v4);
\draw[color=gray!60!blue, very thick] (v2) -- (v4);
\draw[color=gray!60!blue, very thick] (v1) -- (v4);
\draw[color=gray!60!blue, very thick] (v1) -- (v2);
\fill[color=magenta!60!black] (v0) circle (0.9pt) node[below] (0up) {};
\fill[color=gray!60!blue] (v1) circle (0.9pt) node[below] (0up) {};
\fill[color=gray!60!blue] (v2) circle (0.9pt) node[below] (0up) {};
\fill[color=gray!60!blue] (v4) circle (0.9pt) node[above] (0up) {};
\end{tikzpicture}}
\!\!
&  
\!\!
\tikzzbox{\begin{tikzpicture}[thick,scale=1.821,color=gray!60!blue, baseline=-0.3cm, >=stealth, 
				style={x={(-0.6cm,-0.4cm)},y={(1cm,-0.2cm)},z={(0cm,0.9cm)}}]
\coordinate (v1) at (1,0,0);
\coordinate (v2) at (1,1,0);
\coordinate (v3) at (0,0,0);
\coordinate (v4) at (0.25,0.1,0.75);
\coordinate (v5) at (0.7,0.3,0.175);
\fill [magenta!60!black,opacity=0.245] (v1) -- (v2) -- (v3);
\fill [blue!20,opacity=0.545] (v4) -- (v2) -- (v3);
\fill [blue!20,opacity=0.545] (v1) -- (v4) -- (v3);
\fill[color=gray!60!blue] (v3) circle (0.9pt) node[left] (0up) {};
\draw[color=magenta!60!black, very thick] (v1) -- (v5);
\draw[color=magenta!60!black, very thick] (v2) -- (v5);
\draw[color=magenta!60!black, very thick] (v3) -- (v5);
\draw[color=gray!60!blue, very thick] (v4) -- (v5);
\fill[color=magenta!60!black] (v5) circle (0.9pt) node[below] (0up) {};
\fill [magenta!60!black,opacity=0.245] (v1) -- (v2) -- (v5);
\fill [magenta!60!black,opacity=0.245] (v1) -- (v3) -- (v5);
\fill [magenta!60!black,opacity=0.245] (v3) -- (v2) -- (v5);
\draw[color=gray!60!blue, very thick] (v1) -- (v3);
\draw[color=gray!60!blue, very thick] (v2) -- (v3);
\draw[color=gray!60!blue, very thick] (v3) -- (v4);
\fill [blue!20,opacity=0.545] (v1) -- (v2) -- (v4);
\draw[color=gray!60!blue, very thick] (v2) -- (v4);
\draw[color=gray!60!blue, very thick] (v1) -- (v4);
\draw[color=gray!60!blue, very thick] (v1) -- (v2);
\fill[color=gray!60!blue] (v1) circle (0.9pt) node[below] (0up) {};
\fill[color=gray!60!blue] (v2) circle (0.9pt) node[below] (0up) {};
\fill[color=gray!60!blue] (v4) circle (0.9pt) node[above] (0up) {};
\end{tikzpicture}}
\\
};
\path[font=\footnotesize] (m-1-1) edge[<->] node[above] { {\scriptsize bubble}} (m-1-2);
\path[font=\footnotesize] (m-1-2) edge[<->] node[above] { {\scriptsize $\simeq$}} (m-1-3);
\path[font=\footnotesize] (m-1-3) edge[<->] node[above] { {\scriptsize 2-3}} (m-1-4);
\path[font=\footnotesize] (m-1-4) edge[<->] node[above] { {\scriptsize $\simeq$}} (m-1-5);
\end{tikzpicture}
\vspace{1cm}
$$
\end{proof}

As a corollary, the above lemma shows that all oriented 1-4 moves are implied by the oriented bubble and 2-3 moves.

\begin{lemma}
\label{lem:ZAM14}
The invariance condition (ii) in Definition~\ref{def:orbidatan} holds for all oriented 1-4 Pachner moves.
\end{lemma}

\begin{proof}
Thanks to Lemmas~\ref{lem:ZAM23} and~\ref{lem:bubbel14} it suffices to show invariance under the bubble moves~\eqref{eq:bubble}. 
For example, let us consider 
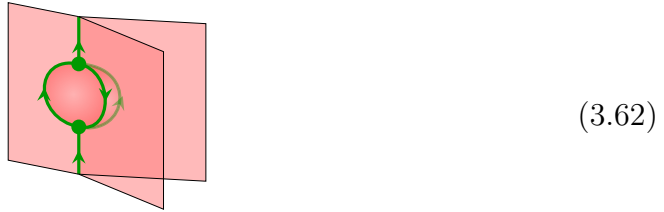
\be
\tikzzbox{\begin{tikzpicture}[thick,scale=2.321,color=blue!50!black, baseline=0.0cm, >=stealth, 
				style={x={(-0.6cm,-0.4cm)},y={(1cm,-0.2cm)},z={(0cm,0.9cm)}}]
	\pgfmathsetmacro{\yy}{0.2}
\coordinate (T) at (0.5, 0.4, 0);
\coordinate (L) at (0.5, 0, 0);
\coordinate (R1) at (0.3, 1, 0);
\coordinate (R2) at (0.7, 1, 0);
\coordinate (1T) at (0.5, 0.4, 1);
\coordinate (1L) at (0.5, 0, 1);
\coordinate (1R1) at (0.3, 1, 1);
\coordinate (1R2) at (0.7, 1, 1);
\coordinate (a) at (0.5, 0.4, 0.3);
\coordinate (1a) at (0.5, 0.4, 0.7);
%
\fill [red!50,opacity=0.545] (L) -- (T) -- (1T) -- (1L);
\fill [red!50,opacity=0.545] (R1) -- (T) -- (1T) -- (1R1);
\fill [red!50,opacity=0.545] (R2) -- (T) -- (1T) -- (1R2);
%
\draw[string=green!60!black, very thick] (T) -- (a);
\draw[string=green!60!black, very thick] (1a) -- (1T);
%
\fill[inner color=red!30!white,outer color=red!55!white, opacity=0.9, very thick] (1a) .. controls +(-0.1,0.28,0) and +(0,0.28,0) .. (a);
\draw[color=green!60!black, opacity=0.4,  very thick, postaction={decorate}, decoration={markings,mark=at position .6 with {\arrow[draw=green!60!black]{<}}}] (1a) .. controls +(-0.1,0.28,0) and +(0,0.28,0) .. (a);
\fill[inner color=red!30!white,outer color=red!55!white, very thick] (a) .. controls +(0,-0.26,0) and +(0,-0.26,0) .. (1a) -- (1a) .. controls +(0,0.2,0) and +(0,0.2,0) .. (a);
\draw[color=green!60!black, very thick, postaction={decorate}, decoration={markings,mark=at position .54 with {\arrow[draw=green!60!black]{>}}}] (a) .. controls +(0,-0.26,0) and +(0,-0.26,0) .. (1a);
\draw[color=green!60!black, very thick, postaction={decorate}, decoration={markings,mark=at position .54 with {\arrow[draw=green!60!black]{>}}}] (1a) .. controls +(0,0.2,0) and +(0,0.2,0) .. (a);
\fill[color=green!60!black] (a) circle (1.2pt) node[color=green!60!black, opacity=1, left, font=\footnotesize] {};
\fill[color=green!60!black] (1a) circle (1.2pt) node[color=green!60!black, opacity=1, left, font=\footnotesize] {};
%
%
\draw [black,opacity=1, very thin] (1T) -- (1L) -- (L) -- (T);
\draw [black,opacity=1, very thin] (1T) -- (1R1) -- (R1) -- (T);
\draw [black,opacity=1, very thin] (1T) -- (1R2) -- (R2) -- (T);
\end{tikzpicture}}
\ee
where here and below to avoid clutter, we do not show the labels $\A, T, \al, \alb$. 
Using isotopy invariance of~$\zz$ as well as the defining properties of orbifold data, we find that under~$\zz$: 
\be
\quad
\tikzzbox
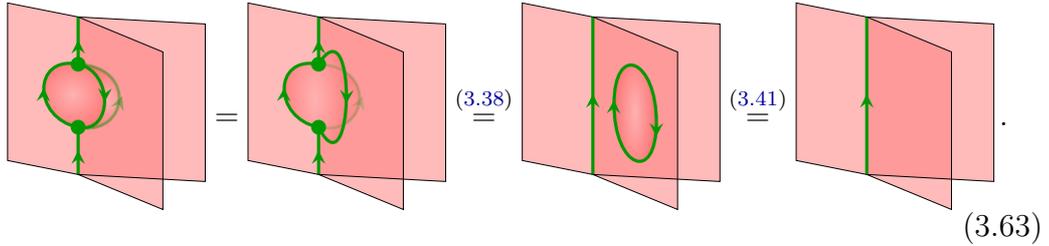
{\begin{tikzpicture}[thick,scale=2.321,color=blue!50!black, baseline=0.0cm, >=stealth, 
				style={x={(-0.6cm,-0.4cm)},y={(1cm,-0.2cm)},z={(0cm,0.9cm)}}]
	\pgfmathsetmacro{\yy}{0.2}
\coordinate (T) at (0.5, 0.4, 0);
\coordinate (L) at (0.5, 0, 0);
\coordinate (R1) at (0.3, 1, 0);
\coordinate (R2) at (0.7, 1, 0);
\coordinate (1T) at (0.5, 0.4, 1);
\coordinate (1L) at (0.5, 0, 1);
\coordinate (1R1) at (0.3, 1, 1);
\coordinate (1R2) at (0.7, 1, 1);
\coordinate (a) at (0.5, 0.4, 0.3);
\coordinate (1a) at (0.5, 0.4, 0.7);
%
\fill [red!50,opacity=0.545] (L) -- (T) -- (1T) -- (1L);
\fill [red!50,opacity=0.545] (R1) -- (T) -- (1T) -- (1R1);
\fill [red!50,opacity=0.545] (R2) -- (T) -- (1T) -- (1R2);
%
\draw[string=green!60!black, very thick] (T) -- (a);
\draw[string=green!60!black, very thick] (1a) -- (1T);
%
\fill[inner color=red!30!white,outer color=red!55!white, opacity=0.9, very thick] (1a) .. controls +(-0.1,0.28,0) and +(0,0.28,0) .. (a);
\draw[color=green!60!black, opacity=0.4,  very thick, postaction={decorate}, decoration={markings,mark=at position .6 with {\arrow[draw=green!60!black]{<}}}] (1a) .. controls +(-0.1,0.28,0) and +(0,0.28,0) .. (a);
\fill[inner color=red!30!white,outer color=red!55!white, very thick] (a) .. controls +(0,-0.26,0) and +(0,-0.26,0) .. (1a) -- (1a) .. controls +(0,0.2,0) and +(0,0.2,0) .. (a);
\draw[color=green!60!black, very thick, postaction={decorate}, decoration={markings,mark=at position .54 with {\arrow[draw=green!60!black]{>}}}] (a) .. controls +(0,-0.26,0) and +(0,-0.26,0) .. (1a);
\draw[color=green!60!black, very thick, postaction={decorate}, decoration={markings,mark=at position .54 with {\arrow[draw=green!60!black]{>}}}] (1a) .. controls +(0,0.2,0) and +(0,0.2,0) .. (a);
\fill[color=green!60!black] (a) circle (1.2pt) node[color=green!60!black, opacity=1, left, font=\footnotesize] {};
\fill[color=green!60!black] (1a) circle (1.2pt) node[color=green!60!black, opacity=1, left, font=\footnotesize] {};
%
%
\draw [black,opacity=1, very thin] (1T) -- (1L) -- (L) -- (T);
\draw [black,opacity=1, very thin] (1T) -- (1R1) -- (R1) -- (T);
\draw [black,opacity=1, very thin] (1T) -- (1R2) -- (R2) -- (T);
\end{tikzpicture}}
= 
\tikzzbox{\begin{tikzpicture}[thick,scale=2.321,color=blue!50!black, baseline=0.0cm, >=stealth, 
				style={x={(-0.6cm,-0.4cm)},y={(1cm,-0.2cm)},z={(0cm,0.9cm)}}]
	\pgfmathsetmacro{\yy}{0.2}
\coordinate (T) at (0.5, 0.4, 0);
\coordinate (L) at (0.5, 0, 0);
\coordinate (R1) at (0.3, 1, 0);
\coordinate (R2) at (0.7, 1, 0);
\coordinate (1T) at (0.5, 0.4, 1);
\coordinate (1L) at (0.5, 0, 1);
\coordinate (1R1) at (0.3, 1, 1);
\coordinate (1R2) at (0.7, 1, 1);
\coordinate (a) at (0.5, 0.4, 0.3);
\coordinate (1a) at (0.5, 0.4, 0.7);
%
\fill [red!50,opacity=0.545] (L) -- (T) -- (1T) -- (1L);
\fill [red!50,opacity=0.545] (R1) -- (T) -- (1T) -- (1R1);
%
%
\fill[inner color=red!30!white,outer color=red!55!white, opacity=0.9, very thick] (1a) .. controls +(-0.1,0.28,0) and +(0,0.28,0) .. (a);
\fill [red!50,opacity=0.545] (R2) -- (T) -- (1T) -- (1R2);
\draw[string=green!60!black, very thick] (T) -- (a);
\draw[string=green!60!black, very thick] (1a) -- (1T);
\fill[inner color=red!30!white,outer color=red!55!white, opacity=1, very thick] (a) .. controls +(0,-0.26,0) and +(0,-0.26,0) .. (1a) -- (1a) .. controls +(0,0.22,0.45) and +(0,0.2,-0.4) .. (a);
\draw[color=green!60!black, very thick, postaction={decorate}, decoration={markings,mark=at position .54 with {\arrow[draw=green!60!black]{>}}}] (a) .. controls +(0,-0.26,0) and +(0,-0.26,0) .. (1a);
\draw[color=green!60!black, very thick, postaction={decorate}, decoration={markings,mark=at position .54 with %
{\arrow[draw=green!60!black]{>}}}] (1a) .. controls +(0,0.22,0.45) and +(0,0.2,-0.4) .. (a);
\fill[color=green!60!black] (a) circle (1.2pt) node[color=green!60!black, opacity=1, left, font=\footnotesize] {};
\fill[color=green!60!black] (1a) circle (1.2pt) node[color=green!60!black, opacity=1, left, font=\footnotesize] {};
\draw[color=green!60!black, opacity=0.2,  very thick, postaction={decorate}, decoration={markings,mark=at position .6 with {\arrow[draw=green!60!black]{<}}}] (1a) .. controls +(-0.1,0.28,0) and +(0,0.28,0) .. (a);
%
%
\draw [black,opacity=1, very thin] (1T) -- (1L) -- (L) -- (T);
\draw [black,opacity=1, very thin] (1T) -- (1R1) -- (R1) -- (T);
\draw [black,opacity=1, very thin] (1T) -- (1R2) -- (R2) -- (T);
\end{tikzpicture}}
\stackrel{\eqref{eq:alpha1}}{=}
\tikzzbox{\begin{tikzpicture}[thick,scale=2.321,color=blue!50!black, baseline=0.0cm, >=stealth, 
				style={x={(-0.6cm,-0.4cm)},y={(1cm,-0.2cm)},z={(0cm,0.9cm)}}]
	\pgfmathsetmacro{\yy}{0.2}
\coordinate (T) at (0.5, 0.4, 0);
\coordinate (L) at (0.5, 0, 0);
\coordinate (R1) at (0.3, 1, 0);
\coordinate (R2) at (0.7, 1, 0);
\coordinate (1T) at (0.5, 0.4, 1);
\coordinate (1L) at (0.5, 0, 1);
\coordinate (1R1) at (0.3, 1, );
\coordinate (1R2) at (0.7, 1, );
\coordinate (d) at (0.55, 0.55, 0.5);
\coordinate (b) at (0.65, 0.85, 0.5);
%
\fill [red!50,opacity=0.545] (L) -- (T) -- (1T) -- (1L);
\fill [red!50,opacity=0.545] (R1) -- (T) -- (1T) -- (1R1);
\fill [red!50,opacity=0.545] (R2) -- (T) -- (1T) -- (1R2);
%
\draw[string=green!60!black, very thick] (T) -- (1T);
%
%
\draw [black,opacity=1, very thin] (1T) -- (1L) -- (L) -- (T);
\draw [black,opacity=1, very thin] (1T) -- (1R1) -- (R1) -- (T);
\draw [black,opacity=1, very thin] (1T) -- (1R2) -- (R2) -- (T);
\fill[inner color=red!30!white,outer color=red!55!white, very thick, rounded corners=0.5mm] (d) .. controls +(0,0,0.4) and +(0,0,0.4) .. (b) -- (b) .. controls +(0,0,-0.4) and +(0,0,-0.4) .. (d);
\draw[color=green!60!black, very thick, rounded corners=0.5mm, postaction={decorate}, decoration={markings,mark=at position .04 with {\arrow[draw=green!60!black]{>}}}] (d) .. controls +(0,0,0.4) and +(0,0,0.4) .. (b);
\draw[color=green!60!black, very thick, rounded corners=0.5mm, postaction={decorate}, decoration={markings,mark=at position .99 with {\arrow[draw=green!60!black]{<}}}] (d) .. controls +(0,0,-0.4) and +(0,0,-0.4) .. (b);
\end{tikzpicture}}
\stackrel{\eqref{eq:bubble2}}{=}
\tikzzbox{\begin{tikzpicture}[thick,scale=2.321,color=blue!50!black, baseline=0.0cm, >=stealth, 
				style={x={(-0.6cm,-0.4cm)},y={(1cm,-0.2cm)},z={(0cm,0.9cm)}}]
	\pgfmathsetmacro{\yy}{0.2}
\coordinate (T) at (0.5, 0.4, 0);
\coordinate (L) at (0.5, 0, 0);
\coordinate (R1) at (0.3, 1, 0);
\coordinate (R2) at (0.7, 1, 0);
\coordinate (1T) at (0.5, 0.4, 1);
\coordinate (1L) at (0.5, 0, 1);
\coordinate (1R1) at (0.3, 1, );
\coordinate (1R2) at (0.7, 1, );
%
\fill [red!50,opacity=0.545] (L) -- (T) -- (1T) -- (1L);
\fill [red!50,opacity=0.545] (R1) -- (T) -- (1T) -- (1R1);
\fill [red!50,opacity=0.545] (R2) -- (T) -- (1T) -- (1R2);
%
\draw[string=green!60!black, very thick] (T) -- (1T);
%
%
\draw [black,opacity=1, very thin] (1T) -- (1L) -- (L) -- (T);
\draw [black,opacity=1, very thin] (1T) -- (1R1) -- (R1) -- (T);
\draw [black,opacity=1, very thin] (1T) -- (1R2) -- (R2) -- (T);
\end{tikzpicture}}
\, . 
\ee
This establishes invariance under one of the
	oriented bubble moves
and the proof of the other ones works analogously. 
\end{proof}

\medskip

Thanks to Lemmas~\ref{lem:ZAM23} and~\ref{lem:ZAM14}, we can now conclude:

\begin{proposition}
A special orbifold datum for a TQFT $\zz\colon  \Bordd[3] \to \Vectk$ in the sense of Definition~\ref{def:orbidata} is an orbifold datum in the sense of Definition~\ref{def:orbidatan}. 
\end{proposition}

In particular, by Theorem~\ref{thmdef:orbifoldtheory}, a special orbifold datum $\A$ defines a closed TQFT $\zz_\orb \colon \Bord_3 \to \Vectk$. 

\begin{remark}
\label{rem:examples}
In \cite{CRS3} we will construct several examples of special orbifold data~$\A$ and the associated orbifold theories. 
These include the following: 
\begin{enumerate}
\item
$\A$ is extracted from the data of a spherical fusion category (which in turn is precisely the input data for Turaev-Viro models \cite{TurVir, BarWes}). 
This is a 3-dimensional analogue of viewing 2-dimensional state sum models \cite{bp9205031,FHK} as orbifolds via $\Delta$-separable symmetric $\Bbbk$-algebras \cite{dkr1107.0495,cr1210.6363}. 
More generally, one can think of special orbifold data for a 3-dimensional defect TQFT~$\zz$ as 
``spherical fusion categories internal to a Gray category with duals'' (see Section~\ref{subsec:specorbdat3}). 
\item 
$\A$ is extracted from a certain type of special symmetric Frobenius algebra internal to a modular tensor category, using an extension of the Reshetikhin-Turaev construction. 
\item
$\A$ is extracted from a ``surface defect with invertible bubble''. 
This is an analogue of the Barr-Beck-type construction with invertible quantum dimensions of \cite{cr1210.6363}, which in fact generalises to arbitrary dimension~$n$. 
\end{enumerate}
\end{remark}

\section{Higher categorical formulation}
\label{sec:highercatfor}

It is expected that the sets of defect labels in an $n$-dimensional defect TQFT arrange themselves into a ``fairly strict $n$-category with duals'' for any $n\in \Z_+$.
This has been made precise for $n=2$ and $n=3$:
\begin{itemize}
\item
It was shown in \cite{dkr1107.0495} that one can naturally extract a strictly pivotal 2-category from a 2-dimensional defect TQFT, and every bicategory with ambidextrous duals for 1-morphisms is biequivalent to a strictly pivotal 2-category, as follows from \cite{bwTV2, NSHFS}. 
Hence for $n=2$ ``fairly strict'' means strict, which in this case is the same as ``as strict as achievable by biequivalence''. 
\item
For $n=3$, it was shown in \cite{CMS} that one naturally obtains a Gray category with duals from every 3-dimensional defect TQFT. 
This is generically the strictest form of 3-categorical structure (where the only non-identity coherence 3-morphism that is allowed is in the interchange law for 2-morphisms) 
with ambidextrous duals for all 1- and 2-morphisms, and every tricategory with ambidextrous duals is triequivalent to a Gray category with duals, cf.\,\cite{Gurskibook, BMS, GregorDiss}. 
In this sense for $n=3$ ``fairly strict'' means ``as strict as achievable by triequivalence''. 
\end{itemize}

For $n>3$ we may expect a similar state of affairs. 
However, to extract a specific notion of $n$-category from a given defect TQFT is no simple combinatorial task: 
Any flavour of $n$-category comes with a prescribed shape that is used to define a notion of source and target for all $k$-morphisms, $k \leqslant n$. 
An element $ x \in D_{k}$ should correspond to an $(n-k)$-morphism, but it does so only after picking a certain decomposition of a sphere in $f_{k}(x)$ into the prescribed shape. 
The dualities in the category should then allow one to relate the morphisms defined by different decompositions. 
Instead of formalising the involved combinatorics, the notion of \textsl{disc-like $n$-category} was invented in \cite{blobcomplex}, where now a $(k-1)$-sphere serves as the combined source and target for a $k$-morphism,
	see Remark~\ref{rem:disclikes} for a more detailed discussion.

\medskip

In Sections~\ref{subsec:specorbdat2} and~\ref{subsec:specorbdat3} below we will reformulate the orbifold constructions of Sections~\ref{subsec:2dimorbis} and~\ref{subsec:3dimorbis} in 2- and 3-categorical language, respectively. 
In fact we formalise the notion of special orbifold data internal to arbitrary pivotal bicategories and Gray categories with duals (which need not be associated with defect TQFTs).

\subsection{Special orbifold data in pivotal bicategories}
\label{subsec:specorbdat2}

Recall from \cite{dkr1107.0495} that to every defect TQFT $\zz\colon  \Bordd[2] \to \Vectk$ one can naturally associate a 2-category $\B_\zz$. 
The objects of $\B_\zz$ are elements of $D_2$, and are to be thought of as $D_2$-decorated planes; 1-morphisms are lists of elements in $D_1 \times \{ \pm \}$ (the signs encode the orientation of 1-strata), which we picture as parallel $D_1$-decorated lines. 
The 2-morphisms are $\Bbbk$-vector spaces which~$\zz$ assigns to certain decorated circles. 
In fact $\B_\zz$ has identical left and right adjoints for all 1-morphisms (corresponding to orientation reversal of 1-strata) and is 
	in fact
a pivotal 2-category. 
This construction is reviewed in detail in \cite{2ddTQFTreview}, which also discusses examples of pivotal bicategories from algebraic and symplectic geometry, differential graded algebras, and categorified quantum groups. 

\medskip
	
Now let $(*,A,\mu,\eta,\Delta,\varepsilon)$ be a special orbifold datum for~$\zz$ as in Section~\ref{subsec:2dimorbis}.
In terms of $\mathcal B_{\zz}$, this means that~$A$ is a 1-endomorphism of the object $* \in \mathcal B_{\zz}$, and we have 2-morphisms $\mu\colon A \otimes A \rightarrow A$, $\eta\colon 1_* \rightarrow A$, $\Delta\colon A \rightarrow A \otimes A$ and $\varepsilon\colon A \rightarrow 1_*$. 
The constraints on these data, written in standard graphical calculus
	(with diagrams read from bottom to top), 
are as follows: 
\vspace{-0.2cm}
\begin{align}
& 
\tikzzbox{\begin{tikzpicture}[very thick,scale=0.53,color=green!50!black, baseline=0.59cm]
\draw[-dot-] (3,0) .. controls +(0,1) and +(0,1) .. (2,0);
\draw[-dot-] (2.5,0.75) .. controls +(0,1) and +(0,1) .. (3.5,0.75);
\draw (3.5,0.75) -- (3.5,0); 
\draw (3,1.5) -- (3,2.25); 
\end{tikzpicture}}
=
\tikzzbox{\begin{tikzpicture}[very thick,scale=0.53,color=green!50!black, baseline=0.59cm]
\draw[-dot-] (3,0) .. controls +(0,1) and +(0,1) .. (2,0);
\draw[-dot-] (2.5,0.75) .. controls +(0,1) and +(0,1) .. (1.5,0.75);
\draw (1.5,0.75) -- (1.5,0); 
\draw (2,1.5) -- (2,2.25); 
\end{tikzpicture}}
\, , 
\qquad
\tikzzbox{\begin{tikzpicture}[very thick,scale=0.4,color=green!50!black, baseline]
\draw (-0.5,-0.5) node[Odot] (unit) {}; 
\fill (0,0.6) circle (5.0pt) node (meet) {};
\draw (unit) .. controls +(0,0.5) and +(-0.5,-0.5) .. (0,0.6);
\draw (0,-1.5) -- (0,1.5); 
\end{tikzpicture}}
=
\tikzzbox{\begin{tikzpicture}[very thick,scale=0.4,color=green!50!black, baseline]
\draw (0,-1.5) -- (0,1.5); 
\end{tikzpicture}}
=
\tikzzbox{\begin{tikzpicture}[very thick,scale=0.4,color=green!50!black, baseline]
\draw (0.5,-0.5) node[Odot] (unit) {}; 
\fill (0,0.6) circle (5.0pt) node (meet) {};
\draw (unit) .. controls +(0,0.5) and +(0.5,-0.5) .. (0,0.6);
\draw (0,-1.5) -- (0,1.5); 
\end{tikzpicture}}
\, , \qquad
\tikzzbox{\begin{tikzpicture}[very thick,scale=0.53,color=green!50!black, baseline=-0.59cm, rotate=180]
\draw[-dot-] (3,0) .. controls +(0,1) and +(0,1) .. (2,0);
\draw[-dot-] (2.5,0.75) .. controls +(0,1) and +(0,1) .. (1.5,0.75);
\draw (1.5,0.75) -- (1.5,0); 
\draw (2,1.5) -- (2,2.25); 
\end{tikzpicture}}
=
\tikzzbox{\begin{tikzpicture}[very thick,scale=0.53,color=green!50!black, baseline=-0.59cm, rotate=180]
\draw[-dot-] (3,0) .. controls +(0,1) and +(0,1) .. (2,0);
\draw[-dot-] (2.5,0.75) .. controls +(0,1) and +(0,1) .. (3.5,0.75);
\draw (3.5,0.75) -- (3.5,0); 
\draw (3,1.5) -- (3,2.25); 
\end{tikzpicture}}
\, , \qquad
\tikzzbox{\begin{tikzpicture}[very thick,scale=0.4,color=green!50!black, baseline=0, rotate=180]
\draw (0.5,-0.5) node[Odot] (unit) {}; 
\fill (0,0.6) circle (5.0pt) node (meet) {};
\draw (unit) .. controls +(0,0.5) and +(0.5,-0.5) .. (0,0.6);
\draw (0,-1.5) -- (0,1.5); 
\end{tikzpicture}}
=
\tikzzbox{\begin{tikzpicture}[very thick,scale=0.4,color=green!50!black, baseline=0, rotate=180]
\draw (0,-1.5) -- (0,1.5); 
\end{tikzpicture}}
=
\tikzzbox{\begin{tikzpicture}[very thick,scale=0.4,color=green!50!black, baseline=0cm, rotate=180]
\draw (-0.5,-0.5) node[Odot] (unit) {}; 
\fill (0,0.6) circle (5.0pt) node (meet) {};
\draw (unit) .. controls +(0,0.5) and +(-0.5,-0.5) .. (0,0.6);
\draw (0,-1.5) -- (0,1.5); 
\end{tikzpicture}}
\ , 
\\
& 
\tikzzbox{\begin{tikzpicture}[very thick,scale=0.4,color=green!50!black, baseline=0cm]
\draw[-dot-] (0,0) .. controls +(0,-1) and +(0,-1) .. (1,0);
\draw[-dot-] (0,0) .. controls +(0,1) and +(0,1) .. (1,0);
\draw (0.5,-0.8) -- (0.5,-1.5); 
\draw (0.5,0.8) -- (0.5,1.5); 
\end{tikzpicture}}
\, = \, 
\tikzzbox{\begin{tikzpicture}[very thick,scale=0.4,color=green!50!black, baseline=0cm]
\draw (0.5,-1.5) -- (0.5,1.5); 
\end{tikzpicture}}
\, , \qquad
\tikzzbox{\begin{tikzpicture}[very thick,scale=0.4,color=green!50!black, baseline=0cm]
\draw[-dot-] (0,0) .. controls +(0,-1) and +(0,-1) .. (-1,0);
\draw[-dot-] (1,0) .. controls +(0,1) and +(0,1) .. (0,0);
\draw (-1,0) -- (-1,1.5); 
\draw (1,0) -- (1,-1.5); 
\draw (0.5,0.8) -- (0.5,1.5); 
\draw (-0.5,-0.8) -- (-0.5,-1.5); 
\end{tikzpicture}}
=
\tikzzbox{\begin{tikzpicture}[very thick,scale=0.4,color=green!50!black, baseline=0cm]
\draw[-dot-] (0,0) .. controls +(0,1) and +(0,1) .. (-1,0);
\draw[-dot-] (1,0) .. controls +(0,-1) and +(0,-1) .. (0,0);
\draw (-1,0) -- (-1,-1.5); 
\draw (1,0) -- (1,1.5); 
\draw (0.5,-0.8) -- (0.5,-1.5); 
\draw (-0.5,0.8) -- (-0.5,1.5); 
\end{tikzpicture}}
\, , \qquad 
\tikzzbox{\begin{tikzpicture}[very thick,scale=0.4,color=green!50!black, baseline=0cm]
\draw[-dot-] (0,0) .. controls +(0,1) and +(0,1) .. (-1,0);
\draw[directedgreen, color=green!50!black] (1,0) .. controls +(0,-1) and +(0,-1) .. (0,0);
\draw (-1,0) -- (-1,-1.5); 
\draw (1,0) -- (1,1.5); 
\draw (-0.5,1.2) node[Odot] (end) {}; 
\draw (-0.5,0.8) -- (end); 
\end{tikzpicture}}
= 
\tikzzbox{\begin{tikzpicture}[very thick,scale=0.4,color=green!50!black, baseline=0cm]
\draw[redirectedgreen, color=green!50!black] (0,0) .. controls +(0,-1) and +(0,-1) .. (-1,0);
\draw[-dot-] (1,0) .. controls +(0,1) and +(0,1) .. (0,0);
\draw (-1,0) -- (-1,1.5); 
\draw (1,0) -- (1,-1.5); 
\draw (0.5,1.2) node[Odot] (end) {}; 
\draw (0.5,0.8) -- (end); 
\end{tikzpicture}}
\label{eq:DeltasepsymFrob}
\, . 
\end{align}
In other words, a special orbifold datum $A$ for~$\zz$ 
	-- the data needed to define the orbifold TQFT $\zz_A \colon \Bord_2 \to \Vectk$ --
is a 
	$\Delta$-separable symmetric Frobenius algebra
in $\mathcal B_{\zz}$. 

\begin{remark}
\label{rem:2dorbicomp}
The above realisation, originally due to \cite{ffrs0909.5013} in the context of conformal field theory, is the starting point of the ``orbifold completion'' construction of \cite{cr1210.6363}: 
The defect TQFT~$\zz$ restricted to $\Bord_2$, viewed as a non-full subcategory of $\Bordd[2]$ where all objects and morphisms are exclusively decorated by $* \in D_2$, is a closed TQFT. 
One may ask whether there is a natural defect TQFT $\zz^{\text{orb}}$ which analogously restricts to the closed orbifold TQFT $\zz_A$. 
Indeed, as explained in \cite{cr1210.6363}, the orbifold construction naturally lifts to produce the ``complete'' orbifold defect TQFT $\zz^{\text{orb}}\colon \Bord_2^{\text{def}}(\D^{\text{orb}}) \to \Vectk$, where $D_2^{\text{orb}}$ and $D_1^{\text{orb}}$ are given by $\Delta$-separable symmetric Frobenius algebras in $\mathcal B_\zz$ and their bimodules, respectively. 
Algebraically, for every pivotal bicategory~$\mathcal{P}$ with idempotent complete morphism categories
this construction motivates the definition of the orbifold completion $\mathcal{P}_\text{orb}$ as the bicategory of $\Delta$-separable symmetric Frobenius algebras, bimodules and bimodule maps internal to~$\mathcal{P}$ \cite[Sec.\,5.1]{cr1210.6363}.
Then one has $\B_{(\zz^{\text{orb}})} \cong (\B_\zz)_{\text{orb}}$ and $(\mathcal{P}_\text{orb})_{\text{orb}} \cong \mathcal{P}_\text{orb}$. 
\end{remark}

\subsection{Special orbifold data in Gray categories with duals}
\label{subsec:specorbdat3}

In analogy to the 2-dimensional case of Section~\ref{subsec:specorbdat2}, to a defect TQFT $\zz\colon  \Bordd[3] \to \Vectk$ one can naturally associate a 3-categorical structure $\tz$, as explained in \cite{CMS}. 
Indeed, it was shown in loc.~cit.~that $\tz$ has the structure of a Gray category with duals. 
We refer to \cite{CMS} for the detailed construction of $\tz$. 
Here we only recall some of the basic structure so that we can formulate the notion of special orbifold data for~$\zz$ internal to $\tz$. 

Roughly, the objects of $\tz$ are elements of $D_3$, which we imagine as $D_3$-decorated patches of $\R^3$, say 3-cubes. 
In this picture, 1-morphisms are stacks of $D_2$-decorated planes, with the spaces in between them decorated compatibly with the maps~$s$ and~$t$. 
Similarly, 2-morphisms can be represented as $D_1$-decorated lines with $D_2$-decorated planes ending on them as allowed by the adjacency map~$f_{1}$. 
Hence we may depict objects~$u$, 1-morphisms $\alpha \colon u \to v$, and 2-morphisms $X \colon \beta \to \gamma$ in $\tz$ as decorated cubes such as
$$
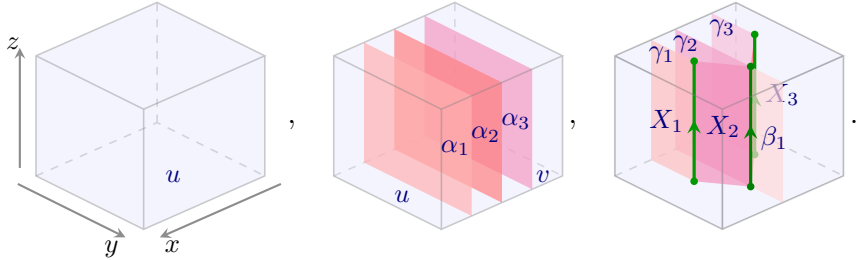
\tikzzbox{
}
\, . 
$$
The functor~$\zz$ is only used to construct the Hom sets for every two parallel 2-morphisms~$X$ and~$Y$, namely as the vector space which~$\zz$ assigns to the sphere around the potential meeting point of the lines~$X$ and~$Y$, decorated accordingly. 

Composition of 1-morphisms in $\tz$, denoted $\sta$, corresponds to stacking $D_2$-decorated planes such as
$$
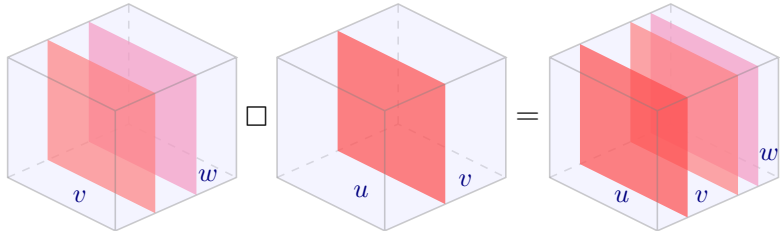
\tikzzbox{%
}
$$
while horizontal composition of 2-morphisms, denoted $\fus$, is concatenation in negative $y$-direction: 
$$
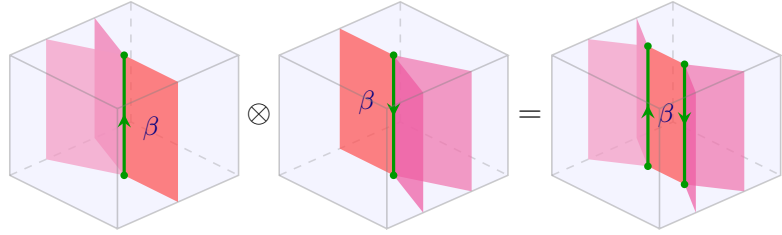
\tikzzbox{%
}
$$
Vertical composition of 3-morphisms (as well as $\sta$- and $\fus$-composition of 3-morphisms) corresponds to cutting and pasting appropriately decorated 3-balls. 

Orientation reversal of planes (1-morphisms) and lines (2-morphisms) gives rise to two distinct notions of duals in $\tz$, $\#$-duals and $\dagger$-duals, respectively, which are compatible with one another in a natural way. 
In the graphical calculus, for 1-morphisms and 2-morphisms we have
$$
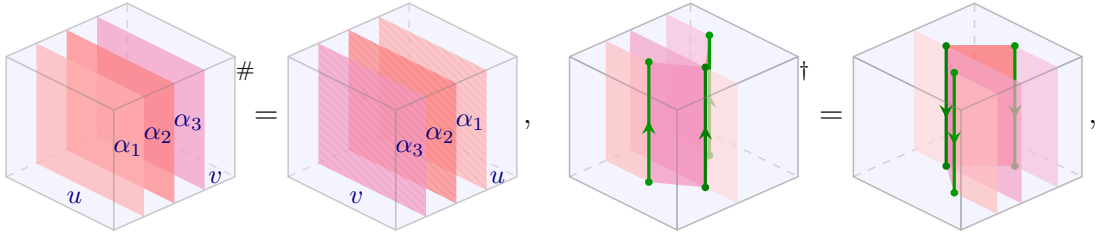
\tikzzbox{%
}
\, , 
$$
respectively, and the associated units of adjunction are the 2- and 3-morphisms 
$$
\coev_\alpha = 
\tikzzbox{%
}
\right) .
$$

\medskip

It is natural to express the constraints on special orbifold data for~$\zz$ in terms of $\tz$, or for any Gray category with duals: 

\begin{definition}
\label{def:orbidataGray}
Let~$\mathcal G$ be a Gray category with duals. 
A set of \textsl{special orbifold datum in $\mathcal G$} is 
\begin{itemize}
\item
an object $* \in \mathcal G$, 
\item
a 1-morphism $\A \in \mathcal G(*,*)$, 
\item 
a 2-morphism $T\colon \A \sta \A \to \A$, 
\item
two 3-isomorphisms $\al\colon T \fus (1_\A \sta T) \rightleftarrows T\otimes (T\sta 1_\A): \alb$, 
\end{itemize}
\be
\tikzzbox{
}
\ee 
such that there are 3-isomorphisms $\phi \in \text{Aut}(1_{1_*})$ and $\psi \in \text{Aut}(1_\A)$, and: 
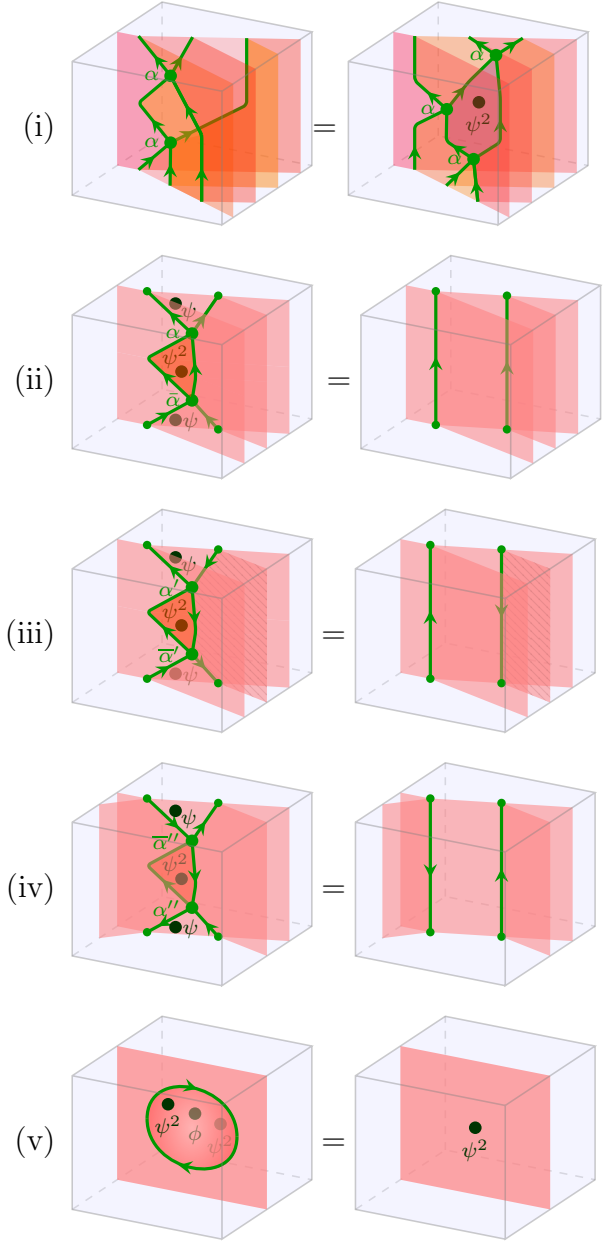
\begin{enumerate}
\item
\tikzzbox{%
}
$
\end{enumerate}
together with the opposite versions of (ii)--(iv) as in \eqref{eq:alpha1voll}--\eqref{eq:alpha3voll}, and with the two other orientations of the semi-spheres in (v) as in \eqref{eq:phipsi}. The vertices $\al', \alb', \al'', \alb''$ are determined from $\al$ and $\alb$ using the duals in $\mathcal G$ (drawing only the $T$-lines around 0-strata for clarity): 
\begin{align}
\al' 
:= 
\tikzzbox{%
}
\, .
\end{align}
In detail, $\al'$ is the 3-morphism (suppressing identities in the right column)
$$
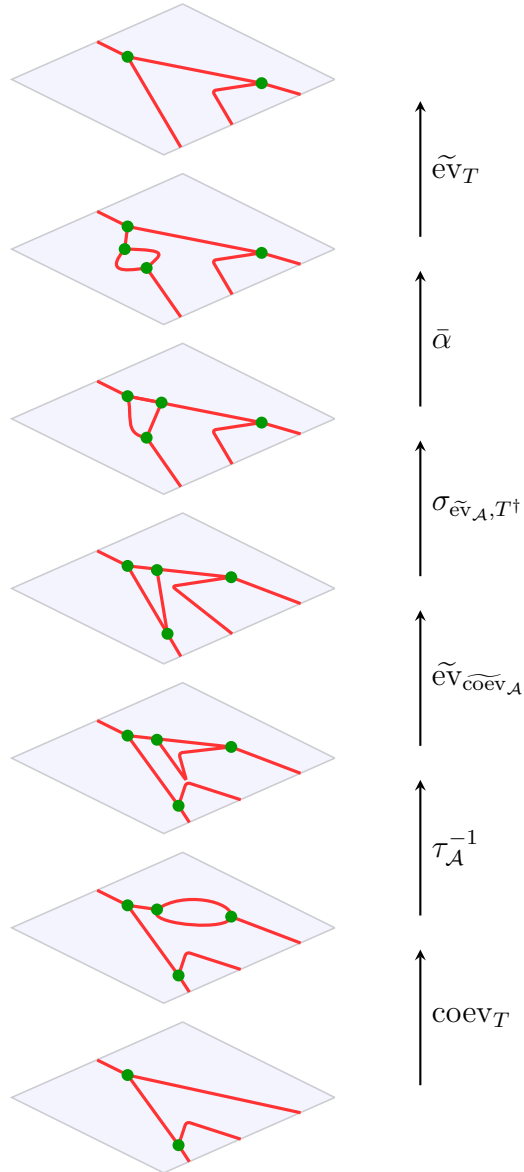
\tikzzbox{%
}
$$
where~$\sigma$ and~$\tau$ are the tensorator and triangulator of~$\mathcal G$, respectively, using the conventions of \cite[Def.\,3.4\,\&\,3.8]{CMS}.
There are similar expressions for $\alb', \al'', \alb''$. 
\end{definition}

Note that in the special case $\psi = 1_{1_{\A}}$ and $\phi = \delta \cdot 1_{1_{1_*}}$ for some $\delta \in \Bbbk$, the last condition~(v) says that the quantum dimension of~$T$ is invertible. 

\medskip

By construction, in three dimensions a special orbifold datum~$\A$ for a defect TQFT~$\zz$ may be identified with a special orbifold datum~$\A$ in~$\tz$. 
As we will explain in \cite{CRS3}, the latter can be thought of as a spherical fusion category internal to the Gray category with duals~$\tz$, just as special orbifold data in two dimensions are $\Delta$-separable symmetric Frobenius algebras.

\begin{remark}
Paralleling the constructions for the 2-dimensional case summarised in Remark~\ref{rem:2dorbicomp}, we expect there to be a natural defect TQFT $\zz^{\text{orb}}$ which for $\A$-decorated bordisms restricts to the 
closed orbifold theory~$\zz_\A$: 
the surface defect labels for $\zz^{\text{orb}}$ are $\A'$-$\A$-bimodule categories internal to~$\tz$, while line defect labels are compatible bimodule functors. 
Furthermore, we expect there to be an associated notion of orbifold completion $\mathcal G_{\text{orb}}$ for any Gray category with duals~$\mathcal G$
which is idempotent complete in an appropriate sense, 
such that $(\tz)_{\text{orb}} \cong \mathcal T_{\zz^{\text{orb}}}$ for $\mathcal G = \tz$. 
\end{remark}

\medskip

Next we address the relation between the Euler completion (see Definition~\ref{def:Eulercompletion}) of a defect TQFT $\zz\colon \Bordd[3] \to \Vectk$ and its Gray category $\tz$. 
We will show that different point insertions on surface defects for~$\zz$ correspond to different pivotal structures on~$\tz$. 

To see this, we fix a collection 
\be
\label{eq:psialphacoll}
\psi_\alpha \in \text{Aut}_{\tz}(1_\alpha) 
\quad
\text{for all } \alpha \in D_2 
\ee
for the remainder of this section. 
Demanding $\psi_{\alpha \sta \beta} = \psi_\alpha \sta \psi_\beta$, we in fact have one 3-isomorphism for every 2-morphism~$\alpha$ in~$\tz$. 
Restricting $\D^\euc$ to point insertions only of type~$\psi_\alpha$ gives us a set of defect data $\D^\psi$ with $D^\psi_2 = \{ (\alpha, \psi_\alpha) \,|\, \alpha \in D_2 \}$ and $D^\psi_j = D_j$ for $j\neq 2$, and we write 
\be
\zz^\psi \colon \Bord_3^{\text{def}}(\D^\psi) \lra \Vectk
\ee
for the restriction of~$\zz^\euc$. 

What is the relation between $\mathcal T_{\zz^\psi}$ and $\tz$? 
The answer involves the following natural notion: 
Given any Gray category with duals~$\mathcal G$ together with a collection $c = \{ c_\alpha \in \text{Aut}_{\mathcal G}(1_\alpha) \,|\, \alpha \text{ is a 1-morphism in } \mathcal G \}$ satisfying $c_{\alpha \sta \beta} = c_\alpha \sta c_\beta$, the \textsl{$c$-twist} $\mathcal G^c$ is the Gray category with duals whose underlying tricategory and $\#$-duals are those of~$\mathcal G$, while the pivotal structures of $\mathcal G^c$ are those with the adjunction 3-morphisms
\begin{align}
\ev_X^{\mathcal G^c} & = c_\alpha^{-1} \circ \ev_X^{\mathcal G} \circ \big( 1_{X^\dagger} \otimes c_\beta \otimes 1_X \big) 
\colon X^\dagger \otimes X \lra 1_\alpha \, , \nonumber
\\
\coev_X^{\mathcal G^c} & =\big( 1_{X} \otimes c_\alpha \otimes 1_{X^\dagger} \big) \circ \coev_X^{\mathcal G} \circ \, c_{\beta}^{-1} 
\colon 1_\beta  \lra X \otimes X^\dagger \, , \nonumber
\\
\tev_X^{\mathcal G^c} & = c_\beta^{-1} \circ \tev_X^{\mathcal G} \circ \big( 1_{X} \otimes c_\alpha \otimes 1_{X^\dagger} \big) 
\colon X \otimes X^\dagger \lra 1_\beta \, , \nonumber
\\
\tcoev_X^{\mathcal G^c} & =\big( 1_{X^\dagger} \otimes c_\beta \otimes 1_{X} \big) \circ \tcoev_X^{\mathcal G} \circ \, c_{\alpha}^{-1} 
\colon 1_\alpha  \lra X^\dagger  \otimes X
\label{eq:twistedcoevs}
\end{align}
for every 2-morphism $X\colon \alpha \to \beta$ in~$\mathcal G$. 
Put differently, $\mathcal G^c$ is the same as~$\mathcal G$, except that the pivotal structures are are ``twisted'' by the maps~$c_\alpha$. 

\begin{proposition}
For a defect TQFT $\zz \colon \Bordd[3] \to \Vectk$ and a collection of invertible 3-morphisms~$\psi_\alpha$ as in~\eqref{eq:psialphacoll}, there is an equivalence
\be
\mathcal T_{\zz^\psi} \cong (\tz)^\psi
\ee
of $\Bbbk$-linear Gray categories with duals. 
\end{proposition}

\begin{proof}
We will first express vertical composition in $\mathcal T_{\zz^\psi}$ in terms of~$\tz$. 
Then we will find a triequivalence $\Gamma \colon \mathcal T_{\zz^\psi} \to (\tz)^\psi$ which maps the adjunction 3-morphisms in $\mathcal T_{\zz^\psi}$ to those in $(\tz)^\psi$. 

Let $f\colon X\to Y$ and $g\colon Y \to Z$ be 3-morphisms in $\mathcal T_{\zz^\psi}$. 
We denote the common source of $X,Y,Z$ by $(\alpha,\psi)$, and the common target by $(\beta,\psi)$. 
By definition (cf.~\cite[Sect.\,3.3]{CMS}), the vertical composition $g\circ_\psi f$ in $\mathcal T_{\zz^\psi}$ is $\zz^\psi(B_{g,f}) \in \Hom_{\mathcal T_{\zz^\psi}}(X,Z)$, where $B_{g,f}$ is the following defect bordism: 
$B_{g,f}$ is the solid 3-ball with two smaller 3-balls $B_f, B_g$ removed whose boundaries are ingoing for $B_{g,f}$ and decorated such that $f \in \zz^\psi(\partial B_f)$ and $g \in \zz^\psi(\partial B_g)$, while the remaining boundary component of $B_{g,f}$ is outgoing and its stratification is locally a cylinder over the boundary stratifications. 
Schematically, 
\be
\label{eq:Bgf}
B_{g,f} = 
\tikzzbox{\begin{tikzpicture}[very thick,scale=1.75,color=green!60!black=-0.1cm, >=stealth, baseline=0]
\fill[ball color=red!50!white, opacity=0.9] (0,0) circle (0.95 cm);
\fill[ball color=red!50!white, opacity=0.07] (0,-0.4) circle (0.27);
\fill[color=white, opacity=0.25] (0,-0.4) circle (0.27);
\fill[ball color=red!50!white, opacity=0.07] (0,0.4) circle (0.27);
\fill[color=white, opacity=0.25] (0,0.4) circle (0.27);
\draw[color=red!80!black, very thick, >=stealth] (0,0) circle (0.95);
\draw[color=red!80!black,  thick, >=stealth] (0,-0.4) circle (0.27);
\draw[color=red!80!black,  thick, >=stealth] (0,0.4) circle (0.27);
\fill[color=green!60!black] (0,0.13) circle (1pt) node[color=red!80!black, font=\tiny] {};
\fill[color=green!60!black] (0,-0.13) circle (1pt) node[color=red!80!black, font=\tiny] {};
\fill[color=green!60!black] (0,0.67) circle (1pt) node[color=red!80!black, font=\tiny] {};
\fill[color=green!60!black] (0,-0.67) circle (1pt) node[color=red!80!black, font=\tiny] {};
\fill[color=green!60!black] (0,0.95) circle (1pt) node[color=red!80!black, font=\tiny] {};
\fill[color=green!60!black] (0,-0.95) circle (1pt) node[color=red!80!black, font=\tiny] {};
\draw[opacity=0.6] (0,-0.95) -- (0,-0.67);
\draw[opacity=0.6] (0,0.13) -- (0,-0.13);
\draw[opacity=0.6] (0,0.95) -- (0,0.67);
\fill (0.65,0) circle (0pt) node[color=black, font=\tiny] {$(\alpha,\psi)$};
\fill (-0.65,0) circle (0pt) node[color=black, font=\tiny] {$(\beta,\psi)$};
\fill (0.1,0.8) circle (0pt) node[color=black, font=\tiny] {$Z$};
\fill (0.1,0) circle (0pt) node[color=black, font=\tiny] {$Y$};
\fill (0.1,-0.8) circle (0pt) node[color=black, font=\tiny] {$X$};
\fill (0,-0.4) circle (0pt) node[color=black, font=\tiny] {$f$};
\fill (0,0.4) circle (0pt) node[color=black, font=\tiny, opacity=1] {$g$};
\end{tikzpicture}}
\ee
where the line-of-sight on the 3-ball is orthogonal to the two red 2-strata. 

To evaluate $\zz^\psi(B_{g,f})$ we note that the two 2-strata in \eqref{eq:Bgf} decorated with $(\alpha,\psi)$ and $(\beta,\psi)$ have two incoming and one outgoing boundary components.
Thus according to our convention in Example~\ref{example:Eulerx} they have symmetric Euler characteristic~$-1$, and it follows that 
\be
g \circ_\psi f = \big( \psi_\beta^{-1} \otimes 1_Z \otimes \psi_\alpha^{-1} \big) \circ g \circ f \, , 
\ee
where $g \circ f$ is the composition in~$\tz$. 
Similarly, we find that the identity 3-morphisms $1_X^{\mathcal T_{\zz^\psi}}$ in $\mathcal T_{\zz^\psi}$ are
\be
1_X^{\mathcal T_{\zz^\psi}}
= 
\zz^\psi 
\Bigg( 
\begin{tikzpicture}[scale=0.55,color=red!80!black, baseline]
\fill[ball color=red!50!white, opacity=0.4] (0,0) circle (1.95 cm);
\draw[ultra thick] (0,0) circle (1.95);
\coordinate (z2) at (90:1.95);
\fill[color=green!60!black] (z2) circle (4.4pt) node[above] {};
\coordinate (x2) at (-90:1.95);
\fill[color=green!60!black] (x2) circle (4.4pt) node[below] {};
%
\draw[color=green!60!black, very thick] 
	(x2) -- (z2);
\fill (0.3,0.7) circle (0pt) node[color=black, font=\tiny] {$X$};
\fill (1.15,0) circle (0pt) node[color=black, font=\tiny] {$(\alpha,\psi)$};
\fill (-1.15,0) circle (0pt) node[color=black, font=\tiny] {$(\beta,\psi)$};
\end{tikzpicture} 
\Bigg)
= 
\psi_\beta \otimes 1_X \otimes \psi_\alpha
\ee
as the $(\alpha,\psi)$- and $(\beta,\psi)$-decorated 2-strata now have symmetric Euler characteristic~$+1$. 

It follows that we obtain an equivalence~$\Gamma$ of Gray categories which is the identity on objects, 1- and 2-morphisms by setting $\Gamma(f) = (\psi_\beta^{-1} \otimes 1_Y \otimes \psi_\alpha^{-1}) \circ f$ 
for every 3-morphism $f\colon X \to Y$. 
The equivalence automatically respects the  $\#$-duals, and to verify that~$\Gamma$ also maps $\ev_X^{\mathcal T_{\zz^\psi}}$ to $\ev_X^{(\mathcal T_{\zz})^\psi}$, we compute 
\begin{align}
\ev_X^{\mathcal T_{\zz^\psi}} 
& 
= 
\zz^\psi
\Bigg(
\begin{tikzpicture}[scale=0.55,color=red!80!black, baseline]
\fill[ball color=red!50!white, opacity=0.4] (0,0) circle (1.95 cm);
\draw[ultra thick] (0,0) circle (1.95);
\coordinate (x2) at (-130:1.95);
\fill[color=green!60!black] (x2) circle (4.4pt) node[left] {};
\coordinate (xd2) at (-50:1.95);
\fill[color=green!60!black] (xd2) circle (4.4pt) node[right] {};
%
\draw[string=green!60!black, very thick] 
	(xd2) .. controls +(0,1.45) and +(0,1.45) .. (x2);
\fill (0.8,-0.3) circle (0pt) node[color=black, font=\tiny] {$X$};
\fill (0,1.15) circle (0pt) node[color=black, font=\tiny] {$(\alpha,\psi)$};
\fill (0,-1.15) circle (0pt) node[color=black, font=\tiny] {$(\beta,\psi)$};
\end{tikzpicture} 
\Bigg)
= 
\psi_\alpha \circ \ev_X^{\mathcal T_{\zz}} \circ \big( 1_{X^\dagger} \otimes \psi_\beta \otimes 1_X \big) 
\nonumber 
\\
& 
\stackrel{\Gamma}{\lmt} 
\psi_\alpha^{-1} \circ \ev_X^{\mathcal T_{\zz}} \circ \big( 1_{X^\dagger} \otimes \psi_\beta \otimes 1_X \big)
= 
\ev_X^{(\mathcal T_{\zz})^\psi}
\end{align}
where the last step is due to the definition in \eqref{eq:twistedcoevs}. 
The argument for the other three types of adjunction maps is analogous. 
\end{proof}

\begin{remark}
\label{rem:disclikes}
Here we give a sketch of  a construction of a disc-like $n$-category $\dz$ from a defect TQFT $\zz \colon\Bordd[n] \rightarrow \Vectk$. 

Recall that the main datum for a disc-like $n$-category is a map that assigns to a $k$-ball~$B$ the set of ``$k$-dimensional field configurations $C_k(B)$'' for every $k\in \{0,\ldots,n\}$. 
The elements of $C_k(B)$ are thought of as the $k$-morphisms of the disc-like $n$-category. 
This assignment has to be functorial with respect to diffeomorphisms of balls, giving a functor 
  \begin{equation}
    \label{eq:balls-}
    C_k\colon \Ball_k \lra \Set
  \end{equation}
from smooth $k$-balls and their diffeomorphisms to sets and bijections. 
For our defect TQFT $\zz \colon \Borddefn{n}(\D)\rightarrow \Vectk$ we define $C_k$ on a $k$-manifold~$X$, for $k\in \{0,\ldots,n-1\}$, as 
\begin{align}
C_{k}(X) &= 
\Big\{ 
X' \;\Big|\;
 \text{$X'$ a decorated stratified $k$-manifold with underlying manifold $X$} 
\nonumber \\[-.6em]
& \hspace{4.2em}
\text{ such that $[X'] \colon \emptyset \to \partial X'$ is a morphism in} \Borddefn{k}(\partial^{n-k+1}\D)
\Big\} \, .   
\end{align}
Restricting to balls, this gives the functor $C_{k}$ as in \eqref{eq:balls-}.

The combined source and target map for $k$-morphisms is specified by a natural transformation $\partial \colon C_k \Rightarrow C_{k-1}$ with components 
\begin{equation}
  \label{eq:sourctarg-sphere}
\partial_X
  \colon C_{k}(X) \lra C_{k-1}(\partial X) \, .
\end{equation}
In the case of $\dz$ we define $\partial_X(X')$ to be the decorated stratified manifold $\partial X'$ for every $X' \in C_{k}(X)$.

Disc-like $n$-categories have many compositions of morphisms modelled by the gluing of balls. 
If a $k$-ball $B=B_{1} \circ_{Y} B_{2}$ is obtained as the gluing of two $k$-balls $B_{1}$ and $B_{2}$ along a (collar over a) $(k-1)$-ball $Y$ with boundary $(k-2)$-sphere $S=\partial Y$, 
it is required that there exists a map 
\be
\label{eq:glY}
\text{gl}_{Y}\colon C(B_{1})_{S} \times_{C(Y)} C(B_{2})_{S} \lra C(B)_{S}
\ee
where the index $S$ denotes elements that are ``splittable along $S$''. 
This notion is intrinsically defined, but in our case it has a clear candidate:
In $\dz$ the spaces $C(B_{i})_{S}$ are those $k$-morphisms that intersect~$S$ transversely, and the map $\text{gl}_{Y}$ is just given by gluing the defects in $B_{1}$ and $B_{2}$ along $Y$. 
Thus gluing is strictly associative. 

The axioms related to the units of disc-like $n$-categories pose a technical issue: 
To model all compositions with various units, \cite{blobcomplex} work in the PL setting. 
In the case of $\dz$ the $k$-dimensional units should be cylinders over lower-dimensional defect balls with certain singularities on the boundary. 
In this remark however, we do not attempt to translate the unit axioms of \cite{blobcomplex} into our setting.

Up to now we have described the $k$-morphisms for $k<n$ without mentioning the functor~$\zz$. 
Indeed, if we took isotopy classes of $(n-1)$-balls we should obtain a disc-like $(n-1)$-category that is ``free over the defect data $\partial \D$'' (compare \cite[Sect.\,3]{CMS}). 
Instead we use~$\zz$ to define $n$-morphisms on an $n$-ball~$B$ as the coproduct over the state spaces associated to all $\D$-decorations of the boundary $\partial B$: 
\be
C_n(B) = 
\coprod_{[S] \in C_{n-1}(\partial B)} \zz(S) 
\, ,
\ee
where the index runs over all representatives of diffeomorphism classes of stratified decorated $(n-1)$-spheres. 

The composition~\eqref{eq:glY} of two $n$-balls $B_{1}$ and $B_{2}$ that are glued along an $(n-1)$-ball $Y$ to produce an $n$-ball $B=B_{1} \circ_{Y} B_{2}$ is given as follows. 
For each chosen decorated stratifications of~$B_1$ and~$B_2$ that are glued to a stratification of~$B$, we need to give a map $\text{gl}'_Y\colon \zz( \partial B_{1}) \times \zz( \partial B_{2}) \to \zz( \partial B)$. 
{}From the coproduct of the maps $\text{gl}'_{Y}$ 
we then obtain the gluing map  $\text{gl}_{Y}$.
To this end we cut out one slightly smaller ball from the inner of both $B_1$ and~$B_2$ inside~$B$. 
This produces a two-holed ball on which we can evaluate~$\zz$, giving the map $\text{gl}_{Y}'$. 
By diffeomorphisms invariance of $\zz$ we expect this to satisfy all axioms of a disc-like $n$-category.
\end{remark}

\end{document}